\documentclass[11pt]{amsart}
\usepackage{amsmath} \usepackage{amsxtra}
\usepackage{amscd} \usepackage{amsthm}  
\usepackage{amsfonts}\usepackage{amssymb} 
\usepackage{eucal}\usepackage{bm}
\usepackage{graphicx} 
\usepackage[latin1]{inputenc}
\usepackage{hyperref}

\newtheorem{Cor}[subsubsection]{Corollary}
\newtheorem{Lm}[subsubsection]{Lemma}
\newtheorem{Pp}[subsubsection]{Proposition}

\newtheorem{Thm}[subsubsection]{Theorem}
\newtheorem{Def}[subsubsection]{Definition}
\newtheorem{Rem}[subsubsection]{Remark}

\theoremstyle{definition}

\theoremstyle{remark}
\newcommand{\nc}{\newcommand}
\nc{\renc}{\renewcommand}
\nc{\ssec}{\subsection}
\nc{\sssec}{\subsubsection}
\nc{\on}{\operatorname}

\newcommand{\cA}{{\mathcal A}}
\newcommand{\cB}{{\mathcal B}}

\newcommand{\cD}{{\mathcal D}}
\newcommand{\cH}{{\mathcal H}}
\newcommand{\cE}{{\mathcal E}}
\newcommand{\cG}{{\mathcal G}}
\newcommand{\cI}{{\mathcal I}}

\newcommand{\cO}{{\mathcal O}}
\newcommand{\cL}{{\mathcal L}}
\newcommand{\cM}{{\mathcal M}}

\newcommand{\cF}{{\mathcal F}}
\newcommand{\cK}{{\mathcal K}}

\newcommand{\cR}{{\mathcal R}}
\newcommand{\cS}{{\mathcal S}}

\newcommand{\cV}{{\mathcal V}}
\newcommand{\cW}{{\mathcal W}}

\newcommand{\cY}{{\mathcal Y}}
\newcommand{\cZ}{{\mathcal Z}}

\newcommand{\NN}{{\mathbb N}}
\newcommand{\ZZ}{{\mathbb Z}}

\newcommand{\PP}{{\mathbb P}}

\newcommand{\ga}{\mathfrak{a}}
 
\renewcommand{\gg}{\mathfrak{g}}  
\newcommand{\gm}{\mathfrak{m}}

\newcommand{\gq}{\mathfrak{q}}
\newcommand{\gu}{\mathfrak{u}}

\newcommand{\gt}{\mathfrak{t}}
\newcommand{\gr}{\mathfrak{r}}
\newcommand{\gs}{\mathfrak{s}}

\nc{\gJ}{\mathfrak{J}}
\nc{\gL}{\mathfrak{L}}
\nc{\gv}{\mathfrak{v}}
\nc{\gF}{\mathfrak{F}}

\newcommand{\Fl}{{\mathcal F}l}

\newcommand{\Rep}{{\on{Rep}}}
\newcommand{\Sch}{{\on{Sch}}}

\newcommand{\toup}[1]{\stackrel{#1}{\to}}
\newcommand{\hook}[1]{\stackrel{#1}{\hookrightarrow}}

\newcommand{\getsup}[1]{\stackrel{#1}{\gets}}

\newcommand{\IC}{\on{IC}}

\newcommand{\Hom}{\on{Hom}}
\newcommand{\Mod}{\on{Mod}}

\newcommand{\End}{\on{End}}

\newcommand{\Ker}{\on{Ker}}

\newcommand{\Aut}{\on{Aut}}

\newcommand{\RG}{\on{R\Gamma}}

\newcommand{\triv}{\on{triv}}

\newcommand{\Bun}{\on{Bun}}

\newcommand{\Bunb}{\on{\overline{Bun}} }
\newcommand{\Bunt}{\on{\widetilde\Bun}}

\newcommand{\Spec}{\on{Spec}}

\newcommand{\supp}{\on{supp}}

\newcommand{\HOM}{{{\mathcal H}om}}

\newcommand{\Gr}{\on{Gr}}

\newcommand{\GL}{\on{GL}}

\newcommand{\Fun}{{\on{Fun}}}

\newcommand{\pr}{\on{pr}}
\newcommand{\id}{\on{id}}
\newcommand{\QED}{$\square$} 
  
\newcommand{\iso}{{\widetilde\to}}

\newcommand{\comp}{\circ}

\renewcommand{\H}{{\on{H}}}   %cohomologies
\newcommand{\DD}{\mathbb{D}}  %for duality
\newcommand{\wt}{\widetilde}
\newcommand{\ov}[1]{\overline{#1}}
\newcommand{\select}[1]{{\it{#1}}}

\newcommand{\Hecke}{\on{Hecke}}
\newcommand{\<}{\langle}
\renewcommand{\>}{\rangle}

\newcommand{\Av}{\on{Av}}
\newcommand{\ev}{\mathit{ev}}

\newcommand{\Conv}{\on{Conv}}

\newcommand{\Lie}{\on{Lie}}
\newcommand{\Sph}{\on{Sph}}
\newcommand{\Res}{\on{Res}}

\newcommand{\ttimes}{\tilde\times}

\newcommand{\act}{\on{act}}
\newcommand{\dimrel}{\on{dim.rel}}
\newcommand{\codim}{\on{codim}}

\newcommand{\tboxtimes}{\,\tilde\boxtimes\,}

\newcommand{\Vect}{\on{Vect}}

\newcommand{\glob}{\on{glob}}
\newcommand{\Ind}{\on{Ind}}

\newcommand{\ra}{\rightarrow}
\newcommand{\la}{\leftarrow}

\nc{\Hi}{\on{Hi}}
\nc{\Lo}{\on{Lo}}
\nc{\Inv}{\on{Inv}}
\nc{\Fix}{\on{Fix}}
\nc{\oV}{\overset{\scriptscriptstyle\circ}{V}}
\nc{\ssharp}{{\scriptstyle\sharp}}
\nc{\oHecke}{\overset{\scriptstyle\bullet}{\Hecke}}
\nc{\Ran}{\on{Ran}}
\nc{\Whit}{\on{Whit}}
\nc{\Maps}{\on{{\mathcal M}aps}}
\nc{\DGCat}{\on{DGCat}}
\nc{\colim}{\on{colim}}
\nc{\PreStk}{\on{PreStk}}
\nc{\otimesshriek}{\stackrel{!}{\otimes}}
\nc{\Cat}{\on{Cat}}
\nc{\QCoh}{\on{QCoh}}
\nc{\FactMod}{\on{FactMod}}
\nc{\oo}[1]{\overset{\scriptscriptstyle\circ}{#1}}
\nc{\inv}{\on{inv}}
\nc{\IndCoh}{\on{IndCoh}}
\nc{\Groth}{\on{Groth}}
\nc{\coGroth}{\on{coGroth}}
\nc{\Conf}{\on{Conf}}
\nc{\Map}{\on{Map}}
\nc{\Tr}{\on{Tr}}
\nc{\oblv}{\on{oblv}}
\nc{\Fact}{\on{Fact}}
\nc{\Spc}{\on{Spc}}
\nc{\Stk}{\on{Stk}}
\nc{\ee}{\mathrm{e}}
\nc{\bfitDelta}{\bm{\mathit{\Delta}}}
\nc{\Grpd}{\on{Grpd}}
\nc{\Corr}{\on{Corr}}
\nc{\Perv}{\on{Perv}}
\nc{\Sat}{\on{Sat}}

\nc{\coind}{\on{coind}}
\nc{\Ad}{\on{Ad}}
\nc{\DG}{\on{DG}}
\nc{\gRes}{\on{gRes}}
\nc{\bvartriangle}{\boldsymbol{\vartriangle}}
\nc{\SI}{\on{SI}}
\nc{\VD}{\on{VD}}
\nc{\ren}{\on{ren}}
\nc{\unren}{\on{un-ren}}
\nc{\astQ}{\overset{Q}{\ast}}
\nc{\astP}{\overset{P}{\ast}}
\nc{\astB}{\overset{B}{\ast}}
\nc{\astI}{\overset{I}{\ast}}
\nc{\astIP}{\overset{I_P}{\ast}}
\nc{\astPminus}{\overset{P^-}{\ast}}
\nc{\astBminus}{\overset{B^-}{\ast}}
\nc{\astIminus}{\overset{I^-}{\ast}}
\nc{\astIPminus}{\overset{I_{P^-}}{\ast}}
\nc{\Nilp}{\on{Nilp}}
\nc{\Cent}{\on{Cent}}
\nc{\Assoc}{\on{Assoc}}
\nc{\cFin}{\on{\cF in}}
\nc{\Mul}{\on{Mul}}

\newcommand{\Step}[1]{\par\noindent{\bf Step {#1}}.}

\makeatletter
\newcommand*{\doublerightarrow}[2]{\mathrel{
  \settowidth{\@tempdima}{$\scriptstyle#1$}
  \settowidth{\@tempdimb}{$\scriptstyle#2$}
  \ifdim\@tempdimb>\@tempdima \@tempdima=\@tempdimb\fi
  \mathop{\vcenter{
    \offinterlineskip\ialign{\hbox to\dimexpr\@tempdima+1em{##}\cr
    \rightarrowfill\cr\noalign{\kern.5ex}
    \rightarrowfill\cr}}}\limits^{\!#1}_{\!#2}}}
\newcommand*{\triplerightarrow}[1]{\mathrel{
  \settowidth{\@tempdima}{$\scriptstyle#1$}
  \mathop{\vcenter{
    \offinterlineskip\ialign{\hbox to\dimexpr\@tempdima+1em{##}\cr
    \rightarrowfill\cr\noalign{\kern.5ex}
    \rightarrowfill\cr\noalign{\kern.5ex}
    \rightarrowfill\cr}}}\limits^{\!#1}}}
\makeatother

\begin{document}
\author{G. Dhillon}
\address{Department of Mathematics, University of California Los Angeles, 
460 Portola Plaza,  Los Angeles, CA, USA 90095}
\email{gsd@math.ucla.edu}
\author{S. Lysenko}
\address{Institut Elie Cartan Lorraine, Universit\'e de Lorraine, 
B.P. 239, F-54506 Vandoeuvre-l\`es-Nancy Cedex, France}
\email{Sergey.Lysenko@univ-lorraine.fr}
\thanks{We are grateful to Misha Finkelberg for regular fruitful discussions and to Sam Raskin for answering the second author's questions. We also thank Roma Bezrukavnikov and Dennis Gaitsgory for helpful correspondence. G.D. was supported by an NSF Postdoctoral Fellowship under grant No. 2103387.}
\title{Semi-infinite parabolic $\IC$-sheaf}
\begin{abstract}
Let $G$ be a connected reductive group, $P$ its parabolic subgroup. We consider the parabolic semi-infinite category of sheaves on the affine Grassmannian of $G$, and construct the parabolic version of the semi-infinite $\IC$-sheaf of each orbit. We establish some of its properties and relate it to sheaves on the Drinfeld compactification $\Bunt_P$ of the moduli stack $\Bun_P$ of $P$-torsors on a curve. We relate the parabolic semi-infinite $\IC$-sheaf with the dual baby Verma object on the spectral side. We also obtain new results on invertibility of some standard objects in parabolic Hecke categories.
\end{abstract}
\maketitle
% keywords: geometric Langlands program, semi-infinite intersection cohomology sheaf

\tableofcontents

\section{Introduction}

\ssec{Some context} 
\sssec{} Let $G$ be a split reductive group. In the usual Langlands correspondence, a basic  tool used in the construction of interesting automorphic forms on $G$, e.g. of constituents of the discrete spectrum, is the operation of taking residues of Eisenstein series. This paper is the first in a series whose aim is to set up an analogous technique in the geometric Langlands program, and study its consequences, e.g., in the geometric representation theory of certain vertex operator algebras.

\sssec{} We recall that the Eisenstein series used in residue constructions are often maximally degenerate, i.e., are defined with respect to a maximal proper parabolic subgroup $P = M \cdot U\footnote{In the main body of the text, we will write instead $U = U(P)$ for the unipotent radical of $P$.}$, and that their analysis involves a local study of poles of standard intertwining operators between the corresponding degenerate principal series representations. In particular, at almost every place $v$, one is in the unramified situation, and encounters spaces of functions of the form 
$$\on{Fun}( M(O_v)\cdot  U(F_v) \backslash G(F_v) / G(O_v)),$$
where $O_v$ and $F_v$ denote the corresponding completed ring of integers and its fraction field, respectively. 

\sssec{} To get started with the geometric theory of residues, for an algebraically closed field $k$, and a parabolic $P \subset G$ defined over $k$, if we write $\cO = k[[t]]$ and $F = k(\!(t)\!)$, we will therefore need certain basic properties of the corresponding (derived) category of sheaves 
$$Shv( M(\cO) \cdot U(F) \backslash G(F) /G(O)),$$ particularly a full subcategory of perverse sheaves therein, and certain local-global compatibilities. 

At one extreme, for $P = G$, this category reduces to the familiar derived Satake category. At the other extreme, for a Borel $P = B$, the corresponding category of sheaves, despite figuring in many influential ideas and conjectures of Feigin--Frenkel and Lusztig from nearly forty years ago, proved elusive due to the highly infinite-dimensional nature of the relevant geometry \cite{FeiFr, LusICM}. A remarkable substitute theory via finite-dimensional global models was developed by Finkelberg--Mirkovi\'c \cite{FM}, and finally a suitable direct definition was found recently by Gaitsgory \cite{Gai19SI}.\footnote{For the convenience of the reader, we also collect some potential errata for  \cite{Gai19SI} in Appendix~\ref{Section_corrections_for_Gai19SI}; specializing the present paper to $P = B$ in particular resolves the issues mentioned there.} 

\subsubsection{} For intermediate parabolics $P$, which we need for a good theory of residues, the analogous theory, particularly the analysis of perverse $t$-structures, and the study of certain basic perverse sheaves therein, the semi-infinite intersection cohomology sheaves, was not yet developed.

However, the existence and properties of such a category of perverse sheaves again has been anticipated elsewhere in representation theory. Notably, the combinatorics of the simple and standard perverse sheaves in this category, and its analogue on the affine flag variety, were described by Lusztig via his  periodic $W$-graphs \cite{LusP}. Relatedly, in recent work Bezrukavnikov and Losev have suggested equivalences between such categories with positive characteristic coefficients and certain blocks of representations in modular representation theory \cite{BezLos}; analogues of their expectations for quantum groups at roots of unity will be formulated among the conjectures in Section \ref{sss:conjectures} below.

Our main goal in the present paper is to construct and study such a category of perverse sheaves. 

We highlight, however, a construction of the Ran version of the semi-infinite parabolic $\IC$-sheaf $\IC^{\frac{\infty}{2}}_{P,\Ran}$ given in (\cite{H}, Section~9) via global methods of factorization geometry (they first define semi-infinite parabolic $\IC$-sheaves of Zastava spaces over configuration of divisors $\Conf$ and then apply some general machinery allowing to pass from unital factorizable objects over $\Conf$ to unital factorizable objects over $\Ran$). The fibre of their sheaf $\IC^{\frac{\infty}{2}}_{P,\Ran}$ at $x\in \Ran$ is the sheaf we construct in our paper by purely local methods. A different construction of $\IC^{\frac{\infty}{2}}_{P,\Ran}$ appears in \cite{DL2}, which is a continuation of the present paper. 

 In addition, we should highlight that Hayash-Faergeman have posted \cite{FH} to the arxiv, which introduces another construction of $\IC^{\frac{\infty}{2}}_{P,\Ran}$ in terms of its Hecke structures. (We mention in passing that the authors of loc. cit. have informed us they plan to release a new version, which, in particular, addresses certain mistakes). 
 
 \subsubsection{} Before turning to a precise description of our results, we should acknowledge right away that at a high level, up to standard complications which arise when dealing with a general parabolic, the techniques we use are largely similar to those employed in previous analyses of the Satake category and the case of the Borel. In particular, we view the work of Gaitsgory \cite{Gai19SI} as a genuine breakthrough in this subject, and the technical advances therein are what make the present analysis possible. However, carrying out this analysis still requires a considerable amount of work, as can be seen e.g. in the length of the present paper.

\ssec{What is done in this paper?} 

 Let us describe our results and compare our situation with the Borel case.

\sssec{} % In the setting of $\cD$-modules the semi-inifinite parabolic category of sheaves has been considered in \cite{LC2, LC, BL}. In this paper we work in the constructible context.

 Recall that $\cO=k[[t]]\subset F=k((t))$. Let $U(P)$ be the unipotent radical of $P$. Set $H=M(\cO)U(P)(F)$. The parabolic semi-infinite category of sheaves is defined as
$$
\SI_P=Shv(\Gr_G)^H
$$
(cf. Section~\ref{Sect_Local automorphic side_begins} for details). 
The $H$-orbits on $\Gr_G$ are indexed by the set $\Lambda^+_M$ of dominant coweights for $M$. Namely, to $\lambda\in\Lambda_M^+$ there corresponds the $H$-orbit $S_P^{\lambda}$ through $t^{\lambda}G(\cO)$.  As in the Borel case, the $H$-orbits on $\Gr_G$ (unless $G=P$) are infinite-dimensional and have infinite codimension. For this reason one needs to work on the level of $\DG$-categories to develop this theory.

 If $\lambda,\nu\in\Lambda^+_M$ then $S_P^{\nu}$ lies in the closure $\bar S_P^{\lambda}$ of $S_P^{\lambda}$ iff $\lambda-\nu\in\Lambda^{pos}$. Here $\Lambda^{pos}$ is the $\ZZ_+$-span of simple positive coroots. While the definition and some first properties of $\SI_P$ were given in \cite{LC2, LC, BL}, we introduce the semi-infinite t-structure on $\SI_P$ and the objects $\IC^{\frac{\infty}{2}}_{P,\lambda}\in \SI_P$ playing the role of the semi-infinite $\IC$-sheaves of $\bar S_P^{\lambda}$. These are our main objects of study. We relate them to the globally defined objects (for a given smooth curve) as well as to the dual baby Verma objects on the spectral side.

 \iffalse 
  Our results extend (from the case of a Borel subgroup to any parabolic $P$ of $G$) the results of Gaitsgory from \cite{Gai19SI}\footnote{ We also collect some mistakes from \cite{Gai19SI} in Appendix~\ref{Section_corrections_for_Gai19SI} and correct them in the present paper.}\fi

\sssec{} We introduce the semi-infinite parabolic category $Shv(\Gr_G)^H$ and its renormalized version $Shv(\Gr_G)^{H, ren}$. We define actions of $\Sph_G=Shv(\Gr_G)^{G(\cO)}$ on these categories, and an action of $\Sph_M=Shv(\Gr_M)^{M(\cO)}$ on $Shv(\Gr_G)^H$. 
We define a natural semi-infinite t-structure on $\SI_P=Shv(\Gr_G)^H$ and show that the action of $\Rep(\check{G})^{\heartsuit}$ and some shifted\footnote{in the sense of (\ref{action_Rep(checkM)_shifted}).} action of $\Rep(\check{M})^{\heartsuit}$ on 
$\SI_P$ are t-exact. We also relate $\SI_P$ to the categories $Shv(\Gr_P)^H$ and $Shv(\Gr_M)^{M(\cO)}$, this is analogous to (\cite{BL}, Section~3.1).  

% For $\lambda\in\Lambda^+_M$ we introduce the standard and costandard objects $\bvartriangle^{\lambda}, \nabla^{\lambda}\in \SI_P$ in Section~\ref{Sect_2.3.9_object_cB}. 
           
\sssec{} Let $T\subset B$ be a maximal torus. Write $\Lambda_{M, ab}$ for the set of coweights of $T$ orthogonal to the roots of $(M, T)$. Let $\Lambda_{M, ab}^+\subset \Lambda_{M, ab}$ be the subset of those coweights, which are dominant for $G$. Write $\check{M}$ (resp., $\check{G}$) for the Langlands dual group of $M$ (resp., of $G$). Set $\check{M}_{ab}=\check{M}/[\check{M}, \check{M}]$. Let $\check{P}\subset \check{G}$ be the parabolic subgroup dual to $P$.

\sssec{} As in the Borel case, for $\lambda\in\Lambda^+_M$ we give a local definition of $\IC^{\frac{\infty}{2}}_{P,\lambda}$ by some colimit formula. It differs from the Borel case in two aspects. First, we propose such a definition for any $H$-orbit on $\Gr_G$. Second, the index category over which the colimit is taken changes, instead of being the category of dominant coweights $\Lambda^+$ of $(G, T)$, now it is the category of $\mu\in\Lambda^+_{M, ab}$ such that $\lambda+\mu\in\Lambda^+$. 

\sssec{Drinfeld-Pl\"ucker formalism} In the case $\lambda=0$ we give a conceptual explanation of this formula. Namely, we propose an analog of the Drinfeld-Pl\"ucker formalism (that only applies for $\lambda=0$ for the moment) in the parabolic setting. There are two versions of this formalism corresponding, for historical reasons, to $\Bunt_P$ and $\Bunb_P$ respectively (these are some relative compactifications of the stack $\Bun_P$ of $P$-torsors on a smooth projective curve, cf. Section~\ref{Sect_Relation between local and global}). It is crucial for this formalism that both $\check{G}/[\check{P},\check{P}]$ and $\check{G}/U(\check{P})$ are quasi-affine, here $U(\check{P})$ is the unipotent radical of $\check{P}$. Denote by $\overline{\check{G}/[\check{P},\check{P}]}$ and $\overline{\check{G}/U(\check{P})}$ the corresponding affine closures.
 
 This formalism allows to produce analogs of the dual baby Verma objects in some situations, when we are given a $\DG$-category $C$ equipped with an action of $\Rep(\check{M})\otimes\Rep(\check{G})$ (resp., of $\Rep(\check{M}_{ab})\otimes\Rep(\check{G})$) for the $\Bunt_P$-version (resp., for $\Bunb_P$-version). 
 
 The corresponding colimit for the $\Bunb_P$-version describes the composition
\begin{multline*}
 \cO(\check{G}/[\check{P}, \check{P}])-mod(C)\to C\otimes_{\Rep(\check{M}_{ab})\otimes\Rep(\check{G})} \Rep(\check{P})\\  \to C\otimes_{\Rep(\check{M}_{ab})\otimes\Rep(\check{G})}\Rep(\check{M})\toup{\oblv} C,
\end{multline*}
where the unnamed arrows are given by pullbacks under natural maps 
$$
B(\check{M})\to B(\check{P})\hook{} \check{G}\backslash \overline{\check{G}/[\check{P},\check{P}]}/\check{M}_{ab},
$$ 
the second map being an open immersion. For the $\Bunt_P$-version it describes the similar composition
\begin{multline*}
 \cO(\check{G}/U(\check{P}))-mod(C)\to
C\otimes_{\Rep(\check{M})\otimes\Rep(\check{G})} \Rep(\check{P})\\  \to C\otimes_{\Rep(\check{M})\otimes\Rep(\check{G})}\Rep(\check{M})\toup{\oblv} C,
\end{multline*}  
where the unnamed arrows are given by pullbacks under natural maps 
$$
B(\check{M})\to B(\check{P})\hook{} \check{G}\backslash \overline{\check{G}/U(\check{P})}/\check{M},
$$
here the second map is an open immersion. A relation between the two versions of the Drinfeld-Pl\"ucker formalism is given by the second proof of our Proposition~\ref{Pp_2.1.10_colimit_for_Bunt_P}  (cf. Section~\ref{Sect_Second proof of Pp2.1.10}). 

We think of the dual baby Verma object as an object of $C$ together with some version of a Hecke property, that is, lifted to an object of $C\otimes_{\Rep(\check{M}_{ab})\otimes\Rep(\check{G})}\Rep(\check{M})$ or $C\otimes_{\Rep(\check{M})\otimes\Rep(\check{G})}\Rep(\check{M})$ respectively.

 Our Drinfeld-Pl\"ucker formalism is further generalized in an essential way in (\cite{Ly10}, Section~6). 
 
\sssec{} We then introduce a version of the dual baby Verma object $\cM_{\check{G}, \check{P}^-}$ for 
$$
C=\IndCoh((\check{\gu}(P^-)\times_{\check{\gg}} 0)/\check{P}^-)
$$
(as well as its version $\cM_{\check{G}, \check{P}}$ with the roles of $P, B$ and $P^-, B^-$ exchanged). Here $\check{\gg}=\Lie \check{G}$ and $\check{\gu}(P^-)$ is the Lie algebra of the unipotent radical of $\check{P}^-$. One of our main results is Theorem~\ref{Thm_4.5.10} giving a precise relation between 
 $\IC^{\frac{\infty}{2}}_{P,0}$ and $\cM_{\check{G}, \check{P}}$ via an equivalence proven by G. Dhillon and H.~Chen (cf. Proposition~\ref{Pp_Chen_Dhillon}) composed with our equivalence (\ref{eq_ren_parahoric_versus_H}) and the so called \select{long intertwining operator}, cf. Section~\ref{Sect_4.5}. This also gives some new insight in the structure of the parahoric Hecke $\DG$-categories (cf. Proposition~\ref{Pp_key_for_baby_Verma_transform_for_P}). 
 
\sssec{} As an aside for our proof of Theorem~\ref{Thm_4.5.10}, we obtain a new result about the intertwining functors between two distinct parabolic Hecke categories. Namely, assume given a $\DG$-category $C$ with an action of $Shv(G)$ and two parabolics $P,Q\subset G$. We determine all the pairs $(P, Q)$ for which the composition
$$
C^P \;\toup{\oblv}\; C^{P\cap Q}\;\toup{\Av^{Q/(P\cap Q)}_*}\; C^Q
$$ 
is an equivalence (cf. Proposition~\ref{Pp_4.5.3} and Theorem~\ref{Thm_B.1.2}). Here $\Av^{Q/(P\cap Q)}_*$ is the right adjoint to the corresponding oblivion functor.
 
\sssec{} We prove a full Hecke property of $\IC^{\frac{\infty}{2}}_{P,0}$. Namely, our Proposition~\ref{Pp_2.5.18_upgrading_IC_semi-infinite_P} asserts that $\IC^{\frac{\infty}{2}}_{P,0}$ naturally upgrades to an object of
$$
\SI_P\otimes_{\Rep(\check{G})\otimes\Rep(\check{M})} \Rep(\check{M})
$$
This is conceptually explained by the Drinfeld-Pl\"ucker formalism in its $\Bunt_P$-version.

\sssec{} We show that for $\lambda\in\Lambda^+_M$, $\IC^{\frac{\infty}{2}}_{P,\lambda}$ lies in the heart $\SI_P^{\heartsuit}$ of the t-structure on $\SI_P$. Write $\Lambda_{G,P}$ for the lattice of cocharacters of $M/[M,M]$. For $\theta\in\Lambda_{G,P}$ we consider the usual diagram of affine Grassmannians
$$
\Gr_M^{\theta}\getsup{\gt^{\theta}_P}\Gr_P^{\theta}\toup{v^{\theta}_P} \Gr_G
$$
giving rise to the geometric restriction functor 
$$
(\gt_P^{\theta})_!(v^{\theta}_P)^*: Shv(\Gr_G)^{M(\cO)}\to Shv(\Gr_M^{\theta})^{M(\cO)}
$$ 
Recall that $\SI_P\subset Shv(\Gr_G)^{M(\cO)}$ is a full subcategory naturally. In Proposition~\ref{Pp_2.5.15_*-restriction} for $\eta\in\Lambda^+_M$ we express $(\gt_P^{\theta})_!(v^{\theta}_P)^*\IC^{\frac{\infty}{2}}_{P, \eta}$ in terms of the Satake equivalence for $M$. This answer is to be compared with the main result of \cite{BFGM}. The advantage is that our isomorphism is canonical, while the related isomorphisms from \cite{BFGM} are not. The importance of this result comes from the relation with the $\IC$-sheaf of $\Bunt_P$, which plays a crucial role in many aspects of the geometric Langlands program.  

 In Proposition~\ref{Pp_action_of_Rep(checkM)_on_SI-IC_P} we show that if $\eta\in\Lambda^+_M$ then $\IC^{\frac{\infty}{2}}_{P, \eta}$ is obtained from $\IC^{\frac{\infty}{2}}_{P, 0}$ by applying a Hecke functor for $M$
corresponding to the irreducible $\check{M}$-module with highest weight $\eta$.
 
\sssec{Relation to global objects} Assume in addition that $[G,G]$ is simply-connected\footnote{This is done to simplify the definition of the Drinfeld compactification of $\Bun_P$}. Pick a smooth projective connected curve $X$ over $k$. Write $\Bun_P$ for the stack of $P$-torsors on $X$, $\Bunt_P$ for its Drinfeld compactification. Pick a closed point $x\in X$.
We introduce a version $_{x,\infty}\Bunt_P$ of $\Bun_P$, cf. Section~\ref{Sect_2.3.7_local_vs_global}. Let $\cY_x$ denote the stack classifying collections: a $G$-torsor $\cF_G$ on $X$, a $M$-torsor $\cF_M$ on $X$ and an isomorphism $\cF_M\times_M G\,\iso\,\cF_G\mid_{X-x}$. We introduce a diagram
$$
M(\cO)\backslash \Gr_G\getsup{\pi_{loc}} \cY_x \toup{\pi_{glob}} {_{x,\infty}\Bunt_P},
$$
in Sections~~\ref{Sect_cY_x_definition}-\ref{Sect_pi_loc_def}, where $M(\cO)\backslash \Gr_G$ denotes the stack quotient of $\Gr_G$. Write $\IC_{\wt{glob}}$ for the $\IC$-sheaf of $\Bunt_P$. Let $j_{glob}: \Bun_P\hook{}\Bunt_P$ be the natural open immersion. 

 Section~\ref{Sect_2.6_Relation} is devoted to establishing precise relations between the semi-infinite parabolic category and the corresponding global objects. Our main results in this direction are Theorems~\ref{Thm_restriction_of_glob_first}, \ref{Th_restriction_of_glob_second} and Proposition~\ref{Pp_4.2.3_now}. They extend the corresponding results of \cite{Gai19SI} from the Borel case. Namely, we show that for $\eta\in\Lambda_M^+$, $\pi_{loc}^!\IC^{\frac{\infty}{2}}_{P,\eta}$ idenitifies canonically with $\pi_{glob}^!\IC^{\eta}_{\wt{glob}}$ up to a cohomological shift, and $\pi_{loc}^!\bvartriangle^0$ idenitifies canonically with $\pi_{glob}^!(j_{glob})_!\IC_{\Bun_P}$ up to a cohomological shift. Here for $\eta\in\Lambda^+_M$ we define $\IC^{\eta}_{\wt{glob}}$ as the $\IC$-sheaf of the closed substack 
$$
_{x, \ge -w_0^M(\eta)}\Bunt_P\hook{} {_{x,\infty}\Bunt_P},
$$ 
cf. Sections~\ref{Sect_2.3.7_local_vs_global}, \ref{Sect_2.3.12_local_vs_global}. 

% Our proof of Theorems~\ref{Thm_restriction_of_glob_first} and \ref{Th_restriction_of_glob_second} is inspired by the argument from \cite{Gai19SI}.

\sssec{} Finally, we establish Theorem~\ref{Thm_4.1.10} saying that for $\lambda\in\Lambda_{M, ab}$, the standard object $\bvartriangle^{\lambda}\in Shv(\Gr_G)^H$ lies in the heart $Shv(\Gr_G)^{H, \heartsuit}$ of the semi-infinite t-structure. In the case $P=B$ this claim was derived in \cite{Gai19SI} from a global result going back to \cite{BG2}. We give two purely local proofs relating the question to the parahoric-equivariant sheaves on $\Gr_G$ and to the t-structure on the spectral side.

\sssec{} In Appendix~\ref{Sect_lax-appendix} we give a formal definition of the category of lax central objects for an $\infty$-category equipped with left and right lax actions of a monoidal $\infty$-category. In (\cite{Gai19SI}, 2.7) this notion has been introduced at informal level. It is essentially used in this paper.
  
\ssec{Applications} 
\label{sss:conjectures}
\sssec{} All the applications of the semi-infinite $\IC$-sheaves foreseen in \cite{Gai19SI} are valid also in the parabolic case. We underline that, as in the Borel case, the semi-infinite parabolic $\IC$-sheaf admits a factorizable version considered in \cite{DL2}, which is more fundamental than our $\IC^{\frac{\infty}{2}}_{P, 0}$. 

 A metaplectic analog of the semi-infinite sheaf $\IC^{\frac{\infty}{2}}_{P, 0}$ is not difficult to define. It plays an important role in the metaplectic geometric Langlands program suggested in \cite{GL}. In particular, it is one of the key ingredients in the forthcoming proof of the factorizable version of the fundamental local equivalence in the metaplectic setting \cite{GHL}. 
 
\sssec{} One of our main motivation to study $\IC^{\frac{\infty}{2}}_{P, 0}$ was also its use for  constructing geometric analogs of residues of the geometric Eisenstein series. This paper is a part of that project.

\sssec{} We would like to propose, in addition, some relations we expect to representations of quantum groups, which we will return to elsewhere. As setup, consider the equivalence
$$
Shv(\Gr_G)^{H, ren} \simeq \on{IndCoh}((\check{\mathfrak{u}}(\check{P}^-) \times_{\check{\mathfrak{g}}} 0)/\check{P}^-),
$$
which is a combination of (\ref{iso_Gurbir_Chen}) and (\ref{eq_ren_parahoric_versus_H}). 
We expect this is $t$-exact, and hence restricts to an equivalence of abelian categories
$$
Shv(\Gr_G)^{H, ren, \heartsuit} \simeq \Rep(\check{P}^-)^{\heartsuit}.
$$
For $P = G$ this is the geometric Satake, and for $P= B$ this follows from
(\cite{Gai19SI}, 1.5.7). 

\sssec{}
To make contact with quantum groups, let us write $\mathring{I}$ for the prounipotent radical of the Iwahori subgroup, and correspondingly pass from the affine Grassmannian to the enhanced affine flag variety
$$
\widetilde{\Fl}_G := G(F)/\mathring{I}.
$$
The semi-infinite category of sheaves 
$Shv(\widetilde{\Fl}_G)^H$ again admits a $t$-structure similar to what is constructed in the present work for the affine Grassmannian.

\sssec{}To describe the corresponding category of quantum group representations, fix any sufficiently large even root of unity $q$. Associated to our parabolic $P$ is a {mixed} quantum group $$\mathfrak{U}_q(\mathfrak{g}, P),$$which has divided powers for the simple raising and lowering operators lying in $P$, and no divided powers for the remaining simple lowering operators. For $P = G$, this is the Lusztig form of the quantum group, and for $P = B$ this was considered by Gaitsgory in \cite{GaiKL}.

It has a renormalized derived category of representations 
$$\on{Rep}_q(\mathfrak{g}, P),$$obtained by ind-completing the pre-triangulated envelope of the parabolic Verma modules in its naive derived category, and we denote its principal block by 
$$\on{Rep}_q(\mathfrak{g}, P)_\circ \hookrightarrow \on{Rep}_q(\mathfrak{g}, P).$$

\sssec{} We conjecture a $t$-exact equivalence 
\begin{equation}
Shv(\widetilde{\Fl}_G)^H \simeq \on{Rep}_q(\mathfrak{g}, P)_\circ.
\end{equation}
This should match the natural structures of highest weight categories on the hearts, and moreover intertwine the pullback$$Shv(\Gr_G)^{H, \heartsuit} \rightarrow Shv(\widetilde{\Fl}_G)^{H, \heartsuit}$$with a suitably defined quantum Frobenius map
$$
\on{Fr}: \on{Rep}(\check{P}^-)^\heartsuit \rightarrow \on{Rep}_q(\mathfrak{g}, P)_\circ^\heartsuit.
$$
For $P=G$, this is a result of Arkhipov--Bezrukavnikov--Ginzburg \cite{ABG}. For all other cases, to the best of our knowledge this may be new.\footnote{In the case of $P = B$, this may be proven, along with its variants for singular blocks, using an equivalence between representations of the mixed quantum group and affine Category $\mathcal{O}$ conjectured by Gaitsgory in \cite{GaiKL} and proven by Chen--Fu \cite{ChenFu}, combined with Kashiwara--Tanisaki localization. We were informed by Losev that he has also obtained a proof by quite different means.} 

\sssec{} In addition, to make contact with small quantum groups, as originally envisioned by Feigin-Frenkel and Lusztig \cite{FeiFr}, \cite{LusICM}, we may proceed as follows.

Our $Shv(\Gr_G)^H$ is naturally a factorization category, and $\on{IC}^{\frac{\infty}{2}}_{P, 0}$ is upgraded to a factorization algebra in this factorization category. For the factorization modules over the semi-infinite $\IC$ sheaf, we conjecture a $t$-exact equivalence of $\DG$-categories
$$
\on{IC}^{\frac{\infty}{2}}_{P, 0}-mod^{{fact}}(Shv(\Gr_G)^{H, ren}) \simeq \on{IndCoh}( (0 \underset{\check{\mathfrak{g}}}{\times} 0)/\check{M}),
$$
which in particular induces an equivalence of abelian categories
$$
\IC^{\frac{\infty}{2}}_{P, 0}-mod^{{fact}}(Shv(\Gr_G)^{H, ren})^{\heartsuit} \simeq \on{Rep}(\check{M})^{\heartsuit};
$$
for $P = B$ this is a forthcoming theorem of Campbell.

\sssec{} To make contact with quantum groups, we again pass to the enhanced affine flag variety, and consider 
$$\on{IC}^{\frac{\infty}{2}}_{P, 0}-mod^{{fact}}(Shv(\widetilde{\Fl}_G)^{H}).$$ 
Associated to the Levi factor $M$ of our parabolic $P$ is a version of the small quantum group, which we denote by 
$$\mathfrak{u}_q(\mathfrak{g}_1M),$$
which contains the small quantum group along with divided powers of the raising and lowering operators corresponding to simple roots of $M$.

It has a renormalized derived category of representations
$$\on{Rep}_q(\mathfrak{g}_1M),$$
obtained by ind-completing the pre-triangulated envelope of the baby parabolic Verma modules within its naive derived category of representations, and we denote its principal block by 
$$\on{Rep}_q(\mathfrak{g}_1M)_\circ \hookrightarrow \on{Rep}_q(\mathfrak{g}_1M).$$

\sssec{} We conjecture a $t$-exact equivalence
\begin{equation}
\on{IC}^{\frac{\infty}{2}}_{P, 0}-mod^{{fact}}(Shv(\widetilde{\Fl}_G)^{H}) \simeq \on{Rep}_q(\mathfrak{g}_1M)_\circ.
\end{equation}
This should match the natural structures of highest weight categories on the hearts, and moreover intertwine the pullback 
$$\IC^{\frac{\infty}{2}}_{P, 0}-mod^{{fact}}(Shv(\Gr_G)^H)^{\heartsuit} \rightarrow \IC^{\frac{\infty}{2}}_{P, 0}-mod^{{fact}}(Shv(\widetilde{\Fl}_G)^H)^{\heartsuit} $$
with a suitably defined quantum Frobenius map
$$\on{Fr}: \on{Rep}(\check{M})^\heartsuit \rightarrow \on{Rep}_q(\mathfrak{g}_1M)_\circ.$$
In the case of the Borel, this is closely related to the proposal of Gaitsgory for localization of Kac-Moody modules at critical level introduced in \cite{Gai19SI}, as well as the work of Arkhipov--Bezrukavnikov--Braverman--Gaitsgory--Mirkovi\'c (\cite{ABBGM}, Theorem 6.1.6). 

\iffalse 
Similarly, when one passes to the enhanced affine flag variety, if we write $\Rep_q(\mathfrak{g}_1 M)_\circ$ for the regular block of modules for the $M$-graded small quantum group,\footnote{For the Borel, this is the usual graded small quantum group. For the general case, briefly one takes the graded small quantum group and adjoins divided powers of the simple raising and lowering operators corresponding to $M$.} we conjecture a $t$-exact equivalence of $\DG$-categories
$$
\on{IC}^{\frac{\infty}{2}}_{P, 0}-mod^{{fact}}(Shv(\widetilde{\Fl}_G)^H) \simeq \on{Rep}_q(  \mathfrak{g}_1 M)_\circ,
$$
which again intertwines pullback from the Grassmannian and a quantum Frobenius.\fi

\ssec{Conventions and notations}

\sssec{} 
\label{Sect_1.4.1}
Work over an algebraically closed field $k$. Write $\Sch^{aff}$ for the category of affine schemes, $\Sch_{ft}$ for the category of schemes of finite type (over $k$).

Let $G$ be a connected reductive group over $k$, $T\subset B\subset G$ be a maximal torus and Borel subgroups, $B^-$ an opposite Borel subgroup with $B\cap B^-=T$. Write $U$ (resp., $U^-$) for the unipotent radical of $B$ (resp., $B^-$). 

 Let $P\subset G$ be a standard parabolic, $P^-$ an opposite parabolic with common Levi subgroup $M=P\cap P^-$. Write $w_0$ for the longest element of the Weyl group $W$, and similarly for $w_0^M\in W_M$, where $W_M$ is the Weyl group of $(M, T)$. Write $U(P)$ (resp., $U(P^-)$) for the unipotent radical of $P$ (resp., of $P^-$). Set $B_M=B\cap M$, $B_M^-=B^-\cap M$. 

 Let $\cI$ be the set of vertices of the Dynkin diagram. For $i\in\cI$ we write $\alpha_i$ (resp., $\check{\alpha}_i$) for the corresponding simple coroot (resp., simple root). Let $\cI_M\subset \cI$ correspond to the Dynkin diagram of $M$. Write $\check{P}, \check{B}, \check{T}$, $U(\check{P}), U(\check{P}^-)$ for the corresponding dual objects.
 
 Write $\Lambda$ (resp., $\check{\Lambda}$) for the coweights (resp., weights) lattice of $T$, $\Lambda^+$ for the dominant coweights. Write $\Lambda^+_M$ for the dominant coweights of $B_M$. Let $\Lambda^{pos}\subset\Lambda$ be the $\ZZ_+$-span of positive coroots. Let $\Lambda_M^{pos}\subset\Lambda$ be the $\ZZ_+$-span of $\alpha_i$, $i\in \cI_M$.

\sssec{} Our conventions about higher categories and sheaf theories are those of \cite{AGKRRV}. In particular, $\Spc$ denotes the $\infty$-category of spaces, $1-\Cat$ is the $\infty$-category of $(\infty, 1)$-categories (\cite{G}, ch. I.1, 1.1.1).  We fix an algebraically closed field $e$ of characteristic zero, the field of coefficients of our sheaf theory. Then $\Vect$ is the $\DG$-category of complexes of $e$-vector spaces defined in (\cite{G}, ch. I.1, 10.1). The categories $\DGCat_{cont}, \DGCat^{non-cocmpl}$ are defined in (\cite{G}, ch. I.1, 10.3). 

 For a scheme $S$ over $\Spec e$ write $\Gamma(S)$ (resp., $\cO(S)\in \Vect^{\heartsuit}$) for the quasi-coherent cohomology $\RG(S, \cO)$ (resp., for the space of functions on $S$). We work in the constructible context.

\section{Analog of the Drinfeld-Pl\"ucker formalism}

\ssec{Case of $\Bunt_P$}
\sssec{} For $\lambda\in\Lambda^+$ write $V^{\lambda}$ for the irreducible $\check{G}$-module with highest weight $\lambda$. We pick vectors $v^{\lambda}\in V^{\lambda}, (v^{\lambda})^*\in (V^{\lambda})^*$ as in (\cite{Gai19SI}, 2.1.2). Namely, $v^{\lambda}$ is a highest weight vector of $V^{\lambda}$. Then $(v^{\lambda})^*$ is characterised by the properties that $\<v^{\lambda}, (v^{\lambda})^*\>=1$, and $(v^{\lambda})^*$ vanished on  the weight spaces $V^{\lambda}(\mu)$ for $\mu\ne\lambda$. 
 
  For $\nu\in\Lambda^+_M$ let $U^{\nu}$ be the irreducible $\check{M}$-module with highest weight $\nu$. If moreover $\nu\in\Lambda^+$ then we simply assume 
$$
U^{\nu}=(V^{\nu})^{U(\check{P})},
$$ 
so we have the highest weight vector $v^{\nu}\in U^{\nu}$. We assume this choice of a highest weight vector $v^{\nu}\in U^{\nu}$ is extended for the whole of $\Lambda^+_M$. 
 
 It is known that for any finite-dimensional $\check{G}$-module $V$ the natural map $V^{U(\check{P})}\to V\to V_{U(\check{P}^-)}$ is an isomorphism. So, for $\nu\in\Lambda^+$ we get canonically
$$
((V^{\nu})^*)^{U(\check{P}^-)}\,\iso\, ((V^{\nu})_{U(\check{P}^-)})^*\,\iso\, (U^{\nu})^*
$$ 
This is an irreducible $\check{M}$-module with highest weight $-w_0^M(\nu)$, and 
$$
(v^{\nu})^*\in ((V^{\nu})^*)^{U(\check{P}^-)}\,\iso\,(U^{\nu})^*
$$
So, $(U^{\nu})^*$ is equipped with the highest weight vector $(v^{\nu})^*$ with respect to the Borel $B_M^-$. %Recall that $\<v^{\lambda}, (v^{\lambda})^*\>=1$. 

 Now define the $B_M^-$-highest weights vectors $(v^{\nu})^*\in (U^{\nu})^*$ for all $\nu\in\Lambda^+_M$ as for $\check{G}$. Namely, they satisfy
$\<v^{\nu},(v^{\nu})^*\>=1$, and $(v^{\nu})^*: U^{\nu}\to e^{\nu}$ vanish on all the weight spaces $U^{\nu}(\mu)$ for $\mu\ne\nu$.   

\sssec{} For $\lambda_i\in\Lambda^+$ we denote by $u^{\lambda_1,\lambda_2}: V^{\lambda_1}\otimes V^{\lambda_2}\to V^{\lambda_1+\lambda_2}$ and $v^{\lambda_1,\lambda_2}: V^{\lambda_1+\lambda_2}\to V^{\lambda_1}\otimes V^{\lambda_2}$ the maps fixed in (\cite{Gai19SI}, Section 2.1.4) as well as their duals. So, we have 
$$
(v^{\lambda_1+\lambda_2})^*\comp u^{\lambda_1,\lambda_2}=(v^{\lambda_1})^*\otimes (v^{\lambda_2})^*\;\;\;\mbox{and}\;\;\; v^{\lambda_1,\lambda_2}\comp
v^{\lambda_1+\lambda_2}=v^{\lambda_1}\otimes v^{\lambda_2}.
$$ 

 For $\lambda_1,\lambda_2\in\Lambda^+_M$ we also denote by $\bar v^{\lambda_1,\lambda_2}: U^{\lambda_1+\lambda_2}\to U^{\lambda_1}\otimes U^{\lambda_2}$ and $\bar u^{\lambda_1,\lambda_2}: U^{\lambda_1}\otimes U^{\lambda_2}\to U^{\lambda_1+\lambda_2}$ the maps defined similarly for $\check{M}$ as well as their duals. 
 
For $\lambda_i\in\Lambda^+$ the diagram commutes
\begin{equation}
\label{diag_for_Sect_1.0.3}
\begin{array}{ccc}
(U^{\lambda_1})^*\otimes (U^{\lambda_2})^* & \hook{} & (V^{\lambda_1})^*\otimes (V^{\lambda_2})^*\\
\downarrow\lefteqn{\scriptstyle \bar v^{\lambda_1,\lambda_2}} && \downarrow\lefteqn{\scriptstyle v^{\lambda_1,\lambda_2}}\\
(U^{\lambda_1+\lambda_2})^* & \hook{} & (V^{\lambda_1+\lambda_2})^*
\end{array}
\end{equation}

\sssec{} 
The above gives 
\begin{equation}
\label{decomp_O(G/U)}
\cO(\check{G}/U(\check{P}^-))\,\iso\,\mathop{\oplus}\limits_{\lambda\in\Lambda^+} V^{\lambda}\otimes (U^{\lambda})^*\in \Rep(\check{G})\otimes\Rep(\check{M})
\end{equation}  
 
  For each finite-dimensional representation $V$ of $\check{G}$ we have the matrix coefficient map $V^*\otimes V^{U(\check{P}^-)}\to \cO(\check{G}/U(\check{P}^-))$, 
$u\otimes v\mapsto (g\mapsto \<u, gv\>)$.  

\sssec{} Recall that $\check{G}/U(\check{P}^-)$ is quasi-affine by (\cite{BG}, 1.1.2), write $\overline{\check{G}/U(\check{P}^-)}$ for the affine closure of $\check{G}/U(\check{P}^-)$.
Consider the diagram
\begin{equation}
\label{diag_one}
\begin{array}{ccccc}
\check{G}\backslash\overline{\check{G}/U(\check{P}^-)}/\check{M} & \getsup{j_M} &
B(\check{P}^-) & \getsup{\eta} & B(\check{M})\\
&\searrow\lefteqn{\scriptstyle \bar q}&\downarrow\lefteqn{\scriptstyle q} & \swarrow\lefteqn{\scriptstyle q_M}\\
&& B(\check{G}\times\check{M}),
\end{array}
\end{equation}
where the maps come from the diagram $\check{M}\to \check{P}^-\to \check{G}\times\check{M}$, the second map being the diagonal morphism. Here $j_M$ is the open immersion obtained by passing to the stack quotient under the action of $\check{G}\times\check{M}$ in 
$$
\check{G}/U(\check{P}^-)\hook{}\overline{\check{G}/U(\check{P}^-)}. 
$$ 
After the base change $\Spec e\to B(\check{G}\times\check{M})$ the map $\eta$ becomes $\bar\eta: \check{G}\to \check{G}/U(\check{P}^-)$.  

 We get an adjoint pair 
\begin{equation}
\label{adjoint_pair_eta}
\eta^*: \QCoh(B(\check{P}^-))\leftrightarrows \QCoh(B(\check{M})): \eta_*
\end{equation}
in $\Rep(\check{G})\otimes\Rep(\check{M})-mod$ by (\cite{G}, ch. I.3, 3.2.4). 
We have $q_*\cO, (q_M)_*\cO\in Alg(\Rep(\check{G})\otimes\Rep(\check{M}))$. By (\cite{G}, ch. I.3, 3.3.3) one has 
$$
\QCoh(B(\check{P}^-))\,\iso\, q_*\cO-mod(\Rep(\check{G})\otimes\Rep(\check{M})),
$$ 
$$
\QCoh(\check{G}\backslash\overline{\check{G}/U(\check{P}^-)}/\check{M})\,\iso\,\cO(\check{G}/U(\check{P}^-))-mod(\Rep(\check{G})\otimes\Rep(\check{M})),
$$
and 
$$
\QCoh(B(\check{M}))\,\iso\, (q_M)_*\cO-mod(\Rep(\check{G})\otimes\Rep(\check{M}))
$$
Here $q_*\cO\,\iso\, \Gamma(\check{G}/U(\check{P}^-))$ and $(q_M)_*\cO\,\iso\, \cO(\check{G})$.  
\sssec{} 
\label{Sect_2.1.5_now}
Let $C\in \Rep(\check{G})\otimes\Rep(\check{M})-mod(\DGCat_{cont})$. Similarly to \cite{Gai19SI}, set $\Hecke_{\check{G}, \check{M}}(C)=C\otimes_{\Rep(\check{G})\otimes\Rep(\check{M})} \Rep(\check{M})$. The diagram (\ref{diag_one}) yields an adjoint pair
$$
\eta^*:
C\otimes_{\Rep(\check{G})\otimes\Rep(\check{M})} \Rep(\check{P}^-)\leftrightarrows \Hecke_{\check{G}, \check{M}}(C): \eta_*
$$ 
in $\DGCat_{cont}$. We want to describe the composition
\begin{equation}
\label{left_adj_DR-PL_for_P}
C\otimes_{\Rep(\check{G})\otimes\Rep(\check{M})} \QCoh(\check{G}\backslash\overline{\check{G}/U(\check{P}^-)}/\check{M})\toup{j_M^*}
C\otimes_{\Rep(\check{G})\otimes\Rep(\check{M})} \Rep(\check{P}^-)\toup{\eta^*}\Hecke_{\check{G}, \check{M}}(C).
\end{equation}
By (\cite{G}, ch. I.1, 8.5.7) 
$$
C\otimes_{\Rep(\check{G})\otimes\Rep(\check{M})} \QCoh(\check{G}\backslash\overline{\check{G}/U(\check{P}^-)}/\check{M})\,\iso\, \cO(\check{G}/U(\check{P}^-))-mod(C).
$$
First we want to describe the composition
\begin{equation}
\label{left_adj_DR-PL_for_P_comoposed}
\cO(\check{G}/U(\check{P}^-))-mod(C)
\toup{j_M^*}C\otimes_{\Rep(\check{G})\otimes\Rep(\check{M})} \Rep(\check{P}^-)\toup{\eta^*}
\Hecke_{\check{G}, \check{M}}(C)\toup{\oblv} C
\end{equation}

 By \select{loc.cit.}, one gets 
$$
C\otimes_{\Rep(\check{G})\otimes\Rep(\check{M})} \Rep(\check{P}^-)\,\iso\, \Gamma(\check{G}/U(\check{P}^-))-mod(C)
$$
and
$
\Hecke_{\check{G}, \check{M}}(C)\,\iso\,  \cO(\check{G})-mod(C)
$. 
Now (\ref{left_adj_DR-PL_for_P}) is the functor 
$$
\cO(\check{G}/U(\check{P}^-))-mod(C)\to \cO(\check{G})-mod(C)
$$  
sending $c$ to $\cO(\check{G})\otimes_{\cO(\check{G}/U(\check{P}^-))} c$ in the sense of (\cite{HA}, 4.4.2.12). The functor $\oblv: \cO(\check{G})-mod(C)\to C$ forgets the $\cO(\check{G})$-module structure.

\sssec{} The category $\cO(\check{G}/U(\check{P}^-))-mod(C)$ admits a description in Pl\"ucker style as follows, this is analogous to (\cite{BG}, Theorem 1.1.2). We will write the action of $\Rep(\check{G})$ on $c\in C$ on the right, and that of $\Rep(\check{M})$ on the left.

  For the convenience of the reader recall the following. If $\cA$ is a monoidal $\infty$-category, $D\in 1-\Cat$ then a lax action of $\cA$ on the left (resp., on the right) on $D$ is a right lax monoidal functor $\cA\to\Fun(D,D)$ (resp., $\cA^{rm}\to \Fun(D,D)$). Here $rm$ stands for the reversed multiplication. 

 Consider the functor $\Rep(\check{G})\to \Rep(\check{M})$, $V\mapsto V^{U(\check{P}^-)}$ of \select{non-derived} invariants, it is right-lax non-unital symmetric monoidal. For $\lambda\in\Lambda^+$ it sends $V^{\lambda}$ to $(V^{\lambda})^{U(\check{P}^-)}\in\Rep(\check{M})^{\heartsuit}$. Restricting the $\Rep(\check{M})$-action on $C$ via this functor, one gets a lax action of $\Rep(\check{G})$ on the left on $C$. Considering also the right action of $\Rep(\check{G})$ on $C$, one gets the category of lax central objects $\Cent^{lax,\ra}_{\Rep(\check{G})}(C)$ for this actions. It is formally defined in Appendix~\ref{Sect_lax-appendix}. 

 Since $\Rep(\check{G})$ is rigid, by Proposition~\ref{Pp_Dennis_result} one has a canonical equivalence 
$$
\cO(\check{G}/U(\check{P}^-))-mod(C)\,\iso\, \Cent^{lax,\ra}_{\Rep(\check{G})}(C).
$$
 
 So, informally, an object $c\in \cO(\check{G}/U(\check{P}^-))-mod(C)$ can be seen as $c\in C$ with the following data. For each $V\in \Rep(\check{G})$ we are given a map $\kappa_V: V^{U(\check{P}^-)}\ast c\to c\ast V$. For a morphism $V_1\to V_2$ in $\Rep(\check{G})$, we are given a commutativity datum for the diagram
$$ 
\begin{array}{ccc}
V_1^{U(\check{P}^-)}\ast c &\toup{\kappa_{V_1}} & c\ast V_1\\
\downarrow && \downarrow\\
V_2^{U(\check{P}^-)}\ast c& \toup{\kappa_{V_2}}& c\ast V_2.
\end{array}
$$
Besides, we are given a commutativity datum for the diagram
$$
\begin{array}{ccc}
(V_1^{U(\check{P}^-)}\otimes V_2^{U(\check{P}^-)})\ast c & \toup{\kappa_{V_2}} & V_1^{U(\check{P}^-)}\ast c\ast V_2 \\
\downarrow && \downarrow\lefteqn{\scriptstyle \kappa_{V_1}}\\
(V_1\otimes V_2)^{U(\check{P}^-)}\ast c & \toup{\kappa_{V_1\otimes V_2}} & c\ast (V_1\otimes V_2)
\end{array}
$$
Besides, for $V$ trivial we are given an identification $\kappa_V\,\iso\,\id$, plus coherent system of higher compatibilities. 

\sssec{} 
\label{Sect_2.1.7_now}
Let $c\in \cO(\check{G}/U(\check{P}^-))-mod(C)$. Then for $\lambda\in\Lambda^+$ we get the action map $\kappa^{\lambda}: (U^{\lambda})^*\ast c\to c\ast (V^{\lambda})^*$. For $\lambda_i\in\Lambda^+$ the above using (\ref{diag_for_Sect_1.0.3}) yields the commutativity datum for the diagram 
\begin{equation}
\label{diag_for_Sect_1.0.8_U(checkP^-)}
\begin{array}{ccc}
(U^{\lambda_1})^*\otimes (U^{\lambda_2})^*\ast c & \toup{\kappa^{\lambda_2}}  (U^{\lambda_1})^*\ast c\ast (V^{\lambda_2})^* \toup{\kappa^{\lambda_1}} & c\ast (V^{\lambda_1})^*\otimes (V^{\lambda_2})^*\\
\downarrow\lefteqn{\scriptstyle \bar v^{\lambda_1,\lambda_2}} && \downarrow\lefteqn{\scriptstyle v^{\lambda_1,\lambda_2}}\\
(U^{\lambda_1+\lambda_2})^*\ast c & \toup{\kappa^{\lambda_1+\lambda_2}} & c\ast (V^{\lambda_1+\lambda_2})^*
\end{array}
\end{equation}

 Consider the following two lax actions of $\Lambda^+$ on $C$. For the first one $\lambda\in\Lambda^+$ sends $x$ to $(U^{\lambda})^*\ast x$, where the lax structure is given by the morphisms 
$$
(U^{\lambda_1})^*\ast ((U^{\lambda_2})^*\ast x)\,\iso\, ((U^{\lambda_1})^*\otimes (U^{\lambda_2})^*) \ast x\toup{\bar v^{\lambda_1,\lambda_2}} (U^{\lambda_1+\lambda_2})^*\ast x
$$
For the second one, $\lambda$ sends $x$ to $x\ast (V^{\lambda})^*$, and the lax structure is given by the morphisms
$$
(x\ast (V^{\lambda_1})^*)\ast (V^{\lambda_2})^*\,\iso\, x\ast ((V^{\lambda_1})^*\otimes (V^{\lambda_2})^*)\toup{v^{\lambda_1,\lambda_2}} (V^{\lambda_1+\lambda_2})^*
$$
Then $c$ inherits a lax central object structure for these actions by functoriality. 
This means that the commutativity datum for (\ref{diag_for_Sect_1.0.8_U(checkP^-)}) is equipped with coherent system of higher compatibilities. The formal definition is given in Appendix~\ref{Sect_lax-appendix}. 

\sssec{} Let $\Lambda_{M, ab}=\{\lambda\in \Lambda\mid \<\lambda, \check{\alpha}_i\>=0\;\mbox{for}\; i\in \cI_M\}$. This is the lattice of characters of $\check{M}_{ab}:=\check{M}/[\check{M},\check{M}]$. Set for brevity $\Lambda_{M, ab}^+=\Lambda_{M, ab}\cap \Lambda^+$. 

 Equip $\Lambda^+_{M, ab}$ with the relation $\lambda_1\le\lambda_2$ iff $\lambda_2-\lambda_1\in\Lambda^+_{M, ab}$. This is not a partial order in general, but makes $\Lambda^+_{M, ab}$ a filtered category. 
 
  Let $c\in \cO(\check{G}/U(\check{P}^-))-mod(C)$. We restrict the lax central object structure on $c$ considered in Section~\ref{Sect_2.1.7_now} under $\Lambda^+_{M,ab}\hook{}\Lambda^+$. 
  
  By Lemma~\ref{Lm_C.4.9}, there is a well-defined functor
$$
f: (\Lambda^+_{M, ab},\le)\to C, \; \lambda\mapsto e^{\lambda}\ast c\ast (V^{\lambda})^* 
$$
For $\lambda_i\in\Lambda^+_{M, ab}$ with $\lambda_2-\lambda_1=\lambda\in\Lambda^+_{M, ab}$ the transition map from $f(\lambda_1)$ to $f(\lambda_2)$ in this diagram is the composition
\begin{multline*}
e^{\lambda_1}\ast c\ast (V^{\lambda_1})^*\,\iso\, e^{\lambda_2}\otimes e^{-\lambda}\ast c\ast (V^{\lambda_1})^*\toup{\kappa^{\lambda}} e^{\lambda_2}\ast (c\ast (V^{\lambda})^*)\ast (V^{\lambda_1})^*\toup{v^{\lambda, \lambda_1}} e^{\lambda_2}\ast c\ast (V^{\lambda_2})^*
\end{multline*}
The higher compatibilities for the above morphisms come automatically from the $\cO(\check{G}/U(\check{P}^-))$-module structure on $c$.

We give two proofs of the following.

\begin{Pp}
\label{Pp_2.1.10_colimit_for_Bunt_P}
The functor (\ref{left_adj_DR-PL_for_P_comoposed}) identifies with
$$
c\mapsto \mathop{\colim}\limits_{\lambda\in(\Lambda^+_{M, ab},\le )}  e^{\lambda}\ast c\ast (V^{\lambda})^*
$$
\end{Pp}
\begin{proof}[First proof]
By Section~\ref{Sect_2.1.5_now}, it suffices to establish the universal case when $C=\Rep(\check{G}\times\check{M})$, and $c=\cO(\check{G}/U(\check{P}^-))\in C$. 

\Step 1 First, define a compatible system of morphisms in $\Rep(\check{G}\times\check{M})$
\begin{equation}
\label{morphism_for_Pp_2.1.10_colimit_for_Bunt_P}
 e^{\lambda}\ast \cO(\check{G}/U(\check{P}^-))\ast (V^{\lambda})^*\to \cO(\check{G})
\end{equation} 
for $\lambda\in\Lambda^+_{M, ab}$ as follows. Let $\check{M}$ act on $\check{G}\times \check{M}$ by right translations via the diagonal homomorphism $\check{M}\to \check{G}\times \check{M}$, form the quotient $(\check{G}\times\check{M})/\check{M}$. Let $\check{G}\times\check{M}$ act by left translations on the latter quotient. We get a $\check{G}\times \check{M}$-equivariant isomorphism $(\check{G}\times\check{M})/\check{M}\,\iso\,\check{G}$, where on the RHS the group $\check{G}$ (resp., $\check{M}$) acts by left (resp., right) translations. 
  
 By Frobenius reciprocity, a datum of (\ref{morphism_for_Pp_2.1.10_colimit_for_Bunt_P}) is the same as a $\check{M}$-equivariant morphism
\begin{equation}
\label{morphism_for_Pp_2.1.10_second}
e^{\lambda}\ast \cO(\check{G}/U(\check{P}^-))\ast (V^{\lambda})^*\to e,
\end{equation}
where on the LHS the action is obtained from the $\check{G}\times \check{M}$-action
by restricting under the diagonal map $\check{M}\to \check{G}\times \check{M}$. Let $\ev: \cO(\check{G}/U(\check{P}^-))\to e$ be the evaluation at $U(\check{P}^-)\in \check{G}/U(\check{P}^-)$, it is invariant under the adjoint $\check{M}$-action. Note that $v^{\lambda}: e^{\lambda}\otimes (V^{\lambda})^*\to e$ is $\check{M}$-equivariant,
define (\ref{morphism_for_Pp_2.1.10_second}) as $\ev\otimes v^{\lambda}$. 

 Let $\lambda,\lambda_1\in\Lambda^+_{M, ab}$ and $\lambda_2=\lambda_1+\lambda$. Let us show that the diagram commutes
$$
\begin{array}{ccc}
e^{\lambda_1}\otimes \cO(\check{G}/U(\check{P}^-))\otimes (V^{\lambda_1})^*& \iso & e^{\lambda_2}\otimes e^{-\lambda}\otimes \cO(\check{G}/U(\check{P}^-))\otimes (V^{\lambda_1})^*\\ 
\downarrow\lefteqn{\scriptstyle v^{\lambda_1}\otimes \ev}
&& \downarrow\lefteqn{\scriptstyle{\kappa^{\lambda}}}\\
e && e^{\lambda_2}\otimes \cO(\check{G}/U(\check{P}^-))\otimes (V^{\lambda})^*\otimes (V^{\lambda_1})^*\\
& \hspace{-1em}\nwarrow\lefteqn{\scriptstyle v^{\lambda_2}\otimes\ev}\;\;\;\;\;\;\;
& \downarrow\lefteqn{\scriptstyle{v^{\lambda, \lambda_1}}}\\
&& e^{\lambda_2}\otimes \cO(\check{G}/U(\check{P}^-))\otimes (V^{\lambda_2})^*
\end{array}
$$
This follows easily from the commutativity of the diagrams
$$
\begin{array}{ccc}
e^{-\lambda} & \getsup{\ev} & e^{-\lambda}\otimes\cO(\check{G}/U(\check{P}^-))\\
\downarrow\lefteqn{\scriptstyle (v^{\lambda})^*} && \downarrow\lefteqn{\scriptstyle  \kappa^{\lambda}}\\
(V^{\lambda})^* &  \getsup{\ev} &\cO(\check{G}/U(\check{P}^-))\otimes (V^{\lambda})^*
\end{array}
$$
and
$$
\begin{array}{ccc}
e^{\lambda_1}\otimes (V^{\lambda_1})^* & \toup{(v^{\lambda})^*} & e^{\lambda_2}\otimes (V^{\lambda})^*\otimes (V^{\lambda_1})^*\\
\downarrow\lefteqn{\scriptstyle v^{\lambda_1}} && \downarrow\lefteqn{\scriptstyle v^{\lambda, \lambda_1}}\\
e & \getsup{v^{\lambda_2}} & e^{\lambda_2}\otimes (V^{\lambda_2})^*
\end{array}
$$

Thus, we get a morphism in $\Rep(\check{G}\times\check{M})$
\begin{equation}
\label{map_iso_would_be_for_Pp_2.1.10}
\mathop{\colim}\limits_{\lambda\in\Lambda^+_{M, ab}}  e^{\lambda}\otimes \cO(\check{G}/U(\check{P}^-))\otimes (V^{\lambda})^*\to \cO(\check{G})
\end{equation}
It remains to show it is an isomorphism.

\smallskip\Step 2
Pick $\nu\in\Lambda^+$. Assuming $\lambda\in\Lambda^+_{M, ab}$ deep enough 
on the wall of the corresponding Weyl chamber, we get
\begin{equation}
\label{exp_for_Step2_Pp2.1.10}
\Hom_{\check{G}}(V^{\nu}, e^{\lambda}\otimes \cO(\check{G}/U(\check{P}^-))\otimes (V^{\lambda})^*)\,\iso\, \Hom_{\check{G}}(V^{\nu}\otimes V^{\lambda}, e^{\lambda}\otimes \cO(\check{G}/U(\check{P}^-)))
\end{equation}
By Lemma~\ref{Lm_2.0.15}, $V^{\nu}\otimes V^{\lambda}\,\iso\, \mathop{\oplus}\limits_{\mu\in\Lambda^+_M} V^{\lambda+\mu}\otimes\Hom_{\check{M}}(U^{\mu}, V^{\nu})$. Using (\ref{decomp_O(G/U)}) now (\ref{exp_for_Step2_Pp2.1.10}) identifies with
\begin{multline*}
\mathop{\oplus}\limits_{\mu\in\Lambda^+_M} \Hom_{\check{G}}(V^{\lambda+\mu}, e^{\lambda}\otimes \Hom_{\check{M}}(U^{\mu}, V^{\nu})^*\otimes V^{\lambda+\mu}\otimes (U^{\lambda+\mu})^*)\,\iso\\ 
\mathop{\oplus}\limits_{\mu\in\Lambda^+_M} (U^{\mu})^*\otimes \Hom_{\check{M}}(U^{\mu}, V^{\nu})^*\,\iso\, (V^{\nu})^*
\end{multline*}
in $\Rep(\check{M})$. We have also in $\Rep(\check{M})$
$$
\Hom_{\check{G}}(V^{\nu}, \cO(\check{G}))\,\iso\, (V^{\nu})^*
$$
and the map $(V^{\nu})^*\to (V^{\nu})^*$ induced by (\ref{map_iso_would_be_for_Pp_2.1.10}) is the identity.
\end{proof}

\sssec{Second proof of Proposition~\ref{Pp_2.1.10_colimit_for_Bunt_P}}
\label{Sect_Second proof of Pp2.1.10}
Let $\check{G}\times\check{M}$ act on $\check{G}/[\check{P}^-, \check{P}^-]$ and on $\check{G}/[\check{M}, \check{M}]$ naturally via its quotient $\check{G}\times\check{M}\to \check{G}\times\check{M}_{ab}$. Then one has a cartesian square in the category of schemes with a $\check{G}\times\check{M}$-action
\begin{equation}
\label{diag_main_cartesian_for_2nd_proof}
\begin{array}{ccc}
\check{G} & \toup{\bar\eta} & \check{G}/U(\check{P}^-)\\
\downarrow &&\downarrow\\
\check{G}/[\check{M}, \check{M}] & \toup{\bar\eta_{ab}} & \check{G}/[\check{P}^-, \check{P}^-].
\end{array}
\end{equation}

 It is obtained as follows. First, consider the diagonal map $\check{P}^-\to \check{G}\times \check{M}_{ab}$ yielding the morphism $B(\check{P}^-)\to B(\check{G}\times \check{M}_{ab})$, which gives in turn
\begin{equation}
\label{map_second_for_the_secon_proof_Pp_2.1.10_colimit_for_Bunt_P}
\eta\times\id: B(\check{M})\times_{B(\check{G}\times \check{M}_{ab})} \Spec k\to B(\check{P}^-)\times_{B(\check{G}\times \check{M}_{ab})} \Spec k.
\end{equation} 

 View (\ref{map_second_for_the_secon_proof_Pp_2.1.10_colimit_for_Bunt_P}) as a morphism in the category of stacks over $B(\check{G}\times \check{M})\times_{B(\check{G}\times \check{M}_{ab})} \Spec k$, here we use the map $B(\check{P}^-)\to B(\check{G}\times \check{M})$ coming from the diagonal morphism $\check{P}^-\to \check{G}\times\check{M}$. Then (\ref{diag_main_cartesian_for_2nd_proof}) is obtained from (\ref{map_second_for_the_secon_proof_Pp_2.1.10_colimit_for_Bunt_P}) by making the base change by 
$$
\Spec k\to B(\check{G}\times \check{M})\times_{B(\check{G}\times \check{M}_{ab})} \Spec k.
$$ 

The diagram (\ref{diag_main_cartesian_for_2nd_proof}) yields a diagram of affine closures
$$
\begin{array}{ccc}
\check{G} & \toup{\bar\eta} & \ov{\check{G}/U(\check{P}^-)}\\
\downarrow &&\downarrow\\
\check{G}/[\check{M}, \check{M}] & \toup{\bar\eta_{ab}} & \ov{\check{G}/[\check{P}^-, \check{P}^-]},
\end{array}
$$
which is also cartesian. This is seen using the Pl\"ucker description of points of $\ov{\check{G}/U(\check{P}^-)}$, $\ov{\check{G}/[\check{P}^-, \check{P}^-]}$ given in (\cite{BG}, 1.1.2). So, we have an isomorphism of algebras in $\Rep(\check{G}\times\check{M})$
$$
\cO(\check{G}/U(\check{P}^-))\otimes_{\cO(\check{G}/[\check{P}^-, \check{P}^-]}\cO(\check{G}/[\check{M}, \check{M}])\,\iso\, \cO(\check{G}).
$$ 

Now for $c\in \cO(\check{G}/U(\check{P}^-))-mod(C)$ one gets
$$
c\otimes_{\cO(\check{G}/U(\check{P}^-))} \cO(\check{G})\,\iso\, c\otimes_{\cO(\check{G}/[\check{P}^-, \check{P}^-])}\cO(\check{G}/[\check{M}, \check{M}])
$$
Our claim follows now from Proposition~\ref{Pp_2.0.11}. \QED

\ssec{Case of $\Bunb_P$}
\sssec{} Given $\lambda\in\Lambda^+$, $(V^{\lambda})^{[\check{P},\check{P}]}$ vanishes unless $\lambda\in\Lambda_{M, ab}^+$, and in the latter case it identifies with the highest weight line $e^{\lambda}\subset V^{\lambda}$ generated by $v^{\lambda}$. So,
$$
\cO(\check{G}/[\check{P},\check{P}])\,\iso\, \mathop{\oplus}\limits_{\lambda\in  \Lambda_{M, ab}^+} (V^{\lambda})^*\otimes e^{\lambda}\in \Rep(\check{G})\otimes\Rep(\check{M}_{ab})
$$
The product in this algebra is given for $\lambda, \mu\in\Lambda^+_{M, ab}$ by the maps
$$
(V^{\lambda})^*\otimes e^{\lambda}\otimes (V^{\mu})^*\otimes e^{\mu}\toup{v^{\lambda,\mu}} (V^{\lambda+\mu})^*\otimes e^{\lambda+\mu}
$$
We could replace in the above $e^{\lambda}$ by $U^{\lambda}$, as the corresponding $\check{M}$-module for $\lambda\in \Lambda_{M, ab}^+$ is 1-dimensional.

\sssec{} One similarly gets
\begin{equation}
\label{functions_for_G/P^-}
\cO(\check{G}/[\check{P}^-,\check{P}^-])\,\iso\, \mathop{\oplus}\limits_{\lambda\in  \Lambda_{M, ab}^+} V^{\lambda}\otimes e^{-\lambda}\in \Rep(\check{G})\otimes\Rep(\check{M}_{ab})
\end{equation}
Here $e^{-\lambda}$ coincides with $(U^{\lambda})^*$. 

 The product in this algebra is given for $\lambda,\mu\in \Lambda_{M, ab}^+$ by the maps
$$
V^{\lambda}\otimes e^{-\lambda}\otimes V^{\mu}\otimes e^{-\mu}\toup{u^{\lambda,\mu}} V^{\lambda+\mu}\otimes e^{-\lambda-\mu}
$$ 

\sssec{} Recall that $\check{G}/[\check{P}^-, \check{P}^-]$ is quasi-affine by (\cite{BG}, 1.1.2). Write $\overline{\check{G}/[\check{P}^-, \check{P}^-]}$ for its affine closure. Consider the diagram
$$
\begin{array}{ccccc}
\check{G}\backslash \overline{\check{G}/[\check{P}^-, \check{P}^-]}/\check{M}_{ab} & \getsup{j_{M, ab}} & B(\check{P}^-) & \getsup{\eta} & B(\check{M})\\
& \searrow\lefteqn{\scriptstyle \bar q_{ab}} & 
\downarrow\lefteqn{\scriptstyle q_{ab}} & \swarrow\lefteqn{\scriptstyle q_{M, ab}}\\
&& B(\check{G}\times \check{M}_{ab})
\end{array}
$$
obtained using the diagonal map $\check{P}^-\to \check{G}\times \check{M}_{ab}$. Here $j_{M, ab}$ is obtained by passing to the stack quotient under the $\check{G}\times \check{M}_{ab}$-action in
$$
\check{G}/[\check{P}^-, \check{P}^-]\hook{}
\overline{\check{G}/[\check{P}^-, \check{P}^-]}.
$$
After the base change $\Spec k\to B(\check{G}\times \check{M}_{ab})$ the map $\eta$ becomes 
$$
\bar\eta_{ab}: \check{G}/[\check{M}, \check{M}]\to \check{G}/[\check{P}^-, \check{P}^-]
$$ 

\sssec{} One has
$$
\cO(\check{G}/[\check{M},\check{M}])\,\iso\, \mathop{\oplus}\limits_{\nu\in\Lambda^+, \; \mu\in\Lambda_{M, ab}} V^{\nu}\otimes e^{-\mu}\otimes \Hom_{\check{M}}(U^{\mu}, V^{\nu})^*
$$
Indeed, for $\nu\in\Lambda^+$, 
$$
(V^{\nu})_{[\check{M},\check{M}]}\,\iso\, \mathop{\oplus}\limits_{\mu\in\Lambda^+_M} (U^{\mu})_{[\check{M},\check{M}]} \otimes \Hom(U^{\mu}, V^{\nu}))
$$
Now $(U^{\mu})_{[\check{M},\check{M}]}$ vanishes unless $\mu\in \Lambda_{M, ab}$, in which case it identifies with $e^{\mu}$ as a $\check{M}_{ab}$-module. Finally,
$$
((V^{\nu})^*)^{[\check{M},\check{M}]}\,\iso\, ((V^{\nu})_{[\check{M},\check{M}]})^*
$$

\sssec{} View now (\ref{adjoint_pair_eta}) as an adjoint pair in $\Rep(\check{G})\otimes\Rep(\check{M}_{ab})-mod$. We get 
$$
(q_{ab})_*\cO, (q_{M, ab})_*\cO\in \Rep(\check{G})\otimes\Rep(\check{M}_{ab})
$$
Since $\check{G}/[\check{P}^-, \check{P}^-]$ is quasi-affine and $\check{G}/[\check{M}, \check{M}]$ is affine, we similarly get
$$
\QCoh(B(\check{P}^-))\,\iso\, (q_{ab})_*\cO-mod(\Rep(\check{G})\otimes\Rep(\check{M}_{ab})),
$$
$$
\QCoh(B(\check{M}))\,\iso\,  (q_{M, ab})_*\cO-mod(\Rep(\check{G})\otimes\Rep(\check{M}_{ab})),
$$
and
$$
\QCoh(\check{G}\backslash \overline{\check{G}/[\check{P}^-, \check{P}^-]}/\check{M}_{ab})\,\iso\, \cO(\check{G}/[\check{P}^-, \check{P}^-])-mod(\Rep(\check{G})\otimes\Rep(\check{M}_{ab})).
$$
Here $(q_{M, ab})_*\cO\,\iso\, \cO(\check{G}/[\check{M}, \check{M}])$ and $(q_{ab})_*\cO\,\iso\, \Gamma(\check{G}/[\check{P}^-, \check{P}^-])$. 

\sssec{} Given $C\in \Rep(\check{G})\otimes\Rep(\check{M}_{ab}))-mod(\DGCat_{cont})$, 
set $\Hecke_{\check{G},\check{M}, ab}(C)=C\otimes_{\Rep(\check{G})\otimes\Rep(\check{M}_{ab}))}\Rep(\check{M})$.
The adjoint pair (\ref{adjoint_pair_eta}) gives an adjoint pair
$$
\eta^*: C\otimes_{\Rep(\check{G})\otimes\Rep(\check{M}_{ab}))}\Rep(\check{P}^-)\leftrightarrows 
\Hecke_{\check{G},\check{M}, ab}(C): \eta_*
$$
Note that 
$$
C\otimes_{\Rep(\check{G})\otimes\Rep(\check{M}_{ab}))} 
\QCoh(\check{G}\backslash \overline{\check{G}/[\check{P}^-, \check{P}^-]}/\check{M}_{ab})
\,\iso\, \cO(\check{G}/[\check{P}^-, \check{P}^-])-mod(C)
$$
We want to better understand the composition
\begin{equation}
\label{left_adjoint_ab_case}
 \cO(\check{G}/[\check{P}^-, \check{P}^-])-mod(C)
\toup{j_{M, ab}^*}C\otimes_{\Rep(\check{G})\otimes\Rep(\check{M}_{ab}))}\Rep(\check{P}^-) \toup{\eta^*} \Hecke_{\check{G},\check{M}, ab}(C)
\end{equation}
and also the composition
\begin{multline}
\label{left_adjoint_ab_case_after_oblv}
 \cO(\check{G}/[\check{P}^-, \check{P}^-])-mod(C)
\toup{j_{M, ab}^*}
C\otimes_{\Rep(\check{G})\otimes\Rep(\check{M}_{ab}))}\Rep(\check{P}^-)\\ \toup{\eta^*} C\otimes_{\Rep(\check{G})\otimes\Rep(\check{M}_{ab}))}\Rep(\check{M})\toup{\oblv} C
\end{multline}

\sssec{} By (\cite{G}, ch. I.1, 8.5.7), one gets 
$$
C\otimes_{\Rep(\check{G})\otimes\Rep(\check{M}_{ab}))}\Rep(\check{P}^-)\,\iso\, \Gamma(\check{G}/[\check{P}^-, \check{P}^-])-mod(C)
$$ 
and 
$$
C\otimes_{\Rep(\check{G})\otimes\Rep(\check{M}_{ab}))}\Rep(\check{M})
\,\iso\, \cO(\check{G}/[\check{M}, \check{M}])-mod(C)
$$
The functor (\ref{left_adjoint_ab_case}) is given by
\begin{equation}
\label{tens_product_in_C_for_Bunb_P}
c\mapsto \cO(\check{G}/[\check{M}, \check{M}])\otimes_{\cO(\check{G}/[\check{P}^-, \check{P}^-])} c
\end{equation} 
in the sense of (\cite{HA}, 4.4.2.12). 

\sssec{} Write the action of $\Rep(\check{G})$ on $C$ on the right, and that of $\Rep(\check{M}_{ab})$ on the left. 

 The category $\cO(\check{G}/[\check{P}^-, \check{P}^-])-mod(C)$ is described as the category of $c\in C$ equipped with maps
$$
\kappa^{\lambda}: c\ast V^{\lambda}\to e^{\lambda}\ast c, \;\; \lambda\in \Lambda_{M, ab}^+
$$ 
with the following additional structures and properties: i) if $\lambda=0$ then $\kappa^{\lambda}$ is identified with the identity map; ii) for $\lambda,\mu\in\Lambda_{M, ab}^+$ we are given a datum of commutativity for the diagram 
$$
\begin{array}{ccc}
(c\ast V^{\lambda})\ast V^{\mu} & \iso & c\ast (V^{\lambda}\otimes V^{\mu})\\
\downarrow\lefteqn{\scriptstyle \kappa^{\lambda}} && \downarrow\lefteqn{\scriptstyle u^{\lambda,\mu}}\\
e^{\lambda}\ast c\ast V^{\mu}&& c\ast V^{\lambda+\mu} \\ 
\downarrow\lefteqn{\scriptstyle \kappa^{\mu}} && \downarrow\lefteqn{\scriptstyle\kappa^{\lambda+\mu}}\\
e^{\lambda}\ast (e^{\mu}\ast c) & \iso & e^{\lambda+\mu}\ast c
\end{array}
$$
iii) a coherent system of higher compatibilities.

\begin{Rem}
For example, if $C$ is equipped with a t-structure, and both actions of $\Rep(\check{G})$ and of $\Rep(\check{M}_{ab})$ are t-exact then for $c\in C^{\heartsuit}$ in the above description of a $\cO(\check{G}/[\check{P}^-, \check{P}^-])$-module on $c$ the higher compatibilities are automatic.
\end{Rem}

\sssec{} 
\label{Sect_2.2.10_Now}
Let $c\in \cO(\check{G}/[\check{P}^-, \check{P}^-])-mod(C)$. For $\lambda\in\Lambda_{M, ab}^+$ by adjointness (and \cite{HA}, 4.6.2.1), rewrite $\kappa^{\lambda}$ as the map 
$$
\tau^{\lambda}: e^{-\lambda}\ast c\to c\ast (V^{\lambda})^*
$$ 

 Consider the following two lax actions of $\Lambda^+_{M, ab}$ on $C$. The left action of $\lambda$ is $c\mapsto e^{-\lambda}\ast c$. The right action of $\lambda$ is $c\ast (V^{\lambda})^*$. For $\lambda,\mu\in\Lambda^+_{M, ab}$ we are using here the lax structure on the right action given by 
$$
(c\ast (V^{\lambda})^*)\ast (V^{\mu})^*\,\iso\, c\ast (V^{\lambda}\otimes V^{\mu})^*\toup{v^{\lambda,\mu}} c\ast (V^{\lambda+\mu})^*
$$ 
\begin{Lm} In the situation of Section~\ref{Sect_2.2.10_Now}, a $\cO(\check{G}/[\check{P}^-, \check{P}^-])$-module structure on $c$ is the same as a structure of a \select{lax central object} on $c$ with respect to the above lax actions of $\Lambda^+_{M, ab}$ on $C$. 
\end{Lm}
\begin{proof}
First, for any $\lambda_i\in\Lambda^+$ the composition $V^{\lambda_1+\lambda_2}\toup{v^{\lambda_1,\lambda_2}} V^{\lambda_1}\otimes V^{\lambda_2}\toup{u^{\lambda_1,\lambda_2}} V^{\lambda_1+\lambda_2}$ is $\id$. Second, 
for any $\lambda_i\in\Lambda^+$ the diagram commutes
$$
\begin{array}{ccc}
V^{\lambda_1}\otimes V^{\lambda_2}\otimes (V^{\lambda_1})^*\otimes (V^{\lambda_2})^* &\getsup{u\otimes u} & e  \\
\downarrow\lefteqn{\scriptstyle u^{\lambda_1,\lambda_2}}&& \downarrow\lefteqn{\scriptstyle u}\\
V^{\lambda_1+\lambda_2}\otimes (V^{\lambda_1})^*\otimes (V^{\lambda_2})^* & \getsup{u^{\lambda_1,\lambda_2}} & V^{\lambda_1+\lambda_2}\otimes(V^{\lambda_1+\lambda_2})^*,
\end{array}
$$ 
where $u$ every time denotes the unit of the corresponding duality.  
The desired claim follows.
\end{proof}

\sssec{} 
\label{Sect_2.0.11_functor_f}
Now given $c\in \cO(\check{G}/[\check{P}^-, \check{P}^-])-mod(C)$, by Lemma~\ref{Lm_C.4.9} we get a well-defined functor 
$$
f: \Lambda^+_{M, ab}\to C, \;\; \lambda\mapsto e^{\lambda}\ast c\ast (V^{\lambda})^*
$$ 
Here we consider $\Lambda^+_{M, ab}$ with the relation $\lambda_1\le\lambda_2$ iff $\lambda_2-\lambda_1\in \Lambda^+_{M, ab}$. This is not a partial order in general, but $\Lambda^+_{M, ab}$ with the relation $\le$ is a filtered category.\footnote{If $G$ is semi-simple then this is a partially ordered set.} For $\lambda_i\in \Lambda^+_{M, ab}$ the transition map from $f(\lambda_1)$ to $f(\lambda_1+\lambda_2)$ is
\begin{multline*}
e^{\lambda_1}\ast c\ast (V^{\lambda_1})^*\to e^{\lambda_1+\lambda_2}\ast e^{-\lambda_2}\ast c\ast (V^{\lambda_1})^*\toup{\tau^{\lambda_2}} \\ e^{\lambda_1+\lambda_2}\ast (c\ast (V^{\lambda_2})^*)\ast (V^{\lambda_1})^*\toup{v^{\lambda_1,\lambda_2}} e^{\lambda_1+\lambda_2}\ast c\ast (V^{\lambda_1+\lambda_2})^*
\end{multline*}
\begin{Pp} 
\label{Pp_2.0.11}
The functor (\ref{left_adjoint_ab_case_after_oblv}) identifies with
$$
c\mapsto \mathop{\colim}\limits_{\lambda\in(\Lambda^+_{M, ab},\le)} e^{\lambda}\ast c\ast (V^{\lambda})^*
$$
taken in $C$.
\end{Pp}
\begin{Rem} i) Proposition~\ref{Pp_2.0.11} is a particular case of the Drinfeld-Pl\"ucker formalism developed in (\cite{Ly10}, Section~6).\\
ii) If $G=P$ then $\Lambda_{M, ab}^+$ is the category equivalent to $pt$, and the above colimit identifies with $c$ itself.
\end{Rem}
\begin{proof}[Proof of Proposition~\ref{Pp_2.0.11}]

\smallskip\noindent
{\bf Step 1} We must show that (\ref{tens_product_in_C_for_Bunb_P}) identifies with the above colimit in $C$. For this it suffices to show that
\begin{equation}
\label{colimit_for_functions_over_Lambda^+_Mab}
\mathop{\colim}\limits_{\lambda\in(\Lambda^+_{M, ab},\le)} e^{\lambda}\ast \cO(\check{G}/[\check{P}^-, \check{P}^-])\ast (V^{\lambda})^*\,\iso\, \cO(\check{G}/[\check{M}, \check{M}])
\end{equation}
in $\Rep(\check{G})\otimes\Rep(\check{M}_{ab})$. Recall the decomposition (\ref{functions_for_G/P^-}). Given $\lambda_i\in \Lambda_{M, ab}^+$, the transition map
$$
e^{\lambda_1}\ast \cO(\check{G}/[\check{P}^-, \check{P}^-])\ast (V^{\lambda_1})^*\to e^{\lambda_1+\lambda_2}\ast \cO(\check{G}/[\check{P}^-, \check{P}^-])\ast (V^{\lambda_1+\lambda_2})^*
$$
restricts for each $\lambda\in\Lambda_{M, ab}^+$ to a morphism
$$
e^{\lambda_1}\ast (V^{\lambda}\otimes e^{-\lambda})\ast (V^{\lambda_1})^*\to e^{\lambda_1+\lambda_2}\ast (V^{\lambda+\lambda_2}\otimes e^{-\lambda-\lambda_2})
\ast (V^{\lambda_1+\lambda_2})^*
$$
So, the LHS of (\ref{colimit_for_functions_over_Lambda^+_Mab}) identifies with the direct sum over $\nu\in\Lambda_{M, ab}$ of
\begin{equation}
\label{mult_space_for_nu_in_LHS_for_Pp2.0.11}
\mathop{\colim}\limits_{\lambda\in(\Lambda^+_{M, ab},\le)} e^{\lambda}\ast (V^{\lambda-\nu}\otimes e^{\nu-\lambda})\ast (V^{\lambda})^*,
\end{equation}
where the colimit is taken over those $\lambda$ satisfying $\lambda-\nu\in \Lambda^+_{M, ab}$. More precisely, (\ref{mult_space_for_nu_in_LHS_for_Pp2.0.11}) is the multiplicity space of $e^{\nu}\in\Rep(\check{M}_{ab})$ in  the LHS of (\ref{colimit_for_functions_over_Lambda^+_Mab}).

 For each $\lambda\in \Lambda^+_{M, ab}$ such that $\lambda-\nu\in \Lambda^+_{M, ab}$ we define the morphism of $\check{G}$-modules
\begin{equation}
\label{map_for_individual_lambda_for_Pp2.0.11}
V^{\lambda-\nu}\otimes (V^{\lambda})^*\to \cO(\check{G}/[\check{M},\check{M}])
\end{equation}
as the map that corresponds via the Frobenius reciprocity to the $[\check{M},\check{M}]$-equivariant morphism $(v^{\lambda-\nu})^*\otimes v^{\lambda}: 
 V^{\lambda-\nu}\otimes (V^{\lambda})^*\to e$. It is easy to see that the maps (\ref{map_for_individual_lambda_for_Pp2.0.11}) are compatible with the transition maps in the diagram (\ref{mult_space_for_nu_in_LHS_for_Pp2.0.11}), so define by passing to the colimit the $[\check{M},\check{M}]$-equivariant morphism from (\ref{mult_space_for_nu_in_LHS_for_Pp2.0.11}) to $e$. This gives by the Frobenius reciprocity a morphism of $\check{G}$-modules
\begin{equation}
\label{map_to_be_proved_iso_Pp2.0.12}
\mathop{\colim}\limits_{\lambda\in(\Lambda^+_{M, ab},\le)} e^{\lambda}\ast (V^{\lambda-\nu}\otimes e^{\nu-\lambda})\ast (V^{\lambda})^*\to  \cO(\check{G}/[\check{M},\check{M}])_{\nu}
\end{equation}
where the subscript $\nu$ stands for the subspace of $\cO(\check{G}/[\check{M},\check{M}])$ on which $\check{M}$ acts by $\nu$. 

 For $v\in V^{\lambda-\nu}, u\in (V^{\lambda})^*$ the map (\ref{map_for_individual_lambda_for_Pp2.0.11}) sends $v\otimes u$ to the function on $G$
$$
g\mapsto \<(v^{\lambda-v})^*, g^{-1} v\>\<v^{\lambda}, g^{-1}u\>
$$ 

{\bf Step 2} Let $\eta\in\Lambda^+$. To finish the proof, it remains to show that for $\lambda\in\Lambda^+_{M, ab}$ large enough with respect to $\eta$ and $\nu$ (that is, $\<\lambda, \check{\alpha}_i\>$ large enough for $i\notin \cI_M$), one has naturally 
$$
\Hom_{\check{G}}(V^{\eta}, V^{\lambda-\nu}\otimes (V^{\lambda})^*)\,\iso\,\Hom_{\check{M}}(e^{-\nu}, V^{\eta})^*
$$
By Lemma~\ref{Lm_2.0.15} below, 
$$
V^{\eta}\otimes V^{\lambda}\,\iso\, \mathop{\oplus}\limits_{\mu\in\Lambda^+_M} V^{\lambda+\mu}\otimes \Hom_{\check{M}}(U^{\mu}, V^{\eta})
$$
Our claim follows.
\end{proof} 

\sssec{} Write $\coind_{\check{P}^-}^{\check{G}}: \QCoh(B(\check{P}^-))\to \QCoh(B(\check{G}))$ for the $*$-direct image map, the right adjoint to the restriction. Then $\coind_{\check{P}^-}^{\check{G}}(U^{\nu})\,\iso\, V^{\nu}$ for $\nu\in \Lambda^+$. These isomorphisms are uniquely normalized by the property that the diagram is required to commute
$$
\begin{array}{ccc}
\coind_{\check{P}^-}^{\check{G}}(U^{\nu}) & \iso & V^{\nu}\\
\downarrow && \downarrow\lefteqn{\scriptstyle (v^{\nu})^*}\\
U^{\nu} & \toup{(v^{\nu})^*} & e^{\nu},
\end{array}
$$
where the left vertical arrow comes from adjunction. 

\begin{Lm} 
\label{Lm_2.0.15}
Let $V\in\Rep(\check{G})^{\heartsuit}$ be finite-dimensional, $\lambda\in\Lambda^+_{M, ab}$. Assume that for any $\nu\in\Lambda_M^+$ appearing in $\Res^{\check{M}} V$, $\nu+\lambda\in\Lambda^+$. Then one has canonically
$$
V\otimes V^{\lambda}\,\iso\, \mathop{\oplus}\limits_{\mu\in\Lambda^+_M} V^{\lambda+\mu}\otimes \Hom_{\check{M}}(U^{\mu}, V)
$$
in $\Rep(\check{G})$.
\end{Lm}
\begin{proof}
By the projection formula, 
$$
V\otimes V^{\lambda}\,\iso\, V\otimes \coind_{\check{P}^-}^{\check{G}}(e^{\lambda})\,\iso\,
\coind_{\check{P}^-}^{\check{G}}(e^{\lambda}\otimes \Res^{\check{P}^-}(V))
$$ 
Now $\Res^{\check{P}^-}(V)$ is filtered with the associated graded being $\mathop{\oplus}\limits_{\mu\in\Lambda^+_M} U^{\mu}\otimes \Hom_{\check{M}}(U^{\mu}, V)$. So, $e^{\lambda}\otimes \Res^{\check{P}^-}(V)$ is filtered with the associated graded being 
$\mathop{\oplus}\limits_{\mu\in\Lambda^+_M} U^{\mu+\lambda}\otimes \Hom_{\check{M}}(U^{\mu}, V)$. For each $\mu$ as above, $\coind_{\check{P}^-}^{\check{G}}(U^{\mu+\lambda})\,\iso\, V^{\mu+\lambda}$. So, the corresponding filtration on $\coind_{\check{P}^-}^{\check{G}}(e^{\lambda}\otimes \Res^{\check{P}^-}(V))$, splits canonically.
\end{proof}

If $V\in\Rep(\check{G})^{\heartsuit}$ is finite-dimensional, $\lambda\in\Lambda^+_{M, ab}$
is large enough for $V$ then Lemma~\ref{Lm_2.0.15} also rewrites as a canonical isomorphism
\begin{equation}
\label{Lm__2.0.15_rewritten}
V\otimes (V^{\lambda})^*\,\iso\, \mathop{\oplus}\limits_{\mu\in\Lambda^+_M} (V^{\lambda+\mu})^*\otimes \Hom_{\check{M}}((U^{\mu})^*, V)
\end{equation}

\sssec{Version of Hecke property} 
\label{Sect_version_of_Hecke_property_2.0.16}
Since $\Rep(\check{G})\otimes\Rep(\check{M}_{ab}))$ and $\Rep(\check{M})$ are rigid, by (\cite{Ly}, 9.2.43) we get
$$
C\otimes_{\Rep(\check{G})\otimes\Rep(\check{M}_{ab})}\Rep(\check{M})\,\iso\,\Fun_{\Rep(\check{G})\otimes\Rep(\check{M}_{ab})}(\Rep(\check{M}), C)
$$

For $c\in \cO(\check{G}/[\check{P}^-, \check{P}^-])-mod(C)$, the $\cO(\check{G}/[\check{M}, \check{M}])$-action on 
$$
\bar c=\mathop{\colim}\limits_{\lambda\in\Lambda^+_{M, ab}} e^{\lambda}\ast c\ast (V^{\lambda})^*
$$ 
is as follows. Let $V\in\Rep(\check{G})^{\heartsuit}$ finite-dimensional. It suffices to provide the action of $V\otimes (V^*)^{[\check{M}, \check{M}]}$ for each such $V$. It is given by a map $\bar c\ast V\to V_{[\check{M}, \check{M}]}\ast \bar c$. Pick $\lambda\in \Lambda^+_{M, ab}$ large enough for $V$. 
Note that 
$$
V_{[\check{M}, \check{M}]}\,\iso\,
\mathop{\oplus}\limits_{\mu\in\Lambda_{M, ab}} e^{-\mu}\otimes \Hom_{\check{M}}((U^{\mu})^*, V)
$$ 
Using (\ref{Lm__2.0.15_rewritten}), the desired map is the composition 
\begin{multline*}
(e^{\lambda}\ast c\ast (V^{\lambda})^*)\ast V\,\iso\, \mathop{\oplus}\limits_{\mu\in\Lambda^+_M} (e^{\lambda}\ast c)\ast (V^{\lambda+\mu})^*\otimes \Hom_{\check{M}}((U^{\mu})^*, V)\,\iso\\ 
\mathop{\oplus}\limits_{\mu\in\Lambda^+_M} (e^{-\mu}\otimes \Hom_{\check{M}}((U^{\mu})^*, V))\ast (e^{\lambda+\mu}\ast c\ast (V^{\lambda+\mu})^*)\to\\ \mathop{\oplus}\limits_{\mu\in\Lambda_{M, ab}} (e^{-\mu}\otimes \Hom_{\check{M}}((U^{\mu})^*, V))\ast (e^{\lambda+\mu}\ast c\ast (V^{\lambda+\mu})^*),
\end{multline*}
where the latter map is the projection on the corresponding summands. More precisely, when we pass to the colimit over $\lambda$, this becomes the desired morphism.

\sssec{} 
\label{Sect_2.0.17}
For our convenience, we spell a version of the above with $\check{P}^-$ replaced by $\check{P}$.

 Let $C\in \Rep(\check{G})\otimes\Rep(\check{M}_{ab})-mod(\DGCat_{cont})$. As above write $\Rep(\check{G})$-action on the right, and $\Rep(\check{M}_{ab})$-action on the left.
 
 The category of $\cO(\check{G}/[\check{P},\check{P}])-mod(C)$ is described as the category of $c\in C$ equipped with maps
$$
\kappa^{\lambda}: e^{\lambda}\ast c\to c\ast V^{\lambda}, \; \lambda\in\Lambda^+_{M, ab}
$$ 
with the following structures and properties: i) if $\lambda=0$ then $\kappa^{\lambda}$ is identified with the identity map; ii) for $\lambda,\mu\in\Lambda^+_{M, ab}$ we are given a datum of commutativity for the diagram
$$
\begin{array}{ccc}
(c\ast V^{\lambda})\ast V^{\mu} & \iso & c\ast (V^{\lambda}\otimes V^{\mu})\\
\uparrow\lefteqn{\scriptstyle \kappa^{\lambda}} && \downarrow\lefteqn{\scriptstyle u^{\lambda,\mu}}\\
e^{\lambda}\ast c\ast V^{\mu}&& c\ast V^{\lambda+\mu} \\ 
\uparrow\lefteqn{\scriptstyle \kappa^{\mu}} && \uparrow\lefteqn{\scriptstyle  \kappa^{\lambda+\mu}}\\
e^{\lambda}\ast (e^{\mu}\ast c) & \iso & e^{\lambda+\mu}\ast c
\end{array}
$$
iii) a coherent system of higher compatibilities.

 For $c\in C$ a $\cO(\check{G}/[\check{P},\check{P}])$-module structure is the same as the structure of a \select{lax central object} on $c$ in the sense of (\cite{Gai19SI}, 2.7) with respect to the following actions of $\Lambda^+_{M, ab}$ on $C$. The left action of $\lambda\in\Lambda^+_{M, ab}$ is $c\mapsto e^{\lambda}\ast c$. The right lax action of $\lambda$ is $c\mapsto c\ast V^{\lambda}$. The lax structure on the right action is given by 
$$
(c\ast V^{\lambda})\ast V^{\mu}\,\iso\, c\ast (V^{\lambda}\otimes V^{\mu})\toup{u^{\lambda,\mu}} c\ast V^{\lambda+\mu}
$$ 
for $\lambda,\mu\in \Lambda^+_{M, ab}$.

 For $c\in \cO(\check{G}/[\check{P},\check{P}])-mod(C)$ we get as above a well-defined functor 
$$
f: \Lambda^+_{M, ab}\to C, \;\; \lambda\mapsto e^{-\lambda}\ast c\ast V^{\lambda}
$$
For $\lambda_i\in\Lambda^+_{M, ab}$ the transition map from $f(\lambda_1)$ to $f(\lambda_1+\lambda_2)$ is
\begin{multline*}
e^{-\lambda_1}\ast c\ast V^{\lambda_1}\,\iso\, e^{-(\lambda_1+\lambda_2)}\ast (e^{\lambda_2}\ast c)\ast V^{\lambda_1}\toup{\kappa^{\lambda_2}}\\ e^{-(\lambda_1+\lambda_2)}\ast c\ast (V^{\lambda_2}\otimes V^{\lambda_1})\toup{u^{\lambda_1,\lambda_2}} 
e^{-(\lambda_1+\lambda_2)}\ast c\ast V^{\lambda_1+\lambda_2}
\end{multline*}

 We have similarly 
$$
C\otimes_{\Rep(\check{G})\otimes\Rep(\check{M}_{ab})} \Rep(\check{P})\,\iso\, \Gamma(\check{G}/[\check{P},\check{P}])-mod(C),
$$ 
$$
C\otimes_{\Rep(\check{G})\otimes \Rep(\check{M}_{ab})}
\QCoh(\check{G}\backslash \overline{\check{G}/[\check{P},\check{P}]}/\check{M}_{ab})\,\iso\, \cO(\check{G}/[\check{P},\check{P}])-mod(C).
$$ 

 Consider the composition
\begin{multline}
\label{functor_for_2.0.17_for_checkP}
\cO(\check{G}/[\check{P},\check{P}])-mod(C)\to
C\otimes_{\Rep(\check{G})\otimes\Rep(\check{M}_{ab})} \Rep(\check{P})\to \\ C\otimes_{\Rep(\check{G})\otimes\Rep(\check{M}_{ab})} \Rep(\check{M})\toup{\oblv} C,
\end{multline}
where the unnamed arrows are the pullbacks along 
$$
B(\check{M})\to B(\check{P})\to \check{G}\backslash \overline{\check{G}/[\check{P},\check{P}]}/\check{M}_{ab}.
$$ 
A version of Proposition~\ref{Pp_2.0.11} in this case affirms that (\ref{functor_for_2.0.17_for_checkP}) identifies with the functor
$$
c\mapsto \mathop{\colim}\limits_{\lambda\in (\Lambda^+_{M, ab},\le )} \; e^{-\lambda}\ast c\ast V^{\lambda},
$$
the colimit being taken in $C$. 

\ssec{A version of dual baby Verma object} 

\sssec{} Assume $C\in \Rep(\check{P}^-)-mod(\DGCat_{cont})$. We let $\Rep(\check{G})\otimes\Rep(\check{M}_{ab})$ act on $C$ via the pull-back along the diagonal map $B(\check{P}^-)\to B(\check{G}\times \check{M}_{ab})$.  

 Let $\check{P}^-$ act on $\check{G}/[\check{P}^-,\check{P}^-]$ adjointly, so that the unit map $\Spec e\to \check{G}/[\check{P}^-,\check{P}^-]$ is $\check{P}^-$-equivariant. Viewing in this way $\cO(\check{G}/[\check{P}^-,\check{P}^-])\in\Rep(\check{P}^-)$, we get the category $\cO(\check{G}/[\check{P}^-,\check{P}^-])-mod(C)$ and the functor 
\begin{equation}
\label{functor_from_C_for_Sect_2.1.1}
C\to \cO(\check{G}/[\check{P}^-,\check{P}^-])-mod(C)
\end{equation}
of restriction of scalars via the evaluation $\cO(\check{G}/[\check{P}^-,\check{P}^-])\to e$ at $1$.

%The functor (\ref{functor_from_C_for_Sect_2.1.1}) sends $c$ to itself with the action maps 
%\begin{equation}
%\label{map_action_abstract_coming_from_P^-}
%e^{-\lambda}\ast c\ast V^{\lambda}\to c
%\end{equation}
%for $\lambda\in\Lambda^+_{M, ab}$ given as follows. Consider the morphism $(v^{\lambda})^*: e^{-\lambda}\otimes V^{\lambda}\to e$ in $\Rep(\check{P}^-)$. Applying to it the functor $\act(\cdot, c): \Rep(\check{P}^-)\to C$ one gets (\ref{map_action_abstract_coming_from_P^-}). 
 
\sssec{}  Let $\Rep(\check{G})\otimes\Rep(\check{M})$ act on $C$ via the pull-back along the diagonal map $B(\check{P}^-)\to B(\check{G}\times \check{M})$. 

 Let $\check{P}^-$ act on $\check{G}/U(\check{P}^-)$ adjointly, so that the unit map $\Spec e\to \check{G}/U(\check{P}^-)$ is $\check{P}^-$-equivariant. Viewing in this way $\cO(\check{G}/U(\check{P}^-))\in \Rep(\check{P}^-)$, we get the category $\cO(\check{G}/U(\check{P}^-))-mod(C)$ and the functor
\begin{equation}
\label{functor_from_C_for_Sect_2.3.2}
C\to \cO(\check{G}/U(\check{P}^-))-mod(C) 
\end{equation} 
of restriction of scalars via the evaluation $\cO(\check{G}/U(\check{P}^-))\to e$ at $1$. 
 
\sssec{} Take for a moment $C=\QCoh(B(\check{P}^-))$. Note that 
$$
\check{P}^-\backslash(\check{M}_{ab}\times\check{G})/\check{P}^-\,\iso\, (\check{G}/[\check{P}^-, \check{P}^-])/\Ad(\check{P}^-),
$$ 
where $\check{P}^-$ acts on $\check{M}_{ab}\times\check{G}$ by left and right translations via the diagonal map $\check{P}^-\to \check{M}_{ab}\times\check{G}$. Write 
$$
j_{ab}: (\check{G}/[\check{P}^-, \check{P}^-])/\Ad(\check{P}^-)\hook{} (\overline{\check{G}/[\check{P}^-, \check{P}^-]})/\Ad(\check{P}^-)
$$
for the natural open immersion.

Consider the diagram where both squares are cartesian
$$
\begin{array}{ccccc}
&& (\overline{\check{G}/[\check{P}^-, \check{P}^-]})/\Ad(\check{P}^-)\\
& \nearrow\lefteqn{\scriptstyle i} & \uparrow\lefteqn{\scriptstyle j_{ab}}\\
B(\check{P}^-) & \gets & \check{P}^-\backslash(\check{M}_{ab}\times\check{G})/\check{P}^- & \gets & \check{M}\backslash (\check{M}_{ab}\times \check{G})/\check{P}^-\\
\downarrow &&\downarrow && \downarrow\\
B(\check{M}_{ab}\times\check{G}) &\gets & B(\check{P}^-) & \getsup{\eta} & B(\check{M}),
\end{array}
$$
we use here the diagonal maps $\check{P}^-\to \check{M}_{ab}\times\check{G}$ to form the diagram. Here $i$ is the closed immersion obtained by passing to the stack quotients under the $\check{P}^-$-actions on the unit map 
$$
\Spec e\to \overline{\check{G}/[\check{P}^-, \check{P}^-]}.
$$ 

 The functor (\ref{functor_from_C_for_Sect_2.1.1}) in these terms is nothing but $i_*$. 
By (\cite{G}, ch. I.1, 3.3.5) we have 
$$
C\otimes_{\Rep(\check{G})\otimes\Rep(\check{M}_{ab})} \Rep(\check{P}^-)\,\iso\, \QCoh(\check{P}^-\backslash(\check{M}_{ab}\times\check{G})/\check{P}^-)
$$
Now 
$$
\check{M}\backslash (\check{M}_{ab}\times G)/\check{P}^-\,\iso\, (\check{G}/[\check{P}^-, \check{P}^-])/\Ad(\check{M})
$$ 
Let 
$$
i_M: B(\check{M})\hook{} (\check{G}/[\check{P}^-, \check{P}^-])/\Ad(\check{M})
$$ 
be the closed immersion given by the point $1$. Taking for $c$ the trivial $\check{P}^-$-module $e\in C$, applying (\ref{functor_from_C_for_Sect_2.1.1}) and further (\ref{tens_product_in_C_for_Bunb_P}) one gets the direct image $(i_M)_*\cO$ of the structure sheaf on $B(\check{M})$. 

 Note that its further direct image to $B(\check{P}^-)$ identifies with $\cO(\check{P}^-/\check{M})\in \Rep(\check{P}^-)$. So, Proposition~\ref{Pp_2.0.11} gives for this $c$ an isomorphism
\begin{equation}
\label{iso_in_Rep_checkP-}
\mathop{\colim}\limits_{\lambda\in(\Lambda^+_{M, ab},\le )} e^{\lambda}\otimes (V^{\lambda})^*\,\iso\,\cO(\check{P}^-/\check{M})
\end{equation}
in $\Rep(\check{P}^-)$, here the colimit is taken in $\Rep(\check{P}^-)$.

\sssec{} Note that $i$ decomposes as
$$
B(\check{P}^-)\toup{i'} \overline{\check{G}/U(\check{P}^-)}/\Ad(\check{P}^-)\to (\overline{\check{G}/[\check{P}^-, \check{P}^-]})/\Ad(\check{P}^-)
$$
The functor (\ref{functor_from_C_for_Sect_2.3.2}) is nothing but $i'_*$. 
Note that 
$$
\check{P}^-\backslash (\check{M}\times\check{G})/\check{P}^-\,\iso\, (\check{G}/U(\check{P}^-))/\Ad(\check{P}^-),
$$
where $\check{P}^-$ acts on $\check{M}\times\check{G}$ by left and right translations via the diagonal map $\check{P}^-\to \check{M}\times\check{G}$. Write
$$
j'_{ab}: (\check{G}/U(\check{P}^-))/\Ad(\check{P}^-)\hook{} (\overline{\check{G}/U(\check{P}^-)})/\Ad(\check{P}^-)
$$
for the natural open immersion.

Consider the diagram, where both squares are cartesian
$$
\begin{array}{ccccc}
&& (\overline{\check{G}/U(\check{P}^-)})/\Ad(\check{P}^-)\\
& \nearrow\lefteqn{\scriptstyle i'} & \uparrow\lefteqn{\scriptstyle j'_{ab}}\\
B(\check{P}^-) & \gets & \check{P}^-\backslash(\check{M}\times\check{G})/\check{P}^- & \gets & \check{M}\backslash (\check{M}\times \check{G})/\check{P}^-\\
\downarrow &&\downarrow && \downarrow\\
B(\check{M}\times\check{G}) &\gets & B(\check{P}^-) & \getsup{\eta} & B(\check{M}),
\end{array}
$$
we used the diagonal map $\check{P}^-\to \check{M}\times\check{G}$ to form this diagram. 
 
 Note that
$$
\check{M}\backslash (\check{M}\times \check{G})/\check{P}^-\,\iso\, (\check{G}/U(\check{P}^-))/\Ad(\check{M})
$$ 
Let 
$$
i'_M: B(\check{M})\to (\check{G}/U(\check{P}^-))/\Ad(\check{M})
$$ 
be the closed immeersion given by the point 1. Taking for $c$ the trivial $\check{P}^-$-module $e\in C$, applying (\ref{functor_from_C_for_Sect_2.3.2}) and further (\ref{left_adj_DR-PL_for_P}) one gets $(i'_M)_*\cO\in \Hecke_{\check{G}, \check{M}}(C)$. Its further direct image to $B(\check{P}^-)$ identifies with $\cO(\check{P}^-/M)\in\Rep(\check{P}^-)$. 

So, as in Proposition~\ref{Pp_2.1.10_colimit_for_Bunt_P}, (\ref{iso_in_Rep_checkP-}) naturally lifts to an object of both $\Hecke_{\check{G}, \check{M}}(C)$ and $\Hecke_{\check{G}, \check{M}, ab}(C)$. In particular, it has the Hecke property similar to that of $\IC_{\Bunt_P}$, cf. Section~\ref{Sect_2.3.12_local_vs_global}.  

\begin{Rem} One may consider the compactification $U(\check{P}^-)\hook{} \check{G}/\check{P}$ and describe (\ref{iso_in_Rep_checkP-}) as the cohomology of the structure sheaf of $\check{G}/\check{P}$ with prescribed poles along the boundary, as the order of poles goes to infinity. Namely, for $i\in \cI$ write $s_i\in W$ for the corresponding simple reflection. Set $W^M=\{w\in W\mid \ell(w s_i)>\ell(w), \; \mbox{for all}\; i\in\cI_M\}$. 
The multiplication $W^M\times W_M\to W$ is bijective by (\cite{Sw}, 2.3.1). If $w\in W$ then $W^M\cap wW_M$ consists of a unique element denoted $w^M$. Then $\check{G}/\check{P}-U(\check{P}^-)$ is a divisor on $\check{G}/\check{P}$, whose irreducible components are the closures of $w_0\check{B}(w_0s_i)^M \check{P}=\check{B}^-s_i \check{P}$ for $i\in \cI-\cI_M$ by (\cite{Sw}, 3.3.3).
\end{Rem}

\sssec{} Let $\check{\gg}=\Lie\check{G}$, $\check{\gu}(P^-)=\Lie U(\check{P}^-)$. Let $\cO=k[[t]]\subset F=k((t))$. Write $\Gr_G$ for the affine Grassmannian of $G$ viewed as the moduli of pairs $(\cF_G, \beta)$, where $\cF_G$ is a $G$-torsor on $D=\Spec\cO$ with a trivialization $\beta: \cF_G\,\iso\,\cF^0_G\mid_{D^*}$, here $D^*=\Spec F$. Write $\Sat: \Rep(\check{G})\to Shv(\Gr_G)^{G(\cO)}$ for the Satake functor.

Write $I_P$ for the preimage of $P$ under $G(\cO)\to G$, this is a parahoric subgroup. Consider the full subcategory 
$$
Shv(\Gr_G)^{I_P, constr}\subset Shv(\Gr_G)^{I_P}
$$
of those objects whose image in $Shv(\Gr_G)$ is compact. Then $Shv(\Gr_G)^{I_P, constr}\in\DGCat^{non-cocmpl}$. Set
$$
Shv(\Gr_G)^{I_P, ren}=\Ind(Shv(\Gr_G)^{I_P, constr})
$$

 The renormalization is a general procedure, for algebraic stacks locally of finite type with an affine diagonal it is studied in (\cite{AGKRRV}, F.5). As in Section~\ref{Sect_A.5.4}, we have an adjoint pair $Shv(\Gr_G)^{I_P}\leftrightarrows Shv(\Gr_G)^{I_P, ren}$ in $\DGCat_{cont}$, where the left adjoint is fully faithful.  

\sssec{} 
\label{Sect_2.1.4_objects_cB}
For $\mu\in\Lambda^+_M$ write $\cB_{\mu, !}, \cB_{\mu, *}\in Shv(\Gr_G)^{I_P}$ for the $\IC$-sheaf of $I_Pt^{\mu}G(\cO)/G(\cO)$ extended by zero (resp., by $*$-extension) to $\Gr_G$. 
 
 If $\mu\in\Lambda_{M, ab}$ then $I_Pt^{\mu}G(\cO)/G(\cO)=It^{\mu}G(\cO)/G(\cO)$, where $I\subset G(\cO)$ is the Iwahori subgroup. If moreover $\mu\in\Lambda^+_{M, ab}$ then $It^{\mu}G(\cO)/G(\cO)=U(\cO)t^{\mu}G(\cO)/G(\cO)=S_B^{\mu}\cap \ov{\Gr}_G^{\mu}$ by (\cite{MV}, proof of Theorem~3.2). In the latter case the open embedding $I_Pt^{\mu}G(\cO)/G(\cO)\hook{} \ov{\Gr}_G^{\mu}$ is affine by (\cite{MV}, 3.1), so that $\cB_{\mu, !}, \cB_{\mu, *}$ are perverse. Note that for $\mu\in\Lambda^+_{M, ab}$ we have a canonical map
\begin{equation}
\label{map_from_Sat(Vmu)_to_cB_mu*}
\Sat(V^{\mu})\to \cB_{\mu, *}
\end{equation}
of $I_P$-equivariant perverse sheaves on $\Gr_G$. 

\sssec{} Let $\Fl_P=G(F)/I_P$ and $\cH_P(G)=Shv(\Fl_P)^{I_P}$. It is well known that 
$(\cH_P(G), *)$ acts on $Shv(\Gr_G)^{I_P}$ by convolutions. It also similarly acts on 
$Shv(\Gr_G)^{I_P, ren}$. 

 Indeed, $\cH_P(G)$ is compactly generated. So, it suffices to show that $\cH_P(G)^c$ acts on $Shv(\Gr_G)^{I_P, constr}$ naturally. Given $K\in \cH_P(G)^c$, there is a $I_P$-invariant closed subscheme of finite type $Y\subset \Fl_P$ such that $\oblv(K)\in Shv(\Fl_P)$ is the extension by zero from $Y$. Let $\tilde Y\subset G(F)$ be the preimage of $Y$ under $G(F)\to \Fl_P$. The desired claim follows now from the fact that the convolution map $\tilde Y\times^{I_P} \Gr_G\to \Gr_G$ is proper.
 
  Write $\tilde W$ for the affine extended Weyl group of $(G, T)$. The $I_P$-orbits on $\Fl_P$ are indexed by $W_M\backslash \tilde W/W_M$. For $w\in \tilde W$ write $j_{w, !}, j_{w, *}\in \cH_P(G)$ for the standard and costandard objects attached to $w\in \tilde W$ and normalized to be perverse on the $I_P$-orbit $I_PwI_P/I_P$. For $\lambda\in\Lambda$ we write for brevity $j_{\lambda, !}=j_{t^{\lambda}, !}$ and $j_{\lambda, *}=j_{t^\lambda, *}$. 

 For the definition of the category $\IndCoh$ on a quasi-smooth Artin stack with a specified singular support condition we refer to (\cite{AG}, Section~8). Gurbir Dhillon and Harrison Chen have proven the following at a point \cite{CD}; and Justin Campbell and Sam Raskin over the Ran space (\cite{CaRas}, Theorem~6.12.3).  
\begin{Pp} 
\label{Pp_Chen_Dhillon}
There is a canonical equivalence
\begin{equation}
\label{iso_Gurbir_Chen}
\IndCoh((\check{\gu}(P^-)\times_{\check{\gg}} 0)/\check{P}^-)\,\iso\, Shv(\Gr_G)^{I_P, ren}
\end{equation}
with the following properties:\\
(i) The $\Rep(\check{G})$-action on $\IndCoh((\check{\gu}(P^-)\times_{\check{\gg}} 0)/\check{P}^-)$ arising from the projection
$$
(\check{\gu}(P^-)\times_{\check{\gg}} 0)/\check{P}^-\to pt/\check{P}^-\to pt/\check{G}
$$
corresponds to the $\Rep(\check{G})$-action on $Shv(\Gr_G)^{I_P, ren}$ via $\Sat: \Rep(\check{G})\to Shv(\Gr_G)^{G(\cO)}$ and the right convolutions.

\smallskip\noindent
(ii) The $\Rep(\check{M}_{ab})$-action on $\IndCoh((\check{\gu}(P^-)\times_{\check{\gg}} 0)/\check{P}^-)$ arising from the projection
$$
(\check{\gu}(P^-)\times_{\check{\gg}} 0)/\check{P}^-\to pt/\check{P}^-\to pt/\check{M}\to pt/\check{M}_{ab}
$$
corresponds to the $\Rep(\check{M}_{ab})$-action on $Shv(\Gr_G)^{I_P, ren}$ such that for $\lambda\in\Lambda_{M, ab}^+$, $e^{\lambda}$ sends $F$ to $j_{\lambda, *}\ast F$. So, it comes from the monoidal functor (\ref{mon_functor_*}). 
  
\smallskip\noindent
iii) The object $\cO_{pt/\check{P}^-}\in \IndCoh((\check{\gu}(P^-)\times_{\check{\gg}} 0)/\check{P}^-)$ corresponds under (\ref{iso_Gurbir_Chen}) to $\delta_{1,\Gr_G}\in Shv(\Gr_G)^{I_P, ren}$. 

\smallskip\noindent
iv)  For $\lambda\in\Lambda^+_{M, ab}$ the map $(v^{\lambda})^*: V^{\lambda}\to e^{\lambda}$ in $\Rep(\check{P}^-)$ corresponds under (\ref{iso_Gurbir_Chen}) to the morphism $\Sat(V^{\lambda})\to \cB_{\lambda,*}$ in $Shv(\Gr_G)^{I_P, ren}$ given by (\ref{map_from_Sat(Vmu)_to_cB_mu*}).

\smallskip\noindent
v) The equivalence (\ref{iso_Gurbir_Chen}) restricts to an equivalence of full subcategories
$$
\IndCoh_{\Nilp}((\check{\gu}(P^-)\times_{\check{\gg}} 0)/\check{P}^-)\,\iso\, Shv(\Gr_G)^{I_P},
$$
here $\Nilp$ stands for the nilpotent singular support.
\end{Pp}
% One makes the following renormalization procedure in the RHS. Namely, consider the full subcategory $\cD\subset Shv(\Gr_G)^{U(P)(F)M(\cO)}$ of those object whose image in $Shv(\Gr_G)^{U(P)(F)}$ remain compact. Then $Shv(\Gr_G)^{U(P)(F)M(\cO), ren}$ is defined as $\Ind(\cD)$. 

\begin{Rem} In fact, the $\Rep(\check{M})$-action on $Shv(\Gr_G)^{H, ren}$
given below by (\ref{action_Rep(checkM)_shifted}) is related 
to the $\Rep(\check{M})$-action on $\IndCoh((\check{\gu}(P^-)\times_{\check{\gg}} 0)/\check{P}^-)$ arising from the projection
$$
(\check{\gu}(P^-)\times_{\check{\gg}} 0)/\check{P}^-\to pt/\check{P}^-\to pt/\check{M}
$$
via (\ref{iso_Gurbir_Chen}) composed with the equivalence (\ref{eq_ren_parahoric_versus_H}). This is established in (\cite{CaRas}, Theorem~6.12.3) for the $\cD$-modules version.
\end{Rem}

\sssec{} Now take $C=\IndCoh((\check{\gu}(P^-)\times_{\check{\gg}} 0)/\check{P}^-)$ equipped with an action of $\Rep(\check{P}^-)$ coming from the projection
$$
(\check{\gu}(P^-)\times_{\check{\gg}} 0)/\check{P}^-\to B(\check{P}^-)
$$
Let $c\in C$ be the direct image of the structure sheaf $\cO$ along the closed embedding
$$
\{0\}/\check{P}^-\to (\check{\gu}(P^-)\times_{\check{\gg}} 0)/\check{P}^-
$$
Applying to $c$ the functors (\ref{functor_from_C_for_Sect_2.1.1}) and further (\ref{left_adjoint_ab_case}), one gets an object of $\Hecke_{\check{G},\check{M} ,ab}(C)$
denoted $\cM_{\check{G},\check{P}^-}$. This is a version of the dual baby Verma object we are interested in.

\sssec{} 
\label{Sect_2.3.10_almost_final_version}
For future applications, we write down an analog of the isomorphism (\ref{iso_in_Rep_checkP-}) with $\check{P}^-$ replaced by $\check{P}$. Take for a moment $C=\QCoh(B(\check{P}))$. Consider the diagram, where both squares are cartesian
$$
\begin{array}{ccccc}
B(\check{P}) & \gets & \check{P}\backslash(\check{M}_{ab}\times\check{G})/\check{P} & \gets & \check{M}\backslash (\check{M}_{ab}\times \check{G})/\check{P}\\
\downarrow &&\downarrow && \downarrow\\
B(\check{M}_{ab}\times\check{G}) &\gets & B(\check{P}) & \gets & B(\check{M}),
\end{array}
$$
we use here the diagonal maps $\check{P}\to \check{M}_{ab}\times\check{G}$ to form the diagram. % It gives 
%$$
%C\otimes_{\Rep(\check{G})\otimes\Rep(\check{M}_{ab})} \Rep(\check{P})\,\iso\, \QCoh(\check{P}\backslash(\check{M}_{ab}\times\check{G})/\check{P})
%$$
%Note that 
%$$
%\check{P}\backslash(\check{M}_{ab}\times\check{G})/\check{P}\,\iso\, (\check{G}/[\check{P}, \check{P}])/\Ad(\check{P})
%$$ 
%Consider the closed immersion 
%$$
%i: B(\check{P})\hook{} (\check{G}/[\check{P}, \check{P}])/\Ad(\check{P})
%$$

 The analog of (\ref{functor_from_C_for_Sect_2.1.1}) for $\check{P}$ is the functor
\begin{equation}
\label{functor_from_C_for_checkP_Sect_2.1.6}
C\to \cO(\check{G}/[\check{P}, \check{P}])-mod(C)
\end{equation} 
sending $c$ to itself with the action maps $e^{\lambda}\ast c\ast (V^{\lambda})^*\to c$ obtained by appling the functor $\act(\cdot, c): \Rep(\check{P})\to C$ to $v^{\lambda}: e^{\lambda}\otimes (V^{\lambda})^*\to e$. 

 Taking for $c$ the trivial $\check{P}$-module, applying (\ref{functor_from_C_for_checkP_Sect_2.1.6}) and further the pullbacks
$$
\cO(\check{G}/[\check{P}, \check{P}])-mod(C)\to C\otimes_{\Rep(\check{G})\otimes\Rep(\check{M}_{ab})} \Rep(\check{P})\to C\otimes_{\Rep(\check{G})\otimes\Rep(\check{M}_{ab})} \Rep(\check{M}),
$$
one gets an object of $\QCoh(\check{M}\backslash (\check{M}_{ab}\times \check{G})/\check{P})$ whose direct image to $B(\check{P})$ identifies with $\cO(\check{P}/\check{M})\in \Rep(\check{P})$. So, an analog of Proposition~\ref{Pp_2.0.11} gives for this $c$ an isomorphism 
\begin{equation}
\label{iso_O(P/M)_for_Sect_2.1.6}
\mathop{\colim}_{\lambda\in(\Lambda^+_{M, ab},\le )} e^{-\lambda}\otimes V^{\lambda}\,\iso\, \cO(\check{P}/\check{M})
\end{equation}
in $\Rep(\check{P})$. Here the colimit is taken in $\Rep(\check{P})$, and the inductive system is described in Section~\ref{Sect_2.0.17}. 

 Exchanging the roles of $P$ and $P^-$, one gets an analog of the object $\cM_{\check{G}, \check{P}^-}$ denoted by $\cM_{\check{G}, \check{P}}\in\Hecke_{\check{G}, \check{M}, ab}(C')$, where $C'=\IndCoh((\check{\gu}(P)\times_{\check{\gg}} 0)/\check{P})$. 

\sssec{} We need the following generalization of the isomorphism (\ref{iso_O(P/M)_for_Sect_2.1.6}). Fix $\eta\in\Lambda^+_M$. Consider the diagram
\begin{equation}
\label{diag_for_Sect_2.1.7}
\{\lambda\in\Lambda_{M, ab}\mid \lambda+\eta\in\Lambda^+\}\to \Rep(\check{P}), \;\lambda\mapsto e^{-\lambda}\otimes V^{\lambda+\eta}
\end{equation}
Here we consider $\{\lambda\in\Lambda_{M, ab}\mid \lambda+\eta\in\Lambda^+\}$ with the relation $\lambda_1\le \lambda_2$ iff $\lambda_2-\lambda_1\in\Lambda^+$. This is not an order relation, but defines instead a structure of a filtered category on this set. 

 For $\lambda_i\in \Lambda_{M, ab}$ with $\lambda_i+\eta\in\Lambda^+$ and $\lambda=\lambda_2-\lambda_1\in\Lambda^+$ the transition map  
$$
e^{-\lambda_1}\otimes V^{\lambda_1+\eta}\to e^{-\lambda_2}\otimes V^{\lambda_2+\eta}
$$
is the composition
\begin{multline*}
e^{-\lambda_1}\otimes V^{\lambda_1+\eta}\,\iso\, e^{-\lambda_2}\otimes e^{\lambda}\otimes V^{\lambda_1+\eta}\toup{v^{\lambda}} e^{-\lambda_2}\otimes V^{\lambda}\otimes V^{\lambda_1+\eta}\toup{u^{\lambda, \lambda_1+\eta}} e^{-\lambda_2}\otimes V^{\lambda_2+\eta},
\end{multline*}

 Write $\coind_{\check{M}}^{\check{P}}: \Rep(\check{M})\to \Rep(\check{P})$ for the right adjoint to the restriction functor.
\begin{Lm} 
\label{Lm_2.1.8_some_colimit}
Let $\eta\in\Lambda^+_M$. One has canonically in $\Rep(\check{P})$
\begin{equation}
\label{iso_for_Lm_2.1.8}
\mathop{\colim}\limits_{\lambda\in\Lambda_{M, ab},\; \lambda+\eta\in\Lambda^+} \; e^{-\lambda}\otimes V^{\lambda+\eta}\,\iso\, \coind_{\check{M}}^{\check{P}}(U^{\eta})
\end{equation}
\end{Lm}
\begin{proof}
{\bf Step 1}. Let $\lambda\in\Lambda_{M, ab}$ with $\lambda+\eta\in\Lambda^+$. By Frobenius reciprocity, a $\check{P}$-equivariant map $e^{-\lambda}\otimes V^{\lambda+\eta}\to \coind_{\check{M}}^{\check{P}}(U^{\eta})$ is the same as a $\check{M}$-equivariant map $e^{-\lambda}\otimes V^{\lambda+\eta}\to U^{\eta}$. The latter is the same as a $\check{M}$-equivariant morphism 
$$
V^{\lambda+\eta}\to U^{\eta}\otimes e^{\lambda}\,\iso\, U^{\lambda+\eta}\,\iso\, \coind_{\check{B}^-_M}^{\check{M}} e^{\lambda+\eta}
$$ 
The latter comes from the $\check{B}^-_M$-equivariant morphism $(v^{\lambda+\eta})^*: V^{\lambda+\eta}\to e^{\lambda+\eta}$. 

 It is easy to check that the morphism so obtained are compatible with the transition maps in the diagram (\ref{diag_for_Sect_2.1.7}). It remains to check that the obtained map (\ref{iso_for_Lm_2.1.8}) is an isomorphism. 
 
  The $\check{P}$-module $\coind_{\check{M}}^{\check{P}}(U^{\eta})$ identifies with $\cO(\check{P}/\check{M})\otimes U^{\eta}$, where $\check{P}$ acts diagonally. Here $\check{P}$ acts by left translations on $\check{P}/\check{M}$,  by functoriality on $\cO(\check{P}/\check{M})$, and via the quotient $\check{P}\to \check{M}$ on $U^{\eta}$.
   
\medskip\noindent
{\bf Step 2} Assume in addition $\eta\in\Lambda^+$. Then we construct a morphism of $\check{P}$-modules 
\begin{equation}
\label{map_for_proof_of_Lm_2.1.8}
\coind_{\check{M}}^{\check{P}}(U^{\eta})\to \mathop{\colim}\limits_{\lambda\in\Lambda_{M, ab},\; \lambda+\eta\in\Lambda^+} \; e^{-\lambda}\otimes V^{\lambda+\eta}
\end{equation}
as follows. For any $\lambda\in\Lambda^+_{M, ab}$ consider the morphism
$$
e^{-\lambda}\otimes V^{\lambda}\otimes V^{\eta}\toup{v^{\lambda,\eta}} e^{-\lambda}\otimes V^{\lambda+\eta}
$$ 
in $\Rep(\check{P})$. These morphisms are compatible with the transition maps in the inductive systems (\ref{diag_for_Sect_2.1.7}) and (\ref{iso_O(P/M)_for_Sect_2.1.6}). Passing to the colimit, from (\ref{iso_O(P/M)_for_Sect_2.1.6}) we get a morphism
$$
\cO(\check{P}/\check{M})\otimes V^{\eta}\to  \mathop{\colim}\limits_{\lambda\in\Lambda_{M, ab},\; \lambda+\eta\in\Lambda^+} \; e^{-\lambda}\otimes V^{\lambda+\eta}
$$
in $\Rep(\check{P})$. Now (\ref{map_for_proof_of_Lm_2.1.8}) is defined as the restriction of the latter map under $U^{\eta}\hook{} V^{\eta}$. The two morphisms so obtained are inverse of each other. 

\medskip\noindent
{\bf Step 3}. Let now $\eta\in\Lambda^+_M$. We reduce our claim to the case of Step 2 as follows. Pick $\lambda_0\in\Lambda^+_{M, ab}$ such that $\eta_0=\eta+\lambda_0\in\Lambda^+$. The LHS of (\ref{iso_for_Lm_2.1.8}) identifies with
$$
\mathop{\colim}\limits_{\lambda\in\Lambda_{M, ab},\; \lambda+\lambda_0+\eta\in\Lambda^+} \; e^{-\lambda-\lambda_0}\otimes V^{\lambda+\lambda_0+\eta}\,\iso\, e^{-\lambda_0}\otimes \coind_{\check{M}}^{\check{P}}(U^{\eta_0})
$$
by Steps 1 and 2. By the projection formula,  
$$
e^{-\lambda_0}\otimes \coind_{\check{M}}^{\check{P}}(U^{\eta_0})\,\iso\, \coind_{\check{M}}^{\check{P}}(e^{-\lambda_0}\otimes U^{\eta_0})
$$
Since $e^{-\lambda_0}\otimes U^{\eta_0}\,\iso\, U^{\eta}$ in $\Rep(\check{M})$, we are done.
\end{proof}

\section{Parabolic semi-infinite category of sheaves}
\ssec{Finite-dimensional counterpart}

\sssec{} 
\label{Sect_Local automorphic side_begins}
Let us explain that $H:=U(P)(F)M(\cO)$ is a placid ind-scheme. We equip $\Lambda^+_{M, ab}$ with the relation $\le$ as in Section~\ref{Sect_2.0.11_functor_f}. For $\lambda\in\Lambda^+_{M, ab}$ set $H_{\lambda}=t^{-\lambda}P(\cO)t^{\lambda}$. This is a placid group scheme. If $\lambda\le\mu$ in $\Lambda^+_{M, ab}$ then $H_{\lambda}\subset H_{\mu}$ is a placid closed immersion, and $H\,\iso\, \mathop{\colim}\limits_{\lambda\in\Lambda^+_{M, ab}} H_{\lambda}$ a placid ind-scheme. So, for $C\in Shv(H)-mod$, $C^H$ makes sense. 

We will relate the RHS of (\ref{iso_Gurbir_Chen}) to $Shv(\Gr_G)^{H}$ in way similar to \cite{Gai19SI}. 

 Note that $H$-orbits on $\Gr_G$ are indexed by $\Lambda^+_M$, to $\mu\in\Lambda^+_M$ we attach the orbit passing through $t^{\mu}$. 
 
\sssec{} By (\cite{LyWhit_loc_glob}, 1.3.4), we get $Shv(\Gr_G)^{H}\,\iso\,\mathop{\lim}\limits_{\lambda\in (\Lambda^+_{M, ab})^{op}} Shv(\Gr_G)^{H_{\lambda}}$. For each $\lambda$ the functor 
$$
\oblv: Shv(\Gr_G)^{H_{\lambda}}\to Shv(\Gr_G)^{M(\cO)}
$$ 
is a full embedding by Section~\ref{Sect_A.0.2}. By (\cite{Ly}, 2.7.7), $\mathop{\lim}\limits_{\lambda\in (\Lambda^+_{M, ab})^{op}} Shv(\Gr_G)^{H_{\lambda}}$ is a full subcategory of $Shv(\Gr_G)^{M(\cO)}$ equal to 
$$
\mathop{\cap}\limits_{\lambda\in \Lambda^+_{M, ab}} Shv(\Gr_G)^{H_{\lambda}}
$$ 
taken in $Shv(\Gr_G)^{M(\cO)}$.
 
\sssec{} Recall that for any smooth affine algebraic group $\cG$ of finite type, $\cG(F)$ is a placid ind-scheme (cf. \cite{Ly4}, 0.0.51). So, $P(F)$ is a placid ind-scheme, and 
$$
P(F)/H\,\iso\, M(F)/M(\cO)\,\iso\, \Gr_M
$$ 
is an ind-scheme of ind-finite type. 

 As in Section~\ref{Sect_A.1}, one gets an action of $Shv(\Gr_M)^{M(\cO)}$ on $Shv(\Gr_G)^H$. In fact, $Shv(M(F))$ acts naturally on $Shv(\Gr_G)^{U(F)}$, and the desired $Shv(\Gr_M)^{M(\cO)}$-action is obtained by functoriality after passing to $M(\cO)$-invariants, cf. Remark~\ref{Rem_A.1.3_action_of_M(F)}.
%Consider also the renormalized version $Shv(\Gr_M)^{M(\cO), ren}$ defined as follows. First, we consider the full subcategory $\cD_0$ of $K\in Shv(\Gr_M)^{M(\cO)}$ whose image under $\oblv: Shv(\Gr_M)^{M(\cO)}\to Shv(\Gr_M)$ remain compact. Set $Shv(\Gr_M)^{M(\cO), ren}=\Ind(\cD_0)$. 

 Composing $\Rep(\check{M})\to Shv(\Gr_M)^{M(\cO)}$ one gets a $\Rep(\check{M})$-action on $Shv(\Gr_G)^{H}$. 
 
\sssec{} 
\label{Sect_3.1.4_right_now}
Write $I$ for the Iwahori subgroup. Let $\Fl=G(F)/I$ be the affine flags. Write $\cH(G)=Shv(\Fl)^I$ for the geometric Iwahori-Hecke algebra. For $\lambda\in\Lambda$ write $j_{\lambda, !}^I, j_{\lambda, *}^I$ for the corresponding objects of $\cH(G)$ attached to $t^\lambda$. Write $\ast^I$ for the convolution in $\cH(G)$. More generally, for $w\in\tilde W$ we have the standard/costandard objects $j_{w, !}^I, j_{w, *}^I\in \cH(G)$.

\begin{Lm} 
\label{Lm_2.2.5} i) Let $\lambda\in\Lambda^+_{M, ab}$. Then $j_{-\lambda, *}\ast j_{\lambda, !}\,\iso\, \delta_1$ in $\cH_P(G)$.\\
ii) Let $\lambda,\mu\in \Lambda^+_{M, ab}$. Then $j_{\lambda, !}\ast j_{\mu, !}\,\iso\, j_{\lambda+\mu, !}$ in $\cH_P(G)$. More generally, for $\oblv: Shv(\Gr_G)^{I_P}\to Shv(\Gr_G)^I$ and $F\in Shv(\Gr_G)^{I_P}$ one has
$$
\oblv(j_{\lambda, !}\ast F)\,\iso\, j_{\lambda, !}^I \ast^I \oblv(F),\;\;\;\;\;
\oblv(j_{\lambda, *}\ast F)\,\iso\, j_{\lambda, *}^I \ast^I \oblv(F)
$$ 
in $Shv(\Gr_G)^I$. 
\end{Lm}
\begin{proof} 
Let $\lambda\in\Lambda^+_{M, ab}$. Then the natural map $It^{\lambda}I/I\to I_Pt^{\lambda}I_P/I_P$ is an isomorphism, both are affine spaces of dimension $\<\lambda, 2\check{\rho}\>$. For the natural map $\tau: \Fl\to \Fl_P$ we get $\tau_!(j_{\lambda, !}^I)\,\iso\, j_{\lambda, !}$. Note that $\tau$ is proper.

 Similarly, the map $It^{-\lambda}I/I\to I_Pt^{-\lambda}I_P/I_P$ is an isomorphism, so that $\tau_!j_{-\lambda, *}^I\,\iso\, j_{-\lambda, *}$. 

\medskip\noindent
ii) We have 
\begin{equation}
\label{equality_of_sets_for_Lm2.2.5}
I_Pt^{\lambda}I_P\times^{I_P} \Fl_P\,\iso\, It^{\lambda}I\times^{I} \Fl_P
\end{equation}
Recall that $j^I_{\lambda, !}\ast^I j^I_{\mu, !}\,\iso\, j_{\lambda+\mu, !}^I$. Applying $\tau_!$ to this isomorphism, one gets the desired result. More generally, for any $F\in Shv(\Fl_P)^{I_P}$ let $\oblv(F)\in Shv(\Fl_P)^I$ then (\ref{equality_of_sets_for_Lm2.2.5}) gives
$$
\oblv(j_{\lambda, !}\ast F)\,\iso\, j_{\lambda, !}^I \ast^I \oblv(F)
$$ 
in $Shv(\Gr_G)^I$. 

\medskip\noindent
i) We have $j_{-\lambda, *}^I\ast^I j_{\lambda, !}^I\,\iso\, \delta_{1, \Fl}$ in $\cH(G)$. 
Applying $\tau_*$ this gives $j_{-\lambda, *}^I \ast^I j_{\lambda, !}\,\iso\, \delta_{1, \Fl_P}$. Finally $j_{-\lambda, *}^I \ast^I j_{\lambda, !}\,\iso\, j_{-\lambda, *}\ast j_{\lambda, !}$ from (\ref{equality_of_sets_for_Lm2.2.5}) also. 
\end{proof}

 From this lemma we conclude that there are monoidal functors $\Lambda_{M, ab}\to \cH_P(G)$, 
\begin{equation}
\label{mon_functor_*}
\lambda\mapsto j_{\lambda, *}, \;\lambda\in\Lambda^+_{M, ab}
\end{equation}
and 
\begin{equation}
\label{mon_functor_!}
\lambda\mapsto j_{\lambda, !}, \;\lambda\in\Lambda^+_{M, ab}
\end{equation} 

\sssec{} Consider the forgetful functor $\oblv: Shv(\Gr_G)^{I_P}\to Shv(\Gr_G)^{M(\cO)}$. It has a continuous right adjoint denoted $\Av^{I_P/M(\cO)}_*$ by Section~\ref{Section_A2_some_invarinats}. Define now $Shv(\Gr_G)^{M(\cO), ren}$ as follows. Let $Shv(\Gr_G)^{M(\cO), constr}\subset Shv(\Gr_G)^{M(\cO)}$ be the full subcategory of those objects which remain compact in $Shv(\Gr_G)$. Set 
$$
Shv(\Gr_G)^{M(\cO), ren}=\Ind(Shv(\Gr_G)^{M(\cO), constr})
$$ 

 We have the evident forgetful functor 
$$
\oblv: Shv(\Gr_G)^{I_P, constr}\to Shv(\Gr_G)^{M(\cO), constr}
$$ 
By construction of $\Av^{I_P/M(\cO)}_*$, we actually get an adjoint pair
$$
\oblv: Shv(\Gr_G)^{I_P, constr}\leftrightarrows Shv(\Gr_G)^{M(\cO), constr}: \Av^{I_P/M(\cO)}_*
$$
Their ind-extensions also give an adjoint pair
$$
\oblv^{ren}: Shv(\Gr_G)^{I_P, ren}\leftrightarrows Shv(\Gr_G)^{M(\cO), ren}: \Av^{I_P/M(\cO), ren}_*
$$ 
This is a general phenomenon, see Remark~\ref{Rem_A.2.2}.
%The functor $\oblv^{ren}$ preserves compact objects. Indeed, the operation of passing to the Karoubi envelope is functorial. So, for any map $h: E_1\to E_2$ in $\DGCat^{non-cocmpl}$, the functor $\Ind(h): \Ind(E_1)\to \Ind(E_2)$ preserves compact objects.

\begin{Pp} 
\label{Pp_2.2.6}
The functor $\Av^{I_P/M(\cO)}_*$ restricted to $Shv(\Gr_G)^H\subset Shv(\Gr_G)^{M(\cO)}$ defines an equivalence
\begin{equation}
\label{equiv_Iwahori_vs_SI}
Shv(\Gr_G)^H\to Shv(\Gr_G)^{I_P}
\end{equation}
\end{Pp} 
\sssec{} Now we define the renormalized version $Shv(\Gr_G)^{H, ren}$ as follows. Denote by  
$$
Shv(\Gr_G)^{H, constr}\subset Shv(\Gr_G)^H
$$ 
the full subcategory that corresponds under  (\ref{equiv_Iwahori_vs_SI}) to $Shv(\Gr_G)^{I_P, constr}\subset Shv(\Gr_G)^{I_P}$. Set $Shv(\Gr_G)^{H, ren}=\Ind(Shv(\Gr_G)^{H, constr})$. 
 
\sssec{Proof of Proposition~\ref{Pp_2.2.6}} The fully faithful inclusion $Shv(\Gr_G)^{H}\subset Shv(\Gr_G)^{M(\cO)}$ admits a left adjoint $\Av_!^{U(P)(F)}: Shv(\Gr_G)^{M(\cO)}\to Shv(\Gr_G)^{H}$. So, we get adjoint pairs
$$
Shv(\Gr_G)^{I_P} \leftrightarrows Shv(\Gr_G)^{M(\cO)}\leftrightarrows Shv(\Gr_G)^H
$$
where the left composition is $\Av_!^{U(P)(F)}\oblv$, and the right composition is (\ref{equiv_Iwahori_vs_SI}).

\smallskip
\noindent
{\bf Step 1} We equip $\Lambda^+_{M, ab}$ with the relation $\le$ as in Section~\ref{Sect_2.0.11_functor_f}. For $\lambda\in\Lambda^+_{M, ab}$ set $U_{\lambda}=t^{-\lambda}U(P)(\cO)t^{\lambda}$. This is a placid group scheme.
 Given $\lambda,\mu\in \Lambda^+_{M, ab}$ with $\lambda\le\mu$ we get a placid closed immersion $U_{\lambda}\subset U_{\mu}$, and 
$$
\mathop{\colim}\limits_{\lambda\in\Lambda^+_{M, ab}} U_{\lambda}\,\iso\, U(P)(F)
$$ 
is a placid ind-scheme. 

 If $\lambda\in \Lambda^+_{M, ab}$ then $U_{\lambda}G(\cO)/G(\cO)=U_{\lambda}/U_0$ is an affine space of dimension $\<\lambda, 2\check{\rho}-2\check{\rho}_M\>=\<\lambda, 2\check{\rho}\>$. For $\lambda\in\Lambda_{M, ab}$ set 
$$
I_P^{\lambda}=t^{-\lambda}I_Pt^{\lambda}
$$  

  A version of the Iwahori decomposition for $P$ is 
$$
I_P=U(P^-)(\cO)_1M(\cO)U(P)(\cO)
$$ 
with $U(P^-)(\cO)_1=\Ker(U(P^-)(\cO)\to U(P^-))$. For $\lambda\in\Lambda^+_{M, ab}$, $t^{-\lambda}U(P^-)(\cO)_1t^{\lambda}\subset U(P^-)(\cO)_1$, so 
$$
I_P^{\lambda}\subset U_{\lambda}M(\cO)U(P^-)(\cO)_1
$$
and 
$$
I^{\lambda}_P I_P/I_P=U_{\lambda}I_P/I_P=U_{\lambda}/U_0
$$ 
Consider the action map $a: I^{\lambda}_PI_P\times^{I_P} \Gr_G\to \Gr_G$. For $F\in Shv(\Gr_G)^{I_P}$ the object $t^{-\lambda}j_{\lambda, !}\ast F[\<\lambda, 2\check{\rho}\>]$ writes as 
$$
a_!(\IC\tboxtimes F)[\<\lambda, 2\check{\rho}\>],
$$ 
where the functor of the twisted exteriour product $\tboxtimes$ is normalized to preserve perversity, and $\IC=e[\<\lambda, 2\check{\rho}\>]$ is the $\IC$-sheaf of the affine space $I^{\lambda}_P I_P/I_P$. We see that the composition 
$$
Shv(\Gr_G)^{I_P}\,\toup{\oblv}\, Shv(\Gr_G)^{P(\cO)}\toup{\Av^{U_{\lambda}}_!} Shv(\Gr_G)^{M(\cO)U_{\lambda}}
$$ 
identifies with the functor $F\mapsto t^{-\lambda}j_{\lambda, !}\ast F[\<\lambda, 2\check{\rho}\>]$. %The meaning of the shift should be clarified later when considering the t-structure on $Shv(\Gr_G)^H$. 

\smallskip\noindent
{\bf Step 2} Consider the left adjoint $\Av_!^{U(P)(F)}: Shv(\Gr_G)^{M(\cO)}\to Shv(\Gr_G)^H$ to the inclusion. Recall that it is given as
\begin{equation}
\label{functor_Av_!_for_Pp2.2.7}
F\mapsto \mathop{\colim}_{\lambda\in\Lambda^+_{M, ab}} \Av^{U_{\lambda}}_!(F)
\end{equation}
as in Lemma~\ref{Lm_A.2.3}.

 For $K\in Shv(\Gr_G)^{M(\cO)}$ and $\lambda\in\Lambda_{M, ab}$ one has canonically
\begin{equation}
\label{iso_for_Step2_Av!_versus_t_lambda}
t^{-\lambda}\Av_!^{U(P)(F)}(t^{\lambda}K)\,\iso\, \Av_!^{U(P)(F)}(K)
\end{equation}
Indeed, for $L\in Shv(\Gr_G)^H$
\begin{multline*}
\HOM_{Shv(\Gr_G)^{M(\cO)}}(t^{-\lambda}\Av_!^{U(P)(F)}(t^{\lambda}K), L)\,\iso\,\HOM_{Shv(\Gr_G)^{M(\cO)}}(\Av_!^{U(P)(F)}(t^{\lambda}K), t^{\lambda}L)\\ \iso\, \HOM_{Shv(\Gr_G)^{M(\cO)}}(t^{\lambda}K, t^{\lambda}L)\,\iso\, \HOM_{Shv(\Gr_G)^{M(\cO)}}(K, L)
\end{multline*}

 Now from Step 1 we see that for $\lambda\in\Lambda^+_{M, ab}$ and $\cF\in Shv(\Gr_G)^{I_P}$ one has
\begin{equation}
\label{iso_Av_!_intertwines}
t^{\lambda}\Av_!^{U(P)(F)}(\cF)[-\<\lambda, 2\check{\rho}\>]\,\iso\, \Av_!^{U(P)(F)}(j_{\lambda, !}\ast \cF)
\end{equation} 

\smallskip\noindent
{\bf Step 3} Let us show that the unit of the adjunction 
$$
\cF\to \Av_*^{I_P/M(\cO)}\Av_!^{U(P)(F)}(\cF)
$$
is an isomorphism for $\cF\in Shv(\Gr_G)^{I_P}$. First, for $\lambda\in\Lambda^+_{M, ab}$ and $\cF'\in Shv(\Gr_G)^{I_P}$ we have $t^{-\lambda}\cF'\in Shv(\Gr_G)^{I_P^{\lambda}}$ naturally. Now the composition
$$
Shv(\Gr_G)^{I_P^{\lambda}}\toup{\oblv} Shv(\Gr_G)^{M(\cO)}\,\toup{\Av_*^{I_P/M(\cO)}}Shv(\Gr_G)^{I_P}
$$
identifies with the composition
$$
Shv(\Gr_G)^{I_P^{\lambda}}\toup{\oblv} Shv(\Gr_G)^{I_P^{\lambda}\cap I_P} \,\toup{\Av_*^{I_P/I_P^{\lambda}\cap I_P}}\; Shv(\Gr_G)^{I_P},
$$
because for for a prounipotent group the inclusion of invariants is fully faithful. The latter functor writes as $K\mapsto \act_*(\IC\tboxtimes K)$ for the action map $\act: I_PI^{\lambda}_P\times^{I^{\lambda}_P} \Gr_G\to \Gr_G$. 

This gives
$$
\act_*(\IC\tboxtimes t^{-\lambda}\cF)\,\iso\, j_{-\lambda, *}\ast \cF[-\<\lambda, 2\check{\rho}\>],
$$ 
because $I_P I_P^{\lambda}=I_Pt^{-\lambda}I_Pt^{\lambda}$. So, for $\cF\in Shv(\Gr_G)^{I_P}$ one gets canonically
\begin{equation}
\label{iso_for_Step3_first} 
\Av_*^{I_P/M(\cO)}(t^{-\lambda}\cF)\,\iso\,  j_{-\lambda, *}\ast \cF[-\<\lambda, 2\check{\rho}\>]
\end{equation}
 
 Thus, for $\lambda\in\Lambda^+_{M, ab}$ we get
$$
\Av_*^{I_P/M(\cO)}\Av_!^{U_{\lambda}}(\cF)\,\iso\, j_{-\lambda, *}\ast j_{\lambda, !}\ast \cF\,\iso\, \cF,
$$ 
where the last isomorphism is given by Lemma~\ref{Lm_2.2.5}. This gives finally
$$
\Av_*^{I_P/M(\cO)}\Av_!^{U(P)(F)}(\cF)\,\iso\, \mathop{\colim}\limits_{\lambda\in\Lambda^+_{M, ab}} \Av_*^{I_P/M(\cO)}\Av_!^{U_{\lambda}}(\cF)\,\iso\, \mathop{\colim}\limits_{\lambda\in\Lambda^+_{M, ab}}\cF\,\iso\cF,
$$
because $\Lambda^+_{M, ab}$ is filtered.

\smallskip\noindent
{\bf Step 4} It suffices now to show that  $\Av_*^{I_P/M(\cO)}: Shv(\Gr_G)^H\to Shv(\Gr_G)^{I_P}$ is conservative. 

 Let $0\ne \cF\in Shv(\Gr_G)^H$. By Section~\ref{Sect_A.2.4}, there is $\cF'\in Shv(\Gr_G)^{K_n}$ for some congruence subgroup $K_n\subset G(\cO)$, $n>0$ such that $\HOM_{Shv(\Gr_G)}(\cF', \cF)\ne 0$, here $\HOM_{Shv(\Gr_G)}\in\Vect$ denotes the inner hom for the $\Vect$-action on $Shv(\Gr_G)$. 
 
  By assumption, $\cF'$ is equivariant with respect to $t^{-\lambda}U(P^-)(\cO)_1t^{\lambda}$ for $\lambda\in\Lambda^+_{M, ab}$ large enough, so $\Av_*^{t^{-\lambda}U(P^-)(\cO)_1t^{\lambda}}(\cF)\ne 0$. Here $\Av_*^{t^{-\lambda}U(P^-)(\cO)_1t^{\lambda}}: Shv(\Gr_G)\to Shv(\Gr_G)^{t^{-\lambda}U(P^-)(\cO)_1t^{\lambda}}$ is the right adjoint to the inclusion. 
  
   Now viewing $\Av_*^{t^{-\lambda}U(P^-)(\cO)_1t^{\lambda}}: Shv(\Gr_G)^{M(\cO)}\to Shv(\Gr_G)^{t^{-\lambda}U(P^-)(\cO)_1t^{\lambda}M(\cO)}$ as the right adjoint to the inclusion, the above also gives $\Av_*^{t^{-\lambda}U(P^-)(\cO)_1t^{\lambda}}(\cF)\ne 0$ in $Shv(\Gr_G)^{t^{-\lambda}U(P^-)(\cO)_1t^{\lambda}M(\cO)}$.

Since $I_P=U(P^-)(\cO)_1M(\cO)U(P)(\cO)$, we get 
$$
\Av_*^{t^{-\lambda}U(P^-)(\cO)_1t^{\lambda}}(\cF)\in Shv(\Gr_G)^{I_P^{\lambda}}
$$ 
for $\lambda\in\Lambda^+_{M, ab}$ large enough. For any $\cF''\in Shv(\Gr_G)^{I_P^{\lambda}}$ we have
$$
t^{\lambda}\cF''\in Shv(\Gr_G)^{I_P}
$$
naturally and 
$$
\Av_*^{U(P^-)(\cO)_1}(\cF'')\,\iso\, j_{-\lambda, *}\ast (t^{\lambda}\cF'')[-\<\lambda, 2\check{\rho}\>]
$$ 
by (\ref{iso_for_Step3_first}). 

 Finally, for $\cF$ as above letting $\cF''=\Av_*^{t^{-\lambda}U(P^-)(\cO)_1t^{\lambda}}(\cF)$ we get
$$
\Av_*^{I_P/M(\cO)}(\cF)\,\iso\, \Av_*^{U(P^-)(\cO)_1}(\cF)\,\iso\, \Av_*^{U(P^-)(\cO)_1}
(\cF'')\,\iso\, j_{-\lambda, *}\ast (t^{\lambda}\cF'')[-\<\lambda, 2\check{\rho}\>]
$$ 
Applying again Lemma~\ref{Lm_2.2.5}, we see that the latter object is nonzero. Proposition~\ref{Pp_2.2.6} is proved. \QED

\sssec{Actions of $\Lambda_{M, ab}$} 
\label{Sect_2.2.10_action_of_Lambda_Mab}
For $\cF\in Shv(\Gr_G)$, $\lambda\in\Lambda$ we denote by $t^{\lambda}\cF$ the direct image of $\cF$ under the multiplication $\Gr_G\to\Gr_G$ by $t^{\lambda}$. 
Consider $\oblv: Shv(\Gr_G)^{M(\cO)}\to Shv(\Gr_G)$. We think of $Shv(\Gr_G)^H$ as a full subcategory of $Shv(\Gr_G)^{M(\cO)}$. There is an action of $\Lambda_{M, ab}$ on $Shv(\Gr_G)^{M(\cO)}$ such that $\lambda\in\Lambda_{M, ab}$ sends $K$ to $t^{\lambda}K[-\<\lambda, 2\check{\rho}\>]$. This means by definition that for $\oblv: Shv(\Gr_G)^{M(\cO)}\to Shv(\Gr_G)$ one has canonically 
$$
\oblv(t^{\lambda}K)\,\iso\, t^{\lambda}(\oblv(K))
$$
This action preserves the full subcategory $Shv(\Gr_G)^{H}$. 

 Consider the $\Lambda_{M, ab}$-action on $Shv(\Gr_G)^{I_P}$ given by restricting the action of $\cH_P(G)$ via the monoidal functor (\ref{mon_functor_!}). Proposition~\ref{Pp_2.2.6} also shows that the equivalence $\Av^{U(P)(F)}_!: Shv(\Gr_G)^{I_P}\,\iso\, Shv(\Gr_G)^H$ intertwines these two actions of $\Lambda_{M, ab}$. Namely, for $\lambda\in\Lambda_{M, ab}^+$, $\cF\in Shv(\Gr_G)^{I_P}$ one has the isomorphism (\ref{iso_Av_!_intertwines}). 

\sssec{} 
\label{Sect_3.1.11_action_of}
The equivalence $\Av^{U(P)(F)}_!: Shv(\Gr_G)^{I_P}\,\iso\, Shv(\Gr_G)^H$ commutes with the actions of $\Rep(\check{G})$ on both sides. Indeed, this can be seen for example from (\ref{functor_Av_!_for_Pp2.2.7}). 

 We equip $Shv(\Gr_G)^{I_P}$, $Shv(\Gr_G)^{I_P, ren}$ with t-structures as in Section~\ref{Sect_A.3}. So, we have the t-exact oblivion functor $\oblv: Shv(\Gr_G)^{I_P}\to Shv(\Gr_G)$.  
 
 The action of $\Rep(\check{G})^c$ on $Shv(\Gr_G)^{I_P}$ preserves the full subcategory $Shv(\Gr_G)^{I_P, constr}$, and the obtained action on $Shv(\Gr_G)^{I_P, constr}$ is t-exact in each variable by (\cite{G_central}, Proposition~6). Passing to the ind-completion this yields a $\Rep(\check{G})$-action on $Shv(\Gr_G)^{I_P, ren}$ which is moreover t-exact in each variable.
 
%  We equip $Shv(\Gr_G)^H$ with the unque t-structure such that the equivalence (\ref{equiv_Iwahori_vs_SI}) is t-exact.  By construction the inclusion $Shv(\Gr_G)^{H, contsr}\subset Shv(\Gr_G)^H$ is compatible with the t-structure (that is, the truncation functors preserve this subcategory). We equip $Shv(\Gr_G)^{H, contst}$ with the induced t-structure.
  
 The equivalence of Proposition~\ref{Pp_2.2.6} yields an equivalence 
$$
Shv(\Gr_G)^{H, constr}\,\iso\, Shv(\Gr_G)^{I_P, constr}
$$ 
which commutes with the actions of $\Rep(G)^c$. Passing to the ind-completion, this gives an equivalence
\begin{equation}
\label{eq_ren_parahoric_versus_H}
\Av_*^{I_P/M(\cO), ren}: Shv(\Gr_G)^{H, ren}\,\iso\, Shv(\Gr_G)^{I_P, ren}
\end{equation}
We equip $Shv(\Gr_G)^{H, ren}$ with $\Rep(\check{G})$-action coming from the ind-completion of the $\Rep(\check{G})^c$-action on $Shv(\Gr_G)^{H, constr}$. % By construction, the $\Rep(\check{G})$-action on $Shv(\Gr_G)^{H, ren}$ is t-exact in both variables.  

\ssec{Relation between local and global: geometry}  
\label{Sect_Relation between local and global}

\sssec{} 
\label{Sect_2.3.1}
From now on we assume $[G,G]$ is simply-connected. Let $\Lambda_{G,P}$ be the lattice of cocharacters of $M/[M,M]$, so $\Lambda_{G,P}$ is the quotient of $\Lambda$ by the span of $\alpha_i, i\in\cI_M$. Let $\check{\Lambda}_{G,P}$ denote the dual lattice. Let $\check{\Lambda}^+$ be the dominant weights for $G$. Write $\Lambda_{G,P}^{pos}$ for the $\ZZ_+$-span of $\alpha_i$, $i\in \cI-\cI_M$ in $\Lambda_{G,P}$. 

For $\theta\in\Lambda_{G,P}$ denote by $\Gr_M^{\theta}$ the connected component of the affine Grassmannian $\Gr_M$ containing $t^{\lambda}M(\cO)$ for any $\lambda\in\Lambda$ over $\theta$. 

As in (\cite{BG}, 4.3.1) for $\theta\in\Lambda_{G,P}$ denote by $\ov{\Gr}_P^{\theta}\subset \Gr_G$ the closed ind-subscheme given by the property that for $\check{\lambda}\in\check{\Lambda}_{G,P}\cap \check{\Lambda}^+_G$ the map
$$
\cL^{\check{\lambda}}_{\cF^0_{M/[M,M]}}(-\<\theta, \check{\lambda}\>)\to \cV^{\check{\lambda}}_{\cF_G}
$$
is regular on the disk $D$. Let $\Gr^{\theta}_P\subset \ov{\Gr}^{\theta}_P$ be the open subscheme where the above maps have no zeros on $D$.

 For $\theta,\theta'\in\Lambda_{G,P}$ one has $\Gr_P^{\theta'}\subset \ov{\Gr}^{\theta}_P$ iff $\theta-\theta'\in\Lambda_{G,P}^{pos}$.

Consider the natural map $\gt^{\theta}_P: \Gr^{\theta}_P\to \Gr_M^{\theta}$ defined in (\cite{BG}, Pp. 4.3.2). For $\mu\in\Lambda^+_M$ write $\Gr_M^{\mu}$ for the $M(\cO)$-orbit on $\Gr_M$ through $t^{\mu}$. For $\mu\in\Lambda^+_M$ over $\theta\in\Lambda_{G,P}$ write $S^{\mu}_P$ for the preimage of $\Gr_M^{\mu}$ under $\gt^{\theta}_P$. So, $\{S^{\mu}_P\}_{\mu\in\Lambda^+_M}$ are the $H$-orbits on $\Gr_G$. 
The restriction of $\gt^{\theta}_P$ is denoted
$$
\gt^{\mu}_P: S^{\mu}_P\to \Gr_M^{\mu}
$$ 

 For $\theta\in\Lambda_{G,P}$ let $i^{\theta}_P: \Gr_M^{\theta}\to \Gr^{\theta}_P$ be the natural map, so that $\gt^{\theta}_P i^{\theta}_P=\id$. Write $v^{\theta}_P: \Gr^{\theta}_P\to \Gr_G$ for the natural inclusion. For $\mu\in\Lambda^+_M$ write $\bar S^{\mu}_P$ for the closure of $S^{\mu}_P$ in $\Gr_G$. 
 
 For $\theta\in\Lambda_{G,P}$ let $\Gr_{P^-}^{\theta}\subset \ov{\Gr}_{P^-}^{\theta}\subset \Gr_G$ be the analogs of the corresponding ind-schemes with $P$ replaced by $P^-$. The corresponding morphisms are denoted $\gt^{\theta}_{P^-}:  \Gr_{P^-}^{\theta}\to \Gr_M^{\theta}$ and
$$
\Gr_M^{\theta}\toup{i^{\theta}_{P^-}} \Gr_{P^-}^{\theta}\toup{v^{\theta}_{P^-}} \Gr_G
$$ 
For $\mu\in\Lambda^+_M$ over $\theta\in\Lambda_{G,P}$ write $S_{P^-}^{\mu}$ for the preimage of $\Gr_M^{\mu}$ under $\gt^{\theta}_{P^-}$. Let $\gt^{\mu}_{P^-}: S_{P^-}^{\mu}\to \Gr_M^{\mu}$ denote the restriction of $\gt^{\theta}_{P^-}$. 

 Recall the following consequence of a theorem of Braden (\cite{DG1}, \cite{Bra}).
\begin{Lm} 
\label{Lm_theorem_of_Braden_for_theta}
Let $\theta\in\Lambda_{G,P}$.\\
a) For $K\in Shv(\Gr_G)^T$ one has canonically
$$
(i^{\theta}_P)^!(v^{\theta}_P)^*K\,\iso\, (i^{\theta}_{P^-})^*(v^{\theta}_{P^-})^!K
$$
b) For $K\in Shv(\Gr_P^{\theta})^T$ one has canonically $(\gt^{\theta}_P)_!K\,\iso\,(i^{\theta}_P)^!K$ and $(\gt^{\theta}_P)_*K\,\iso\,(i^{\theta}_P)^*K$ in $Shv(\Gr_M^{\theta})^T$, and similarly for $\Gr_{P^-}^{\theta}$. \QED
\end{Lm}

\sssec{} 
\label{Sect_2.3.3_loc_vs_glob}
Let $X$ be a smooth projective connected curve. The stack $\Bunb_P$ is defined in (\cite{BG}, 1.3.2). 

Fix a point of our curve $x\in X$. Let $_{x,\infty}\Bunb_P$ be the stack classifying $M/[M,M]$-torsor $\cF_{M/[M,M]}$ on $X$, $G$-torsor $\cF_G$ on $X$, and a collection of maps
$$
\kappa^{\check{\lambda}}: \cL^{\check{\lambda}}_{\cF_{M/[M,M]}}\to \cV^{\check{\lambda}}_{\cF_G}(\infty x), \check{\lambda}\in\check{\Lambda}^+\cap \check{\Lambda}_{G,P}
$$
satisfying the Pl\"ucker relations. 

 Pick a uniformizer $t_x\in \cO_x$, hence an isomorphism $\cO\,\iso\, \cO_x$. This allows to view $\Gr_G$ as the ind-scheme classifying $(\cF_G, \beta)$, where $\cF_G$ is a $G$-torsor on $X$, $\beta: \cF_G\,\iso\,\cF^0_G$ is trivialization over $X-x$. We get the morphism $\pi: \Gr_G\to {_{x,\infty}\Bunb_P}$ sending $(\cF_G, \beta)$ to $(\cF^0_{M/[M,M]}, \cF_G, \kappa)$, where $\kappa$ is induced by the $P$-structure on the trivial $P$-torsor.
 
 The preimage $\pi^{-1}\Bunb_P$ identifies with $\ov{\Gr}^0_P$.  

We let $\Rep(\check{G})$ act on $Shv(_{x,\infty}\Bunb_P)$ so that $V\in \Rep(\check{G})$ acts as 
$$
_x\H^{\ra}_G(\Sat(V), \cdot): Shv(_{x,\infty}\Bunb_P)\to Shv(_{x,\infty}\Bunb_P)
$$ 
in the notations of (\cite{BG}, 3.2.4). Since we are in the constructible context, we have the adjoint pair in $\DGCat_{cont}$
$$
\pi_!: Shv(\Gr_G)\leftrightarrows Shv(_{x,\infty}\Bunb_P): \pi^!
$$
By $(*, !)$-base change, both these functors commute with $\Rep(\check{G})$-actions. 
 
\sssec{} For $\theta\in\Lambda_{G,P}$ let $_{\le\theta, x}\Bunb_P\subset {_{x,\infty}\Bunb_P}$ be the closed substack given by the property that for any $\check{\lambda}\in\check{\Lambda}_{G,P}\cap \check{\Lambda}^+$ the map
\begin{equation}
\label{map_for_x_infty_Bunb_P_Sect_2.3.3}
\cL^{\check{\lambda}}_{\cF_{M/[M,M]}}(-\<\theta, \check{\lambda}\>x)\to \cV^{\check{\lambda}}_{\cF_G}
\end{equation}
is regular on $X$. Let also $_{=\theta, x}\Bunb_P\subset {_{\le\theta, x}\Bunb_P}$ be the open substack given by the property that (\ref{map_for_x_infty_Bunb_P_Sect_2.3.3}) have no zeros everywhere on $X$.  Note that 
$$
\pi^{-1}(_{\le\theta, x}\Bunb_P)=\ov{\Gr}^{\theta}_P\;\;\; \mbox{and}\;\;\; \pi^{-1}(_{=\theta, x}\Bunb_P)=\Gr^{\theta}_P
$$
% If $\theta,\theta'\in\Lambda_{G,P}$ with $\theta-\theta'\in\Lambda_{G,P}^{pos}$ then $\ov{\Gr}^{\theta'}_P\subset \ov{\Gr}^{\theta}_P$. 

\sssec{} For $\lambda\in\Lambda$ write $\Bun_T^{\lambda}$ for the connected component of $\Bun_T$ classifying $\cF_T\in\Bun_T$ such that for any $\check{\lambda}\in\check{\Lambda}$, $\deg\cL^{\check{\lambda}}_{\cF_T}=-\<\lambda, \check{\lambda}\>$. Similarly, for $\theta\in\Lambda_{G,P}$ let $\Bun_M^{\theta}$ be the preimage of $\Bun_{M/[M,M]}^{\theta}$, this normalization agrees with \cite{BG}. 

 For $\theta'\in\Lambda_{G,P}$ let $\Bun_P^{\theta'}, \Bunb_P^{\theta'}$ and so on be the preimage of the component $\Bun_M^{\theta'}$. We have 
$$
\dim\Bun_P^{\theta}=(g-1)\dim P+\<\theta, 2\check{\rho}-2\check{\rho}_M\>=\dim(_{=\theta, x}\Bunb_P^0)
$$ 
This explains the shift in the definition of the t-structure on $Shv(\Gr^{\theta}_P)^H$ in Section~\ref{Sect_t-str_on_S_theta_P}.

 Let $\Lambda_{G,P}$ act on $_{x,\infty}\Bunb_P$ so that $\theta\in\Lambda_{G,P}$ acts as
$$
(\cF_{M/[M,M]}, \cF_G,\kappa)\mapsto (\cF_{M/[M,M]}(\theta x), \cF_G, \kappa)
$$ 
Let now $\Lambda_{M, ab}$ act on $_{x,\infty}\Bunb_P$ via the inclusion $\Lambda_{M, ab}\hook{} \Lambda_{G,P}$. Then $\pi: \Gr_G\to {_{x,\infty}\Bunb_P}$ is $\Lambda_{M,ab}$-equivariant, where $\lambda\in\Lambda_{M, ab}$ acts on $\Gr_G$ as $t^{\lambda}$. 
 
\sssec{} 
\label{Sect_2.3.6_positive part_Gr_M^+}
Set $\Lambda_{M,G}^+=\Lambda_M^+\cap w_0^M(\Lambda^{pos})$. 

 We define the positive part of the affine Grassmannian $\Gr_M^+\subset\Gr_M$ as the subscheme of $(\cF_M, \beta)\in\Gr_M$, where $\cF_M$ is a $M$-torsor on the disk $D$, and $\beta: \cF_M\,\iso\cF^0_M\mid_{D^*}$ such that for any $V\in\Rep(G)^{\heartsuit}$ finite-dimensional, the natural map
$$
V^{U(P)}_{\cF_M}\toup{\beta} V^{U(P)}_{\cF^0_M} 
$$
is regular over $D$.
Recall that for $\nu\in\Lambda^+_M$ we have $\Gr_M^{\nu}\subset \Gr_M^+$ iff $\nu\in \Lambda_{M,G}^+$ by (\cite{BG}, Proposition~6.2.3). For $\theta\in\Lambda_{G,P}$ we set $\Gr_M^{+,\theta}=\Gr_M^{\theta}\cap \Gr_M^+$.

\sssec{} 
\label{Sect_2.3.7_local_vs_global}
Let $\Bunt_P$ and $_{x,\infty}\Bunt_P$ be defined as in (\cite{BG}, 4.1.1). As in \select{loc.cit}. for $\nu\in\Lambda^+_M$ define the closed substack $_{x,\ge\nu}\Bunt_P\subset {_{x,\infty}\Bunt_P}$ by requiring that for any finite-dimensional $G$-module $\cV$ whose weights are $\le \check{\lambda}$, the map
\begin{equation}
\label{map_for_2.3.5}
\cV^{U(P)}_{\cF_M}\to \cV_{\cF_G}(-\<w_0^M(\nu), \check{\lambda}\>x)
\end{equation}
is regular on $X$. In particular, $\Bunt_P={_{x,\ge 0}\Bunt_P}$. Let $_{x,\nu}\Bunt_P\subset {_{x,\ge\nu}\Bunt_P}$ be the open substack defined as in (\cite{BG}, 4.2.2). In fact, it classifies $(\cF_M, \cF_G, \kappa)\in {_{x,\infty}\Bunt_P}$ such that there is a modification $\cF_M\,\iso\, \cF'_M\mid_{X-x}$ of $M$-torsors at $x$ for which $\cF'_M$ defines a true $P$-structure on $\cF_G$ in a neighbourhood of $x$, and such that $\cF_M$ is in the position $\nu$ with respect to $\cF'_M$ at $x$. 

 Recall that by (\cite{BG}, 4.2.3) the stacks $_{x,\nu'}\Bunt_P$ for $\nu'\in\Lambda_M^+$ with $w_0^M(\nu'-\nu)\in\Lambda^{pos}$ form a stratification of $_{x,\ge \nu}\Bunt_P$. 
 
 For $\theta\in\Lambda_{G,P}$ denote by $\Bunt_P^{\theta}$ be the preimage of $\Bun_M^{\theta}$ under $\Bunt_P\to\Bun_M$.  
 
\sssec{} 
\label{Sect_2.3.8_local_vs_global}
Define the morphism $\tilde\pi: \Gr_G\to {_{x,\infty}\Bunt_P}$ sending $(\cF_G,\beta)$ to $(\cF_M^0, \cF_G, \kappa)$. The composition 
$$
\Gr_G\toup{\tilde\pi}{_{x,\infty}\Bunt_P}\toup{\gr}{_{x,\infty}\Bunb_P}
$$
equals $\pi$. Here $r$ is the map sending $(\cF_M, \cF_G,\kappa)$ to $(\cF_{M/[M,M]},\cF_G, \kappa)$ with $\cF_{M/[M,M]}$ induced from $\cF_M$.

If $\nu\in\Lambda_M^+$ then $\tilde\pi^{-1}(_{x,-w_0^M(\nu)}\Bunt_P)$ coincides with $S_P^{\nu}$. This gives the fact that if $\nu\in\Lambda_M^+$ then $\bar S^{\nu}_P$ is stratified by locally closed ind-schemes $S^{\mu}_P$ for $\mu\in\Lambda^+_M$ satisfying  $\nu-\mu\in\Lambda^{pos}$. 

 We let $Shv(\Gr_G)^{G(\cO)}$ act on $Shv(_{x,\infty}\Bunt_P)$, so that $\cS\in Shv(\Gr_G)^{G(\cO)}$ acts as 
$$
_x\H^{\ra}_{P, G}(\cS, \cdot): Shv(_{x,\infty}\Bunt_P)\to Shv(_{x,\infty}\Bunt_P)
$$ 
in the notations of (\cite{BG}, 4.1.4). Write for brevity $\_\ast \cS={_x\H^{\ra}_{P, G}(\cS, \_)}$. As above, we have the adjoint pair
$$
\tilde\pi_! : Shv(\Gr_G)\leftrightarrows Shv(_{x,\infty}\Bunt_P): \tilde\pi^!,
$$
and both these functors commute with $\Rep(\check{G})$-actions. Similarly, the functors
$$
\gr_!: Shv(_{x,\infty}\Bunt_P)\leftrightarrows Shv(_{x,\infty}\Bunb_P): \gr^!
$$
commute with $\Rep(\check{G})$-actions at $x$.  Note that $\tilde\pi^{-1}(\Bunt_P)=\bar S^0_P$. 

 For $\nu\in\Lambda^+_M$ write $i_{\nu, glob}: {_{x,-w_0^M(\nu)}\Bunt_P}\to {_{x,\infty}\Bunt_P}$ for the natural inclusion and set
$$
\bvartriangle^{\nu}_{\glob}=(i_{\nu, glob})_!\IC(_{x,-w_0^M(\nu)}\Bunt_P),\;\;\;\;\; 
\nabla^{\nu}_{glob}=(i_{\nu, glob})_*\IC(_{x,-w_0^M(\nu)}\Bunt_P)
$$ 

\sssec{} We let $\Rep(\check{M})$ act on $Shv(_{x,\infty}\Bunt_P)$, so that $V\in \Rep(\check{M})$ acts as 
$$
_x\H^{\la}_{P, M}(\Sat_M(V), \cdot): Shv(_{x,\infty}\Bunt_P)\to Shv(_{x,\infty}\Bunt_P)
$$ 
in the notations of (\cite{BG}, 4.1.2). 

 For $\cS\in \Sph_M$ and $K\in Shv(\Gr_G)^{M(\cO)}$ we write $\cS\ast K$ for the natural left action of $\cS$ on $K$. For $\cS\in Shv(\Gr_M)^{M(\cO)}$ we also write for brevity $\cS\ast \_={_x\H^{\la}_{P, M}(\cS, \_)}$. 

\sssec{} 
\label{Sect_cY_x_definition}
Let $\cY_x$ be the stack classifying $\cF_G\in\Bun_G, \cF_M\in\Bun_M$ and an isomorphism $\xi: \cF_M\times_M G\,\iso\, \cF_G\mid_{X-x}$. Let 
$$
\pi_{glob}: \cY_x\to {_{x,\infty}\Bunt_P}
$$ 
be the map sending the above point to $(\cF_G, \cF_M,\kappa)$. Here $\kappa$ is obtained from the $P$-structure $\cF_M\times_M P$ on $\cF_G\mid_{X-x}$. Note that $\cY_x$ is locally of finite type as a prestack. 

 We define the $\Sph_{G, x}$ and $\Sph_{M,x}$-actions on $Shv(\cY_x)$ along the lines of the corresponding actions on $Shv(_{x,\infty}\Bunt_P)$. One easily checks that the functors in the adjoint pair
$$
(\pi_{glob})_!: Shv(\cY_x)\leftrightarrows Shv(_{x,\infty}\Bunt_P): \pi_{glob}^!
$$
are both $\Sph_{G, x}$ and $\Sph_{M, x}$-linear.   

 For $\theta\in\Lambda_{G,P}$ write $\cY^{\theta}_x$ for the preimage of $\Bun_M^{\theta}$ under the projection $\cY_x\to \Bun_M$, $(\cF_G,\cF_M,\xi)\mapsto \cF_M$. 

\sssec{} 
\label{Sect_pi_loc_def}
Write $\pi_{loc}: \cY_x\to M(\cO_x)\backslash \Gr_{G, x}$ for the map sending $(\cF_G, \cF_M, \xi)$ to their restrictions to $D_x$. Note that $\pi_{loc}$ is schematic. Indeed, the prestack classifying $\cF_M\in\Bun_M$ with a trivialization over $D_x$ is known to be a scheme. 

 For a $M(\cO_x)$-stable locally closed ind-subscheme $Z$ of $\Gr_{M, x}$ (or of $\Gr_{G, x}$) set 
$$
Z(\cY)=\Bun_M\times_{pt/M(\cO_x)} (M(\cO_x)\backslash Z).
$$ 
For a $M(\cO_x)$-equivariant map $h: Z\to Z'$ between two such ind-subschemes, we write $h_{\cY}: Z(\cY)\to Z'(\cY)$ for the morphism obtained from $M(\cO_x)\backslash Z\to M(\cO_x)\backslash Z'$ by base change via $\Bun_M\to pt/M(\cO_x)$.  

  In particula, for $\nu\in\Lambda_M^+$ we get the inclusion $i_{\nu, \cY}: S^{\nu}_P(\cY)\hook{} \cY_x$.

%%% label{functor_tilde_pi^M^!} vezde ubrat'! this map does not exist.
  
\sssec{}  Let us explain that the functor $\pi_{loc}^!: Shv(M(\cO_x)\backslash \Gr_{G, x})\to Shv(\cY_x)$ is well-defined.

 We write $\Gr_G\,\iso\, \colim_{i\in i} Z_i$, where $I$ is small filtered, $Z_i\subset \Gr_G$ is a closed $M(\cO)$-invariant subscheme of finite type, and for $i\to i'$ the map $Z_i\to Z_{i'}$ is a closed immersion. Recall that $Shv(M(\cO_x)\backslash \Gr_{G, x})\,\iso\, \colim_i Shv(M(\cO_x)\backslash Z_i)$, and similarly $Shv(\cY_x)\,\iso\, \colim_i Shv(\Bun_M\times_{pt/M(\cO_x)} (M(\cO_x)\backslash Z_i)$. It suffices to define the corresponding functor
$$
\pi_{i, loc}^!: Shv(M(\cO_x)\backslash Z_i)\to Shv(\Bun_M\times_{pt/M(\cO_x)} (M(\cO_x)\backslash Z_i)
$$
then $\pi_{loc}^!$ is obtained by passing to the colimit over $i\in I$. Fix $i\in I$ and pick a quotient group scheme $M(\cO_x)\to M_i$ with $M_i$ smooth of finite type such that the $M(\cO_x)$-action on $Z_i$ factors through $M_i$. The corresponding morphism
\begin{equation}
\label{smooth_map_part_of_pi_loc}
\Bun_M\times_{pt/M_i} (M_i\backslash Z_i)\to M_i\backslash Z_i
\end{equation}
is obtained by base change from $\Bun_M\to pt/M_i$. Since (\ref{smooth_map_part_of_pi_loc}) is a morphism of algebraic stacks locally of finite type, the !-inverse image under this map is well-defined (and commutes with the transition functors in the above diagrams). This concludes the construction of $\pi_{loc}^!$. 

 For $i\in I$ the map (\ref{smooth_map_part_of_pi_loc}) is smooth. This implies, in particular, that for any $\nu\in\Lambda^+_M$ one has canonically
$$
\pi_{loc}^!(i_{\nu})_!\omega\,\iso\, (i_{\nu, \cY})_!\omega
$$ 
in $Shv(\cY_x)$. 

\sssec{} One easily checks that $\pi_{loc}^!: Shv(M(\cO_x)\backslash \Gr_{G, x})\to Shv(\cY_x)$ is both $\Sph_{G, x}$ and $\Sph_{M, x}$-linear.   
  
\begin{Rem} 
\label{Rem_2.3.7}
i) Let $\nu\in\Lambda_M^+$ and $\lambda\in\Lambda^+$. If $S^{\nu}_P\cap \ov{\Gr}_G^{\lambda}\ne\emptyset$ then $\lambda-\nu, \nu+w_0(\lambda)\in \Lambda^{pos}$. Indeed, the map $\gt^{\nu}_P$ is $M(\cO)$-equivariant, so $S^{\nu}\cap \ov{\Gr}_G^{\lambda}\ne\emptyset$, where $S^{\nu}$ is the $U(F)$-orbit through $t^{\nu}$.

\smallskip\noindent
ii) For $\mu\in \Lambda_{M, ab}, \nu\in \Lambda^+_M$ one has $t^{\mu}S^{\nu}_P=S^{\nu+\mu}_P$. 

\smallskip\noindent
iii) Let $\mu\in\Lambda_{M, ab}$ over $\bar\mu\in\Lambda_{G,P}$, let $\theta\in\Lambda_{G,P}$. Then $t^{\mu}\Gr_M^{\theta}=\Gr_M^{\theta+\bar\mu}$ and $t^{\mu}\Gr^{\theta}_P=\Gr^{\theta+\bar\mu}_P$. The natural map $\Lambda_{M, ab}\to \Lambda_{G,P}$ is injective.
\end{Rem}

\sssec{} 
\label{Sect_2.3.12_local_vs_global}
Write $\IC_{\wt{glob}}$ for the $\IC$-sheaf of $\Bunt_P$. Its Hecke property is given by (\cite{BG}, 4.1.5). It says that for $V\in\Rep(\check{G})$ one has isomorphisms
$$
\IC_{\wt{glob}}\ast V\,\iso\, \Res(V)\ast \IC_{\wt{glob}}
$$
in a way compatible with the monoidal structures on $\Rep(\check{G}), \Rep(\check{M})$. Here $\Res: \Rep(\check{G})\to \Rep(\check{M})$ is the restriction.  

 This shows that $\IC_{\wt{glob}}$ naturally upgrades to an object of 
$$
Shv(_{x,\infty}\Bunt_P)\otimes_{\Rep(\check{G})\otimes\Rep(\check{M})} \Rep(\check{M})
$$ 

 For $\nu\in\Lambda^+_M$ write $\IC^{\nu}_{\wt{glob}}$ for the $\IC$-sheaf of $_{x,\ge -w_0^M(\nu)}\Bunt_P$ extended by zero to $_{x,\infty}\Bunt_P$. 
  
\ssec{Structure of $\SI_P$}

\sssec{} 
\label{Sect_2.3.8}
For $\nu\in\Lambda_M^+$ write $j_{\nu}: S^{\nu}_P\hook{} \bar S^{\nu}_P$ for the open immersion. Let $\bar i_{\nu}: \bar S^{\nu}_P\hook{} \Gr_G$ be the closed immersion and $i_{\nu}=\bar i_{\nu}\comp j_{\nu}$. The ind-schemes $S^{\nu}_P$, $\bar S^{\nu}_P$ are acted on by $H$, so we also consider the categories of invariants 
$$
Shv(\bar S^{\nu}_P)^H, \;\; Shv(S^{\nu}_P)^H
$$ 
By restriction we get the adjoint pairs $(\bar i_{\nu})_!: Shv(\bar S^{\nu}_P)^H\leftrightarrows Shv(\Gr_G)^H: (\bar i_{\nu})^!$ and 
$$
j_{\nu}^*: Shv(\bar S^{\nu}_P)^H\leftrightarrows Shv(S^{\nu}_P)^H: (j_{\nu})_*
$$ 
with $(\bar i_{\nu})_!, (j_{\nu})_*$ fully faithful. 
By Lemma~\ref{Lm_A.5.3} and Section~\ref{Sect_A.5.5}, we have the adjoint pair 
$$
(j_{\nu})_!:  Shv(S^{\nu}_P)^H\leftrightarrows Shv(\bar S^{\nu}_P)^H: j_{\nu}^!
$$ 
with $(j_{\nu})_!$ fully faithful. Let $i_P^{\nu}: \Gr_M^{\nu}\to S^{\nu}_P$ be the closed embedding, the $M(\cO)$-orbit through $t^{\nu}$, so that $\gt^{\nu}_Pi_P^{\nu}=\id$. 
The following is close to (\cite{LC2}, Lemma~2.1.5).

\begin{Lm} 
\label{Lm_2.3.8}
i) Let $\nu\in\Lambda_M^+$. The $!$-restriction under $\gt^{\nu}_P: S^{\nu}_P\to \Gr_M^{\nu}$ yields a fully faithfull embedding $Shv(\Gr_M^{\nu})^{M(\cO)}\hook{} Shv(S^{\nu}_P)^{M(\cO)}$ whose image is $Shv(S^{\nu}_P)^H$. The composition 
\begin{equation}
\label{equiv_for_Lm_2.3.8}
Shv(S^{\nu}_P)^H\hook{} Shv(S^{\nu}_P)^{M(\cO)}\toup{(i^{\nu}_P)^!} Shv(\Gr_M^{\nu})^{M(\cO)}
\end{equation}
is an equivalence.  

\smallskip\noindent
ii) Let $\theta\in\Lambda_{G,P}$. The functor $(\gt^{\theta}_P)^!: Shv(\Gr_M^{\theta})^{M(\cO)}\to Shv(\Gr^{\theta}_P)^{M(\cO)}$ is fully faithful, its essential image is $Shv(\Gr^{\theta}_P)^H$. The composition 
\begin{equation}
\label{composition_for_Lm_2.3.8}
Shv(\Gr^{\theta}_P)^H\hook{} Shv(\Gr^{\theta}_P)^{M(\cO)}\toup{(i^{\theta}_P)^!}Shv(\Gr_M^{\theta})^{M(\cO)}
\end{equation}
is an equivalence. The natural transformation $(\gt_P^{\theta})_!(\gt_P^{\theta})^!\to \id$ on $Shv(\Gr_M^{\theta})^{M(\cO)}$ is an equivalence.
\end{Lm}
\begin{proof} i) We have $S^{\nu}_P\,\iso\,\mathop{\colim}\limits_{\lambda\in\Lambda^+_{M, ab}} H_{\lambda}t^{\nu}$. So, 
$$
Shv(S^{\nu}_P)^H\,\iso\, \mathop{\lim}\limits_{\lambda\le\lambda'\in (\Fun([1], \Lambda^+_{M, ab}))^{op}} Shv(H_{\lambda'}t^{\nu})^{H_{\lambda}}
$$ 
However, the diagonal map $\Lambda^+_{M, ab}\to \Fun([1], \Lambda^+_{M, ab})$ is cofinal, so the latter limit identifies with 
$$
\mathop{\lim}\limits_{\lambda\in\Lambda^+_{M, ab}} Shv(H_{\lambda}t^{\nu})^{H_{\lambda}}
$$
The stabilizor $St_{\nu}$ of $t^{\nu}$ in $H$ is the preimage of $St^M_{\nu}:=M(\cO)\cap t^{\nu}M(\cO)t^{-\nu}$ under $t^{\nu}P(\cO)t^{-\nu}\to t^{\nu}M(\cO)t^{-\nu}$. For for $\lambda$ large enough we have $St_{\nu}\subset H_{\lambda}$ and 
$$
Shv(H_{\lambda}t^{\nu})^{H_{\lambda}}\,\iso\, Shv(B(St_{\nu}))
$$ 
This gives an equivalence
$$
Shv(S^{\nu}_P)^H\,\iso\, Shv(B(St_{\nu}))
$$
The kernel of $St_{\nu}\to St^M_{\nu}$ is prounipotent, so 
$$
Shv(B(St_{\nu}))\,\iso\, Shv(B(St^M_{\nu}))\,\iso\, Shv(\Gr_M^{\nu})^{M(\cO)}
$$

\smallskip\noindent
ii)  Since $H/M(\cO)$ is ind-pro-unipotent, $\oblv: Shv(\Gr^{\theta}_P)^H\to Shv(\Gr^{\theta}_P)^{M(\cO)}$ is fully faithful. The map $\gt_P^{\theta}$ is $H$-equivariant, so 
\begin{equation}
\label{functor_gt_P^theta^!_for_Lm_2.4.2}
 (\gt_P^{\theta})^!: Shv(\Gr_M^{\theta})^{M(\cO)}\to Shv(\Gr^{\theta}_P)^{M(\cO)}
\end{equation}
takes values in $Shv(\Gr^{\theta}_P)^H$.
 
 The group $U(P(F))$ acts transitively on the fibres of $\Gr_P^{\theta}\to \Gr_M^{\theta}$.
By Lemma~\ref{Lm_A.3.6_fully_faithful_functors}, (\ref{functor_gt_P^theta^!_for_Lm_2.4.2}) is fully faithful. It remains to show that 
$$
(\gt^{\theta}_P)^!: Shv(\Gr_M^{\theta})^{M(\cO)}\to Shv(\Gr^{\theta}_P)^H
$$ 
is essentially surjective.

 The objects of the form $(i_{\nu})_!F$ for $\nu\in\Lambda_M^+$ over $\theta$, $F\in Shv(S^{\nu}_P)^H$, generate $Shv(\Gr_P^{\theta})^H$. The desired claim follows now from part i).
\end{proof}

\sssec{} 
\label{Sect_2.3.9_object_cB}
By Lemma~\ref{Lm_2.3.8}, for each $\nu\in\Lambda^+_M$ we have the object $\omega\in Shv(S^{\nu}_P)^H$ and $(i_{\nu})_!\omega\in Shv(\Gr_G)^H$. For $\lambda\in\Lambda_M^+$ set 
$$
\bvartriangle^{\lambda}=(i_{\lambda})_!\omega[-\<\lambda, 2\check{\rho}\>],\;\;\;\;
\nabla^{\lambda}=(i_{\lambda})_*\omega[-\<\lambda, 2\check{\rho}\>]
$$
in $Shv(\Gr_G)^H$. 

For $\lambda\in\Lambda_{M, ab}$ the ind-scheme $S^{\lambda}_P$ coincides with the $U(P)(F)$-orbit of $t^{\lambda}\in\Gr_G$. 
 
\sssec{} For $\lambda\in \Lambda_{M, ab}$ let $\cW_{\lambda}\in \cH_P(G)$ denote the Wakimoto object defined as follows. Writing $\lambda=\lambda_1-\lambda_2$ with $\lambda_i\in\Lambda_{M, ab}^+$ we set $\cW_{\lambda}=j_{\lambda_1, !}\ast j_{-\lambda_2, *}$. 

\begin{Lm} 
\label{Lm_2.3.13_about_Av!}
i) For $\lambda\in\Lambda_{M, ab}$ one has naturally $\Av^{U(P)(F)}_!(\cW_{\lambda}\ast \delta_{1,\Gr_G})\,\iso\, \bvartriangle^{\lambda}$ in $Shv(\Gr_G)^H$. Besides, we have canonically
$$
\Av_!^{U(P)(F)}(\delta_{t^{\lambda},\Gr_G})[-\<\lambda, 2\check{\rho}\>]\,\iso\, \bvartriangle^{\lambda}
$$   
for $\Av_!^{U(P)(F)}: Shv(\Gr_G)^{M(\cO)}\to Shv(\Gr_G)^H$.
\\
ii) If $\lambda\in\Lambda^+_{M, ab}$ then 
$$
\cB_{\lambda, !}\,\iso\, j_{\lambda, !}\ast \delta_{1,\Gr_G}
$$ 
%so that $\vartriangle^{\lambda}$ lies in the heart of the t-structure of $Shv(\Gr_G)^H$ and of $Shv(\Gr_G)^{H, ren}$.
iii) For $\lambda\in\Lambda_{M, ab}$, $\mu\in\Lambda^+_M$, one has canonically $t^{\lambda}\bvartriangle^{\mu}[-\<\lambda, 2\check{\rho}\>]\,\iso\, \bvartriangle^{\mu+\lambda}$.
\end{Lm}
\begin{proof} i) By (\ref{iso_Av_!_intertwines}) for any $\lambda\in\Lambda_{M, ab}$ one has canonically 
$$
t^{\lambda}\Av_!^{U(P)(F)}(\delta_{1,\Gr_G})[-\<\lambda, 2\check{\rho}\>]\,\iso\, \Av_!^{U(P)(F)}(\cW_{\lambda}\ast \delta_{1,\Gr_G})
$$ 
One has $t^{\lambda}\bvartriangle^0[-\<\lambda, 2\check{\rho}\>]\,\iso\, \bvartriangle^{\lambda}$. So, we are reduced to the case $\lambda=0$. For each $\mu\in\Lambda^+_{M, ab}$ consider the embedding $\act: U_{\mu}/U_0\to \Gr_G$, $zU_0\mapsto zG(\cO)$. By definition, $\Av_!^{U_{\mu}}(\delta_{1,\Gr_G})\,\iso\, \act_!\omega$. Now
$$
\Av_!^{U(P)(F)}(\delta_{1,\Gr_G})\,\iso\,\mathop{\colim}\limits_{\mu\in\Lambda^+_{M, ab}} \Av^{U_{\mu}}_!(\delta_{1,\Gr_G})\,\iso\, \mathop{\colim}\limits_{\mu\in\Lambda^+_{M, ab}} \omega_{U_{\mu}/U_0}\,\iso\,\bvartriangle^0
$$

 For the second claim, by (\ref{iso_for_Step2_Av!_versus_t_lambda}) we have $\Av_!^{U(P)(F)}(\delta_{\lambda,\Gr_G})\,\iso\, t^{\mu}\Av_!^{U(P)(F)}(\delta_{1,\Gr_G})$. The claim follows now from the above.

\smallskip\noindent
ii) Let $\act: t^{\lambda}U_{\lambda}/U_0\to\Gr_G$ be the embedding sending $zU_0$ to $zG(\cO)$. By Section~\ref{Sect_2.1.4_objects_cB}, $j_{\lambda, !}\ast \delta_{1,\Gr_G}\,\iso\, \act_!\IC$.
\end{proof}

%\sssec{} The collection of objects $\{(i_{\nu})_!\omega\}_{\nu\in\Lambda^+_M}$ generate $Shv(\Gr_G)^{H, ren}$?

\sssec{$t$-structure on $Shv(\Gr_G)^H$} First, for $\nu\in\Lambda^+_M$ define a new t-structrure on $Shv(S^{\nu}_P)^H$ by declaring that $K\in Shv(S^{\nu}_P)^H$ lies in $Shv(S^{\nu}_P)^{H,\le 0}$ iff $(i^{\nu}_P)^!(K)$ lies in perverse degrees $\le \<\nu, 2\check{\rho}-2\check{\rho}_M\>$. In fact, $Shv(S^{\nu}_P)^{H,\le 0}\subset Shv(S^{\nu}_P)^H$ is the smallest full subcategory containing $\bvartriangle^{\nu}$, stable under extensions and small colimits. 

Define 
$$
Shv(\Gr_G)^{H, \le 0}\subset Shv(\Gr_G)^H
$$ 
as the smallest full subcategory containing $(i_{\nu})_!F$ for $\nu\in\Lambda_M^+, F\in Shv(S^{\nu}_P)^{H, \le 0}$, closed under extensions and small colimits. By (\cite{HA}, 1.4.4.11), this defines an accessible t-structure on $Shv(\Gr_G)^H$. 

 So, $F\in Shv(\Gr_G)^H$ lies in $Shv(\Gr_G)^{H, >0}$ iff for any $\nu\in\Lambda^+_M$, $i_{\nu}^!F\in Shv(S^{\nu}_P)^{H, >0}$. This shows that the t-structure on $Shv(\Gr_G)^H$ is compatible with filtered colimits.

\begin{Rem} 
\label{Rem_2.3.14}
i) The objects of the form $(i_{\nu})_!F$ for $\nu\in\Lambda_M^+, F\in Shv(S^{\nu}_P)^H$ generate $Shv(\Gr_G)^H$.\\
ii) The objects of the form $(v^{\theta}_P)_!F$ for $\theta\in\Lambda_{G,P}$, $F\in Shv(\Gr_P^{\theta})^H$ generate $Shv(\Gr_G)^H$.
\end{Rem}
\begin{proof}
i) If $S\in\Sch_{ft}$ and $S\to \Gr_G$ is a map then the stratification of $\Gr_G$ by $S^{\nu}_P$, $\nu\in\Lambda_M^+$ defines a finite stratification of $S$ by $S\cap S^{\nu}_P$. 
So, if $i_{\nu}^!F=0$ for all $\nu$ then $F=0$.\\
ii) is similar. 
\end{proof}

\begin{Lm} 
\label{Lm_functor_i^nu^*}
For each $\nu\in\Lambda_M^+$ the adjoint functors
$$
i_{\nu}^*: Shv(\Gr_G)^{M(\cO)}\leftrightarrows Shv(S^{\nu}_P)^{M(\cO)}: (i_{\nu})_*
$$ 
preserve the full subcategories of $H$-invariants and give rise to an adjoint pair 
$$
i_{\nu}^*: Shv(\Gr_G)^{H}\leftrightarrows Shv(S^{\nu}_P)^{H}: (i_{\nu})_*
$$
Besides, the diagram canonically commutes 
$$
\begin{array}{ccc}
Shv(\Gr_G)^{M(\cO)} & \toup{i_{\nu}^*} &Shv(S^{\nu}_P)^{M(\cO)}\\
\downarrow\lefteqn{\scriptstyle \Av^{U(P)(F)}_!} && \downarrow\lefteqn{\scriptstyle \Av^{U(P)(F)}_!}\\
Shv(\Gr_G)^H & \toup{i_{\nu}^*} &Shv(S^{\nu}_P)^H
\end{array}
$$
\end{Lm}
\begin{proof}
%By Remark~\ref{Rem_2.3.14}, it suffices to show that $i_{\nu}^*$ is defined on objects of the form $(i_{\lambda})_! F$ for $\lambda\in\Lambda_M^+, F\in Shv(S^{\lambda}_P)^H$. However, $i_{\nu}^*(i_{\lambda})_! F$ equals $F$ for $\nu=\lambda$ and vanishes otherwise.
This follows from the results of Section~\ref{Sect_A.5}.
\end{proof}
\begin{Rem} 
\label{Rem_2.3.17_about_t-structure}
Let $\mu\in\Lambda^+_M$ and $F\in Shv(\Gr_G)^H$ be the extension by zero from $\bar S^{\mu}_P$. Then $F\in Shv(\Gr_G)^{H, \le 0}$ iff $i_{\nu}^*F\in Shv(S^{\nu}_P)^{H, \le 0}$ for all $\nu\in\Lambda_M^+$.
\end{Rem}
\begin{proof} First, let $\lambda\in\Lambda_M^+$, $K\in Shv(S^{\lambda}_P)^{H,\le 0}$. Then for any $\nu\in\Lambda_M^+$, $i_{\nu}^*(i_{\lambda})_!K\in Shv(S^{\nu}_P)^{H, \le 0}$. So, if $F\in Shv(\Gr_G)^{H, \le 0}$ then $i_{\nu}^*F\in Shv(S^{\nu}_P)^{H, \le 0}$ for all $\nu\in\Lambda_M^+$.

 Conversely, let $0\ne F\in Shv(\Gr_G)^{H, >0}$ be the extension by zero from $\bar S^{\mu}_P$. Assume $i_{\nu}^*F\in Shv(S^{\nu}_P)^{H, \le 0}$ for all $\nu\in\Lambda_M^+$. We must show that $F=0$. Let $\lambda$ be a largest orbit $S^{\lambda}_P$ such that $i_{\lambda}^!F\ne 0$. By definition of the t-structure, $i_{\lambda}^!F\in Shv(S^{\lambda}_P)^{H, >0}$. On the other hand, $i_{\lambda}^!F\,\iso\, i_{\lambda}^*F$, and the assertion follows.
\end{proof}

\medskip

 Note that if $G=P$ then the above t-structure on $Shv(\Gr_G)^H\,\iso\, Shv(\Gr_G)^{G(\cO)}$ is the perverse t-structure on the latter category.  
 
\sssec{} 
\label{Sect_t-str_on_S_theta_P}
For $\theta\in\Lambda_{G,P}$ define a t-structure on $Shv(\Gr^{\theta}_P)^H$ by declaring that $K\in Shv(\Gr^{\theta}_P)^H$ lies in $
Shv(\Gr^{\theta}_P)^{H,\le 0}$ iff $(i^{\theta}_P)^!K\in Shv(\Gr_M^{\theta})^{M(\cO)}$ lies in perverse degrees  $\le \<\theta, 2\check{\rho}-2\check{\rho}_M\>$. By Lemma~\ref{Lm_2.3.8}, this is equivalent to the property that 
$$
(\gt_P^{\theta})_!K\in Shv(\Gr_M^{\theta})^{M(\cO)}
$$ 
lies in perverse degrees $\le  \<\theta, 2\check{\rho}-2\check{\rho}_M\>$. 

 Now $Shv(\Gr_G)^{H,\le 0}$ is the smallest full subcategory containing $(v^{\theta}_P)_!K$ for $\theta\in\Lambda_{G,P}$, $K\in Shv(\Gr^{\theta}_P)^{H,\le 0}$, stable under extensions and small colimits. 

 This implies that $K\in Shv(\Gr_G)^H$ lies in $Shv(\Gr_G)^{H, >0}$ iff for any $\theta\in\Lambda_{G,P}$, $(v^{\theta}_P)^!K\in Shv(\Gr^{\theta}_P)^{H, >0}$. 
 
 Set $Sph_M=Shv(\Gr_M)^{M(\cO)}$.
\begin{Lm} For $\theta\in\Lambda_{G,P}$ the adjoint functors
$$
(v^{\theta}_P)^*: Shv(\Gr_G)^{M(\cO)}\leftrightarrows Shv(\Gr^{\theta}_P)^{M(\cO)}: (v^{\theta}_P)_*
$$ 
preserve the full subcategories of $H$-invariants and give rise to an adjoint pair
$$
(v^{\theta}_P)^*: Shv(\Gr_G)^H\leftrightarrows Shv(\Gr^{\theta}_P)^H: (v^{\theta}_P)_*
$$ 
The functors $(v_P)_*, v_P^*$ are $Sph_M$-linear. Besides, the diagram canonically commutes
$$
\begin{array}{ccc}
Shv(\Gr_G)^{M(\cO)} & \toup{(v^{\theta}_P)^*} & Shv(\Gr^{\theta}_P)^{M(\cO)}\\
\downarrow\lefteqn{\scriptstyle \Av^{U(P)(F)}_!} && \downarrow\lefteqn{\scriptstyle \Av^{U(P)(F)}_!} \\
Shv(\Gr_G)^H & \toup{(v^{\theta}_P)^*} & Shv(\Gr^{\theta}_P)^H
\end{array}
$$
The functor $\Av^{U(P)(F)}_!: Shv(\Gr^{\theta}_P)^{M(\cO)}\to Shv(\Gr^{\theta}_P)^H$ identifies canonically with $(\gt_P^{\theta})^!(\gt_P^{\theta})_!$. 
\end{Lm}  
\begin{proof}
The first claims are obtained as in Lemma~\ref{Lm_functor_i^nu^*}. Now using Lemma~\ref{Lm_2.3.8} for $K\in Shv(\Gr_P)^{M(\cO)}, L\in Shv(\Gr_M^{\theta})^{M(\cO)}$ we get
\begin{multline*}
\HOM_{Shv(\Gr_P)^{M(\cO)}}(K, (\gt_P^{\theta})^!L)\,\iso\, \HOM_{Shv(\Gr_M^{\theta})^{M(\cO)}}((\gt_P^{\theta})_! K, L)\,\iso\\
\HOM_{Shv(\Gr_P)^{M(\cO)}}((\gt_P^{\theta})^!(\gt_P^{\theta})_! K, (\gt_P^{\theta})^!L)
\end{multline*}
This gives the last claim. Since $v_P^*$ is left-lax $Sph_M$-linear, its  $Sph_M$-linearity follows from Remark~\ref{Rem_3.3.12_strictness} below.
\end{proof}

\begin{Rem} 
\label{Rem_3.3.12_strictness}
If $C,C'\in Sph_M-mod$ and $h: C\to C'$ is a left-lax morphism of $Sph_M$-module categories then $h$ is $Sph_M$-linear. Indeed, $\Rep(\check{M})$ is rigid, so it suffices to show that $\Sat: \Rep(\check{M})\to \Sph_M$ generates the target under colimits. This is true, because $\Sph_M$ is compactly generated, and each compact object of $\Sph_M$ is cohomologically bounded. In the case of $\cD$-modules this also follows from (\cite{GaLocWhit}, D.4.4)). 
\end{Rem}

\sssec{} Write $\upsilon: Shv(\Gr_G)^{G(\cO)}\to Shv(\Gr_G)^{G(\cO)}$ for the functor induced by $G(F)\to G(F), g\mapsto g^{-1}$, and similarly for $\upsilon: Shv(\Gr_M)^{M(\cO)}\to Shv(\Gr_M)^{M(\cO)}$. 

By Section~\ref{Sect_A.3}, we have canonical Verdier self-dualities 
$$
\VD: Shv(\Gr_G)^{I_P}\,\iso\, (Shv(\Gr_G)^{I_P})^{\vee}
$$ 
and $Shv(\Gr_G)^{I_P, ren}\,\iso\, (Shv(\Gr_G)^{I_P, ren})^{\vee}$. 
Under this duality, for $K\in Shv(\Gr_G)^{G(\cO)}$ the dual of the functor $Shv(\Gr_G)^{I_P}\to Shv(\Gr_G)^{I_P}$, $F\mapsto F\ast K$ identifies with the functor $F\mapsto F\ast (\upsilon K)$. The same holds for the renormalized versions. 

 In view of Proposition~\ref{Pp_2.2.6} and the equivalence (\ref{eq_ren_parahoric_versus_H}), the above yields self-dualities
$$
\VD_H: Shv(\Gr_G)^H\,\iso\,(Shv(\Gr_G)^H)^{\vee}\;\;\;\mbox{and}\;\;\; Shv(\Gr_G)^{H, ren}\,\iso\, (Shv(\Gr_G)^{H, ren})^{\vee}
$$

\sssec{} Fix a Chevalley involution $\sigma\in \Aut(G)$ acting as $z\mapsto z^{-1}$ on $T$ and sending each simple root subspace $\gg_{\alpha_i}$ isomorphically to $\gg_{-\alpha_i}$ (so, sending $B$ to $B^-$). Then $\sigma(P)=P^-$. Write $I_{P^-}$ for the preimage of $P^-$ under $G(\cO)\to G$. Note that $\sigma(I_P)=I_{P^-}$. By abuse of notations, write also $\sigma: \Gr_G\to\Gr_G$ for the map $zG(\cO)\mapsto \sigma(z)G(\cO)$. It induces an equivalence 
$$
\sigma: Shv(\Gr_G)^{I_P}\,\iso\, Shv(\Gr_G)^{I_{P^-}}
$$

 Let $H^-=M(\cO)U(P^-)(F)$. The map $\sigma: \Gr_G\to\Gr_G$ intertwines the $H$ and $H^-$-actions via the isomorphism also denoted $\sigma: H\,\iso\, H^-$, hence induces an equivalence
$$
\sigma: Shv(\Gr_G)^H\,\iso\, Shv(\Gr_G)^{H^-}
$$
The following diagram commutes
$$
\begin{array}{ccccc}
Shv(\Gr_G)^H & \toup{\oblv} & Shv(\Gr_G)^{M(\cO)} &\toup{\Av_*^{I_P/M(\cO)}} & Shv(\Gr_G)^{I_P}\\
\downarrow\lefteqn{\scriptstyle \sigma} && \downarrow\lefteqn{\scriptstyle \sigma} &&\downarrow\lefteqn{\scriptstyle \sigma}\\
Shv(\Gr_G)^{H^-} & \toup{\oblv} & Shv(\Gr_G)^{M(\cO)} &\toup{\Av_*^{I_{P^-}/M(\cO)}} & Shv(\Gr_G)^{I_{P^-}}
\end{array}
$$
The equivalence $\sigma$ is compatible with $\VD$, namely the diagram canonically commutes
$$
\begin{array}{ccc}
Shv(\Gr_G)^{I_P} & \toup{\sigma} & Shv(\Gr_G)^{I_{P^-}}\\
\downarrow\lefteqn{\scriptstyle \VD} && \downarrow\lefteqn{\scriptstyle \VD}\\
(Shv(\Gr_G)^{I_P})^{\vee} & \getsup{\sigma^{\vee}} & (Shv(\Gr_G)^{I_{P^-}})^{\vee}
\end{array}
$$
This gives the commutativity of the diagram
$$
\begin{array}{ccc}
Shv(\Gr_G)^H & \toup{\VD_H} & ((Shv(\Gr_G)^H)^{\vee}\\
\downarrow\lefteqn{\scriptstyle \sigma} && \uparrow\lefteqn{\scriptstyle \sigma^{\vee}}\\
Shv(\Gr_G)^{H^-} & \toup{\VD_{H^-}} & (Shv(\Gr_G)^{H^-})^{\vee},
\end{array}
$$
where $\VD_{H^-}$ is defined similarly to $\VD_H$ (replacing $P$ by $P^-$). 

 The involution $\sigma$ similarly yields a monoidal equivalence $\sigma: \Sph_G\,\iso\,\Sph_G$, for $\lambda\in\Lambda^+$ we have $\sigma\Sat(V^{\lambda})\,\iso\, V^{-w_0(\lambda)}$. For $K\in Shv(\Gr_G)^H$, $F\in \Sph_G$ one has canonically $\sigma(K\ast F)\,\iso\, \sigma(K)\ast \sigma(F)$ in $Shv(\Gr_G)^{H^-}$.

\begin{Pp} 
\label{Pp_2.3.21}
The action of $\Rep(\check{G})^{\heartsuit}$ on $Shv(\Gr_G)^H$ is t-exact.
\end{Pp}
\begin{proof} For $G=P$ the claim follows from (\cite{G_central}, Proposition~6). For $P=B$ this is (\cite{Gai19SI}, Proposition~2.8.2). Consider now any $P$. Take $V\in \Rep(\check{G})^{\heartsuit}$. It suffices to show that for $V$ finite-dimensional the functor $\_ \ast \Sat(V)$ is t-exact.

\smallskip\noindent
{\bf Step 1} We show that $\_ \ast \Sat(V)$ is right t-exact. Note that $Shv(\Gr_G)^{H,\le 0}\subset Shv(\Gr_G)^H$ is the smallest full subcategory containing $\bvartriangle^{\nu}$ for $\nu\in\Lambda^+_M$, stable under extensions and small colimits. 
 
 Let $\nu\in \Lambda^+_M$ and $\theta$ be its image in $\Lambda_{G,P}$. Now it suffices to show $\bvartriangle^{\nu}\ast \Sat(V)\in Shv(\Gr_G)^{H,\le 0}$. For this by Remark~\ref{Rem_2.3.17_about_t-structure} it suffices to show that for any $\theta'\in\Lambda_{G,P}$, 
$$
(v^{\theta'}_P)^*(\bvartriangle^{\nu}\ast \Sat(V))\in Shv(\Gr^{\theta'}_P)^{H, \le 0}
$$ 
 
Recall the convolution diagram $\Gr_G\ttimes \Gr_G:=\Conv_G=G(F)\times^{G(\cO)}\Gr_G$ with the product map $\act: \Conv_G\to\Gr_G$. 

  Define the ind-scheme $\Gr^{\theta}_P\ttimes \Gr^{\theta'-\theta}_P$ as follows. Let $P(F)^{\theta}$ be the preimage of $\Gr_M^{\theta}$ under $P(F)\to M(F)\to \Gr_M$.  After passing to reduced ind-schemes, one has an isomorphism $\Gr_P^{\theta}\,\iso\, P(F)^{\theta}/P(\cO)$.  We ignore the nilpotents here, as they do not change the category of sheaves on a given ind-scheme of ind-finite type. Set 
$$
\Gr^{\theta}_P\ttimes \Gr^{\theta'-\theta}_P=P(F)^{\theta}\times^{P(\cO)} \Gr^{\theta'-\theta}_P
$$

 The base change of $\act: \Gr^{\theta}_P\ttimes \Gr_G\to \Gr_G$ by $v^{\theta'}_P: \Gr^{\theta'}_P\hook{} \Gr_G$ is the convolution map
$$
\act': \Gr^{\theta}_P\ttimes \Gr^{\theta'-\theta}_P\to \Gr^{\theta'}_P.
$$ 
So, we must show that $\act'_!(\bvartriangle^{\nu}\tboxtimes (v^{\theta'-\theta}_P)^*\Sat(V))$ lies in $Shv(\Gr^{\theta'}_P)^{H,\le 0}$ or, equivalently, that 
\begin{equation}
\label{complex_one_for_Pp_2.3.21}
(\gt_P^{\theta'})_!\act'_!(\bvartriangle^{\nu}\tboxtimes (v^{\theta'-\theta}_P)^*\Sat(V))
\end{equation}
lies in perverse degrees $\le \<\theta', 2\check{\rho}-2\check{\rho}_M\>$. 

 Note that $(\gt^{\theta}_P)_!\bvartriangle^{\nu}$ is the extension by zero of $\omega[-\<\nu, 2\check{\rho}\>]$ under $\Gr_M^{\nu}\hook{} \Gr_M^{\theta}$. So, $(\gt^{\theta}_P)_!\bvartriangle^{\nu}$ is placed in perverse degrees $\le \<\theta, 2\check{\rho}-2\check{\rho}_M\>$.
 
  By Proposition~\ref{Pp_gRes_for_Levi} below, 
$$
(\gt_P^{\theta'-\theta})_!(v^{\theta'-\theta}_P)^*\Sat(V)[\<\theta'-\theta, 2\check{\rho}-2\check{\rho}_M\>]\,\iso\, \Sat_M(\Res(V)_{\theta'-\theta}),
$$ 
where $\Res(V)$ denotes its restriction under $\check{M}\to \check{G}$, $\Sat_M$ is the Satake functor for $M$, and for $\cV\in\Rep(\check{M})$ we denote by $\cV_{\theta}$ the direct summand on which $Z(\check{M})$ acts by $\theta$. Now (\ref{complex_one_for_Pp_2.3.21}) identifies with 
$$
((\gt^{\theta}_P)_!\bvartriangle^{\nu})\ast \Sat_M(\Res(V)_{\theta'-\theta})[\<\theta-\theta', 2\check{\rho}-2\check{\rho}_M\>],
$$
where the convolution is for $M$ now.  Our claim follows now from (\cite{G_central}, Proposition~6). 

\smallskip\noindent
{\bf Step 2} Recall that $\dim V<\infty$. The left and right adjoint of $Shv(\Gr_G)^H\to Shv(\Gr_G)^H$, $K\mapsto K\ast \Sat(V)$ is the functor $K\mapsto K\ast \Sat(V^*)$. Since the left adjoint if right t-exact, the right adjoint is left t-exact.
\end{proof}

\sssec{} Recall the following more precise version of (\cite{BG}, Theorem~4.3.4).   

\begin{Pp} 
\label{Pp_gRes_for_Levi} Let $\theta\in\Lambda_{G,P}$. The functor 
$$
(\gt_P^{\theta})_!(v_P^{\theta})^*[\<\theta, 2\check{\rho}-2\check{\rho}_M\>]: \Perv(\Gr_G)^{G(\cO)}\to Shv(\Gr_M^{\theta})^{M(\cO)}
$$ 
takes values in $\Perv(\Gr_M^{\theta})^{M(\cO)}$. Let $\gRes: \Perv(\Gr_G)^{G(\cO)}\to \Perv(\Gr_M)^{M(\cO)}$ be the functor 
$$
\mathop{\oplus}\limits_{\theta\in\Lambda_{G,P}} (\gt_P^{\theta})_!(v_P^{\theta})^*[\<\theta, 2\check{\rho}-2\check{\rho}_M\>]
$$
The diagram commutes
$$
\begin{array}{ccc}
\Rep(\check{G})^{\heartsuit} & \toup{\Sat_G} & \Perv(\Gr_G)^{G(\cO)}\\
\downarrow && \downarrow\lefteqn{\scriptstyle \gRes}\\
\Rep(\check{M})^{\heartsuit} & \toup{\Sat_M} & \Perv(\Gr_M)^{M(\cO)},
\end{array}
$$
where $\Sat_G, \Sat_M$ denote the Satake functors for $G,M$, and the left vertical arrow is the restriction along $\check{M}\hook{} \check{G}$.  
\end{Pp}
\begin{Rem} For $K\in \Perv(\Gr_G)^{G(\cO)}$ one has canonically $\upsilon \gRes(K)\,\iso\,\gRes(\upsilon K)$ in $\Perv(\Gr_M^{\theta})^{M(\cO)}$, as $\DD\upsilon =\upsilon \DD$ corresponds to the contragredient duality on $\Rep(\check{G})^{\heartsuit}$ by (\cite{FG}, 5.2.6). 
\end{Rem}

 This gives the following dual version of Proposition~\ref{Pp_gRes_for_Levi}. For $\theta\in\Lambda_{G,P}$ the functor 
$$
(\gt_P^{\theta})_*(v_P^{\theta})^![-\<\theta, 2\check{\rho}-2\check{\rho}_M\>]: \Perv(\Gr_G)^{G(\cO)}\to Shv(\Gr_M^{\theta})^{M(\cO)}
$$ 
takes values in $\Perv(\Gr_M^{\theta})^{M(\cO)}$. Let $\gRes^!: \Perv(\Gr_G)^{G(\cO)}\to \Perv(\Gr_M)^{M(\cO)}$ be the functor 
$$
\mathop{\oplus}\limits_{\theta\in\Lambda_{G,P}} (\gt_P^{\theta})_*(v_P^{\theta})^![-\<\theta, 2\check{\rho}-2\check{\rho}_M\>]
$$
The diagram commutes
$$
\begin{array}{ccc}
\Rep(\check{G})^{\heartsuit} & \toup{\Sat_G} & \Perv(\Gr_G)^{G(\cO)}\\
\downarrow && \downarrow\lefteqn{\scriptstyle \gRes^!}\\
\Rep(\check{M})^{\heartsuit} & \toup{\Sat_M} & \Perv(\Gr_M)^{M(\cO)},
\end{array}
$$
where the left vertical arrow is the restriction along $\check{M}\hook{} \check{G}$.

\sssec{} 
\label{Sect_true_Rep(checkM)-action_on_SI}
The category $Shv(\Gr_M)^{M(\cO)}$ acts on $Shv(\Gr_G)^{M(\cO)}$ by convolutions on the left. For $F\in Shv(\Gr_M)^{M(\cO)}$, $K\in Shv(\Gr_G)^{M(\cO)}$ we denote this (left) action by $(F, K)\mapsto F\ast K$. 

 Consider the action of $\Rep(\check{M})$ on $Shv(\Gr_G)^{M(\cO)}$ such that $V\in\Rep(\check{M})$ on which $Z(\check{M})$ acts by $\theta\in\Lambda_{G,P}$ sends $K\in Shv(\Gr_G)^{M(\cO)}$ to
\begin{equation}
\label{action_Rep(checkM)_shifted}
\Sat_M(V)\ast K[-\<\theta, 2\check{\rho}-2\check{\rho}_M\>]
\end{equation} 

 Write $\Gr_M\ttimes \Gr_G=M(F)\times^{M(\cO)} \Gr_G$ with the action map $\act: \Gr_M\ttimes \Gr_G\to\Gr_G$ coming from $M(F)\times \Gr_G\to \Gr_G$, $(m, gG(\cO))\mapsto mgG(\cO)$. More generally, for a $M(\cO)$-invariant ind-subscheme $Y\subset \Gr_G$, one similarly has $\Gr_M\ttimes \, Y$ and $\act: \Gr_M\ttimes \, Y\to \Gr_G$. 

\begin{Pp} 
\label{Pp_2.4.19}
i) The full subcategory $Shv(\Gr_G)^H\subset Shv(\Gr_G)^{M(\cO)}$ is preserved under the action of $Shv(\Gr_M)^{M(\cO)}$ on $Shv(\Gr_G)^{M(\cO)}$ by convolutions on the left. The functor 
$$
\Av^{U(P)(F)}_!: Shv(\Gr_G)^{M(\cO)}\to Shv(\Gr_G)^H
$$ 
is $\Rep(\check{M})$-linear.\\
ii) If $V\in \Rep(\check{M})^{\heartsuit}$ and $Z(\check{M})$ acts on $V$ by $\theta'\in\Lambda_{G,P}$ then the functor 
\begin{equation}
\label{functor_for_shifted_action_of_Rep_checkM}
Shv(\Gr_G)^H\to Shv(\Gr_G)^H, \,  K\mapsto \Sat_M(V)\ast K[-\<\theta', 2\check{\rho}-2\check{\rho}_M\>]
\end{equation}
is t-exact.\\
iii) Let $\theta,\theta'\in\Lambda_{G,P}$, $F\in Shv(\Gr_M^{\theta'})^{M(\cO)}$, $K\in Shv(\Gr_G)^{M(\cO)}$. One has a canonical isomorphism
$$
F\ast ((\gt_P^{\theta})_!(v^{\theta}_P)^*K)\,\iso\, (\gt_P^{\theta'+\theta})_!(v^{\theta'+\theta}_P)^*(F\ast K)
$$
in $Shv(\Gr_M^{\theta'+\theta})^{M(\cO)}$ functorial in $F$ and $K$. The analog of the latter isomorphism with $P$ replaced by $P^-$ also holds.
\end{Pp}
\begin{proof}
i) Let $\theta,\theta'\in \Lambda_{G,P}$, $F\in Shv(\Gr_M^{\theta'})^{M(\cO)}$, and $K\in Shv(\Gr_M^{\theta})^{M(\cO)}$. By Remark~\ref{Rem_2.3.14} and Lemma~\ref{Lm_2.3.8}, it suffices to show that $F \ast ((\gt^{\theta}_P)^!K)\in Shv(\Gr_G)^H$. The action map $\act: \Gr_M^{\theta'}\ttimes \Gr_P^{\theta}\to \Gr_G$ factors through $v^{\theta+\theta'}_P: \Gr_P^{\theta+\theta'}\hook{} \Gr_G$, and the square is cartesian
\begin{equation}
\label{diag_for_Pp_2.4.19}
\begin{array}{ccc}
\Gr_M^{\theta'}\ttimes \Gr_P^{\theta} &\toup{\act} & \Gr_P^{\theta+\theta'}\\
\downarrow\lefteqn{\scriptstyle\id\times \gt^{\theta}_P} && \downarrow\lefteqn{\scriptstyle \gt^{\theta+\theta'}_P}\\
\Gr_M^{\theta'}\ttimes \Gr_M^{\theta} &\toup{\act} & \Gr_M^{\theta+\theta'}
\end{array}
\end{equation}
We have 
$$
F\ast ((\gt^{\theta}_P)^!K)\,\iso\, \act_*(\id\times\gt_P^{\theta})^!(F\tboxtimes K)\,\iso\, (\gt^{\theta+\theta'}_P)^!(F\ast K)
$$
The first claim follows now from Lemma~\ref{Lm_2.3.8}. 

 Since $\Rep(\check{M})$ is rigid, the second claim follows from (\cite{G}, ch. I.1, 9.3.6). 

\smallskip\noindent
ii) We may assume $V\in \Rep(\check{M})^{\heartsuit}$ is finite-dimensional. 
Let us first show that (\ref{functor_for_shifted_action_of_Rep_checkM}) is right t-exact. Take $\theta\in\Lambda_{G,P}$. It suffices to show that for any $K\in Shv(\Gr_M^{\theta})^{M(\cO)}$ placed in perverse degrees $\le \<\theta, 2\check{\rho}-2\check{\rho}_M\>$, the object 
$$
\Sat_M(V)\ast ((\gt_P^{\theta})^!K)[-\<\theta', 2\check{\rho}-2\check{\rho}_M\>]
$$
lies in $Shv(\Gr_P^{\theta+\theta'})^{H, \le 0}$, that is, 
$$
(\gt_P^{\theta+\theta'})_!(\Sat_M(V)\ast ((\gt_P^{\theta})^!K))
$$
is placed in perverse degrees $\le \<\theta, 2\check{\rho}-2\check{\rho}_M\>$ over $\Gr_M^{\theta+\theta'}$. The latter complex identifies with $\Sat_M(V)\ast K$. Our claim follows now from (\cite{G_central}, Proposition~6). 

 To see that (\ref{functor_for_shifted_action_of_Rep_checkM}) is left t-exact argue as in Step 2 of Proposition~\ref{Pp_2.3.21}. Here we use the fact that $Z(\check{M})$ acts on $V^*$ as $-\theta'$.
 
\medskip\noindent
iii) We may and do assume that $K$ is the extension by zero from a closed $M(\cO)$-invariant subscheme of finite type in $\Gr_G$. 

\smallskip\noindent
{\bf Step 1}. We claim that 
$$
(v^{\theta'+\theta}_P)^*(F\ast K)\,\iso\, F\ast ((v^{\theta}_P)^*K)
$$ 
canonically in $Shv(\Gr_P^{\theta+\theta'})^{M(\cO)}$. Indeed, $K$ in $Shv(\Gr_G)^{M(\cO)}$ admits a finite filtration with succesive quotients $(v^{\mu}_P)_!(v^{\mu}_P)^*K$, $\mu\in \Lambda_{G,P}$. So, $F\ast K$ admits a finite filtration in $Shv(\Gr_P^{\theta+\theta'})^{M(\cO)}$ with succesive quotients $F\ast ((v^{\mu}_P)_!(v^{\mu}_P)^*K)$, $\mu\in \Lambda_{G,P}$. Our claim follows.

\smallskip\noindent
{\bf Step 2}. From (\ref{diag_for_Pp_2.4.19}) we get  
$$
(\gt^{\theta+\theta'}_P)_!(F\ast (v^{\theta}_P)^*K)\,\iso\, F\ast ((\gt_P^{\theta})_!(v^{\theta}_P)^*K)
$$
Our claim follows.
\end{proof}  

\sssec{} The action of $\Lambda_{M, ab}$ on $Shv(\Gr_G)^H$ from Section~\ref{Sect_2.2.10_action_of_Lambda_Mab} is obtained by restricting the $\Rep(\check{M})^{\heartsuit}$-action given by (\ref{functor_for_shifted_action_of_Rep_checkM}) to $\Rep(\check{M}_{ab})^{\heartsuit}$. 

\sssec{} For $\nu\in\Lambda^+_M$ denote by $j_{\nu}^-: S^{\nu}_{P^-}\hook{} \bar S^{\nu}_{P^-}$ this open immersion. Let $\bar i_{\nu}^-: \bar S^{\nu}_{P^-}\hook{} \Gr_G$ be the closed immersion and $i_{\nu}^-=\bar i_{\nu}^-\comp j_{\nu}^-$. We also have the closed embedding $i^{\nu}_{P^-}: \Gr_M^{\nu}\hook{} S^{\nu}_{P^-}$.

\sssec{} The following is not used in the paper. To complete the properties of $\SI_P$, we recall the following result of Lin Chen. 
\begin{Pp}[\cite{LC}, Corollary 1.4.5] Consider the composition
$$
Shv(\Gr_G)^H\otimes Shv(\Gr_G)^{H^-}\toup{\oblv\otimes\oblv} Shv(\Gr_G)^{M(\cO)}\otimes Shv(\Gr_G)^{M(\cO)}\to\Vect
$$
where the right arrow is the Verdier duality from Section~\ref{Sect_A.3}. This is a counit of a duality $Shv(\Gr_G)^H\,\iso\, (Shv(\Gr_G)^{H^-})^{\vee}$.  
\end{Pp}

\section{The semi-infinite $\IC$-sheaf $\IC^{\frac{\infty}{2}}_P$}

\ssec{Definition and first properties}
\sssec{} Recall that $\Lambda^+$ is the set of dominant coweights of $G$. Equip it with the relation $\lambda\le\mu$ iff $\mu-\lambda\in\Lambda^+$. Then $(\Lambda^+, \le)$ is a filtered category. In (\cite{Gai19SI}, Sections 2.3, 2.7) the functor 
\begin{equation}
\label{functor_initial_for_IC_SI}
(\Lambda^+,\le)\to Shv(\Gr_G)
\end{equation}
given by $\lambda\mapsto t^{-\lambda}\Sat(V^{\lambda})[\<\lambda, 2\check{\rho}\>]$ was constructed. Here $\Sat: \Rep(\check{G})\to Shv(\Gr_G)^{G(\cO)}$ is the Satake functor.  
Recall the construction of the transition maps in (\ref{functor_initial_for_IC_SI}). 

 First, for $\lambda\in\Lambda^+$ one has a canonical map
\begin{equation}
\label{map_fibre_of_Sat(V^lambda)} 
\delta_{1, \Gr_G}\to t^{-\lambda}\Sat(V^{\lambda})[\<\lambda, 2\check{\rho}\>]
\end{equation}
If moreover $\lambda\in\Lambda^+_{M, ab}$ then (\ref{map_fibre_of_Sat(V^lambda)}) is a map in $Shv(\Gr_G)^{M(\cO)}$ naturally. Now for $\lambda_i\in\Lambda^+$ with $\lambda_2-\lambda_1=\lambda\in\Lambda^+$ the morphism
\begin{equation}
\label{transition_map_for_IC_SI}
t^{-\lambda_1}\Sat(V^{\lambda_1})[\<\lambda_1, 2\check{\rho}\>]\to t^{-\lambda_2}\Sat(V^{\lambda_2})[\<\lambda_2, 2\check{\rho}\>]
\end{equation}
is defined as the composition
\begin{multline*}
t^{-\lambda_1}\Sat(V^{\lambda_1})[\<\lambda_1, 2\check{\rho}\>]\,\iso\, t^{-\lambda_1}\delta_{1,\Gr_G} \ast \Sat(V^{\lambda_1})[\<\lambda_1, 2\check{\rho}\>]\toup{(\ref{map_fibre_of_Sat(V^lambda)})}\\  t^{-\lambda_1}t^{-\lambda}\Sat(V^{\lambda})\ast \Sat(V^{\lambda_1})[\<\lambda+\lambda_1, 2\check{\rho}\>]\,\iso\, t^{-\lambda_2}\Sat(V^{\lambda}\otimes V^{\lambda_1})[\<\lambda_2, 2\check{\rho}\>]\toup{u^{\lambda, \lambda_1}}\\  t^{-\lambda_2}\Sat(V^{\lambda_2})[\<\lambda_2, 2\check{\rho}\>]
\end{multline*}

\sssec{} Consider the restriction
\begin{equation}
\label{functor_for_IC_SI_parabolic}
(\Lambda^+_{M, ab},\le)\to Shv(\Gr_G)
\end{equation}
of (\ref{functor_initial_for_IC_SI}) under $(\Lambda^+_{M, ab},\le)\to (\Lambda^+,\le)$. 
 If $\lambda\in\Lambda^+_{M, ab}$ then $t^{-\lambda}\Sat(V^{\lambda})$ is naturally $M(\cO)$-equivariant. Besides, (\ref{transition_map_for_IC_SI}) naturally upgraded to a morphism in $Shv(\Gr_G)^{M(\cO)}$. Our next purpose is to show that (\ref{functor_for_IC_SI_parabolic}) naturally upgraded to a functor with values in $Shv(\Gr_G)^{M(\cO)}$. We derive this from Lemma~\ref{Lm_C.4.9}.
  
 To do so, consider a left action of $\Lambda^+_{M, ab}$ on $Shv(\Gr_G)^{M(\cO)}$ such that $\lambda\in\Lambda^+_{M, ab}$ sends $K$ to $t^{\lambda}K[-\<\lambda, 2\check{\rho}\>]$, this is the action given by (\ref{action_Rep(checkM)_shifted}). Consider also the right lax action of $\Lambda^+_{M, ab}$ on $Shv(\Gr_G)^{M(\cO)}$ such that $\lambda$ acts on $K$ as $K\ast \Sat(V^{\lambda})$. The right lax structure on this action is given by the morphisms
$$
(K\ast \Sat(V^{\lambda}))\ast \Sat(V^{\mu})\,\iso\, K\ast \Sat(V^{\lambda}\otimes V^{\mu})\toup{u^{\lambda,\mu}} K\ast \Sat(V^{\lambda+\mu})
$$ 
for $\lambda,\mu\in\Lambda^+_{M, ab}$. We claim that $\delta_{1,\Gr_G}$  acquires a structure of a lax central object with respect to these actions in the sense of (\cite{Gai19SI}, 2.7.1), cf. also Appendix~\ref{Sect_lax-appendix}. The corresponding map
$$
t^{\lambda}\ast \delta_{1,\Gr_G}[-\<\lambda, 2\check{\rho}\>]\to \Sat(V^{\lambda})
$$
is (\ref{map_fibre_of_Sat(V^lambda)}). 

 Indeed, by Lemma~\ref{Lm_C.4.7_Dennis} it suffices to show that for any $\lambda\in\Lambda^+_M$ the object
\begin{equation}
\label{space_is_indeed_discrete_for_Sect_2.5.2}
\Map_{Shv(\Gr_G)^{M(\cO)}}(\delta_{t^{\lambda}}[-\<\lambda, 2\check{\rho}\>], \Sat(V^{\lambda}))\in\Spc
\end{equation}
is discrete. The corresponding internal hom in $\Vect$
$$
\HOM_{Shv(\Gr_G)^{M(\cO)}}(\delta_{t^{\lambda}}[-\<\lambda, 2\check{\rho}\>], \Sat(V^{\lambda}))
$$
is placed in degrees $\ge 0$. Applying the Dold-Kan functor $\Vect\to\Vect^{\le 0}\to \Spc$, we see that (\ref{space_is_indeed_discrete_for_Sect_2.5.2}) is discrete. Thus, (\ref{functor_for_IC_SI_parabolic}) naturally upgrades to a functor with values in $Shv(\Gr_G)^{M(\cO)}$. 
 
\begin{Def} 
\label{Def_IC_SI_parabolic}
Let $\IC^{\frac{\infty}{2}}_P\in Shv(\Gr_G)^{M(\cO)}$ be given by
\begin{equation}
\label{Def_IC_semi_inf}
\IC^{\frac{\infty}{2}}_P=\mathop{\colim}\limits_{\lambda\in\Lambda^+_{M, ab}}  t^{-\lambda}\Sat(V^{\lambda})[\<\lambda, 2\check{\rho}\>]
\end{equation}
\end{Def}

\sssec{} Let $\mu\in\Lambda^+_M$. One similarly gets a diagram
$$
(\{\lambda\in \Lambda^+_{M, ab}, \lambda+\mu\in\Lambda^+\}, \le)\to Shv(\Gr_G)^{M(\cO)}, \lambda\mapsto t^{-\lambda}\Sat(V^{\lambda+\mu})[\<\lambda, 2\check{\rho}\>]
$$
Let us only indicate the transition maps.

 Let $\lambda, \lambda_i\in\Lambda^+_{M, ab}$ with $\lambda_2=\lambda_1+\lambda$ and $\lambda\in\Lambda^+$. The transition morphism
$$
t^{-\lambda_1}\Sat(V^{\lambda_1+\mu})[\<\lambda_1, 2\check{\rho}\>]\to 
t^{-\lambda_2}\Sat(V^{\lambda_2+\mu})[\<\lambda_2, 2\check{\rho}\>]
$$
is the composition
\begin{multline*}
t^{-\lambda_1}\Sat(V^{\lambda_1+\mu})[\<\lambda_1, 2\check{\rho}\>]\,\iso\, t^{-\lambda_1}\delta_{1,\Gr_G} \ast \Sat(V^{\lambda_1+\mu})[\<\lambda_1, 2\check{\rho}\>]\toup{(\ref{map_fibre_of_Sat(V^lambda)})}\\  t^{-\lambda_1}t^{-\lambda}\Sat(V^{\lambda})\ast \Sat(V^{\lambda_1+\mu})[\<\lambda+\lambda_1, 2\check{\rho}\>]\,\iso\, t^{-\lambda_2}\Sat(V^{\lambda}\otimes V^{\lambda_1+\mu})[\<\lambda_2, 2\check{\rho}\>]\toup{u^{\lambda, \lambda_1+\mu}}\\  t^{-\lambda_2}\Sat(V^{\lambda_2+\mu})[\<\lambda_2, 2\check{\rho}\>]
\end{multline*}
\begin{Def} For $\mu\in\Lambda^+_M$ define $\IC^{\frac{\infty}{2}}_{P,\mu}\in Shv(\Gr_G)^{M(\cO)}$ by
\begin{equation}
\label{def_SI_IC_for_mu_stratum}
\IC^{\frac{\infty}{2}}_{P,\mu}=\mathop{\colim}\limits_{\lambda\in\Lambda^+_{M, ab}, \; \lambda+\mu\in\Lambda^+}  \, t^{-\lambda}\Sat(V^{\lambda+\mu})[\<\lambda, 2\check{\rho}\>]
\end{equation}
So, $\IC^{\frac{\infty}{2}}_P=\IC^{\frac{\infty}{2}}_{P,0}$.
\end{Def}

\sssec{} For $\theta\in\Lambda_{G,P}$, $V\in\Rep(\check{M})^{\heartsuit}$ let $V_{\theta}$ be the subspace of $V$ on which the center $Z(\check{M})$ of $\check{M}$ acts by $\theta$. 

 The first properties of $\IC^{\frac{\infty}{2}}_{P,\mu}$ are as follows.

\begin{Pp}
\label{Pp_2.4.7_first_propeties} Let $\eta\in\Lambda_M^+$.\\
a) The object $\IC^{\frac{\infty}{2}}_{P,\eta}$ belongs to $Shv(\Gr_G)^H\subset Shv(\Gr_G)^{M(\cO)}$.\\
b) $\IC^{\frac{\infty}{2}}_{P,\eta}$ is the extension by zero from $\bar S^{\eta}_P$. The ind-scheme $\bar S^{\eta}_P$ is stratified by $S^{\nu}_P$ with $\nu\in\Lambda^+_M$ such that $\eta-\nu\in\Lambda^{pos}$.

\medskip\noindent
c) One has $i_{\eta}^*\IC^{\frac{\infty}{2}}_{P,\eta}\,\iso\, i_{\eta}^!\IC^{\frac{\infty}{2}}_{P,\eta}\,\iso\, \omega[-\<\eta, 2\check{\rho}\>]$ over $S^{\eta}_P$. 

\medskip\noindent
d) $\IC^{\frac{\infty}{2}}_{P,\eta}$ belongs to $Shv(\Gr_G)^{H, \heartsuit}$. It admits no subobjects in $Shv(\Gr_G)^{H, \heartsuit}$, which are extensions by zero from $\bar S^{\eta}_P - S^{\eta}_P$.
\end{Pp}
\begin{proof}
a) Pick $\lambda\in\Lambda^+_{M, ab}$. Let us show that $\IC^{\frac{\infty}{2}}_{P,\eta}$ is $H_{\lambda}$-equivariant. For any $\mu\ge\lambda$, that is, with $\mu-\lambda\in\Lambda^+_{M, ab}$ the object $t^{-\mu}\Sat(V^{\mu+\eta})$ is naturally $t^{-\mu}G(\cO)t^{\mu}$-equivariant. Since $H_{\lambda}\subset H_{\mu}\subset t^{-\mu}G(\cO)t^{\mu}$, our claim follows, because $\{\mu\in\Lambda^+_{M, ab}\mid\mu\ge\lambda\}\subset \Lambda^+_{M, ab}$ is cofinal.

\smallskip\noindent
b) It suffices to show that for any $\lambda\in\Lambda^+_{M, ab}$ one has $t^{-\lambda}\ov{\Gr}_G^{\lambda+\eta}\subset \bar S^{\eta}_P$. Indeed, let $\nu\in\Lambda^+_M, \lambda\in\Lambda^+_{M, ab}$ and 
$$
S^{\nu}_P\cap (t^{-\lambda}\ov{\Gr}_G^{\lambda+\eta})\ne \emptyset
$$ 
By Remark~\ref{Rem_2.3.7}, $S^{\nu+\lambda}_P\cap \ov{\Gr}_G^{\lambda+\eta}\ne \emptyset$ and $\eta-\nu\in \Lambda^{pos}$. 

\smallskip\noindent
c) Let $i_{t^{\eta}}: \Spec k\to \Gr_G$ be the point $t^{\eta}$. By Lemma~\ref{Lm_2.3.8}, it suffices to show that 
$$
(i_{t^{\eta}})^!\IC^{\frac{\infty}{2}}_{P,\eta}\,\iso\,e[-\<\eta, 2\check{\rho}\>]
$$ 
We are calculating the colimit of the $!$-fibres at $t^{\eta}\in\Gr_G$ of 
$$
t^{-\lambda}\Sat(V^{\lambda+\eta})[\<\lambda, 2\check{\rho}\>]
$$ 
over $\lambda\in\Lambda^+_{M, ab}$. Each term of this diagram identifies canonically with $e[-\<\eta, 2\check{\rho}\>]$, and the transition maps are the identity. The claim follows.

\smallskip\noindent
d) {\bf Step 1}. We show that $\IC^{\frac{\infty}{2}}_{P,\eta}\in Shv(\Gr_G)^{H,\ge 0}$. 
Let $0\ne \eta-\nu\in\Lambda^{pos}$ with $\nu\in\Lambda^+_M$. We check that $(i^{\nu}_P)^!i_{\nu}^!\IC^{\frac{\infty}{2}}_{P,\eta}$ is placed in perverse degrees $> \<\nu, 2\check{\rho}-2\check{\rho}_M\>$ over $\Gr_M^{\nu}$. If $\lambda\in\Lambda_{M, ab}^+$ is large enough for $\nu$ then $\lambda+\nu\in\Lambda^+_G$. It suffices to show that for such $\lambda$, 
$$
(i^{\nu+\lambda}_P)^!i_{\nu+\lambda}^!\Sat(V^{\lambda+\eta})[\<\lambda, 2\check{\rho}\>] 
$$
is placed in perverse degrees $> \<\nu, 2\check{\rho}-2\check{\rho}_M\>$ over $\Gr_M^{\nu+\lambda}$. 

 We have $\Gr_M^{\nu+\lambda}\subset \Gr_G^{\nu+\lambda}$. The !-restriction of $\Sat(V^{\lambda+\eta})$ to $\Gr_G^{\nu+\lambda}$ is placed in perverse degrees $>0$, and has smooth perverse cohomology sheaves. Recall that $\dim\Gr_G^{\nu+\lambda}=\<\nu+\lambda, 2\check{\rho}\>$. 
 
 For any bounded complex on $\Gr_G^{\nu+\lambda}$ placed in perverse degrees $>0$ and having smooth perverse cohomology sheaves, its !-restriction to $\Gr_M^{\nu+\lambda}$ is placed in perverse degrees $>\codim(\Gr_M^{\nu+\lambda}, \Gr_G^{\nu+\lambda})=\<\nu+\lambda, 2\check{\rho}-2\check{\rho}_M\>$. 
Since $\<\lambda, 2\check{\rho}_M\>=0$, our claim follows. We proved actually that for $\nu$ as above, $(i_{\nu})^!\IC^{\frac{\infty}{2}}_{P,\eta}\in Shv(S^{\nu}_P)^{H, >0}$. 

\medskip\noindent
{\bf Step 2}. Let us show that $\IC^{\frac{\infty}{2}}_{P,\eta}\in Shv(\Gr_G)^{H,\le 0}$. It suffices to show that for $\theta\in\Lambda_{G,P}$, 
$$
(v^{\theta}_P)^*\IC^{\frac{\infty}{2}}_{P,\eta}\in Shv(\Gr^{\theta}_P)^{H, \le 0}
$$ 
or, equivalently, $(\gt^{\theta}_P)_!(v^{\theta}_P)^*\IC^{\frac{\infty}{2}}_{P,\eta}$ is placed in perverse degrees $\le \<\theta, 2\check{\rho}-2\check{\rho}_M\>$ over $\Gr_M^{\theta}$. We have 
$$
(\gt^{\theta}_P)_!(v^{\theta}_P)^*\IC^{\frac{\infty}{2}}_{P,\eta}\,\iso\,\mathop{\colim}\limits_{\lambda\in\Lambda_{M,ab}^+, \; \lambda+\eta\in\Lambda^+} \; (\gt^{\theta}_P)_!(v^{\theta}_P)^*(t^{-\lambda}\Sat(V^{\lambda+\eta}))[\<\lambda, 2\check{\rho}\>]  
$$
The latter identifies with 
\begin{equation}
\label{complex_for_d_Pp_2.4.6}
\mathop{\colim}\limits_{\lambda\in\Lambda_{M, ab}^+, \; \lambda+\eta\in\Lambda^+} \; t^{-\lambda}(\gt^{\theta+\bar\lambda}_P)_!(v^{\theta+\bar\lambda}_P)^*\Sat(V^{\lambda+\eta})[\<\lambda, 2\check{\rho}\>]  
\end{equation}
By Proposition~\ref{Pp_gRes_for_Levi}, 
$$
(\gt^{\theta+\bar\lambda}_P)_!(v^{\theta+\bar\lambda}_P)^*\Sat(V^{\lambda+\eta})[\<\theta+\bar\lambda, 2\check{\rho}-2\check{\rho}_M\>]\,\iso\, \Sat_M((V^{\lambda+\eta})_{\theta+\bar\lambda})
$$ 
So, (\ref{complex_for_d_Pp_2.4.6}) identifies with
\begin{equation}
\label{diag_after_Jacquet_functors_for_theta}
\mathop{\colim}\limits_{\lambda\in\Lambda_{M, ab}^+, \; \lambda+\eta\in\Lambda^+} \; t^{-\lambda}\Sat_M((V^{\lambda+\eta})_{\theta+\bar\lambda})[-\<\theta, 2\check{\rho}-2\check{\rho}_M\>]
\end{equation}
This shows that $(v^{\theta}_P)^*\IC^{\frac{\infty}{2}}_P\in Shv(\Gr^{\theta}_P)^{H, \heartsuit}$ for any $\theta\in\Lambda_{G,P}$. Our claim follows. 
\end{proof}

\begin{Rem} i) The object $\IC^{\frac{\infty}{2}}_{P,\eta}$ is the extension by zero from the connected component $\Gr_G^{\bar\eta}$ of $\Gr_G$, where $\bar\eta\in \Lambda_{G,G}$ is the image of $\eta$.\\
ii) If $P=G$ then $\IC^{\frac{\infty}{2}}_{P,\eta}\,\iso\, \Sat(V^{\eta})$ canonically. 
\end{Rem}

\sssec{Another presentation of $\IC^{\frac{\infty}{2}}_{P,\eta}$} The inclusion $Shv(\Gr_G)^H\hook{} Shv(\Gr_G)^{M(\cO)}$ commutes with $\Rep(\check{G})$-actions. Since $\Rep(\check{G})$ is rigid, on the left adjoint $\Av^{U(P)(F)}_!: Shv(\Gr_G)^{M(\cO)}\to Shv(\Gr_G)^H$ the left-lax $\Rep(\check{G})$-structure is strict by (\cite{G}, ch. I.1, 9.3.6). Let $\eta\in\Lambda^+_M$. Using Lemma~\ref{Lm_2.3.13_about_Av!} and applying $\Av^{U(P)(F)}_!$ to (\ref{def_SI_IC_for_mu_stratum}) we get
$$
\IC^{\frac{\infty}{2}}_{P,\eta}\,\iso\,\mathop{\colim}\limits_{\lambda\in\Lambda_{M, ab}^+,\, \lambda+\eta\in\Lambda^+}\, \bvartriangle^{-\lambda}\ast \Sat(V^{\lambda+\eta}),
$$
where the colimit is taken in $Shv(\Gr_G)^H$.

Generalizing (\cite{Gai19SI}, 1.5.6) we have the following.
\begin{Thm}
\label{Thm_4.1.10}
For any $\lambda\in\Lambda_{M, ab}$, $\bvartriangle^{\lambda}\in Shv(\Gr_G)^{H, \heartsuit}$.
\end{Thm}
The proof is postponed to Section~\ref{Sect_4.6}. 

\sssec{} Let $\theta\in\Lambda_{G,P}, \eta\in\Lambda^+_M$. Consider the diagram 
\begin{equation}
\label{diag_for_Rep(checkM)}
\{\lambda\in\Lambda_{M, ab}, \; \lambda+\eta\in\Lambda^+\}
\to \Rep(\check{M}), \;\; \lambda\mapsto (e^{-\lambda}\otimes V^{\lambda+\eta})_{\theta}
\end{equation}
obtained from (\ref{diag_for_Sect_2.1.7}) by restricting to $\check{M}$ and imposing on each term the condition that $Z(\check{M})$ acts by $\theta$. 

Write $\cO(U(\check{P}))$ for the ring of regular functions on $U(\check{P})$. 
By Lemma~\ref{Lm_2.1.8_some_colimit}, 
$$
\mathop{\colim}_{\lambda\in \Lambda_{M, ab}, \; \lambda+\eta\in\Lambda^+}\; 
(e^{-\lambda}\otimes V^{\lambda+\eta})_{\theta}\,\iso\, (\cO(U(\check{P}))\otimes U^{\eta})_{\theta}
$$
Here $\check{M}$ acts (on the left) diagonally on $\cO(U(\check{P}))\otimes U^{\eta}$, the action on the first factor comes from the adjoint $\check{M}$-action on $U(\check{P})$. 

\begin{Pp}  
\label{Pp_2.5.15_*-restriction}
Let $\theta\in\Lambda_{G,P}$, $\eta\in\Lambda^+_M$. One has canonically
$$
(\gt^{\theta}_P)_!(v^{\theta}_P)^*\IC^{\frac{\infty}{2}}_{P,\eta}[\<\theta, 2\check{\rho}-2\check{\rho}_M\>]\;\iso\; Sat_M((\cO(U(\check{P}))\otimes U^{\eta})_{\theta})
$$
in $Shv(\Gr_M)^{M(\cO)}$. So, 
$$
(v^{\theta}_P)^*\IC^{\frac{\infty}{2}}_{P,\eta}[\<\theta, 2\check{\rho}-2\check{\rho}_M\>]\;\iso\; (\gt^{\theta}_P)^!Sat_M((\cO(U(\check{P}))\otimes U^{\eta})_{\theta})
$$
in $Shv(\Gr_P^{\theta})^H$. 
\end{Pp} 
\begin{proof}
Applying $\Sat_M: \Rep(\check{M})\to \Perv(\Gr_M)^{M(\cO)}$ to (\ref{diag_for_Rep(checkM)}) and further tensoring by $e[-\<\theta, 2\check{\rho}-2\check{\rho}_M\>]$ one gets the diagram (\ref{diag_after_Jacquet_functors_for_theta}). The first claim follows as in Step 2 in the proof of Proposition~\ref{Pp_2.4.7_first_propeties}  d). The second one follows from Lemma~\ref{Lm_2.3.8} ii). 
\end{proof}

\begin{Rem} Assume that $G\ne P$. It is easy to see that if $\theta\in\Lambda_{G,P}, \eta\in\Lambda^+_M$ then $(\gt_P^{\theta})_*(v^{\theta}_P)^!\IC_{P,\eta}^{\frac{\infty}{2}}$ is infinitely coconnective in the t-structure on $Shv(\Gr_M^{\theta})^{M(\cO)}$.
\end{Rem}

\begin{Rem}
Let $\eta\in\Lambda^+_M$. If $G\ne P$ then $\IC^{\frac{\infty}{2}}_{P, \eta}$ is not the intermediate extension under $S^{\eta}_P\hook{} \bar S^{\eta}_P$. Indeed, one can easily find $0\ne \theta'\in -\Lambda_{G,P}^{pos}$ such that $\cO(U(\check{P}))_{\theta'}\ne 0$. Let $\theta-\theta'$ be the image of $\eta$ in $\Lambda_{G,P}$. Then $(v^{\theta}_P)^*\IC^{\frac{\infty}{2}}_{P,\eta}\in Shv(\Gr_P^{\theta})^{H, \heartsuit}$ is not zero by Proposition~\ref{Pp_2.5.15_*-restriction}. 
\end{Rem}

 For $\eta\in\Lambda^+$ write $i^{\eta}_M: \ov{\Gr}_M^{\eta}\hook{} \ov{\Gr}_G^{\eta}$ for the natural closed immersion.

\begin{Lm} Let $\eta\in\Lambda^+$. \\
i) For any $\nu\in \Lambda^+$ with $\eta-\nu\in\Lambda^{pos}$ the scheme $\ov{\Gr}_M^{\eta}\cap \Gr_G^{\nu}$ is empty unless $\eta-\nu\in\Lambda_M^{pos}$, and in the latter case 
$$
\ov{\Gr}_M^{\eta}\cap \Gr_G^{\nu}=\Gr_M^{\nu}.
$$
If $\tau\in\Lambda^+_M$ with $\Gr_M^{\tau}\subset \ov{\Gr}_M^{\eta}$ then $\tau\in\Lambda^+$.\\
ii) The complex $(i^{\eta}_M)^!\Sat(V^{\eta})$ is placed in perverse degrees $\ge \<\eta, 2\check{\rho}-2\check{\rho}_M\>$. \\
iii) There is a canonical morphism
\begin{equation}
\label{map_for_Lm_2.5.17}
\Sat_M(U^{\eta})\ast \delta_{1,\Gr_G}[-\<\eta, 2\check{\rho}-2\check{\rho}_M\>]\to \Sat(V^{\eta})
\end{equation}
in $Shv(\Gr_G)^{M(\cO)}$, which after applying $(i^{\eta}_M)^!$ becomes an isomorphism on perverse cohomology sheaves in degree $\<\eta, 2\check{\rho}-2\check{\rho}_M\>$.
\end{Lm}
\begin{proof}
{\bf Step 1} Let us show that $\Gr_G^{\eta}\cap \ov{\Gr}_M^{\eta}=\Gr_M^{\eta}$. Let $\nu\in\Lambda^+_M$ with $0\ne \eta-\nu\in\Lambda^{pos}_M$ such that $\Gr_G^{\eta}\cap\Gr_M^{\nu}\ne\emptyset$. We must get a contradiction. Pick $w\in W$ such that $\nu=w\eta$ and the length of $w$ is minimal with with property. 

 We claim that if $i\in \cI_M$ then $w^{-1}(\check{\alpha}_i)$ is a positive root. Indeed, if $w^{-1}(\check{\alpha}_i)$ is negative then $\<\eta, w^{-1}(\check{\alpha}_i)\>\le 0$ on one hand, and on the other hand $\<\nu, \check{\alpha}_i\>=\<\eta, w^{-1}(\check{\alpha}_i)\>\ge 0$. So, $\<\eta, w^{-1}(\check{\alpha}_i)\>=0$, hence $s_{\alpha_i}w\eta=w\eta$. By (\cite{Spr}, Lemma~8.3.2), $\ell(s_{\alpha_i}w)<\ell(w)$. This conradicts our choice of $w$.
 
 The above implies that for $i\in \cI_M$, $w^{-1}(\alpha_i)$ is a positive coroot. Applying $w^{-1}$ to $\eta-w\eta\in \Lambda^{pos}_M$ we get $0\ne w^{-1}\eta-\eta\in \Lambda^{pos}_M$. This is impossible, because $\eta\in\Lambda^+$ and for $w'\in W$, $\eta-w'\eta\in\Lambda^{pos}$. Our claim follows. 
 
 Note that for $i: \Gr_M^{\eta}\hook{} \Gr_G^{\eta}$ we have canonically
\begin{equation}
\label{iso_for_Step1_Lm_2.5.17}
i^!\Sat(V^{\eta})\,\iso\, \Sat_M(U^{\eta})[-\<\eta, 2\check{\rho}-2\check{\rho}_M\>]\mid_{\Gr_M^{\eta}}
\end{equation} 

\medskip\noindent
{\bf Step 2}. For $\nu\in\Lambda^+$ let $A(\eta, \nu)=\{\tau\in \Lambda^+_M\mid \eta-\tau\in\Lambda_M^{pos}, \tau\in W\nu\}$. Note that 
$$
\ov{\Gr}_M^{\eta}\cap \Gr_G^{\nu}=\underset{\tau\in A(\eta, \nu)}{\cup} \Gr_M^{\tau}.
$$ 

 Let $\nu\in\Lambda^+$ with $\Gr_G^{\nu}\subset \ov{\Gr}_G^{\eta}$, so $\eta-\nu\in\Lambda^{pos}$. Let $\tau\in \Lambda^+_M$ such that $\Gr_M^{\tau}\subset \ov{\Gr}_M^{\eta}\cap \Gr_G^{\nu}$, so $\tau\in A(\eta, \nu)$. We have $\eta-\tau\in\Lambda_M^{pos}$, and $\nu-\tau, \eta-\nu\in \Lambda^{pos}$. Since 
$$
(\eta-\nu)+(\nu-\tau)=\eta-\tau\in \Lambda_M^{pos},
$$
we conclude that 
$$
\nu-\tau, \eta-\nu\in \Lambda^{pos}_M.
$$ 
So, $\Gr_M^{\tau}\subset \ov{\Gr}_M^{\nu}\subset \ov{\Gr}_M^{\eta}$. By Step 1, $\ov{\Gr}_M^{\nu}\cap \Gr_G^{\nu}=\Gr_M^{\nu}$, so $\Gr_M^{\tau}\subset \ov{\Gr}_M^{\nu}\cap \Gr_G^{\nu}=\Gr_M^{\nu}$ and $\tau=\nu$. The first claim of i) is proved.

 For the second claim, let $\tau\in\Lambda^+_M$ with $\Gr_M^{\tau}\subset \ov{\Gr}_M^{\eta}$. Let $\nu\in\Lambda^+$ with $\tau\in W\nu$.Then $\Gr_M^{\tau}\subset \Gr_G^{\nu}\cap \ov{\Gr}_G^{\eta}$, and the second claim follows from the first. 

\medskip\noindent
{\bf Step 3} Let $\tau\in\Lambda^+_M$ such that $\Gr_M^{\tau}\subset \ov{\Gr}_M^{\eta}$ and $\tau\ne \eta$. We show that the !-restriction of $\Sat(V^{\eta})$ to $\Gr^{\tau}_M$ is placed in perverse degrees $> \<\eta, 2\check{\rho}-2\check{\rho}_M\>$. 

 By i), $\tau\in\Lambda^+$ and $\Gr_G^{\tau}\subset \ov{\Gr}_G^{\eta}$. The $!$-restriction of $\Sat(V^{\eta})$ to $\Gr_G^{\tau}$ is placed in perverse degrees $>0$ and has smooth perverse cohomology sheaves.

 For any bounded complex on $\Gr_G^{\tau}$ placed in perverse degrees $>0$ and having smooth perverse cohomology sheaves, its !-restriction to $\Gr_M^{\tau}$ is placed in perverse degrees $>\codim(\Gr_M^{\tau}, \Gr_G^{\tau})=\<\tau, 2\check{\rho}-2\check{\rho}_M\>=\<\eta, 2\check{\rho}-2\check{\rho}_M\>$, because the image of $\eta-\nu$ vanishes in $\Lambda_{G,P}$ by i). Part ii) follows. Moreover, the above gives a unique map (\ref{map_for_Lm_2.5.17}) whose restriction to $\Gr_M^{\eta}$ comes from (\ref{iso_for_Step1_Lm_2.5.17}).
\end{proof}

\begin{Pp} 
\label{Pp_action_of_Rep(checkM)_on_SI-IC_P}
Let $\eta\in\Lambda^+_M$. One has canonically in $Shv(\Gr_G)^H$
\begin{equation}
\label{iso_M-action_on_IC-SI_Prop_2.5.17}
\Sat_M(U^{\eta})\ast \IC^{\frac{\infty}{2}}_P[-\<\eta, 2\check{\rho}-2\check{\rho}_M\>]\,\iso\, \IC^{\frac{\infty}{2}}_{P,\eta}
\end{equation}
\end{Pp}
\begin{proof}
{\bf Step 1}. In the case $\eta\in\Lambda_{M, ab}, \mu\in\Lambda^+_M$ from (\ref{def_SI_IC_for_mu_stratum}) making the change of variables in the colimit one gets $t^{\eta}\IC^{\frac{\infty}{2}}_{P,\mu}[-\<\eta, 2\check{\rho}\>]\,\iso\, \IC^{\frac{\infty}{2}}_{P,\mu+\eta}$. 

 To establish (\ref{iso_M-action_on_IC-SI_Prop_2.5.17}) in general, we first reduce to the case $\eta\in\Lambda^+$. For this pick $\lambda\in\Lambda_{M, ab}$ such that $\lambda+\eta\in\Lambda^+$. If 
$$
\Sat_M(U^{\eta+\lambda})\ast \IC^{\frac{\infty}{2}}_P[-\<\eta+\lambda, 2\check{\rho}-2\check{\rho}_M\>]\,\iso\, \IC^{\frac{\infty}{2}}_{P,\eta+\lambda}
$$
then applying $t^{-\lambda}[\<\lambda, 2\check{\rho}\>]$ to the latter isomorphism, one gets (\ref{iso_M-action_on_IC-SI_Prop_2.5.17}) by the above.

\medskip\noindent
{\bf Step 2} Assume $\eta\in\Lambda^+$. Let us construct a morphism 
of functors 
\begin{equation}
\label{map_of_functors_for_Pp_2.5.18}
(\Lambda^+_{M, ab}, \le)\to Shv(\Gr_G)^{M(\cO)}
\end{equation}
sending $\lambda$ to
\begin{equation}
\label{map_first_for_Pp_2.5.18}
\Sat_M(U^{\eta})\ast t^{-\lambda}\Sat(V^{\lambda})[\<\lambda, 2\check{\rho}\>-\<\eta, 2\check{\rho}-2\check{\rho}_M\>]\to t^{-\lambda}\Sat(V^{\lambda+\eta})[\<\lambda, 2\check{\rho}\>],
\end{equation}
here we used the diagram defining $\IC^{\frac{\infty}{2}}_P$ in the LHS and $\IC^{\frac{\infty}{2}}_{P,\eta}$ in the RHS.

 Since $t^{\lambda}\Sat_M(U^{\eta})\ast \delta_{t^{-\lambda}}\,\iso\, \Sat_M(U^{\eta})$, (\ref{map_first_for_Pp_2.5.18}) rewrites as 
$$
\Sat_M(U^{\eta})\ast \Sat(V^{\lambda})[-\<\eta, 2\check{\rho}-2\check{\rho}_M\>]\to \Sat(V^{\lambda+\eta})
$$ 
We define the latter morphism as the composition
\begin{multline*}
\Sat_M(U^{\eta})\ast \Sat(V^{\lambda})[-\<\eta, 2\check{\rho}-2\check{\rho}_M\>]\toup{(\ref{map_for_Lm_2.5.17})} \Sat(V^{\eta})\ast \Sat(V^{\lambda})\\ \iso\,\Sat(V^{\eta}\otimes V^{\lambda})\toup{u^{\eta,\lambda}} \Sat(V^{\lambda+\eta})
\end{multline*}
These maps naturally upgrade to a morphism of functors (\ref{map_of_functors_for_Pp_2.5.18}). Passing to the colimit, this gives the morphism (\ref{iso_M-action_on_IC-SI_Prop_2.5.17}). 

 Let $\theta'$ be the image of $\eta$ in $\Lambda_{G,P}$. To show that (\ref{iso_M-action_on_IC-SI_Prop_2.5.17}) is an isomorphism, it suffices, in view of Lemma~\ref{Lm_2.3.8} ii), to prove that for any $\theta\in\Lambda_{G,P}$ applying the functor 
$$
(\gt_P^{\theta+\theta'})_!(v^{\theta+\theta'}_P)^*
: Shv(\Gr_G)^{M(\cO)}\to Shv(\Gr_M^{\theta+\theta'})^{M(\cO)}
$$ 
to (\ref{iso_M-action_on_IC-SI_Prop_2.5.17}) one gets an isomorphism.
 
 By Propositions~\ref{Pp_2.4.19} iii) and \ref{Pp_2.5.15_*-restriction} we get 
\begin{multline}
\label{complex_LHS_for_Pp_2.5.18}
(\gt_P^{\theta+\theta'})_!(v^{\theta+\theta'}_P)^*(\Sat_M(U^{\eta})\ast \IC^{\frac{\infty}{2}}_P)\,\iso\, \Sat_M(U^{\eta})\ast ((\gt_P^{\theta})_!(v^{\theta}_P)^*\IC^{\frac{\infty}{2}}_P)\,\iso\\
\Sat_M(U^{\eta})\ast \Sat_M((\cO(U(\check{P}))_{\theta})[-\<\theta, 2\check{\rho}-2\check{\rho}_M\>]\,\iso\\ \Sat_M((\cO(U(\check{P})\otimes U^{\eta})_{\theta+\theta'})[-\<\theta, 2\check{\rho}-2\check{\rho}_M\>]
\end{multline}
and
\begin{multline}
\label{complex_RHS_for_Pp_2.5.18}
(\gt_P^{\theta+\theta'})_!(v^{\theta+\theta'}_P)^* \IC^{\frac{\infty}{2}}_{P, \eta}[\<\eta, 2\check{\rho}-2\check{\rho}_M\>]\,\iso\, \Sat_M((\cO(U(\check{P})\otimes U^{\eta})_{\theta+\theta'})[-\<\theta, 2\check{\rho}-2\check{\rho}_M\>]
\end{multline}
One checks that the so obtained map from (\ref{complex_LHS_for_Pp_2.5.18}) to (\ref{complex_RHS_for_Pp_2.5.18}) in $Shv(\Gr_M^{\theta+\theta'})^{M(\cO)}$ is an isomorphism. We are done.
\end{proof}

\sssec{} According to Section~\ref{Sect_2.0.17}, we interpret Definition~\ref{Def_IC_SI_parabolic} as follows. Consider the $\Lambda_{M, ab}$-action on $Shv(\Gr_G)^H$ defined 
in Section~\ref{Sect_2.2.10_action_of_Lambda_Mab}, we also think of it as $\Rep(\check{M}_{ab})$-action. It is also equipped with $\Rep(\check{G})$-action by right convolutions. Then $\delta_{1,\Gr_G}\in \cO(\check{G}/[\check{P},\check{P}])-mod(C)$ for 
$C=Shv(\Gr_G)^H$, so that $\IC^{\frac{\infty}{2}}_P$ naturally upgrades to an object of 
$$
C\otimes_{\Rep(\check{G})\otimes\Rep(\check{M}_{ab})} \Rep(\check{M}),
$$  
so has the Hecke property described as in Section~\ref{Sect_version_of_Hecke_property_2.0.16}.

 In fact, it has a stronger Hecke property given as follows.
 
\begin{Pp} 
\label{Pp_2.5.18_upgrading_IC_semi-infinite_P}
$\IC^{\frac{\infty}{2}}_P$ naturally upgrades to an object of 
$$
(Shv(\Gr_G)^H)\otimes_{\Rep(\check{G})\otimes\Rep(\check{M})} \Rep(\check{M}),
$$ 
where we consider the $\Rep(\check{M})$-action given by (\ref{action_Rep(checkM)_shifted}), and $\Rep(\check{G})$-action by right convolutions.
\end{Pp}
\begin{proof}
The structure under consideration is equivalent to the following Hecke property. For $V\in\Rep(\check{G})^{\heartsuit}$ one has canonically
\begin{equation}
\label{iso_Hecke_propert_of_ICinfty/2_P}
\IC^{\frac{\infty}{2}}_P\ast \Sat(V)\,\iso\, \mathop{\oplus}\limits_{\mu\in\Lambda^+_M} \Sat_M(U^{\mu})\ast
\IC^{\frac{\infty}{2}}_P\otimes \Hom_{\check{M}}(U^{\mu}, V)[-\<\mu, 2\check{\rho}-2\check{\rho}_M\>]
\end{equation}
in a way compatible with the monoidal structures on $\Rep(\check{G})^{\heartsuit}$, $\Rep(\check{M})^{\heartsuit}$. 

 By Proposition~\ref{Pp_2.4.19} the isomorphisms (\ref{iso_Hecke_propert_of_ICinfty/2_P}) take place in the abelian category $Shv(\Gr_G)^{H,\heartsuit}$, so that higher compatibilities will be easy to check.

 To establish (\ref{iso_Hecke_propert_of_ICinfty/2_P}), we may assume $V$ finite-dimensional. Write $\Lambda^+_{M, ab}(V)$ for the set of $\lambda\in\Lambda^+_{M, ab}$ such that if $\mu\in\Lambda^+_M$ and $U^{\mu}$ appears in $\Res^{\check{M}} V$ then $\mu+\lambda\in\Lambda^+$. Applying Lemma~\ref{Lm_2.0.15}, for $\lambda\in \Lambda^+_{M, ab}(V)$ we get
$$
t^{-\lambda}\Sat(V^{\lambda})[\<\lambda, 2\check{\rho}\>]\ast \Sat(V)\,\iso\,
\mathop{\oplus}\limits_{\mu\in\Lambda^+_M} t^{-\lambda}\Sat(V^{\lambda+\mu})\otimes\Hom_{\check{M}}(U^{\mu}, V)[\<\lambda, 2\check{\rho}\>]\
$$
The above isomorphism upgrades to an isomorphism of functors
$$
(\Lambda^+_{M, ab}(V), \le)\to Shv(\Gr_G)^{M(\cO)},
$$
where we used the diagram (\ref{Def_IC_semi_inf}) in the LHS, and the diagram (\ref{def_SI_IC_for_mu_stratum}) in the RHS respectively. 

 Passing to the colimit over $(\Lambda^+_{M, ab}(V),\le)$, this gives an isomorphism
$$
\IC^{\frac{\infty}{2}}_P\ast \Sat(V)\,\iso\, \mathop{\oplus}\limits_{\mu\in\Lambda^+_M}  \IC^{\frac{\infty}{2}}_{P,\mu}\otimes \Hom_{\check{M}}(U^{\mu}, V)
$$ 
The isomorphism (\ref{iso_Hecke_propert_of_ICinfty/2_P}) follows now from Proposition~\ref{Pp_action_of_Rep(checkM)_on_SI-IC_P}. The compatibility of (\ref{iso_Hecke_propert_of_ICinfty/2_P}) with the monoidal structure on $\Rep(\check{G})^{\heartsuit}$ also follows from the construction.
\end{proof}

\ssec{Relation to the $\IC$-sheaf of $\Bunt_P$}
\label{Sect_2.6_Relation}

\sssec{} Main results of this subsection are Theorems~\ref{Thm_restriction_of_glob_first},\ref{Th_restriction_of_glob_second} and Proposition~\ref{Pp_4.2.3_now}. They relate explicitly the standard/costandard objects and the semi-infinite $\IC$-sheaves $\IC_{P, \eta}^{\frac{\infty}{2}}$ of the orbits $S^{\eta}_P$ on the affine Grassmannian $\Gr_G$ with the corresponding objects on $_{x,\infty}\Bunt_P$. 

\sssec{} Similarly to the case of $B=P$ studied in \cite{Gai19SI}, one has the following.
Recall the diagram from Sections~\ref{Sect_cY_x_definition}-\ref{Sect_pi_loc_def}.
$$
M(\cO_x)\backslash \Gr_{G, x} \getsup{\pi_{loc}} \cY_x  \toup{\pi_{glob}} {_{x,\infty}\Bunt_P}.
$$
For $\epsilon\in\Lambda_{G,P}$ the stack $\cY_x^{\epsilon}$ is defined in Section~\ref{Sect_cY_x_definition}. 

\begin{Thm} 
\label{Thm_restriction_of_glob_first}
Let $\eta\in\Lambda^+_M$. For any $\epsilon\in\Lambda_{G,P}$
there is a canonical isomorphism in $Shv(\cY_x^{\epsilon})$
\begin{equation}
\label{map_for_Thm_restriction_of_glob_first}
\pi_{loc}^!\IC_{P, \eta}^{\frac{\infty}{2}}\,\iso\, \pi_{glob}^!\IC_{\wt{\glob}}^{\eta}[(g-1)\dim P+\<\epsilon, 2\check{\rho}-2\check{\rho}_M\>].
\end{equation}
\end{Thm} 

\begin{Pp} 
\label{Pp_4.2.3_now}
For $\eta\in\Lambda_M^+,\epsilon\in\Lambda_{G,P}$ over $\cY_x^{\epsilon}$
one has canonically
$$
\pi_{glob}^!\nabla^{\eta}_{glob}[(g-1)\dim P+\<\epsilon, 2\check{\rho}-2\check{\rho}_M\>]\,\iso\, \pi^!_{loc}\nabla^{\eta}.
$$
\end{Pp}
\begin{proof}
The square is cartesian
$$
\begin{array}{ccc}
S^{\eta}_P(\cY) & \toup{i_{\eta, \cY}} & \cY_x\\
\downarrow && \downarrow\lefteqn{\scriptstyle \pi_{glob}}\\
_{x, -w_0^M(\eta)}\Bunt_P & \toup{i_{\eta, glob}} & _{x, \infty}\Bunt_P,
\end{array}
$$
the map $i_{\eta, \cY}$ was defined in Section~\ref{Sect_pi_loc_def}. 

Let $_{x, -w_0^M(\eta)}\Bunt_P^{\epsilon}$ be the intersection $_{x,\infty}\Bunt_P^{\epsilon}\cap \; {_{x, -w_0^M(\eta)}\Bunt_P}$. Then 
$$
\dim(_{x, -w_0^M(\eta)}\Bunt_P^{\epsilon})=(g-1)\dim P+\<\epsilon, 2\check{\rho}-2\check{\rho}_M\>+\<\eta, 2\check{\rho}\>.
$$
So, over $\cY_x^{\epsilon}$ for $\eta\in\Lambda^+_M$ by base change we get the desired isomorphism. 
\end{proof}

\sssec{} Let $j_{glob}: \Bun_P\hook{} \Bunt_P$ be the natural open immersion. View $(j_{glob})_!\IC(\Bun_P)$ as extended by zero under $\Bunt_P\hook{} {_{x,\infty}\Bunt_P}$. For $\epsilon\in\Lambda_{G,P}$ one has a natural map in $Shv(\cY_x^{\epsilon})$
\begin{equation}
\label{map_for_Th_2.6.4}
(i_{0,\cY})_!\omega\to \pi_{glob}^!(j_{glob})_!\IC_{\Bun_P}[(g-1)\dim P+\<\epsilon, 2\check{\rho}-2\check{\rho}_M\>].
\end{equation}
It comes from the cartesian square
$$
\begin{array}{ccc}
S^0_P(\cY) & \toup{i_{0,\cY}} & \cY_x\\
\downarrow\lefteqn{\scriptstyle \pi_{glob}^0} && \downarrow\lefteqn{\scriptstyle \pi_{glob}}\\
\Bun_P & \hook{} & _{x,\infty}\Bunt_P
\end{array}
$$  
and the natural map $(\pi_{glob}^0)_!(\pi_{glob}^0)^!\omega\to\omega$. 

%\label{map_for_Th_2.6.4} is kept from the previous version!

\begin{Thm} 
\label{Th_restriction_of_glob_second}
The map (\ref{map_for_Th_2.6.4}) is an isomorphism in $Shv(\cY_x^{\epsilon})$. 
\end{Thm}

\sssec{} In the rest of Section~\ref{Sect_2.6_Relation} we prove Theorems~\ref{Thm_restriction_of_glob_first} and \ref{Th_restriction_of_glob_second}. 

By (\cite{BG}, 4.1.3) for $\eta\in\Lambda^+_M$, one has canonically
$$
\Sat(U^{\eta})\ast \IC_{\wt{glob}}\,\iso\, \IC_{\wt{glob}}^{\eta}
$$
Since $\pi_{glob}^!$ commutes with $\Rep(\check{M})$-actions, Proposition~\ref{Pp_action_of_Rep(checkM)_on_SI-IC_P} immediately reduces Theorem~\ref{Thm_restriction_of_glob_first} to its special case $\eta=0$. 

\sssec{} Let us construct the morphism (\ref{map_for_Thm_restriction_of_glob_first}) for $\eta=0$. It is equivalent to providing a morphism
\begin{equation}
\label{map_for_Thm_restriction_of_glob_second}
(\pi_{glob})_!\pi_{loc}^!\IC_P^{\frac{\infty}{2}}\to \IC_{\wt{\glob}}[(g-1)\dim P+\<\epsilon, 2\check{\rho}-2\check{\rho}_M\>]
\end{equation}
over $_{x,\infty}\Bunt_P^{\epsilon}$. We first construct for $\lambda\in\Lambda^+_{M, ab}$ a morphism
\begin{equation}
\label{map_for_Sect_2.6.5_from_lambda_to_ICglob}
(\pi_{glob})_!\pi_{loc}^!(t^{-\lambda}\Sat(V^{\lambda}))[\<\lambda, 2\check{\rho}\>]\to \IC_{\wt{\glob}}[(g-1)\dim P+\<\epsilon, 2\check{\rho}-2\check{\rho}_M\>]
\end{equation}  
in $Shv(_{x,\infty}\Bunt_P^{\epsilon})$. 

 The functor $Shv(_{x,\infty}\Bunt_P)\to Shv(_{x,\infty}\Bunt_P)$, $K\mapsto K\ast \Sat(V^{\lambda})$ is both left and right adjoint to the functor $K\mapsto K\ast \Sat((V^{\lambda})^*)$. So, the datum of (\ref{map_for_Sect_2.6.5_from_lambda_to_ICglob}) is equivalent to a morphism
$$
(\pi_{glob})_!\pi_{loc}^!(\delta_{t^{-\lambda}})[\<\lambda, 2\check{\rho}\>]\to \IC_{\wt{\glob}}\ast \Sat((V^{\lambda})^*)[(g-1)\dim P+\<\epsilon, 2\check{\rho}-2\check{\rho}_M\>]
$$
in $Shv(_{x,\infty}\Bunt_P^{\epsilon})$. 

We have the map $i_{t^{-\lambda}, \cY}: \Bun_M\to \cY_x$ obtained by base change of $M(\cO_x)\backslash t^{-\lambda}\hook{} M(\cO_x)\backslash \Gr_{G, x}$ via $\pi_{loc}$. Then 
$$
\pi_{loc}^!(\delta_{t^{-\lambda}})\,\iso\, (i_{t^{-\lambda}, \cY})_!\omega.
$$ 
So, we are constructing the map
$$
(\pi_{glob})_!(i_{t^{-\lambda}, \cY})_!\IC(\Bun_M) [\<\lambda, 2\check{\rho}\>]\to \IC_{\wt{\glob}}\ast \Sat((V^{\lambda})^*)[(g-1)\dim U(P)+\<\epsilon, 2\check{\rho}-2\check{\rho}_M\>]
$$
over $_{x,\infty}\Bunt_P^{\epsilon}$.

 Let $\theta\in\Lambda_{G,P}$ be the image of $\lambda$. By the Hecke property of $\IC_{\wt{glob}}$, 
$$
\IC_{\wt{\glob}}\ast \Sat((V^{\lambda})^*)\,\iso\,\mathop{\oplus}\limits_{\mu\in\Lambda^+_M} \IC_{\wt{glob}}^{\mu}\otimes \Hom_{\check{M}}(U^{\mu}, (V^{\lambda})^*).
$$
Note that if $\mu\in\Lambda^+_M$ and $\Hom_{\check{M}}(U^{\mu}, (V^{\lambda})^*)\ne 0$ then $\lambda+\mu\in\Lambda^{pos}$, and $\pi_{glob}i_{t^{-\lambda}, \cY}$ factors through
$$
 {_{x, \lambda}\Bunt_P}\hook{} {_{x, \ge -w_0^M(\mu)}\Bunt_P}\hook{} {_{x,\infty}\Bunt_P}.
$$

 As in \cite{BG}, for $\nu\in\Lambda^+_M$ write $_x\cH^{\nu}_M$ for the stack classifying $(\cF_M, \cF'_M,\beta)$, where $\cF_M, \cF'_M$ are $M$-torsors on $X$ with an isomorphism $\beta: \cF_M\,\iso\,\cF'_M\mid_{X-x}$ such that $\cF'_M$ is in the position $\le \nu$ with respect to $\cF_M$ at $x$. Let 
$$
h^{\la}_M, h^{\ra}_M: {_x\cH^{\nu}_M}\to\Bun_M
$$ 
be the map sending the above point to $\cF_M$ and $\cF'_M$ respectively.

 Let $\mu\in\Lambda^+_M$ and $\Hom_{\check{M}}(U^{\mu}, (V^{\lambda})^*)\ne 0$. 
 Consider the stack $_x\cH^{-\lambda}_M\times_{\Bun_M}\Bun_P^{\epsilon-\theta}$, where we used the map $h^{\ra}_M$ to form the fibred product. Consider the locally closed immersion
$$
h: {_x\cH^{-\lambda}_M\times_{\Bun_M}\Bun_P^{\epsilon-\theta}}\hook{} {_{x,\ge -w_0^M(\mu)}\Bunt_P^{\epsilon}}
$$
sending 
$$
(\cF_M, \cF'_P, \cF'_M=\cF'_P\times^P M, \beta: \cF_M\,\iso\, \cF'_M\mid_{X-x})
$$ 
to $(\cF_M, \cF_G, \kappa)$, where $\cF_G=\cF'_P\times^P G$. Then $h^!\IC_{\wt{glob}}^{\mu}$ has smooth perverse cohomology sheaves, it is placed in perverse degrees $\ge 0$, and the inequality is strict unless $\mu=-\lambda$. So, 
$$
(i_{t^{-\lambda}, \cY})^!\pi_{glob}^!\IC_{\wt{glob}}^{\mu}
$$
is placed in perverse degrees 
$$
\ge \codim(\Bun_M, \; {_x\cH^{-\lambda}_M\times_{\Bun_M}\Bun_P^{\epsilon-\theta}})=(g-1)\dim U(P)+\<\epsilon-\theta, 2\check{\rho}-2\check{\rho}_M\>,
$$ 
and the inequality is strict unless $\mu=-\lambda$.  
 
 We have $\<\lambda, 2\check{\rho}\>=\<\theta, 2\check{\rho}-2\check{\rho}_M\>$. We see that the complex
$$
(i_{t^{-\lambda}, \cY})^!\pi_{glob}^!\IC_{\wt{\glob}}\ast \Sat((V^{\lambda})^*)[(g-1)\dim U(P)+\<\epsilon-\theta, 2\check{\rho}-2\check{\rho}_M\>]
$$ 
over $\Bun_M^{\epsilon}$ is placed in perverse degrees $\ge 0$, and its 0-th perverse cohomology sheaf identifies with the 0-th perverse cohomology sheaf of 
$$
(i_{t^{-\lambda}, \cY})^!\pi_{glob}^!\IC_{\wt{\glob}}^{-\lambda}[(g-1)\dim U(P)+\<\epsilon-\theta, 2\check{\rho}-2\check{\rho}_M\>],
$$
which in turn identifies with $\IC(\Bun_M)$ canonically. This gives the desired map (\ref{map_for_Sect_2.6.5_from_lambda_to_ICglob}). 

\sssec{} One checks that the maps (\ref{map_for_Sect_2.6.5_from_lambda_to_ICglob}) are compatible in the homotopy category with the transition maps in the diagram (\ref{Def_IC_semi_inf}). 

\sssec{} Since the internal hom with respect to the $\Vect$-action
$$
\HOM_{Shv(_{x,\infty}\Bunt_P^{\epsilon})}((\pi_{glob})_!\pi_{loc}^!(t^{-\lambda}\Sat(V^{\lambda}))[\<\lambda, 2\check{\rho}\>], \IC_{\wt{\glob}}[(g-1)\dim P+\<\epsilon, 2\check{\rho}-2\check{\rho}_M\>])
$$
is placed in degrees $\ge 0$, we conclude applying the Dold-Kan functor $\Vect\to \Vect^{\le 0}\to\Spc$
that the corresponding mapping space
$$
\Map_{Shv(_{x,\infty}\Bunt_P^{\epsilon})}((\pi_{glob})_!\pi_{loc}^!(t^{-\lambda}\Sat(V^{\lambda}))[\<\lambda, 2\check{\rho}\>], \IC_{\wt{\glob}}[(g-1)\dim P+\<\epsilon, 2\check{\rho}-2\check{\rho}_M\>])
$$
is discrete. So, as in (\cite{Gai19SI}, 3.4.7) the maps (\ref{map_for_Sect_2.6.5_from_lambda_to_ICglob}) uniquely combine to the desired morphism (\ref{map_for_Thm_restriction_of_glob_second}).

\sssec{} 
\label{Sect_2.6.9_describing_the_composition}
As in (\cite{Gai19SI}, 3.4.8), for $\lambda\in\Lambda^+_{M, ab}$ one may describe the composition
\begin{multline}
\label{map_for_Sect_2.6.8_long_composition}
\pi_{loc}^!(t^{-\lambda}\Sat(V^{\lambda}))[\<\lambda, 2\check{\rho}\>]\to \pi_{loc}^!\IC^{\frac{\infty}{2}}_P\,\toup{(\ref{map_for_Thm_restriction_of_glob_first})} \, \pi_{glob}^!\IC_{\wt{\glob}}[(g-1)\dim P+\<\epsilon, 2\check{\rho}-2\check{\rho}_M\>]
 \\ \to \pi_{glob}^!\nabla_{\glob}^0[(g-1)\dim P+\<\epsilon, 2\check{\rho}-2\check{\rho}_M\>]\,\iso\,\pi_{loc}^! \nabla^0
\end{multline}
over $\cY^{\epsilon}_x$ as follows. It is obtained by applying $\pi_{loc}^!$ to the following morphism
$$
t^{-\lambda}\Sat(V^{\lambda})[\<\lambda, 2\check{\rho}\>]\to \nabla^0
$$
in $Shv(\Gr_{G, x})^{M(\cO_x)}$. 

 The base change under $t^{-\lambda}\ov{\Gr}_G^{\lambda}\subset \bar S^0_P$ of $j_0: S^0_P\hook{}\bar S^0_P$ is the open immersion $S^0_P\cap (t^{-\lambda}\ov{\Gr}_G^{\lambda})\subset t^{-\lambda}\ov{\Gr}_G^{\lambda}$, and $\nabla^0$ is the extension by zero under $\bar i_0: \bar S^0_P\hook{}\Gr_G$. Besides, $S^{\lambda}_P\cap \ov{\Gr}_G^{\lambda}\subset \Gr_G^{\lambda}$. As a morphism on $\bar S^0_P$, the map (\ref{map_for_Sect_2.6.8_long_composition}) equals
$$
t^{-\lambda}\Sat(V^{\lambda})[\<\lambda, 2\check{\rho}\>]\to (j_0)_*(j_0)^*(t^{-\lambda}\Sat(V^{\lambda})[\<\lambda, 2\check{\rho}\>])\,\iso\, (j_0)_*\omega_{(t^{-\lambda}\Gr_G^{\lambda})\cap S^0_P}\to (j_0)_*\omega_{S^0_P}
$$
We used the isomorphism 
$(j_0)^*(t^{-\lambda}\Sat(V^{\lambda})[\<\lambda, 2\check{\rho}\>])\,\iso\,\omega_{(t^{-\lambda}\Gr_G^{\lambda})\cap S^0_P}$,
where the RHS is considered as extended by zero under the closed immersion 
$$
(t^{-\lambda}\Gr_G^{\lambda})\cap S^0_P\hook{} S^0_P.
$$ 

\sssec{} For $\theta\in\Lambda_{G,P}$ and a $M$-torsor $\cF'_M$ on $X$ write $\Gr_{M, x}^{\theta}(\cF'_M)$ for the ind-scheme classifying $(\cF_M, \beta_M)$, where $\cF_M$ is a $M$-torsor on $X$, $\beta_M: \cF_M\,\iso\, \cF'_M\mid_{X-x}$ is an isomorphism such that $\beta_M$ induces an isomorphism of $M/[M,M]$-torsors on $X$
$$
\bar \beta_M: \cF_{M/[M,M]}\,\iso\, \cF'_{M/[M,M]}(-\theta x)
$$  

Write $_{=\theta, x}\Bunt_P$ for the preimage of $_{=\theta, x}\Bunb_P$ under $\gr: {_{x,\infty}\Bunt_P}\to {_{x,\infty}\Bunb_P}$. One has an isomorphism $\Bun_P\,\iso\, {_{=\theta, x}\Bunb_P}$ sending $\cF_P$ to $(\cF_{M/[M,M]}(\theta x), \cF_P\times_P G, \kappa)$, where $\cF_{M/[M,M]}$ is obtained from $\cF_P$ by the extension of scalars. 
 
  The stack $_{=\theta, x}\Bunt_P$ classifies triples $(\cF'_P, \cF_M,\beta_M)$, where $\cF'_P$ is a $P$-torsor on $X$ with $\cF'_M=\cF'_P\times_P M$, $\cF_M$ is a $M$-torsor on $X$, $\beta_M: \cF_M\,\iso\, \cF'_M\mid_{X-x}$ is an isomorphism such that $(\cF'_M, \beta_M)\in \Gr_{M, x}^{\theta}(\cF_M)$. 
  
  The map $\tilde\pi$ fits into a cartesian square
$$
\begin{array}{ccc}
\Gr_P^{\theta} & \toup{v^{\theta}_P} & \Gr_G\\
\downarrow && \downarrow\lefteqn{\scriptstyle \tilde\pi}\\
_{=\theta, x}\Bunt_P & \to &  {_{x,\infty}\Bunt_P}
\end{array}
$$ 
\sssec{} For $\theta\in\Lambda_{G,P}$ and a $M$-torsor $\cF'_M$ on $X$ write $\Gr_{M}^{+, \theta}(\cF'_M)$ for the version of $\Gr_M^{+,\theta}$, where the background torsor $\cF^0_M$ is replaced by $\cF'_M$. One has 
$$
\Gr_{M}^{+, \theta}(\cF'_M)\subset \Gr_{M}^{\theta}(\cF'_M)
$$ 
 
  For $\theta\in\Lambda_{G,P}^{pos}$ write $\Gr_M^{-,-\theta}$ for the scheme of those 
$
(\cF_M, \beta_M: \cF_M\,\iso\, \cF^0_M\mid_{X-x})\in\Gr_M^{-\theta}
$ 
for which $(\cF^0_M, \beta_M)\in \Gr_M^{+,\theta}(\cF_M)$. 

\sssec{} For $\theta\in \Lambda_{G,P}^{pos}$ write $X^{\theta}$ for the moduli space of $\Lambda_{G,P}^{pos}$-valued divisor of degree $\theta$.

For $\theta\in -\Lambda^{pos}_{G,P}$ write $\Mod_M^{-, \theta}$ for the moduli scheme classifying $D\in X^{-\theta}$, a $M$-torsor $\cF_M$ on $X$ with a trivialization $\beta_M: \cF_M\,\iso\, \cF^0_M\mid_{X-\supp(D)}$ such that for any finite-dimensional $G$-module $\cV$ the map
$$
\cV^{U(P)}_{\cF^0_M}\toup{\beta_M^{-1}} \cV^{U(P)}_{\cF_M}
$$
is regular on $X$, and $\beta_M$ induces an isomorphism $\cF_{M/[M,M]}\,\iso\, \cF^0_{M/[M,M]}(D)$ on $X$. 

 To be precise, here we pick a homomorphism $\Theta: \Lambda_{G,P}^{pos}\to \ZZ_+$ of semigroups sending each $\alpha_i$, $i\notin \cI_M$ to a strictly positive integer, which allows to associate to $D\in X^{-\theta}$ an effective divisor $\Theta(D)$. Then $X-\supp(D)$ is defined as $X-\supp(\Theta(D))$. 

Let $\pi_M: \Mod_M^{-,\theta}\to X^{-\theta}$ be the projection sending the above point to $D$.

\sssec{} Let $\Bun_{M, x}$ be the stack classifying $\cF_M\in\Bun_M$ with a trivialization $\beta_M: \cF_M\,\iso\, \cF^0_M\mid_{D_x}$. Let $q_{\cY}: \Bun_{M, x}\times \Gr_{G, x}\to \cY_x$ be the map sending 
$$
(\cF_M,\beta_M, \cF_G\in\Bun_G, \beta: \cF_G\,\iso\,\cF^0_G\mid_{X-x})
$$ 
to $(\cF_M, \cF'_G)$, where $\cF'_G$ is the gluing of $\cF_M\times_M G\mid_{X-x}$ with $\cF_G\mid_{D_x}$ over $D_x^*$ via 
$$
\cF_M\times_M G\toup{\beta_M} \cF^0_G\toup{\beta^{-1}} \cF_G.
$$ 
So, $q_{\cY}$ is a $M(\cO_x)$-torsor. 

The following is established exactly as in (\cite{Gai19SI}, Proposition~3.5.2). 
\begin{Pp}
Both $q_{\cY}^!\pi_{glob}^!\IC_{\wt{glob}}$ and $q_{\cY}^!\pi_{glob}^!(j_{glob})_!\IC_{\Bun_P}$ belong to the full subcategory 
$$
Shv(\Bun_{M, x}\times \Gr_{G, x})^{U(P)(F_x)}\subset Shv(\Bun_{M, x}\times \Gr_{G, x}),
$$ 
here $U(P)(F_x)$ acts via its action on $\Gr_{G,x}$. \QED
\end{Pp}

\begin{Rem} 
\label{Rem_intersesion_Gr_P^0_with_barS^0_P}
One has $\Gr_P^0\cap \bar S^0_P=S^0_P$. Indeed, let $\eta\in\Lambda^+_M\cap (-\Lambda^{pos})$ such that $S^{\eta}_P\subset \Gr_P^0$. Then $\eta=0$ in $\Lambda_{G,P}$. Our claim follows from the fact that, since $[M,M]$ is semi-simple and simply-connected, $\Lambda^+_{[M,M]}\subset (\ZZ_+$-span of $\alpha_i$ for $ i\in\cI_M$ in $\Lambda)$. Here $\Lambda^+_{[M,M]}$ is the semigroup of dominant coweights of $[M,M]$. 
\end{Rem}

\sssec{} In order to prove Theorem~\ref{Thm_restriction_of_glob_first} for $\eta=0$ it suffices to show that applying $q_{\cY}^!$ to (\ref{map_for_Thm_restriction_of_glob_first}) one gets an isomorphism. Moreover, it suffices to show that for any field-valued point $\tau=(\cF^b_m, \beta_M)\in \Bun_{M,x}^{\epsilon}$ the $!$-restriction of the latter map under 
$\tau\times\id: \Gr_{G, x}\to \Bun_{M,x}^{\epsilon}\times \Gr_{G, x}$ becomes an isomorphism. 

 Similar considerations are also valid for the proof of Theorem~\ref{Th_restriction_of_glob_second}. Denote by $\tilde\pi_{\tau}: \Gr_{G,x}\to {_{x,\infty}\Bunt_P}$ the composition $\Gr_{G, x}\to \Bun_{M,x}^{\epsilon}\times \Gr_{G, x}\toup{\pi_{glob}q_{\cY}}{_{x,\infty}\Bunt_P}$. So, Theorem~\ref{Thm_restriction_of_glob_first} for $\eta=0$ and Theorem
 ~\ref{Th_restriction_of_glob_second} are reduced to the following.
 
\begin{Pp} 
\label{Pp_4.2.18_now}
i) The map 
$$
\IC^{\frac{\infty}{2}}_P\to \tilde\pi_{\tau}^! \IC_{\wt{\glob}}[(g-1)\dim P+\<\epsilon, 2\check{\rho}-2\check{\rho}_M\>]
$$
obtained from (\ref{map_for_Thm_restriction_of_glob_first}) is an isomorphism.

\smallskip\noindent
ii) The object
$
\tilde\pi_{\tau}^!(j_{glob})_!\IC_{\Bun_P}
$
is the extension by zero under $i_0: S^0_P\hook{} \Gr_{G, x}$. 
\end{Pp}   

 In what follows, to simplify notations, we assume $\tau=(\cF^b, \beta_M)$ is given by the trivial torsor $\cF^0_M$ with the natural trivialization. Then $\epsilon=0$ and $\tilde\pi_{\tau}=\tilde\pi$. The proof for general $\tau$ is similar. 
    
\sssec{} In order to prove Proposition~\ref{Pp_4.2.18_now} i)
it suffices to show that for any $\theta\in -\Lambda_{G,P}^{pos}$ the map
$$
(v^{\theta}_P)^*\tilde\pi^!\IC^{\frac{\infty}{2}}_{P}
\to (v^{\theta}_P)^*\tilde\pi^!\IC_{\wt{\glob}}[(g-1)\dim P]
$$
induced by (\ref{map_for_Thm_restriction_of_glob_first}) is an isomorphism. In view of Remark~\ref{Rem_intersesion_Gr_P^0_with_barS^0_P}, to prove Proposition~\ref{Pp_4.2.18_now} ii)  it suffices to show that for any $0\ne \theta\in -\Lambda^{pos}_{G,P}$ one has
$$
(v^{\theta}_P)^*\tilde\pi^!(j_{glob})_!\IC_{\Bun_P}=0.
$$

\sssec{} In view of Lemmas~\ref{Lm_theorem_of_Braden_for_theta} and \ref{Lm_2.3.8} ii), Proposition~\ref{Pp_4.2.18_now} is reduced to the following. 

\begin{Pp} 
\label{Pp_2.6.18_main_technical}
Let $\theta\in -\Lambda_{G,P}^{pos}$.\\ 
i) The map
$$
(\gt^{\theta}_{P^-})_*(v_{P^-}^{\theta})^!\IC^{\frac{\infty}{2}}_P\to (\gt^{\theta}_{P^-})_*(v_{P^-}^{\theta})^!\tilde\pi^!\IC_{\wt{\glob}}[(g-1)\dim P]
$$
induced by (\ref{map_for_Thm_restriction_of_glob_first}) is an isomorphism over $\Gr^{\theta}_M$. 

\smallskip\noindent
ii) If $\theta\ne 0$ then $(\gt^{\theta}_{P^-})_*(v_{P^-}^{\theta})^!\tilde\pi^!(j_{glob})_!\IC_{\Bun_P}=0$.
\end{Pp}

\sssec{Recollection about the Zastava space} For $\theta\in -\Lambda_{G,P}^{pos}$ denote by $\cZ^{\theta}$ the moduli stack classifying $(\cF_M,\beta_M)\in \Mod_M^{-,\theta}$ for which we set $D=\pi_M(\cF_M,\beta_M)\in X^{-\theta}$, a $G$-torsor $\cF_G$ on $X$ together with a trivialization $\beta: \cF_G\,\iso\,\cF_M\times_M G\mid_{X-\supp(D)}$ such that two conditions hold: 
\begin{itemize}
\item For any finite-dimensional $G$-module $\cV$, the map $\cV\to \cV_{U(P^-)}$ extends to a regular surjective map of vector bundles on $X$
$$
\cV_{\cF_G}\toup{\beta}\cV_{\cF_M}\to  (\cV_{U(P^-)})_{\cF_M}
$$
\item For any finite-dimensional $G$-module $\cV$, the composition 
$$
\cV^{U(P)}_{\cF^0_M}\,\hook{\beta_M}\, \cV^{U(P)}_{\cF_M}\to \cV_{\cF_M}\toup{\beta^{-1}} \cV_{\cF_G}
$$
is a regular morphism of coherent sheaves on $X$.
\end{itemize}

\sssec{} As in \cite{BFGM}, $\cZ^{\theta}$ is representable by an irreducible quasi-projective scheme. Write $\pi_P: \cZ^{\theta}\to \Mod_M^{-,\theta}$ for the map sending the above point to $(\cF_M,\beta_M)$. Let $\gs: \Mod_M^{-,\theta}\to \cZ^{\theta}$ denote the natural section  of $\pi_P$. Let $\gq_{\cZ}: \cZ^{\theta}\to \Bunt_P$ be the map sending the above point to $(\cF^0_M, \cF_G,\kappa)$. 

 Set $\oo{\cZ}{}^{\theta}=\gq_{\cZ}^{-1}(\Bun_P)$. The scheme $\oo{\cZ}{}^{\theta}$ is smooth, and $\dim\cZ^{\theta}=-\<\theta, 2\check{\rho}-2\check{\rho}_M\>$. 
  
 Let $\gF^{\theta}$ be the \select{central fibre} of $\cZ^{\theta}$ over $X^{-\theta}$, that is, the preimage of $-\theta x\in X^{-\theta}$ under $\pi_M\comp\pi_P$. Set also $\oo{\gF}{}^{\theta}=\gF^{\theta}\cap \oo{\cZ}{}^{\theta}$. Write $i_{\Mod}: \Gr_M^{-,\theta}\hook{} \Mod_M^{-, \theta}$ for the base change of $\Mod_M^{-,\theta}$ under $-\theta x: \Spec k\to X^{-\theta}$.  As in \cite{BFGM}, one gets an isomorphism 
$$
\gF^{\theta}\,\iso\, \bar S^0_P\cap \Gr_{P^-}^{\theta},
$$ 
which restricts to an isomorphism of open subshemes $\oo{\gF}{}^{\theta}\,\iso\, S^0_P\cap \Gr_{P^-}^{\theta}$. The diagram commutes
$$
\begin{array}{ccc}
\cZ^{\theta} & \xrightarrow{\; \;\;\;\;\;\;\;\;\;\gq_{\cZ}\;\;\;\;\;\;\;\;\;\;}
& \Bunt_P^0\\
\uparrow && \uparrow\lefteqn{\scriptstyle \tilde\pi}\\
\gF^{\theta} & \iso\,  \bar S^0_P\cap \Gr_{P^-}^{\theta}\hook{} & \bar S^0_P
\end{array}
$$

 Let $j: \oo{\cZ}{}^{\theta}\hook{} \cZ^{\theta}$ be the natural open immersions. The following is an analog of (\cite{Gai19SI}, Proposition~3.6.5) (though formally in the case $B=P$ they differ).  
\begin{Pp} 
\label{Pp_2.6.21_Zastava}
Let $\theta\in -\Lambda_{G,P}^{pos}$.\\
a) One has a canonical isomorphism
$$
\gq_{\cZ}^!\IC_{\wt{glob}}[(g-1)\dim P]\,\iso\, \IC_{\cZ^{\theta}}[-\<\theta, 2\check{\rho}-2\check{\rho}_M\>]
$$
extending the tautological one over $\oo{\cZ}{}^{\theta}$. Note that $\dim\Bun_P^0=(g-1)\dim P$.

\smallskip\noindent
b) One has a canonical isomorphism $j_!(\omega_{\oo{\cZ}{}^{\theta}})\,\iso\, \gq_{\cZ}^!(j_{glob})_!\omega_{\Bun_P}$ 
\end{Pp}
\sssec{}  The proof of Proposition~\ref{Pp_2.6.21_Zastava} is given in Section~\ref{Sect_Proof_Pp_2.6.21_Zastava}. Note that $\gq_{\cZ}$ naturally extends to a map $\cZ^{\theta}\to \Bunt_P^0\times_{\Bun_G} \Bun_{P^-}^{\theta}$.
 
 Let $\oo{\pi}_{\gF}: \oo{\gF}{}^{\theta}\to \Gr_M^{-, \theta}$ be the restriction of $\pi_P$. 
Write $U(\gu(\check{P}))$ for the universal enveloping algebra of $\gu(\check{P})$. Recall that $U(\gu(\check{P}))$ is the graded dual of $\cO(U(\check{P}))$ (for $P=B$ this is proved in (\cite{GLS}, Proposition 5.1)). The following is a version of (\cite{BFGM}, Proposition~5.9).
\begin{Pp} 
\label{Pp_2.6.23_about_Gr}
The complex $(\oo{\pi}_{\gF})_!e[-\<\theta, 2\check{\rho}-2\check{\rho}_M\>]$ on $\Gr_M^{-,\theta}$ is placed in perverse degrees $\le 0$, and the $0$-th perverse cohomology sheaf identifies with 
$$
\DD\Sat_M(\cO(U(\check{P}))_{\theta})\,\iso\, \upsilon \Sat_M(U(\gu(\check{P}))_{-\theta})
$$ 
\end{Pp}
 For completeness, we supply a proof of Proposition~\ref{Pp_2.6.23_about_Gr} in Section~\ref{Sect_2.8}. 
  
 The following is a version of (\cite{BFGM}, Proposition~5.7 and 5.8).  
\begin{Pp}
\label{Pp_2.6.23_citation_from_BFGM}
a) The complex $i_{\Mod}^!(\pi_P)_*\IC_{\cZ^{\theta}}$ is concentrated in perverse cohomological degree zero on $\Gr_M^{-,\theta}$. \\
b) The map 
\begin{equation}
\label{map_for_Pp_2.6.23}
i_{\Mod}^!(\pi_P)_*\IC_{\cZ^{\theta}}\to i_{\Mod}^!(\pi_P)_*j_*\IC_{\oo{\cZ}{}^{\theta}}\,\iso\, (\oo{\pi}_{\gF})_*\omega_{\oo{\gF}{}^{\theta}}[\<\theta, 2\check{\rho}-2\check{\rho}_M\>]
\end{equation} 
over $\Gr_M^{-,\theta}$ induces an isomorphism in the (lowest) perverse cohomological degree zero. Here we used the fact that $\oo{\cZ}{}^{\theta}$ is smooth of dimension $-\<\theta, 2\check{\rho}-2\check{\rho}_M\>$. \QED
 \end{Pp}

\sssec{Proof of Proposition~\ref{Pp_2.6.18_main_technical} ii)}
In view of (\cite{BFGM}, Proposition~5.2), we are reduced to show that
$$
i_{\Mod}^!(\pi_P)_*\gq_{\cZ}^! (j_{glob})_!\IC_{\Bun_P}=0
$$
Applying Proposition~\ref{Pp_2.6.21_Zastava} b), it suffices to show that 
$$
i_{\Mod}^!(\pi_P)_*j_!(\omega_{\oo{\cZ}{}^{\theta}})=0
$$
As in (\cite{BFGM}, Proposition~5.2), $(\pi_P)_*j_!(\omega_{\oo{\cZ}{}^{\theta}})\,\iso\, \gs^*j_!(\omega_{\oo{\cZ}{}^{\theta}})$. If $\theta\ne 0$ then $\gs^*j_!(\omega_{\oo{\cZ}{}^{\theta}})=0$, because the corresponding fibre product is empty. \QED

\sssec{} The following is established as in (\cite{BFGM}, Proposition~6.6).
\begin{Lm} 
\label{Lm_2.6.27_inclusion}
Let $\nu,\nu'\in\Lambda^+_M$ then $S_P^{\nu}\cap S_{P^-}^{\nu'}$ is a scheme of finite type. If $\lambda\in \Lambda^+_{M, ab}$ is deep enough on the wall of the corresponding Weyl chamber then 
$$
S_P^{\nu+\lambda}\cap S_{P^-}^{\nu'+\lambda}\subset \Gr_G^{\nu+\lambda}
$$
\end{Lm}
\begin{proof}
Recall the closed immersion $i^{\nu}_P: \Gr_M^{\nu}\hook{}S_P^{\nu}$, so $S_P^{\nu}=U(P)(F)\Gr_M^{\nu}$. Since $S_P^{\nu}\cap S_{P^-}^{\nu'}$ is finite type, there is $\lambda\in \Lambda^+_{M, ab}$ deep enough such that the preimage of $S_P^{\nu}\cap S_{P^-}^{\nu'}$ under the action map
$$
U(P)(F)\times \Gr_M^{\nu}\to S_P^{\nu}
$$
is contained in $\Ad_{t^{-\lambda}}(U(P)(\cO))\times \Gr_M^{\nu}$. So, the action of $t^{\lambda}$ sends $S_P^{\nu}\cap S_{P^-}^{\nu'}$ inside $U(P)(\cO)\Gr_M^{\nu+\lambda}\subset \Gr_G^{\nu+\lambda}$. 
\end{proof}
\begin{Cor} 
\label{Cor_2.6.28}
Let $\theta\in -\Lambda_{G,P}^{pos}$. Then for $\lambda\in\Lambda^+_{M, ab}$ deep enough on the wall of the corresponding Weyl chamber one has $S^0_P\cap \Gr_{P^-}^{\theta}\subset t^{-\lambda}\Gr_G^{\lambda}$.
\end{Cor}
\begin{proof}
We know that $\oo{\gF}{}^{\theta}$ is of finite type. It is stratified by locally closed subschemes $S^0_P\cap S_{P^-}^{\nu}$ for $\nu\in\Lambda^+_M$ over $\theta$. Since this stratification is finite, the claim follows from Lemma~\ref{Lm_2.6.27_inclusion}. 
\end{proof}

\sssec{Proof of Proposition~\ref{Pp_2.6.18_main_technical} i)} For $\lambda\in\Lambda^+_{M, ab}$ over $\bar\lambda\in\Lambda_{G,P}$
consider the composition 
$$
t^{-\lambda}\Sat(V^{\lambda})[\<\lambda, 2\check{\rho}\>]\to \IC^{\frac{\infty}{2}}_{P}\to
\tilde\pi^!\IC_{\wt{\glob}}[(g-1)\dim P].
$$
It gives rise to the map
\begin{multline}
\label{map_for_Sect_2.6.25}
(\gt^{\theta}_{P^-})_*(v_{P^-}^{\theta})^!(t^{-\lambda}\Sat(V^{\lambda}))[\<\lambda, 2\check{\rho}\>]\to (\gt^{\theta}_{P^-})_*(v_{P^-}^{\theta})^!\tilde\pi^!\IC_{\wt{\glob}}[(g-1)\dim P].
\end{multline}
By Lemma~\ref{Lm_theorem_of_Braden_for_theta} and Proposition~\ref{Pp_gRes_for_Levi}, the LHS of (\ref{map_for_Sect_2.6.25}) identifies with
$$
t^{-\lambda}(\gt_P^{\theta+\bar\lambda})_!(v_P^{\theta+\bar\lambda})^*\Sat(V^{\lambda})[\<\lambda, 2\check{\rho}\>]\,\iso\, t^{-\lambda}\Sat_M((V^{\lambda})_{\theta+\bar\lambda})[-\<\theta, 2\check{\rho}-2\check{\rho}_M\>]
$$
and sits in perverse degree $\<\theta, 2\check{\rho}-2\check{\rho}_M\>$ on $\Gr_M^{\theta}$. As $\lambda$ runs through $\Lambda^+_{M, ab}$ the corresponding diagram $\Lambda^+_{M, ab}\to Shv(\Gr_M)^{\theta}$ was described in Proposition~\ref{Pp_2.5.15_*-restriction}. 

 By Propositions~\ref{Pp_2.6.21_Zastava} and \ref{Pp_2.6.23_citation_from_BFGM}, the RHS of (\ref{map_for_Sect_2.6.25}) identifies with
$$
i_{\Mod}^!(\pi_P)_*\gq_{\cZ}^!\IC_{\wt{glob}}[(g-1)\dim P]\,\iso\, i_{\Mod}^!(\pi_P)_* \IC_{\cZ^{\theta}}[-\<\theta, 2\check{\rho}-2\check{\rho}_M\>]
$$
and is concentrated in perverse cohomological degree $\<\theta, 2\check{\rho}-2\check{\rho}_M\>$. Thus, it suffices to show that the map (\ref{map_for_Sect_2.6.25}), after taking the colimit in the LHS over $\lambda\in\Lambda^+_{M, ab}$, induces an isomorphism on the perverse cohomology sheaves in degree $\<\theta, 2\check{\rho}-2\check{\rho}_M\>$. 

\sssec{} Write 
$$
\oo{\pi}{}^{\lambda}: S^0_P\cap \Gr_{P^-}^{\theta}\cap (t^{-\lambda}\Gr_G^{\lambda})\to \Gr_M^{-,\theta}
$$ 
for the restriction of $\oo{\pi}_{\gF}$. 
Since $t^{-\lambda}\ov{\Gr}_G^{\lambda}\subset \bar S^0_P$, we have the open immersion
$$
S^0_P\cap \Gr_{P^-}^{\theta}\cap (t^{-\lambda}\Gr_G^{\lambda})\subset  \Gr_{P^-}^{\theta}\cap (t^{-\lambda}\Gr_G^{\lambda})
$$

 For $\lambda\in\Lambda^+_{M, ab}$ large enough in the corresponding wall of the Weyl chamber, by Corollary~\ref{Cor_2.6.28},
$$ 
S^0_P\cap \Gr_{P^-}^{\theta}\cap (t^{-\lambda}\Gr_G^{\lambda})=S^0_P\cap \Gr_{P^-}^{\theta}=\oo{\gF}{}^{\theta}
$$
Now it suffices to show that for $\lambda\in\Lambda^+_{M, ab}$ large enough
the diagram commutes
$$
\begin{array}{ccc}
(\gt^{\theta}_{P^-})_*(v^{\theta}_{P^-})^!(t^{-\lambda}\Sat(V^{\lambda}))[\<\lambda, 2\check{\rho}\>] & \to & (\oo{\pi}{}^{\lambda})_*\omega_{S^0_P\cap \Gr_{P^-}^{\theta}\cap (t^{-\lambda}\Gr_G^{\lambda})}\\
\downarrow\lefteqn{\scriptstyle (\ref{map_for_Sect_2.6.25})} && \uparrow\\
(\gt^{\theta}_{P^-})_*(v^{\theta}_{P^-})^!\tilde\pi^!\IC_{\wt{glob}}[(g-1)\dim P] & \toup{(\ref{map_for_Pp_2.6.23})} & (\oo{\pi}_{\gF})_*\omega_{\oo{\gF}{}^{\theta}}
\end{array}
$$ 
As in (\cite{Gai19SI}, Section~3.7.3), this follows from the description of the map (\ref{map_for_Sect_2.6.8_long_composition}) in Section~\ref{Sect_2.6.9_describing_the_composition}. Proposition~\ref{Pp_2.6.18_main_technical} i) is proved. \QED

\ssec{Proof of Proposition~\ref{Pp_2.6.21_Zastava}}
\label{Sect_Proof_Pp_2.6.21_Zastava}

\sssec{} Write $\Bunt_{U(P)}$ for the version of $\Bunt_P$, where the corresponding $M$-torsor is fixed and identified with $\cF^0_M$. By abuse of notations, the natural map $\cZ^{\theta}\to \Bunt_{U(P)}$ is also denoted by $\gq_{\cZ}$. 

Let $\gu(P)$ be the Lie algebra of $U(P)$. Write $\Bun_M^{sm}\subset\Bun_M$ for the open substack given by the property that for all $M$-modules $V$ appearing as subquotients of $\gu(P)$ one has $\H^1(X, V_{\cF_M})=0$. Write $\Bun_{P^-}^{sm}\subset \Bun_{P^-}$ for the preimage of $\Bun_M^{sm}$ under $\Bun_{P^-}\to\Bun_M$. The map $\Bun_{P^-}^{sm}\to\Bun_G$ is smooth. Write $\cZ^{\theta, sm}$ for the preimage of $\Bun_{M}^{sm}$ under the projection $\cZ^{\theta}\to \Bun_M$, $(\cF_M, \cF_G, \beta,\beta_M)\mapsto \cF_M$.  

 If we have $\cZ^{\theta}=\cZ^{\theta, sm}$ then the claim is easy. It follows in this case as a combination of the fact that the projection $\cZ^{\theta}\to \Bunt_{U(P)}$ is smooth with the fact that both $\IC_{\wt{glob}}$ and $(j_{glob})_!\IC$ are ULA with respect to $\Bunt_P\to \Bun_M$ by (\cite{BG}, Theorem 5.1.5). We reduce to this case as in (\cite{Gai19SI}, Proposition~3.6.5). 

 Set $\Bun_{P^-}^{\theta, sm}=\Bun_{P^-}^{\theta}\cap \Bun_{P^-}^{sm}$. Pick $\theta'\in -\Lambda_{G,P}^{pos}$. Let 
$$
(X^{-\theta}\times X^{-\theta'})_{disj}\subset X^{-\theta}\times X^{-\theta'}
$$
denote the open locus of divisors whose supports are disjoint. Recall the factorization property 
\begin{equation}
\label{fact_property}
\cZ^{\theta+\theta'}\times_{X^{-\theta-\theta'}} (X^{-\theta}\times X^{-\theta'})_{disj}\,\iso\,
(\cZ^{\theta}\times \cZ^{\theta'})\times_{X^{-\theta}\times X^{-\theta'}} (X^{-\theta}\times X^{-\theta'})_{disj}
\end{equation}

\sssec{} Let $(\Bunt_{U(P)}\times X^{-\theta'})^{good}\subset \Bunt_{U(P)}\times X^{-\theta'}$ be the open subscheme given by the property that the maps
$$
\kappa^{\cV}: \cV^{U(P)}_{\cF^0_M}\hook{} \cV_{\cF_G}
$$
have no zero at the support of the point of $X^{-\theta'}$. 

 Given $S\in\Sch^{aff}$ and $D\in X^{-\theta'}(S)$ write $\hat \cD_D$ for the formal completion of $\supp(D)$ in $S\times X$. Let $\cD_D$ be the affine scheme corresponding to 
$\hat \cD_D$, i.e., the image of $\hat\cD_D$ under the functor
$$
\colim: \Ind(\Sch^{aff})\to \Sch^{aff}
$$
Let $\oo{\cD}_D\subset \cD_D$ be the open subscheme obtained by removing $\supp(D)$. 

Write $\gL^+(U(P))_{\theta'}$ for the group scheme over $X^{-\theta'}$ classifying a point $D\in X^{-\theta'}$ and a section $\cD_D\to U(P)$. We leave it to a reader to formulate a version of this definition with $S$-points for a test scheme $S\in\Sch^{aff}$.  
 
 For a point of $(\Bunt_{U(P)}\times X^{-\theta'})^{good}$ one gets a $U(P)$-torsor $\cF_{U(P)}$ over $\cD_D$. Let 
$$
(\Bunt_{U(P)}\times X^{-\theta'})^{level}\to (\Bunt_{U(P)}\times X^{-\theta'})^{good}
$$
be the stack classifying a point of $(\Bunt_{U(P)}\times X^{-\theta'})^{good}$ as above together with an trivialization
$$
\cF_{U(P)}\,\iso\, \cF^0_{U(P)}
$$
of this $U(P)$-torsors on $\cD_D$. This is a torsor under the group scheme $\gL^+(U(P))_{\theta'}$.

 Write $\gL(U(P))_{\theta'}$ for the group ind-scheme over $(\Bunt_{U(P)}\times X^{-\theta'})^{good}$ classifying a point $D\in X^{-\theta'}$ and a section $\oo{\cD}_D\to U(P)$. The usual gluing procedure allows to extend the action of $\gL^+(U(P))_{\theta'}$ on $(\Bunt_{U(P)}\times X^{-\theta'})^{level}$ to that of $\gL(U(P))_{\theta'}$. 
 
 Pick a group subscheme 
$$
\gL^+(U(P))_{\theta'}\subset U(P)'_{\theta'}\subset \gL(U(P))_{\theta'}
$$
pro-smooth over $X^{-\theta'}$, where the first inclusion is a placid closed embedding, and $U(P)'_{\theta'}/\gL^+(U(P))_{\theta'}$ is smooth. Then the projection 
$$
(\Bunt_{U(P)}\times X^{-\theta'})^{good}\to  (\Bunt_{U(P)}\times X^{-\theta'})^{level}/U(P)'_{\theta'}
$$
is smooth, where the RHS denotes the stack quotient. 

\sssec{} Recall that for $\Mod_M^{-,\theta}$ one also has the factorization isomorphism
\begin{multline}
\label{fact_property_for_Mod_M^-}
(\Mod_M^{-,\theta}\times \Mod_M^{-,\theta'})\times_{(X^{-\theta}\times X^{-\theta'})} (X^{-\theta}\times X^{-\theta'})_{disj}\,\iso\\ \Mod_M^{-, \theta+\theta'}\times_{X^{-\theta-\theta'}} (X^{-\theta}\times X^{-\theta'})_{disj}
\end{multline}

 Denote by 
$$
(\Mod_M^{-,\theta}\times \Mod_M^{-,\theta'})^{sm}_{disj}
$$ 
the preimage of $\Bun_M^{\theta+\theta', sm}$ under (\ref{fact_property_for_Mod_M^-}) composed with the projection to $\Bun_M^{\theta+\theta'}$. 
Write $(\cZ^{\theta}\times \oo{\cZ}{}^{\theta'})^{sm}_{disj}$ for the preimage of $(\Mod_M^{-,\theta}\times \Mod_M^{-,\theta'})^{sm}_{disj}$ under 
$$
\cZ^{\theta}\times \oo{\cZ}{}^{\theta'}\to \cZ^{\theta}\times \cZ{}^{\theta'}\to \Mod_M^{-,\theta}\times \Mod_M^{-,\theta'}
$$
% Write $\cZ^{\theta+\theta', sm}$ for the preimage of $\Bun_M^{\theta+\theta', sm}$ under the composition $\cZ^{\theta+\theta'}\to \Bun_{P^-}^{\theta+\theta'}\to \Bun_M^{\theta+\theta'}$. 

\sssec{} Let $\gq^{sm}$ be the composition
$$
(\cZ^{\theta}\times \oo{\cZ}{}^{\theta'})^{sm}_{disj} \toup{(\ref{fact_property})}  \cZ^{\theta+\theta', sm}\times_{X^{-\theta-\theta'}} (X^{-\theta}\times X^{-\theta'})_{disj}\to \Bunt_{U(P)}\times X^{-\theta'}.
$$
The map $\gq^{sm}$ is smooth and factors through the open immersion
$$
(\Bunt_{U(P)}\times X^{-\theta'})^{good}\hook{} \Bunt_{U(P)}\times X^{-\theta'}.
$$ 
For $U(P)'_{\theta'}$ large enough the diagram commutes
\begin{equation}
\label{diag_for_chasing}
\begin{array}{ccc}
(\cZ^{\theta}\times \oo{\cZ}{}^{\theta'})^{sm}_{disj} \\
\downarrow & \searrow\lefteqn{\scriptstyle \gq^{sm}}\\
(\cZ^{\theta}\times X^{-\theta'})\times_{(X^{-\theta}\times X^{-\theta'})} (X^{-\theta}\times X^{-\theta'})_{disj} && (\Bunt_{U(P)}\times X^{-\theta'})^{good}\\
\downarrow\lefteqn{\scriptstyle \gq_{\cZ}} && \downarrow \\
(\Bunt_{U(P)}\times X^{-\theta'})^{good} & \to & (\Bunt_{U(P)}\times  X^{-\theta'})^{level}/U(P)'_{\theta'}
\end{array}
\end{equation}
 
\begin{Lm} 
\label{Lm_2.7.5}
Given $\theta\in -\Lambda_{G,P}^{pos}$, there is $\theta'\in -\Lambda_{G,P}^{pos}$ such that the projection
$$
(\Mod_M^{-,\theta}\times \Mod_M^{-,\theta'})^{sm}_{disj}\to \Mod_M^{-,\theta}
$$
is surjective. \QED
\end{Lm} 
 
\sssec{} Pick $\theta'$ as in Lemma~\ref{Lm_2.7.5}, so that the projection $(\cZ^{\theta}\times \oo{\cZ}{}^{\theta'})^{sm}_{disj}\to \cZ^{\theta}$ is smooth and surjective. Now one finishes te proof of Proposition~\ref{Pp_2.6.21_Zastava} precisely as (\cite{Gai19SI}, Section~3.9.4) by chasing the diagram (\ref{diag_for_chasing}). 

 To get point a), it suffices to show that the $!$-pull-back of the $\IC$-sheaf along the composite left vertical map is isomorphic up to a shift to the $\IC$-sheaf. Since the bottom horizontal arrow is smooth, it suffices to show that the pull-back of the $\IC$-sheaf of $(\Bunt_{U(P)}\times  X^{-\theta'})^{level}/U(P)'_{\theta'}$ to $(\cZ^{\theta}\times \oo{\cZ}{}^{\theta'})^{sm}_{disj}$ is isomorphic to the $\IC$-sheaf up to a shift. This follows from the fact that both the slanted arrow and the right vertical arrows in (\ref{diag_for_chasing}) are smooth.
 
 For point b) note the following. Denote by $(\Bun_{U(P)}\times X^{-\theta'})^{level}$ the preimage of $\Bun_{U(P)}\times X^{-\theta'}$ under 
$(\Bunt_{U(P)}\times X^{-\theta'})^{level}\to (\Bunt_{U(P)}\times X^{-\theta'})^{good}.$. 
The $\gL(U(P))_{\theta'}$-action preserves the open part  
$$
(\Bun_{U(P)}\times X^{-\theta'})^{level}\subset (\Bunt_{U(P)}\times X^{-\theta'})^{level}.
$$
The preimage of $(\Bun_{U(P)}\times X^{-\theta'})^{level}/U(P)'_{\theta'}$ under bottom horizontal arrow in (\ref{diag_for_chasing}) is $\Bun_{U(P)}\times X^{-\theta'}$. Our claim follows. 
\QED 

\ssec{Proof of Proposition~\ref{Pp_2.6.23_about_Gr}}
\label{Sect_2.8}

\sssec{} By Section~\ref{Sect_2.3.6_positive part_Gr_M^+}, given $\nu\in\Lambda^+_M$ over some $\mu\in\Lambda_{G,P}$ one has $\Gr_M^{\nu}\subset\Gr_M^{-,\mu}$ iff $\nu\in -\Lambda^{pos}$. 

  Pick $\nu\in\Lambda^+_M$ over $\theta$ such that $\Gr_M^{\nu}\subset \Gr_M^{-,\theta}$, so $\nu\in -\Lambda^{pos}$. By (\cite{BFGM}, Section~6.6 after Proposition~6.6), 
$$
\dim(S^0_P\cap S_{P^-}^{\nu})\le \<-w_0^M(\nu), \check{\rho}\>
$$ 
So, the fibres of $S^0_P\cap S_{P^-}^{\nu}\to \Gr_M^{\nu}$ are of dimension at most
\begin{equation}
\label{dim_fibres_for_Sect_2.8.1}
-\<w_0^M(\nu), \check{\rho}\>-\<\nu, 2\check{\rho}_M\>
\end{equation} 
Since $w_0^M(\nu)-\nu$ vanishes in $\Lambda_{G,P}$, $\<w_0^M(\nu)-\nu, \check{\rho}-\check{\rho}_M\>=0$. This implies that (\ref{dim_fibres_for_Sect_2.8.1}) equals
$-\<\nu, \check{\rho}\>$. Since $\dim \Gr_M^{\nu}=\<\nu, 2\check{\rho}_M\>$, this implies that 
$
(\oo{\pi}_{\gF})_!e
$
is placed in perverse degrees $\le -\<\theta, 2\check{\rho}-2\check{\rho}_M\>$. 

 It remains to show that the set of irreducible components of $S^0_P\cap S_{P^-}^{\nu}$ of (maximal) dimension $\<-w_0^M(\nu), \check{\rho}\>$ naturally form a base in 
\begin{equation}
\label{vector_space_for_Sect_2.8.1}
\Hom_{\check{M}}(U^{-w_0^M(\nu)}, U(\gu(\check{P}))_{-\theta})
\end{equation}

 Pick $\mu'\in\Lambda^+_{M, ab}$ deep enough in the corresponding wall of the Weyl chamber (that is, we require $\<\mu', \check{\alpha}_i\>$ large enough for all $i\in \cI-\cI_M$). By (\cite{BFGM}, 6.6), one gets
\begin{equation}
\label{inclusion_for_Sect_2.8.1}
S_P^{-\mu'}\cap S_{P^-}^{\nu-\mu'}\subset S_P^{-\mu'}\cap \Gr_G^{w_0(w_0^M(\nu-\mu'))},
\end{equation}
and the action of $t^{\mu'}$ gives an isomorphism $S_P^{-\mu'}\cap S_{P^-}^{\nu-\mu'}\,\iso\,S_P^0\cap S_{P^-}^{\nu}$. 
By (\cite{MV}, Theorem~3.2), 
$$
S_P^{-\mu'}\cap \Gr_G^{w_0(w_0^M(\nu-\mu'))}
$$ 
is of pure dimension $-\<\check{\rho}, w_0^M(\nu)\>$. As in (\cite{BFGM}, Section~6.5) the inclusion (\ref{inclusion_for_Sect_2.8.1}) yields a bijection on the set of irreducible components of (maximal) dimension $-\<\check{\rho}, w_0^M(\nu)\>$ of both sides. By (\cite{BFGM}, Theorem~6.2), the set of ireducible components of 
$$
S_P^{-\mu'}\cap \Gr_G^{w_0(w_0^M(\nu-\mu'))}
$$ 
form a base of
$
\Hom_{\check{M}}(U^{-\mu'}, V^{w_0(w_0^M(\nu-\mu'))})
$ naturally.
Under our assumption on $\mu'$ the latter vector space identifies canonically with (\ref{vector_space_for_Sect_2.8.1}). 
\QED

\ssec{Relation between the dual baby Verma objects}
\label{Sect_4.5}

\sssec{} Main result of this subsection is Theorem~\ref{Thm_4.5.10}, which provides a precise relation between $\IC_P^{\frac{\infty}{2}}$ and the dual babdy Verma object $\cM_{\check{G},\check{P}}$ in the Hecke category of $\IndCoh((\check{\gu}(P)\times_{\check{\gg}} 0)/\check{P})$ defined in Section~\ref{Sect_2.3.10_almost_final_version}. The passage between the two uses the equivalence of G. Dhillon and H. Chen given by Proposition~\ref{Pp_Chen_Dhillon}. 

\sssec{} 
\label{Sect_4.5.1}
Write 
$$
j_{PP^-}: P\backslash PP^-/P^-\hook{} P\backslash G/P^-,\;\;\;\;\;\;\;\;\; j_{P^-P}: P^-\backslash P^-P/P\hook{} P^-\backslash G/P
$$
for the natural open immersions. We get the objects 
$$
j_{*}^P, j_{!}^P\in Shv(P\backslash G/P^-),\;\;\;\;\;\;\; j_{*}^{P^-}, j_{!}^{P^-}\in Shv(P^-\backslash G/P)
$$ 
defined as the corresponding extensions under $j_{PP^-}$ and $j_{P^-P}$ of the $\IC$-sheaves on $P\backslash PP^-/P^-$ and $P^-\backslash P^-P/P$ respectively. 

 One similarly defines 
$$
j_*^B, j_!^B\in Shv(B\backslash G/B^-),\;\;\;\;\;\;\;\; j_*^{B^-}, j_!^{B^-}\in Shv(B^-\backslash G/B)
$$ 

\sssec{} For $C\in Shv(G)-mod(\DGCat_{cont})$ we have the functors $\_\astP j_*^P: C^P\to C^{P^-}$ and $\_\astPminus j_!^{P^-}: C^{P^-}\to C^P$ defined in Section~\ref{Sect_A.7.5}. 

The following result is established in Section~\ref{Sect_appendixB}. %%% add a reference to appendix!!

\begin{Pp} 
\label{Pp_4.5.3}
The diagram commutes
$$
\begin{array}{ccccc}
C^{P} & \toup{\_\astP j^P_{*}} & C^{P-} & \toup{\_\astPminus j_!^{P^-}} & C^P \\
\downarrow\lefteqn{\scriptstyle\oblv} &&\downarrow\lefteqn{\scriptstyle\oblv} && \downarrow\lefteqn{\scriptstyle\oblv}\\
C^B & \toup{\_\astB j^B_*} & C^{B^-} & \toup{\_\astBminus j_!^{B^-}} & C^B,
\end{array} 
$$ 
and the horizontal arrows are equivalences. Besides, the composition in each line is canonically isomorphic to the identity functor. 
\end{Pp} 
\begin{Rem} For $P=B$ this is well-known. Our contribution is to generalize this to an arbitrary standard parabolic $P$.
\end{Rem}
 
\sssec{Parahoric version} By abuse of notations we also denote by 
$$
j_{PP^-}: I_P\backslash I_PI_{P^-}/I_{P^-}\hook{} I_P\backslash G(\cO)/I_{P^-},\;\;\;\;\;\; j_{P^-P}: I_{P^-}\backslash I_{P^-}I_P/I_P\hook{} I_{P^-}\backslash G(\cO)/I_P
$$ 
the corresponding immersions. We analogously get the objects
$$
j_{*}^P, j_{!}^P\in Shv(I_P\backslash G(\cO)/I_{P^-}),\;\;\;\;\;\;\; j_{*}^{P^-}, j_{!}^{P^-}\in Shv(I_{P^-}\backslash G(\cO)/I_P)
$$ 
defined as the corresponding extensions under $j_{PP^-}$ and $j_{P^-P}$, of the $\IC$-sheaves on $I_P\backslash I_PI_{P^-}/I_{P^-}$ and $I_{P^-}\backslash I_{P^-}I_P/I_P$ respectively. 

 We similarly have 
$$
j_{*}^B, j_{!}^B\in Shv(I\backslash G(\cO)/I^-),\;\;\;\;\;\;\; j_{*}^{B^-}, j_{!}^{B^-}\in Shv(I^-\backslash G(\cO)/I)
$$  
The above notations are in ambiguity with those of Section~\ref{Sect_4.5.1},  
the precise sense will be clear from the context.

 Proposition~\ref{Pp_4.5.3} immediately implies the following.
\begin{Cor}
Let $C\in Shv(G(\cO))-mod(\DGCat_{cont})$. Then the diagram commutes
$$
\begin{array}{ccccc}
C^{I_P} & \toup{\_\astIP j^P_{*}} & C^{I_{P-}} & \toup{\_\astIPminus j_!^{P^-}} & C^{I_P} \\
\downarrow\lefteqn{\scriptstyle\oblv} &&\downarrow\lefteqn{\scriptstyle\oblv} && \downarrow\lefteqn{\scriptstyle\oblv}\\
C^I & \toup{\_\astI j^B_*} & C^{I^-} & \toup{\_\astIminus j_!^{B^-}} & C^I,
\end{array} 
$$ 
and the horizontal arrows are equivalences. Besides, the composition in each line is canonically isomorphic to the identity functor. 
\end{Cor} 

\sssec{} Set $\Fl_{P^-}=G(F)/I_{P^-}$. For $\lambda\in\Lambda$ we denote by 
$$
j_{\lambda, !}^-,\;  j_{\lambda, *}^-\in \cH_{P^-}(G):=Shv(\Fl_{P^-})^{I_{P^-}}
$$ 
the standard and costandard objects attached to $t^{\lambda}\in \tilde W$. Let us reformulate Proposition~\ref{Pp_Chen_Dhillon} with $B$ and $P$ replaced by $B^-$ and $P^-$ respectively. 
\begin{Pp} 
\label{Pp_Gurbir_Chen_for_P_reformulation}
There is a canonical equivalence
\begin{equation}
\label{equivalence_Gurbir_Chen_for_P_instead}
\IndCoh((\check{\gu}(P)\times_{\check{\gg}} 0)/\check{P})\,\iso\,Shv(\Gr_G)^{I_{P^-}, ren}
\end{equation}
with the following properties:

\smallskip\noindent
(i) The $\Rep(\check{G})$-action on $\IndCoh((\check{\gu}(P)\times_{\check{\gg}} 0)/\check{P})$ arising from the projection
$$
(\check{\gu}(P)\times_{\check{\gg}} 0)/\check{P}\to pt/\check{P}\to pt/\check{G}
$$
corresponds to the action of $\Rep(\check{G})$ on $Shv(\Gr_G)^{I_{P^-}, ren}$ via $\Sat: \Rep(\check{G})\to Shv(\Gr_G)^{G(\cO)}$ and the right convolutions.

\smallskip\noindent
(ii) The $\Rep(\check{M}_{ab})$-action on $\IndCoh((\check{\gu}(P)\times_{\check{\gg}} 0)/\check{P})$ arising from the projection
$$
(\check{\gu}(P)\times_{\check{\gg}} 0)/\check{P}\to pt/\check{M}\to pt/\check{M}_{ab}
$$
corresponds to the $\Rep(\check{M}_{ab})$-action on $Shv(\Gr_G)^{I_{P^-}, ren}$ such that for $\lambda\in -\Lambda^+_{M, ab}$, $e^{\lambda}$ sends $F$ to $j_{\lambda,*}^-\ast F$. 

\smallskip\noindent
(iii) The object $\cO_{pt/\check{P}}\in \IndCoh((\check{\gu}(P)\times_{\check{\gg}} 0)/\check{P})$ corresponds under (\ref{equivalence_Gurbir_Chen_for_P_instead}) to $\delta_{1,\Gr_G}\in Shv(\Gr_G)^{I_{P^-}, ren}$. 

\smallskip\noindent
(iv) The equivalence (\ref{equivalence_Gurbir_Chen_for_P_instead}) restricts to an equivalence
$$
\IndCoh_{\Nilp}((\check{\gu}(P)\times_{\check{\gg}} 0)/\check{P})\,\iso\,Shv(\Gr_G)^{I_{P^-}}
$$
\QED
\end{Pp}

\sssec{} Recall the fully faithful embedding $\ren: Shv(\Gr_G)^{I_P}\hook{} Shv(\Gr_G)^{I_P, ren}$. Applying (\ref{eq_ren_parahoric_versus_H}), it gives a full embedding 
$$
\ren: Shv(\Gr_G)^H\hook{} Shv(\Gr_G)^{H, ren}.
$$ 
By abuse of notations, the image of $\IC_P^{\frac{\infty}{2}}$ under the latter functor is also denoted $\IC_P^{\frac{\infty}{2}}$.

 The following is one of our main results.

\begin{Thm} 
\label{Thm_4.5.10}
The image of $\IC_P^{\frac{\infty}{2}}$ under the composition
\begin{multline*}
Shv(\Gr_G)^{H, ren} \;\toup{\Av_*^{I_P/M(\cO), ren}} \;Shv(\Gr_G)^{I_P, ren}\;\toup{j_*^{P^-}\,\astIP\_}\\ Shv(\Gr_G)^{I_{P^-}, ren}\toup{(\ref{equivalence_Gurbir_Chen_for_P_instead})} \IndCoh((\check{\gu}(P)\times_{\check{\gg}} 0)/\check{P})
\end{multline*}
is canonically isomorphic to the object $\cM_{\check{G}, \check{P}}[\dim U(P^-)]$ defined in Section~\ref{Sect_2.3.10_almost_final_version}. 
\end{Thm}

In the rest of Section~\ref{Sect_4.5} we prove Theorem~\ref{Thm_4.5.10}.

\begin{Rem}[given by the referee] As explained in Sections~\ref{Sect_3.1.4_right_now}-\ref{Lm_2.2.5}, there are two families of Wakimoto sheaves in $\cH_P(G)$ given by 
$$
j_{\lambda, !}\;\;\;\mbox{or}\;\;\; j_{\lambda, *}\;\;\ \mbox{for}\; \lambda\in\Lambda^+_{M, ab},
$$
hence two actions of $\Lambda_{M, ab}$ on $Shv(\Gr_G)^{I_P, ren}$ obtained by restricting the $\cH_P(G)$-action via the monoidal functors (\ref{mon_functor_!}) and (\ref{mon_functor_*}) respectively. By Section~\ref{Sect_2.2.10_action_of_Lambda_Mab}, the first action is compatible with the equivalence $\SI^{ren}_P\,\iso\, Shv(\Gr_G)^{I_P, ren}$, when $\SI^{ren}_P$ is equipped with the natural translation action of $\Lambda_{M, ab}$, while the second is compatible with the equivalence (\ref{iso_Gurbir_Chen}). To handle this discrepancy, one observes that the equivalence 
$$
j_*^{P^-}\,\ast\_: Shv(\Gr_G)^{I_P, ren}\,\iso\, Shv(\Gr_G)^{I_{P^-}, ren}
$$
exchanges the two types of $\Lambda_{M, ab}$-actions. So, the composed equivalence 
$$
\SI^{ren}_P\,\iso\, \IndCoh((\check{\gu}(P)\times_{\check{\gg}} 0)/\check{P})
$$ 
of Theorem~\ref{Thm_4.5.10} exchanges the translation action of $\Lambda_{M, ab}$ with the natural action of $\Rep(\check{M}_{ab})$. 
\end{Rem}

\sssec{} Our starting point is the following result from (\cite{Gai19SI}, Section~4.1.3). Given $\lambda\in\Lambda^+$ there is a canonical isomorphism in $\cH(G)$
\begin{equation}
\label{iso_original_conjug_by_w_0}
j_{w_0,*}^I \astI j_{w_0(\lambda),*}^I \,\iso\, j_{\lambda, *}^I \astI j_{w_0, *}^I
\end{equation}

If $\mu\in \Lambda$, $w\in W$ then $wt^{\mu}w^{-1}=t^{w\mu}$. For this reason, 
$$
w_0It^{\mu}Iw_0=I^-w_0t^{\mu}w_0I^-=I^-t^{w_0\mu}I^-
$$ 
Consider the isomorphism $w_0: \Gr_G\,\iso\, \Gr_G, gG(\cO)\mapsto w_0gG(\cO)$. It intertwines the $I$-actions on the source by left translations and the $I$-action on the target such that $i\in I$ acts as $w_0iw_0$. This gives an involution also denoted 
$$
w_0: Shv(I\backslash \Gr_G)\,\iso\, Shv(I^-\backslash \Gr_G)
$$ 

 We assume that for an ind-scheme of finite type $Y$ and a placid group scheme $\cG$ acting on $Y$ the functor $\oblv: Shv(Y)^G\to Shv(Y)$ is related to the !-pullback via our convention of Section~\ref{Sect_A.3}.
 
For $\mu\in\Lambda$ the diagram commutes  
$$
\begin{array}{ccc}
Shv(I\backslash \Gr_G) &\toup{j_{\mu, *}^I \astI} & Shv(I\backslash\Gr_G)\\
\downarrow\lefteqn{\scriptstyle w_0} && \downarrow\lefteqn{\scriptstyle w_0} \\
Shv(I^-\backslash\Gr_G) & \toup{j_{w_0(\mu),*}^{I^-} \astIminus} & Shv(I^-\backslash\Gr_G)
\end{array}
$$
 
  The composition 
$$
Shv(I^-\backslash \Gr_G)\toup{w_0} Shv(I\backslash \Gr_G)\toup{j_{w_0, *}^I \astI} Shv(I\backslash \Gr_G)
$$ 
is $j^B_*\astI\_$. So, (\ref{iso_original_conjug_by_w_0}) implies that for $\lambda\in\Lambda^+$ one has an isomorphism
\begin{equation}
\label{iso_original_conjug_by_w_0_second}
j_*^B\astIminus j_{\lambda, *}^{I^-}\,\iso\, j_{\lambda, *}^I \astI j_*^B
\end{equation}
in $Shv(I\backslash G(F)/I^-)$. 

\sssec{} Assume in addition $\lambda\in\Lambda^+_{M, ab}$. Recall that the diagram commutes
$$
\begin{array}{ccc}
Shv(\Gr_G)^{I_P} & \toup{j_{\lambda, *}\astIP\_} & Shv(\Gr_G)^{I_P} \\
\downarrow\lefteqn{\scriptstyle \oblv} && \downarrow\lefteqn{\scriptstyle \oblv}\\
Shv(\Gr_G)^I & \toup{j^I_{\lambda, *}\astI\_} & Shv(\Gr_G)^{I} 
\end{array}
$$
\begin{Pp} 
\label{Pp_key_for_baby_Verma_transform_for_P}
For $\lambda\in\Lambda^+_{M, ab}$ one has an isomorphism
$$
j_*^P\astIPminus j_{\lambda, *}^-\,\iso\, j_{\lambda, *}\astIP j_*^P
$$
in $Shv(I_P\backslash G(F)/I_{P^-})$.
\end{Pp}
\begin{proof}
It is obtained by applying the direct image under $\eta: I\backslash G(F)/I^-\to I\backslash G(F)/I_{P^-}$ to (\ref{iso_original_conjug_by_w_0_second}). It turns out that the result is already the pullback from $I_P\backslash G(F)/I_{P^-}$. This is similar to Lemma~\ref{Lm_2.2.5}.  

 Namely, recall the natural isomorphisms 
$$
It^{\lambda}I/I\;\iso\; I_Pt^{\lambda}I_P/I_P,\;\;\;\;\;\;\;  I^-t^{\lambda}I^-/I^-\;\iso\; I_{P^-}t^{\lambda}I_{P^-}/I_{P^-}$$ 
from Lemma~\ref{Lm_2.2.5}. One has $II_{P^-}=I_PI_{P^-}$. 

 The natural map 
\begin{equation}
\label{map_to_determine_hope_iso_first}
II^-\times^{I^-} (I^-t^{\lambda}I_{P^-})/I_{P^-}\to II_{P^-}\times^{I_{P^-}} (I_{P^-}t^{\lambda}I_{P^-})/I_{P^-}
\end{equation}
is an isomorphism. Indeed, $I=K_1B, I_{P^-}=K_1P^-$, so the RHS of (\ref{map_to_determine_hope_iso_first}) is 
$$
K_1BP^-\times^{I_{P^-}}(I_{P^-}t^{\lambda}I_{P^-})/I_{P^-}\,\iso\, B\times (I_{P^-}t^{\lambda}I_{P^-})/I_{P^-}
$$ 
and the LHS of (\ref{map_to_determine_hope_iso_first}) is 
$$
K_1BB^-\times^{I^-} (I^-t^{\lambda}I_{P^-})/I_{P^-} \,\iso\,
B\times (I_{P^-}t^{\lambda}I_{P^-})/I_{P^-}
$$
So, 
$$
\eta_!(j_*^B\astIminus j_{\lambda, *}^{I^-})\,\iso\,\oblv(j_*^P\astIPminus j_{\lambda, *}^-)
$$

 Similarly, one shows that 
$$
\eta_!(j_{\lambda, *}\astI j_*^B)\,\iso\, \oblv(j_{\lambda, *}\astIP j_*^P)
$$
Our claim follows.
\end{proof}

\sssec{} Exchanging $B$ with $B^-$ (and $P$ with $P^-$) from Proposition~\ref{Pp_key_for_baby_Verma_transform_for_P} one gets the following.

\begin{Cor}
\label{Pp_key_for_baby_Verma_transform_for_Pminus}
For $\mu\in\Lambda^+_{M, ab}$ one has an isomorphism
$$
j_*^{P^-}\astIP j_{-\mu, *}\,\iso\, j^-_{-\mu, *}\astIPminus j_*^{P^-}
$$
in $Shv(I_{P^-}\backslash G(F)/I_P)$. \QED
\end{Cor}

\sssec{} From the properties of $\Av_*^{I_P/M(\cO)}$ in Sections~\ref{Sect_2.2.10_action_of_Lambda_Mab} - \ref{Sect_3.1.11_action_of}, we get
\begin{equation}
\label{expression_1_for_dual_baby}
\Av^{I_P/M(\cO), ren}_*(\IC_P^{\frac{\infty}{2}})\,\iso\, 
\mathop{\colim}\limits_{\lambda\in\Lambda_{M, ab}^+} j_{-\lambda, *}\ast \Sat(V^{\lambda}) \in Shv(\Gr_G)^{I_P, ren}
\end{equation}
where the colimit is taken in $Shv(\Gr_G)^{I_P, ren}$. 

 Note that acting by $j_*^{P^-}\in Shv(I_{P^-}\backslash G(\cO)/I_P)$ on $\delta_{1,\Gr_G}\in Shv(\Gr_G)^{I_P}$ one gets 
$$
j_*^{P^-}\astIP \delta_{1,\Gr_G}\,\iso\, \delta_{1,\Gr_G}[\dim U(P^-)]\in Shv(\Gr_G)^{I_{P^-}}
$$ 

Applying $j_*^{P^-}\astIP\_$ to (\ref{expression_1_for_dual_baby}) one gets
$$
\mathop{\colim}\limits_{\lambda\in\Lambda_{M, ab}^+} (j_*^{P^-}\astIP j_{-\lambda, *})\ast \Sat(V^{\lambda}) \in Shv(\Gr_G)^{I_{P^-}, ren},
$$
which by Corollary~\ref{Pp_key_for_baby_Verma_transform_for_Pminus} identifies with
$$
\mathop{\colim}\limits_{\lambda\in\Lambda_{M, ab}^+} (j^-_{-\lambda, *}\astIPminus j_*^{P^-})\ast \Sat(V^{\lambda}) \in Shv(\Gr_G)^{I_{P^-}, ren}
$$
The latter object identifies with
\begin{equation}
\label{complex_for_Sect_4.5.16}
\mathop{\colim}\limits_{\lambda\in\Lambda_{M, ab}^+} j^-_{-\lambda, *}\ast \Sat(V^{\lambda})[\dim U(P^-)]\in Shv(\Gr_G)^{I_{P^-}, ren},
\end{equation}
where in the latter formula we use the action of $\cH_{P^-}(G)$ on $Shv(\Gr_G)^{I_{P^-}, ren}$.

 By Proposition~\ref{Pp_Gurbir_Chen_for_P_reformulation}, (\ref{complex_for_Sect_4.5.16}) under the equivalence (\ref{equivalence_Gurbir_Chen_for_P_instead}) identifies with
$$
\mathop{\colim}\limits_{\lambda\in\Lambda_{M, ab}^+} e^{-\lambda}\otimes V^{\lambda}[\dim U(P^-)]\in 
\IndCoh((\check{\gu}(P)\times_{\check{\gg}} 0)/\check{P})
$$
By (\ref{iso_O(P/M)_for_Sect_2.1.6}), the latter identifies with the direct image of $\cO(\check{P}/\check{M})[\dim U(P^-)]\in\Rep(\check{P})$ under $B(\check{P})\to  (\check{\gu}(P)\times_{\check{\gg}} 0)/\check{P}$. That is, we get the object $\cM_{\check{G},\check{P}}[\dim U(P^-)]$. Theorem~\ref{Thm_4.5.10} is proved. \QED

\ssec{Proof of Theorem~\ref{Thm_4.1.10}}
\label{Sect_4.6}

\sssec{Semi-infinite relative dimensions} It is convenient first to introduce the semi-infinite relative dimensions of orbits. For $\lambda\in \Lambda^+_M$ the relative dimension $\dimrel(S^{\lambda}_P : S^0_P)$ is defined as follows. Write $H^{\lambda}$ for the stabilizer of $t^{\lambda}G(\cO)\in \Gr_G$ in $H$. Set 
$$
\dimrel(S^{\lambda}_P: S^0_P)=\dimrel(H^0: H^{\lambda})=\dim(H^0/H^0\cap H^{\lambda})-\dim(H^{\lambda}/H^0\cap H^{\lambda})
$$ 
This is easy to calculate, one gets $\dimrel(S^{\lambda}_P: S^0_P)=\<\lambda, 2\check{\rho}\>$. 

\sssec{} For $\theta\in \Lambda_{G,P}$ define the relative dimension
$\dimrel(\Gr_P^{\theta}: \Gr_P^0)$ as follows. Set $M'=[M,M]$. Set $\cK=M'(F)U(P)(F)$. Recall that $\cK$ acts transitively on $\Gr_P^{\theta}$. Pick any $\lambda\in\Lambda$ over $\theta$. Let $\cK^{\lambda}$ be the stabilizer of $t^{\lambda}G(\cO)$ in $\cK$. Set
$$
\dimrel(\Gr_P^{\theta}: \Gr_P^0)=\dimrel(\cK^0: \cK^{\lambda})=\dim(\cK^0/\cK^0\cap \cK^{\lambda})-\dim(\cK^{\lambda}/\cK^0\cap \cK^{\lambda}).
$$
One checks that this is independent of a choice of $\lambda$. More precisely, 
$$
\dimrel(\Gr_P^{\theta}: \Gr_P^0)=\<\theta, 2\check{\rho}-2\check{\rho}_M\>.
$$
If $\Gr_P^0\subset \ov{\Gr}_P^{\theta}$ and $\theta\ne 0$ then $\<\theta, 2\check{\rho}-2\check{\rho}_M\> >0$. 

\sssec{Another system of generators} Let $\nu\in\Lambda^+_M$ and $\theta\in\Lambda_{G,P}$ its image. Denote by $J_{\nu, !}\in Shv(\Gr_G)^{H, \le 0}$ the object 
$$
J_{\nu, !}=(v_P^{\theta})_!(\gt_P^{\theta})^!\Sat_M(U^{\nu})[\<\theta, 2\check{\rho}_M-2\check{\rho}\>]
$$ 

 Note that $Shv(\Gr_G)^{H, \le 0}\subset Shv(\Gr_G)^H$ is the smallest full subcategory containing $J_{\nu, !}$ for $\nu\in\Lambda^+_M$, stable under extensions and small colimits.

\sssec{} For $F\in \SI_P$ the property $F\in Shv(\Gr_G)^{H,\ge 0}$ is equivalent to $\HOM_{\SI_P}(\bvartriangle^{\mu}, F)\in\Vect^{\ge 0}$ for any $\mu\in\Lambda^+_M$. 

 It is easy to see that for $F\in \SI_P$ the property $F\in Shv(\Gr_G)^{H,\ge 0}$ is also equivalent to $\HOM_{\SI_P}(J_{\nu, !}, F)\in\Vect^{\ge 0}$ for all $\nu\in\Lambda^+_M$.
 
\sssec{}  Recall that if $V\in\Rep(\check{M}), \theta\in\Lambda_{G,P}$ then $V_{\theta}$ denotes the direct summand of $V$, on which $Z(\check{M})$ acts by $\theta$. 

\begin{Lm} 
\label{Lm_4.6.6}
Let $\gamma\in\Lambda^+$ and $\theta$ be the image of $w_0(\gamma)$ in $\Lambda_{G, P}$. Then $V^{\gamma}_{\theta}$ is an irreducible $\check{M}$-module with highest weight $w_0^Mw_0(\gamma)$. 
\end{Lm}
\begin{proof}
Let $\check{\gu}(P^-), \check{\gu}(P), \check{\gm}$ denote the Lie algebras of $U(\check{P}^-), U(\check{P}), \check{M}$ respectively. We have 
$$
U(\check{\gg})\,\iso\, U(\check{\gu}(P))\otimes U(\check{\gm})\otimes U(\check{\gu}(P^-))
$$ 
for the universal enveloping algebras. Let $v\in V^{\gamma}$ be a lowest weight vector. Then $V^{\gamma}=U(\check{\gu}(P))\otimes U(\check{\gm})v$. Moreover, $U(\check{\gm})v\subset V$ is an irreducible $\check{M}$-module. Indeed, otherwise there would exist another nontrivial lowest weight vector $v'\in U(\check{\gm})v$. But then $v'$ would be a lowest vector for $G$ itself, because $U(\check{P}^-)$ is normal in $U(\check{B}^-)$. Here $U(\check{B}^-)$ is the unipotent radical of $\check{B}^-$. 
Our claim follows now from the observation that the $\check{T}$-weights on $\check{\gu}(P)$ are nonzero elements of $\Lambda_{G,P}^{pos}$.
\end{proof}

\begin{Lm} 
\label{Lm_4.6.7}
Let $\gamma\in\Lambda^+$. There is a canonical fibre sequence in $\SI_P$
$$
K\to \bvartriangle^0\ast \Sat(V^{\gamma})\to J_{w_0^Mw_0(\gamma), !},
$$ 
where $K$ admits a finite filtration by objects $J_{\mu, !}$ with $\mu\in\Lambda^+_M$ such that $U^{\mu}$ appears in $\Res(V^{\gamma})$ and satisfies $\mu\ne w_0^Mw_0(\gamma)$. For such $\mu$ we have 
$$
\<\mu-w_0^Mw_0(\gamma), 2\check{\rho}-2\check{\rho}_M\> > 0
$$
\end{Lm}
\begin{proof}
Consider the filtration on $\bvartriangle^0\ast \Sat(V^{\gamma})$ coming from the stratification of $\Gr_G$ by $\{\Gr_P^{\theta}\}_{\theta\in\Lambda_{G,P}}$. The successive subquotients of this filtrations are 
\begin{equation}
\label{object_succ_quotient_for_t-str}
(v_P^{\theta})_!(v_P^{\theta})^*(\bvartriangle^0\ast \Sat(V^{\gamma})),\;\;\; \mbox{for}\;\;\; \theta\in\Lambda_{G,P}.
\end{equation} 
As we have seen in the proof of Proposition~\ref{Pp_2.3.21}, for $\theta\in\Lambda_{G,P}$ one has canonically
$$
(\gt^{\theta}_P)_!(v_P^{\theta})^*(\bvartriangle^0\ast \Sat(V^{\gamma}))\,\iso\, \Sat_M(V^{\gamma}_{\theta})[\<\theta, 2\check{\rho}_M-2\check{\rho}\>]
$$
We see that (\ref{object_succ_quotient_for_t-str}) is a direct sum of finitely many objects of the form $J_{\nu, !}$ for $\nu\in \Lambda^+_M$ with $U^{\nu}$ appearing in $V^{\gamma}_{\theta}$. 

 Let now $\theta$ be the image of $w_0(\gamma)$ in $\Lambda_{G,P}$. Write $Y$ for the support of $\bvartriangle^0\ast \Sat(V^{\gamma})$. It is easy to see that $\Gr_P^{\theta}\cap Y$  is closed in $Y$. This follows from the fact that for $\theta_1,\theta_2\in\Lambda_{G,P}$ one has $\Gr_P^{\theta_1}\subset \ov{\Gr}_P^{\theta_2}$ iff $\theta_2-\theta_1\in\Lambda_{G,P}^{pos}$. 

 Our claim follows now from Lemma~\ref{Lm_4.6.6}.
\end{proof}

\sssec{} For $\mu\in\Lambda_{M, ab}, \gamma\in\Lambda^+$ set 
$$
\Upsilon^{\mu, \gamma}=\Sat_M(U^{\mu})\ast \bvartriangle^0\ast \Sat(V^{\gamma})[-\<\mu, 2\check{\rho}\>] \in Shv(\Gr_G)^{H,\le 0}
$$

It is easy to see that for any $\mu\in\Lambda_{M, ab}$, $\nu\in \Lambda^+_M$ one has canonically
$$
t^{\mu}J_{\nu, !}[-\<\mu,2\check{\rho}\>]\,\iso\, J_{\nu+\mu, !}
$$ 
Besides, for any $\nu\in\Lambda^+_M$ there is $\mu\in \Lambda^+_{M, ab}$ such that $w_0w_0^M(\nu-\mu)\in \Lambda^+$. From Lemma~\ref{Lm_4.6.7} we immediately derive the following.
\begin{Cor} For any $\nu\in\Lambda^+_M$ there is $\mu\in\Lambda_{M, ab}^+, \gamma\in\Lambda^+$ and a fibre sequence 
\begin{equation}
\label{fibre_sequence_for_Gurbir_induction}
K\to \Upsilon^{\mu, \gamma}\to J_{\nu, !},
\end{equation}
where $K$ admits a finite filtration by objects $J_{\nu', !}$ with $\nu'\in\Lambda^+_M$ satisfying 
$$
\<\nu'-\nu, 2\check{\rho}-2\check{\rho}_M\> >0.
$$ 
\end{Cor}

\begin{Pp} 
\label{Pp_induction_step}
Let $F\in Shv(\Gr_G)^H$. Assume there is $N\in\ZZ$ such that for any $\nu\in\Lambda^+_M$ with $\<\nu, 2\check{\rho}-2\check{\rho}_M\> >N$ one has $\HOM_{\SI_P}(J_{\nu, !}, F)\in\Vect^{\ge 0}$. Then the following properties are equivalent:
\begin{itemize}
\item[i)] $F\in Shv(\Gr_G)^{H,\ge 0}$;
\item[ii)] if $\mu\in\Lambda_{M, ab}^+, \gamma\in\Lambda^+$ then $\HOM_{\SI_P}(\Upsilon^{\mu, \gamma}, F)\in \Vect^{\ge 0}$.
\end{itemize}
\end{Pp}
\begin{proof} i) implies ii), because for $\mu\in\Lambda_{M, ab}, \gamma\in\Lambda^+$ one has $\Upsilon^{\mu, \gamma}\in Shv(\Gr_G)^{H,\le 0}$ by Propositions~\ref{Pp_2.4.19} and \ref{Pp_2.3.21}. 

Assume ii). To get i), it suffices to show that for any $\nu\in\Lambda^+_M$ one has 
$$
\HOM_{\SI_P}(J_{\nu, !}, F)\in\Vect^{\ge 0}.
$$ 
We proceed by the descending induction on $\<\nu, 2\check{\rho}-2\check{\rho}_M\>$. Let $N'\in\ZZ$. Assume for any $\nu\in \Lambda^+_M$ with $\<\nu, 2\check{\rho}-2\check{\rho}_M\> >N'$ one has $\HOM_{\SI_P}(J_{\nu, !}, F)\in\Vect^{\ge 0}$. Let $\nu\in \Lambda^+_M$ with $\<\nu, 2\check{\rho}-2\check{\rho}_M\>=N'$. Pick the fibre sequence (\ref{fibre_sequence_for_Gurbir_induction}). It yields a fibre sequence in $\Vect$
$$
\HOM_{\SI_P}(J_{\nu, !}, F)\to \HOM_{\SI_P}(\Upsilon^{\mu, \gamma}, F)\to \HOM_{\SI_P}(K, F),
$$
which shows that $\HOM_{\SI_P}(J_{\nu, !}, F)\in\Vect^{\ge 0}$. 
\end{proof}

\begin{proof}[Proof of Theorem~\ref{Thm_4.1.10}] 
By Lemma~\ref{Lm_2.3.13_about_Av!} and Proposition~\ref{Pp_2.4.19}, we may and do assume $\mu=0$. 

 If $i\in\cI-\cI_M$ then $\<\alpha_i, 2\check{\rho}_M\>\le 0$. So, for $\theta\in\Lambda_{G,P}^{pos}$ one has $\<\theta, 2\check{\rho}-2\check{\rho}_M\>\ge 0$. 
Note that $\bvartriangle^0$ is the extension by zero from $\ov{\Gr}_P^0$. So, if $\nu\in\Lambda_M^+$ with $\<\nu, 2\check{\rho}-2\check{\rho}_M\> >0$ then $\HOM_{\SI_P}(J_{\nu, !},\bvartriangle^0)=0$. Thus, By Proposition~\ref{Pp_induction_step}, it suffices to show that for any $\mu\in\Lambda_{M, ab}^+, \gamma\in\Lambda^+$ one has 
$$
\HOM_{\SI_P}(\Upsilon^{\mu, \gamma}, \bvartriangle^0)\in \Vect^{\ge 0}
$$
Applying (\ref{equiv_Iwahori_vs_SI}), we are reduced to show that
\begin{equation}
\label{property_for_step_one_Thm_4.1.10}
\HOM_{Shv(\Gr_G)^{I_P}}(j_{\mu, !}\ast \Sat(V^{\gamma}), \delta_{1,\Gr_G})\in \Vect^{\ge 0}.
\end{equation}

 By Lemma~\ref{Lm_2.3.13_about_Av!} ii), we are reduced to show that
$$
\HOM_{Shv(\Gr_G)^{I_P}}(\cB_{\mu, !}, \Sat((V^{\gamma})^*))\in \Vect^{\ge 0}
$$
Equip $Shv(\Gr_G)^{I_P}$ with the perverse t-strcture defined in Section~\ref{Sect_A.5.4_right_now}. Taking into account Section~\ref{Sect_2.1.4_objects_cB}, both $\cB_{\mu, !}$ and $\Sat((V^{\gamma})^*)$ lie in the heart of the t-structure of $Shv(\Gr_G)^{I_P}$. Our claim follows.
\end{proof}

\sssec{Second proof of Theorem~\ref{Thm_4.1.10}} Using the notations of the first proof, we establish (\ref{property_for_step_one_Thm_4.1.10}) in another way. Consider the adjoint pair
$$
I: \IndCoh_{\Nilp}(\check{\gu}(P)\times_{\check{\gg}} 0)/\check{P})\leftrightarrows \IndCoh(\check{\gu}(P)\times_{\check{\gg}} 0)/\check{P}): I^!
$$
corresponding under (\ref{equivalence_Gurbir_Chen_for_P_instead}) to the adjoint pair $\ren: Shv(\Gr_G)^{I_{P^-}}\leftrightarrows Shv(\Gr_G)^{I_{P^-}, ren}: \unren
$. By Corollary~\ref{Pp_key_for_baby_Verma_transform_for_Pminus} and Lemma~\ref{Lm_2.2.5}, one has
$$
j^-_{\mu, !}\astIPminus j_*^{P^-}\,\iso\, j_*^{P^-} \astIP j_{\mu, !}
$$
 
Now applying 
$$
Shv(\Gr_G)^{I_P}\,\toup{j_*^{P^-}\astIP}Shv(\Gr_G)^{I_{P^-}}\toup{(88)}\IndCoh_{\Nilp}(\check{\gu}(P)\times_{\check{\gg}} 0)/\check{P}), 
$$
the property (\ref{property_for_step_one_Thm_4.1.10}) becomes
$$
\HOM_{\IndCoh_{\Nilp}(\check{\gu}(P)\times_{\check{\gg}} 0)/\check{P})}(I^!i_*(e^{-\mu}\otimes V^{\gamma}), I^!i_*e),
$$ 
where we have denoted by 
$$
i: B(P)\to (\check{\gu}(P)\times_{\check{\gg}} 0)/\check{P}
$$ 
the natural map. Recall that $\IndCoh_{\Nilp}(\check{\gu}(P)\times_{\check{\gg}} 0)/\check{P})$ carries a unique t-structure such that 
$$
i^!: \IndCoh_{\Nilp}(\check{\gu}(P)\times_{\check{\gg}} 0)/\check{P})\to \QCoh(\check{\gu}(P)\times_{\check{\gg}} 0)/\check{P})
$$ 
is t-exact. Moreover, $i^!$ induces an equivalence
$$
\IndCoh_{\Nilp}(\check{\gu}(P)\times_{\check{\gg}} 0)/\check{P})^+\,\iso\,\QCoh(\check{\gu}(P)\times_{\check{\gg}} 0)/\check{P})^+
$$
on the subcategories of eventually coconnective objects. For the t-structures on $\IndCoh$ of an Artin stack we refer to (\cite{Gai_IndCoh}, Proposition~11.7.5). Moreover, the composition
\begin{multline*}
\QCoh(B(P))\toup{i_*} \IndCoh(\check{\gu}(P)\times_{\check{\gg}} 0)/\check{P})\toup{I^!} \IndCoh_{\Nilp}(\check{\gu}(P)\times_{\check{\gg}} 0)/\check{P})\\ \toup{i^!} \QCoh(\check{\gu}(P)\times_{\check{\gg}} 0)/\check{P})
\end{multline*}
is the usual pushforward $i_*: \QCoh(B(P))\to \QCoh(\check{\gu}(P)\times_{\check{\gg}} 0)/\check{P})$. Theorem~\ref{Thm_4.1.10} is proved. \QED

\appendix

\section{Generalities}

\ssec{Some adjoint pairs}
\begin{Pp} 
\label{Pp_A.1}
Let $H, G$ be placid group schemes, $G\hook{} H$ be a subgroup (not necessarily a placid closed immersion). Assume $G\,\iso\, \lim_{i\in I^{op}} G_i$, where $G_i$ is a smooth group scheme of finite type, $I\in 1-\Cat$ is filtered, and for $i\to j$ in $I$ the map $G_j\to G_i$ is smooth affine surjective homomoprhism of group schemes. Write $K_i=\Ker(G\to G_i)$. Assume $H\,\iso\, \lim_{i\in I^{op}} H/K_i$ in $\PreStk$. Assume $H/G$ is a pro-smooth placid scheme. Consider the projection $p: H/G\to \Spec k$ as $H$-equivariant map. Then 

\smallskip\noindent
i) the adjoint pair $p^*: \Vect\leftrightarrows Shv(H/G): p_*$ takes place in $Shv(H)-mod$; \\
ii) assume $C\in Shv(H)-mod(\DGCat_{cont})$. The above adjoint pair gives an adjoint pair in $\DGCat_{cont}$
$$
\oblv: C^H\leftrightarrows C^G: \Av^{H/G}_*
$$
\end{Pp}
\begin{proof} i) Since $p$ is $H$-equivariant map, $p_*: Shv(H/G)\to \Vect$ is a morphism of $Shv(H)$-modules. Now the diagram is cartesian
$$
\begin{array}{ccc}
H\times H/G & \toup{act} & H/G\\
\downarrow\lefteqn{\scriptstyle\pr} && \downarrow\\
H & \to & \Spec k
\end{array}
$$
So, for $K\in Shv(H)$, $\act_*(K\boxtimes p^*e)\,\iso\, p^*\RG(H, K)$ canonically. Here we used the base change isomorphism given by Lemma~\ref{Lm_A.1.3_base_change}.

\smallskip\noindent
ii) Applying $\Fun_{Shv(H)}(\cdot, C)$, we get the adjoint pair $\oblv: \Fun_{Shv(H)}(\Vect, C)\leftrightarrows \Fun_{Shv(H)}(Shv(H/G), C): \Av^{H/G}_*$. Now using the assumption $H\,\iso\, \lim_{i\in I^{op}} H/K_i$ in $\PreStk$ from (\cite{Ly4}, 0.0.36) 
we get $Shv(H/G)\,\iso\, Shv(H)^G$ with respect to the $G$-action on $H$ by right translations.

 Recall that $Shv(H)^G\,\iso\, Shv(H)_G$, because $G$ is a placid group scheme. Finally, 
$$
\Fun_{Shv(H)}(Shv(H/G), C)\,\iso\, \Fun_{Shv(H)}(Shv(H)\otimes_{Shv(G)}\Vect, C)\,\iso\, C^G
$$
\end{proof}

\sssec{} 
\label{Sect_A.0.2}
An example of the situation as in Proposition~\ref{Pp_A.1}. Assume $H=G\rtimes \bar H$, where $\bar H\subset H$ is a normal subgroup, $\bar H$ is a placid group scheme, and $G$ acts on $\bar H$ by conjugation. Here $G$ is a placid group scheme. 

If moreover $\bar H$ is prounipotent then Proposition~\ref{Pp_A.1} also shows that the functor $\oblv: C^H\to C^G$ is fully faithful. Indeed, the natural map $\id\to \Av^{H/G}_*\oblv$ is the identity, because $p_*p^*: \Vect\to \Vect$ identifies with $\id$.

\sssec{} Consider a cartesian square 
\begin{equation}
\label{square_1}
\begin{array}{ccc}
X& \toup{f_X} & X'\\
\downarrow\lefteqn{\scriptstyle g} &&\downarrow\lefteqn{\scriptstyle g'}\\
Y& \toup{f_Y} & Y',
\end{array}
\end{equation}
in $\PreStk$.
\begin{Lm}[\cite{Ly4}, Lemma~0.0.20]
\label{Lm_A.1.3_base_change}
Assume $Y'\in Sch_{ft}$, and $Y, X'$ are placid schemes over $Y'$. Then $X$ is also a placid scheme. Assume $Y\,\iso\, \lim_{i\in I^{op}} Y_i$, where $I$ is filtered, $f_{Y,i}: Y_i\to Y'$ is smooth, $Y_i\in Sch_{ft}$, and for $i\to j$ in $I$, $Y_j\to Y_i$ is smooth affine surjective morphism in $\Sch_{ft}$. Then one has $f_Y^*g'_*\,\iso\, g_*f_X^*$ as functors $Shv(X')\to Shv(Y)$.
\end{Lm} 

\ssec{Some coinvariants} 
\label{Sect_A.1}

\sssec{} 
Let $P\subset G$ be a parabolic in a connected split reductive group with Levi $M$ and unipotent radical $U$. Let $F=k((t)), \cO=k[[t]]$. Let $H=M(\cO)U(F)$. This is a placid ind-scheme, closed in $P(F)$. We have also $P(F)/H\,\iso\, M(F)/M(\cO)$. Since the object $\delta_1\in Shv(\Gr_M)$ is $H$-invariant, the functor $\Vect\to Shv(\Gr_M)$, $e\mapsto \delta_1$ is $Shv(H)$-linear. Now the $Shv(H)$-action on $Shv(\Gr_M)$ comes as the restriction of a $Shv(P(F))$-action, hence we get by adjointness a canonical functor 
\begin{equation}
\label{functor_for_Sect_A.1}
Shv(P(F))\otimes_{Shv(H)}\Vect\to Shv(\Gr_M)
\end{equation}
\begin{Lm} 
\label{Lm_A.1.2_coinvariants}
The functor (\ref{functor_for_Sect_A.1}) is an equivalence. 
\end{Lm}
\begin{proof} Pick a presentation $U(F)=\colim_{n\in\NN} U_n$, where $U_n$ is a prounipotent group scheme, for $n\le m$, $U_n\to U_m$ is a placid closed immersion and a homomorphism of group schemes. Assume $M(\cO)$ normalizes each $U_n$, so $M(\cO)U_n=:H_n$ is a placid group scheme, and $H\,\iso\,\colim_{n\in\NN} H_n$. We have $P(F)\,\iso\, \colim_{n\in \NN} M(F)U_n$ in $\PreStk$, as colimits in $\PreStk$ are universal. Recall that for any morphism $f: Y_1\to Y_2$ of placid ind-schemes the functor $f_*: Shv(Y_1)\to Shv(Y_2)$ is well-defined. Now 
$$
Shv(P(F))\,\iso\, \colim_{n\in \NN} Shv(M(F)U_n)
$$ 
with respect to the $*$-push-forwards. Indeed, pick a presentation $M(F)\,\iso\,\colim_{i\in I} M_i$, where $M_i$ is a placid scheme, $I$ is small filtered, for $i\to i'$ the map $M_i\to M_{i'}$ is a placid closed immersion. Then $Shv(M(F)U_n)\,\iso\, \colim_{i\in I} Shv(M_iU_n)$.  

The above gives
$$
Shv(P(F))\otimes_{Shv(H)}\Vect \,\iso\,\mathop{\colim}\limits_{(n\le m)\in\Fun([1],\NN)} Shv(M(F)U_m)\otimes_{Shv(H_n)}\Vect
$$
The diagonal map $\NN\to \Fun([1],\NN)$ is cofinal, so the above identifies with
$$
\colim_{n\in \NN} Shv(M(F)U_n)\otimes_{Shv(H_n)}\Vect
$$
Now each term of the latter diagram identifies with $Shv(M(F)/M(\cO))$ using (\cite{Ly4}, 0.0.36), and we are done. Indeed, for any $I\in 1-\Cat$ the natural map $I\to \mid I\mid$ is cofinal, and for $I$ filtered we get $\mid I\mid\,\iso\, *$. 
\end{proof}

 Further, let $C\in Shv(P(F))-mod(\DGCat_{cont})$. We get 
$$
C^H=
\Fun_{Shv(H)}(\Vect, C)\,\iso\, \Fun_{Shv(P(F))}(Shv(P(F))\otimes_{Shv(H)}\Vect, C)
$$ 
Thus, $\Fun_{Shv(P(F))}(Shv(\Gr_M), Shv(\Gr_M))$ acts on $C^H$. Now 
$$
\Fun_{Shv(P(F))}(Shv(\Gr_M), Shv(\Gr_M))\,\iso Shv(\Gr_M)^H\,\iso\, Shv(\Gr_M)^{M(\cO)}
$$
acts on $C^H$.

\begin{Rem} 
\label{Rem_A.1.3_action_of_M(F)}
As in Lemma~\ref{Lm_A.1.2_coinvariants}, one shows that $Shv(P(F))_{U(F)}\,\iso\, Shv(M(F))$ naturally in $Shv(P(F))-mod(\DGCat_{cont})$, where we used the $U(F)$-action on $P(F)$ by right translations. This similarly implies that for any $C\in Shv(P(F))-mod(\DGCat_{cont})$, $C^{U(F)}\in Shv(M(F))-mod(\DGCat_{cont})$ naturally. 
\end{Rem}

\ssec{Some invariants}
\label{Section_A2_some_invarinats}

\sssec{} An observation about categories of invariants. Let $Y=\colim_{j\in J} Y_j$ be an ind-scheme of ind-finite type, here $J$ is a filtered category, $Y_j$ is a scheme of finite type, and for $j\to j'$ in $J$ the map $Y_j\to Y_{j'}$ is a closed immersion. 

Let $\alpha: H\to G$ be a homomorphism of group schemes, which are placid schemes. Assume $I$ is a filtered category and $H\,\iso\lim_{i\in I^{op}} H_i$, $G\,\iso\,\lim_{i\in I^{op}} G_i$, where $H_i$, $G_i$ is a smooth group scheme of finite type, for $i\to j$ in $I$ the transition maps $H_j\to H_i$, $G_j\to G_i$ are smooth, affine, surjective homomorphisms. Besides, we are given a diagram $I^{op}\times [1]\to Grp(\Sch_{ft})$, sending $i$ to $\alpha_i: H_i\to G_i$, where $\alpha_i$ is a closed subgroup. Here $Grp(\Sch_{ft})$ is the category of groups in $\Sch_{ft}$. We assume $\alpha=\lim_{i\in I^{op}} \alpha_i$. We assume for each $i$, $\Ker(H\to H_i)$ and $\Ker(G\to G_i)$ are prounipotent.

 Assume $G$ acts on $Y$. Moreover, for any $j\in J$, $Y_j$ is $G$-stable, and $G$ acts on $Y_j$ through the quotient $G\to G_i$ for some $i\in I$. We claim that 
%(for any of our 4 sheaf theories) 
$oblv: Shv(Y)^G\to Shv(Y)^H$ admits a continuous right adjoint $\Av_*$.
 
\begin{proof} We have $Shv(Y)^G\,\iso\,\lim_{j\in J^{op}} Shv(Y_j)^G$ with respect to the !-pullbacks, similarly $Shv(Y)^H\,\iso\,\lim_{j\in J^{op}} Shv(Y_j)^H$, and $oblv: Shv(Y)^G\to Shv(Y)^H$ is the limit over $j\in J^{op}$ of $\oblv_j: Shv(Y_j)^G\to Shv(Y_j)^H$. For given $j\in J$ the functor $\oblv_j$ admits a continuous right adjoint $\Av_{j, *}$. Indeed, pick $i\in I$ such that $G$-action on $Y_j$ factors through $G_i$. Then $\oblv_j$ identifies with the functor $f^!: Shv(Y_j/G_i)\to Shv(Y_j/H_i)$ for the projection $f: Y_j/H_i\to Y_j/G_i$. Since $G_i/H_i$ is smooth, $f$ is smooth. So, $f^!$ admits a continuous right adjoint (as $f$ is schematic of finite type).

 Let now $j\to j'$ be a map in $J$. Pick $i$ such that the $G$-action on $Y_j, Y_{j'}$ factors through $G_i$. Then we get a cartesian square
$$
\begin{array}{ccc}
Y_j/H_i & \toup{h} & Y_{j'}/H_i\\
\downarrow\lefteqn{\scriptstyle f_j}&& \downarrow\lefteqn{\scriptstyle f_{j'}}\\
Y_j/G_i & \toup{h'} & Y_{j'}/G_i,
\end{array}
$$ 
where $h, h'$ are closed immersions. We have $(h')^!f_{j', *}\,\iso\, f_{j, *} h^!$. Since $f_j, f_{j'}$ are of the same relative dimension, we see that the diagram commutes
$$
\begin{array}{ccc}
Shv(Y_j)^H & \getsup{h^!} & Shv(Y_{j'})^H\\
\downarrow\lefteqn{\scriptstyle \Av_{j, *}}&&\downarrow\lefteqn{\scriptstyle \Av_{j', *}}\\
Shv(Y_j)^G & \getsup{(h')^!} & Shv(Y_{j'})^G
\end{array}
$$
By (\cite{G}, ch. I.1, 2.6.4), $\oblv$ admits a right adjoint $\Av_*$, and for the evaliation maps $\ev_j: Shv(Y)^G\to Shv(Y_j)^G$, $\ev_j: Shv(Y)^H\to Shv(Y_j)^H$ one gets $\ev_j\Av_*\,\iso\, \Av_{j, *}\ev_j$. 
So, $\Av_*$ is continuous.  
\end{proof} 

\begin{Rem} 
\label{Rem_A.2.2}
If $L: C\leftrightarrows C': R$ is an adjoint pair in $\DGCat^{non-cocmpl}$ then $\Ind(L): \Ind(C)\leftrightarrows\Ind(C'): \Ind(R)$ is an adjoint pair in $\DGCat_{cont}$.
\end{Rem}

\begin{Lm} 
\label{Lm_A.2.3}
Let $C\in \DGCat_{cont}$, $C_i\subset C$ be a full subcategory, this is a map in $\DGCat_{cont}$ for $i\in I$. Here $I\in 1-\Cat$ is filtered. Assume for $i\to j$ in $I$, $C_j\subset C_i$. Set $D=\cap_i C_i=\lim_{i\in I^{op}} C_i$, where the limit is calculated in $\DGCat_{cont}$. Assume $L_i: C\to C_i$ is a left adjoint to the inclusion. Then $D$ is a localization of $C$, and the localization functor $L: C\to D$ is given by $L(c)=\colim_{i\in I} L_i(c)$, where the transition maps are the localization morphisms for $C_j\subset C_i$, and the colimit is calculated in $C$.
\end{Lm} 
\begin{proof} For $x\in \cap_i C_i$, $c\in C$ we get 
\begin{multline*}
\Map_C(\colim_i L_i(c), x)\,\iso\, \lim_{i\in I^{op}} \Map(L_i(c), x)\,\iso\,  \lim_{i\in I^{op}} \Map(c, x)\\
\,\iso\, \Map(c,x)\,\iso\,\Fun(I^{op}, \Map(c,x))
\end{multline*}
For $J\in 1-\Cat$, $Z\in \Spc$ we have $\Fun(J, Z)\,\iso\, \Fun(\mid J\mid, Z)$, where $\mid J\mid\in\Spc$ is obtained by inverting all arrows. Since a filtered category is contractible, we are done. 

 To explain that $L$ takes values in $\cap C_i$, note that we may equally understand $\colim_i L_i(c)$ as taken in $C_j$ over $i\in I_{j/}$, because the inclusion $C_j\subset C$ is continuous, so the colimit lies in $C_j$ for any $j$.
\end{proof}

\sssec{About representations of $G(F)$} 
\label{Sect_A.2.4}
Let $G$ be a connected reductive group over $k$. Recall from (\cite{GaLocWhit}, D.1.2) that for any $C\in \DGCat_{cont}$ with an action of $Shv(G(F))$, $C\,\iso\, \mathop{\colim}\limits_{n\in\NN} C^{K_n}$, where $K_n=\Ker(G(\cO)\to G(\cO/t^n))$. So, for any $c\in C$, 
$$
c\,\iso\, \mathop{\colim}\limits_{n\in\NN} \oblv_n \Av^{K_n}_*(c),
$$ 
where $\oblv_n: C^{K_n}\to C$ and $\Av^{K_n}_*: C\to C^{K_n}$ are adjoint functors (by \cite{Ly}, 9.2.6).

\ssec{Actions} 

\sssec{}
The theory of placid group ind-schemes acting on categories is developed for $\cD$-modules in (\cite{LC}, Appendix B). A version of this theory in the constructible context is developed to some extent in (\cite{Ly9}, Sections 1.3.2 - 1.3.24). 

 Recall the following. If $Y\in\PreStk_{lft}$ then by a placid group (ind)-scheme over $Y$ we mean a group object $G\to Y$ in $\PreStk_{/Y}$ such that for any $S\to Y$ with $S\in\Sch_{ft}$, $G\times_Y S$ is a placid group (ind)-scheme over $S$. 
 
 If $Y\in\PreStk_{lft}$ and $G\to Y$ is a placid group ind-scheme over $Y$, say that $G$ is ind-pro-unipotent if for any $S\in \Sch_{ft}$, $G\times_Y S$ is ind-pro-unipotent. In other words, there is a small filtered category $I$, and a presentation $G\times_Y S\,\iso\,\colim_{i\in I} G_{S, i}$, where $G_{S,i}$ is a prounipotent group scheme over $S$, for $i\to j$ in $I$, $G_{S, i}\to G_{S, j}$ is a placid closed immersion and a homomorphism of group schemes over $S$.
 
\sssec{} Let $f: Y\to Z$ be a morphism of ind-schemes of ind-finite type, $G\to Z$ be a placid group ind-scheme over $Z$. As in (\cite{Ly9}, 1.3.12) for $S\to Z$ with $S\in\Sch_{ft}$ set $G_S=G\times_Z S$, $Y_S=Y\times_Z S$ and view $Shv(G_S)$ as an object of $Alg(Shv(Z)-mod)$. Assume $G$ acts on $Y$ over $Z$. Then set
$$
Shv(Y)^G=\lim_{(S\to Y)\in (\Sch_{ft}/Y)^{op}} Shv(Y_S)^{G_S},
$$
see (\cite{Ly4}, 0.0.42) for details. The category of invariants is defined in (\cite{Ly9}, 1.3.2) as 
$$
Shv(Y_S)^{G_S}=\Fun_{Shv(G_S)}(Shv(S), Shv(Y_S))\in Shv(S)-mod(\DGCat_{cont})
$$
 
\sssec{} 
\label{Sect_A.2.5}
 Let $f: Y\to Z$ be a morphism of ind-schemes of ind-finite type, $G\to Z$ be a placid group ind-scheme over $Z$. Assume $G$ acts on $Y$ over $Z$, and $G$ is ind-pro-unipotent. Let $s: Z\to Y$ be a section of $f$. The stabilizer $St_s$ of $s$ is defined as the fibred product $G\times_{Y\times Y} Y$, that is, by the equation $gs(\bar g)=s(\bar g)$ for $g\in G$. Here $\bar g\in Z$ is the projection of $g$, and the two maps $G\to Y$, $G\to Y$ are $g\mapsto gs(\bar g)$ and $g\mapsto s(\bar g)$.  Assume $St_s$ is a placid group scheme over $Z$. 
 
 Consider the quotient $G/St_s\to Z$ over $Z$ in the sense of stacks. We get a natural map $\bar f: G/St_s\to Y$ over $Z$. Assume that $\bar f$ is an isomorphism, so $G$ acts transitively on the fibres of $f$.
  
\begin{Lm} 
\label{Lm_A.2.6}
In the situation of Section~\ref{Sect_A.2.5} one has the following.\\
i) The composition
$$
Shv(Y)^G\toup{\oblv} Shv(Y)\toup{s^!} Shv(Z)
$$
is an equivalence. \\
ii) The functor $f^!: Shv(Z)\to Shv(Y)$ is fully faithful and takes values in the full subcategory $Shv(Y)^G\subset Shv(Y)$.
\end{Lm}
\begin{proof}
i) Let $S\in\Sch_{ft}$ with a map $S\to Z$. Making the base change by this map, we get a diagram $f_S: Y\times_Z S\to S$ and its section $s_S$. Let $G_S=G\times_Z S$. By (\cite{Ly9}, Lemma~1.3.20), the composition
$$
Shv(Y\times_Z S)^{G_S}\toup{\oblv} Shv(Y\times_Z S)\toup{s_S^!} Shv(S)
$$
is an equivalence. Passing to the limit over $(S\to Z)\in (\Sch_{ft}/Z)^{op}$, we conclude.

\smallskip\noindent
ii) We may assume $Z\in \Sch_{ft}$. Consider the stabilizor $St_s$ of $s$ in $G$ as in Section~\ref{Sect_A.2.5}. Pick a presentation $G\,\iso\, \colim_{j\in J} G^j$, where $G^j$ is a placid group scheme over $Z$ for $j\in J$, $J$ is a small filtered category, for $j\to j'$ in $J$ the map $G^j\to G^{j'}$ is a placid closed immersion and a homomorphism of group schemes over $Z$. Besides we assume $0\in J$ is an initial object, and $G^0=St_z$. 
 
 Write $G/G^0$ for the stack quotient, so $Y\,\iso\, G/G^0\,\iso\,\colim_{j\in J} G^j/G^0$. Write $f^j$ for the composition $G^j/G^0\hook{} G/G^0\toup{f} Z$. For each $j$, the functor $(f^j)^!: Shv(Z)\to Shv(G^j/G^0)$ is fully faithful, because $G^j$ is prounipotent. So, $f^!: Shv(Z)\to Shv(Y)$ is fully faithful
\end{proof}
 
\begin{Lm} 
\label{Lm_A.3.6_fully_faithful_functors}
In the situation of Lemma~\ref{Lm_A.2.6} assume in addition that $M$ is a placid group scheme acting on $Y, Z$ and $f$ is $M$-equivariant. Then the functor $f^!: Shv(Z)^M\to Shv(Y)^M$ is also fully faithful. 
\end{Lm}
\begin{proof}
As in Lemma~\ref{Lm_A.2.6}, the standard argument reduces our claim to the case when $Z\in\Sch_{ft}$, so we assume this. % We may also assume $M$ acts on $Z$ through its quotient group scheme of finite type $M_0$ such that $\Ker(M\to M_0)$ is prounipotent. 
 
  Recall that $f^!: Shv(Z)^M\to Shv(Y)^M$ is obtained by passing to the limit over $[n]\in \bfitDelta$ in the diagram
$$
\Fun_{e, cont}(Shv(M)^{\otimes n}, Shv(Z))\to \Fun_{e, cont}(Shv(M)^{\otimes n}, Shv(Y))
$$
For each $[n]\in \bfitDelta$ the latter functor is fully faithful by (\cite{Ly}, 9.2.64). So, passing to the limit we obtain a fully faithful functor by (\cite{Ly}, 2.2.17).  
\end{proof}

\ssec{t-structures}
\label{Sect_A.3}
 
\sssec{} Let $Y\to S$ be a morphism in $\Sch_{ft}$. Equip $Shv(Y)$ with the perverse t-structure. Let $G\to S$ be a placid group scheme over $S$ acting on $Y$ over $S$ through its quotient group scheme $G\to G_0$, where $G_0$ is smooth of finite type over $S$. Assume that $\Ker(G\to G_0)$ is a prounipotent group scheme over $S$. We equip $Shv(Y)^G$ with the perverse t-structure as follows. 

 Define $Shv(Y)^{G, \le 0}$ as the preimage of $Shv(Y)^{\le 0}$ under $\oblv: Shv(Y)^G\to Shv(Y)$. This is a t-structure by (\cite{HA}, 1.4.4.11). The latter functor $\oblv$ is t-exact. Indeed, $Shv(Y)^G\,\iso\, Shv(Y)^{G_0}$ by (\cite{Ly9}, 1.3.21) in such a way that $\oblv$ identifies with $q^!: Shv(Y/G_0)\to Shv(Y)$ up to a shift for $q: Y\to Y/G_0$. 
 
 \smallskip\noindent
 
{\bf Convention}: We identify $Shv(Y)^G$ with $Shv(Y/G_0)$ in such a way that $\oblv: Shv(Y)^G\to Shv(Y)$ identifies with $q^!$ for $q: Y\to Y/G_0$.  

\sssec{} Recall that $Y/G_0$ is duality adapted in the sense of (\cite{AGKRRV}, F.2.6). So, the Verdier duality gives an equivalence 
$$
\DD: (Shv(Y/G_0)^c)^{op}\,\iso\, (Shv(Y/G_0)^c)
$$ 
This in turn gives an equivalence $Shv(Y/G_0)^{\vee}\,\iso\, Shv(Y/G_0)$ such that the corresponding counit map $Shv(Y/G_0)\otimes Shv(Y/G_0)\to \Vect$ is 
$$
(F_1, F_2)\mapsto \RG_{\blacktriangle}(Y/G_0, F_1\otimes^! F_2), 
$$
where $\RG_{\blacktriangle}: Shv(Y/G_0)\to\Vect$ is the functor dual to $\Vect\to Shv(Y/G_0)$, $e\mapsto \omega_{Y/G_0}$, see (\cite{AGKRRV}, F.2 and \cite{AGKRRV2}, A.4). This counit does not depend on the choice of the above quotient $G\to G_0$, so yields a canonical functor $Shv(Y)^G\otimes Shv(Y)^G\to\Vect$. 

We have a canonical morphism
$$
\RG_{\blacktriangle}(Y/G_0, F_1\otimes^! F_2)\to \RG(Y/G_0, F_1\otimes^! F_2)
$$
If $F_1$ or $F_2$ lies in $Shv(Y/G_0)^c$ then the latter map is an isomorphism by (\cite{AGKRRV}, F.4.5). 
 
 Define $Shv(Y/G_0)^{constr}\subset Shv(Y/G_0)$ as the full subcategory of those objects whose restriction to $Y$ lies in $Shv(Y)^c$. By definition of the renormalized category of sheaves from (\cite{AGKRRV}, F.5.1), $Shv(Y/G_0)^{ren}=\Ind(Shv(Y/G_0)^{constr})$. This category does not depend on the choice of $G_0$ up to a canonical equivalence, so gives rise to the category $Shv(Y)^{G, ren}$.
 
 The Verdier duality extends to an equivalence 
$$
\DD: (Shv(Y/G_0)^{constr})^{op}\,\iso\, Shv(Y/G_0)^{constr},
$$ 
see (\cite{AGKRRV}, F.2.5). Passing to the ind-completions, we get 
$$
(Shv(Y/G_0)^{ren})^{\vee}\,\iso\, Shv(Y/G_0)^{ren}
$$ 
As in \select{loc.cit.}, one has an adjoint pair in $\DGCat_{cont}$
$$
\ren: Shv(Y/G_0)\leftrightarrows Shv(Y/G_0)^{ren}: \unren,
$$
where $ren$ is fully faithful. This adjoint pair does not depend on a choice of the above quotient $G\to G_0$, so gives rise to a canonical adjoint pair
$$
\ren: Shv(Y)^G\leftrightarrows Shv(Y)^{G, ren}: \unren.
$$

\sssec{} 
\label{Sect_A.4.2_on_t-structures}
If $Y\to Y'$ is a closed immersion in $(\Sch_{ft})_{/S}$, assume it is $G$-equivariant, where the $G$-action factors through some finite-dimensional quotient group scheme $G\to G_0$ as above. By functoriality, we get an adjoint pair $i_!: Shv(Y)^G\leftrightarrows Shv(Y')^G: i^!$, where both $i_!, i^!$ commute with the corresponding oblivion functors. In this case $i^!$ is left t-exact, and $i_!$ is t-exact and fully faithful.

 Indeed, this follows from the fact that $K\in Shv(Y)^G$ is connective (coconnective) iff $\oblv(K)\in Shv(Y)$ is connective (coconnective). 
   
Pick a quotient group scheme $G\to G_0$ as above. Since we are in the constructible context, the functor $i^!: Shv(Y'/G_0)\to Shv(Y/G_0)$ has a continuous right adjoint, so we get adjoint pairs 
$$
i_!: Shv(Y/G_0)^c\leftrightarrows Shv(Y'/G_0)^c: i^!
$$
and $i_!: Shv(Y/G_0)^{constr}\leftrightarrows Shv(Y'/G_0)^{constr}: i^!$. Under the above duality, the dual of $i_!: Shv(Y/G_0)\to Shv(Y'/G_0)$ identifies with $i^!: Shv(Y'/G_0)\to Shv(Y/G_0)$, and similarly for the renormalized version. We similarly get the adjoint pair 
$$
i_!: Shv(Y)^{G, ren}\leftrightarrows Shv(Y')^{G, ren}: i^!,
$$
where the left adjoint is fully faithful.
 
\sssec{} 
\label{Sect_A.5.4_right_now}
Let $S\in \Sch_{ft}$, let $Y\to S$ be an ind-scheme of ind-finite type over $S$ equipped with a $G$-action over $S$. We assume there is a presentation $Y\,\iso\, \colim_{i\in I} Y_i$ in $\PreStk_{lft}$ such that $I$ is small filtered, $Y_i$ is a scheme of finite type over $S$, for $i\to j$ in $I$ the map $Y_i\to Y_j$ is a closed immersion. Moreover, each $Y_i$ is $G$-stable and $G$-action on $Y_i$ factors through some quotient group scheme $G\to G_i$ such that $G_i$ is smooth of finite type over $S$, and 
$\Ker(G\to G_i)$ is a prounipotent group scheme over $S$. 

 In this case $Shv(Y)^G\,\iso \lim_{i\in I^{op}} Shv(Y_i)^G$ with respect to the $!$-inverse images. Passing to the left adjoint, we may also write $Shv(Y)^G\,\iso\,\colim_{i\in I} Shv(Y_i)^G$ with respect to the $!$-direct images. For each $i$ the $!$-extension $Shv(Y_i)^G\to Shv(Y)^G$ is fully faithful (by \cite{GaiDG}, Lemma 1.3.6). 
 We define $(Shv(Y)^G)^{\le 0}$ as the smallest full subcategory containing $(Shv(Y_i)^G)^{\le 0}$ for all $i$, closed under extensions and small colimits. By (\cite{HA}, 1.4.4.11), $(Shv(Y)^G)^{\le 0}$ is presentable and defines an accessible t-structure on $Shv(Y)^G$.  

 Note that $K\in Shv(Y)^G$ lies in $(Shv(Y)^G)^{>0}$ iff for any $i$ its $!$-restriction to $Y_i$ lies in $(Shv(Y_i)^G)^{>0}$. This shows that this t-structure is compatible with filtered colimits.
 
 Let us show that $\oblv: Shv(Y)^G\to Shv(Y)$ is t-exact. Indeed, it is right t-exact by construction. For $i\in I$ let $h_i: Y_i\to Y$ be the natural map. Let $K\in Shv(Y)^{G, >0}$. To see that $\oblv(K)\in Shv(Y)^{>0}$, it suffices to show that if $i\in I$ then $\oblv(h_i^!K)\,\iso\, h_i^!\oblv(K)\in Shv(Y_i)^{>0}$. Since 
$\oblv: Shv(Y_i)^G\to Shv(Y_i)$ is t-exact, it suffices to show that $h_i^!K\in Shv(Y_i)^{G, >0}$. Our claim follows from the fact that, since $(h_i)_!: Shv(Y_i)^G\to Shv(Y)^G$ is right t-exact, its right adjoint is left t-exact. 

\sssec{} The above self-duality $(Shv(Y_i)^G)^{\vee}\,\iso\, Shv(Y_i)^G$ for each $i$ yields a self-duality 
$$
(Shv(Y)^G)^{\vee}\,\iso\, Shv(Y)^G
$$ 
using (\cite{GaiDG}, Lemma~2.2.2) and Section~\ref{Sect_A.4.2_on_t-structures}. 

\sssec{} 
\label{Sect_A.5.4}
Define $Shv(Y)^{G, constr}$ as the full subcategory of those $K\in Shv(Y)^G$ such that $\oblv(K)\in Shv(Y)^c$. Then $Shv(Y)^{G, constr}\in\DGCat^{non-cocmpl}$. Set $Shv(Y)^{G, ren}=\Ind(Shv(Y)^{G, constr})$. 

 Now $Shv(Y)^{G, constr}$ acquires a unique t-structure such that both projections 
$$
Shv(Y)^c\gets Shv(Y)^{G, constr}\to Shv(Y)^G
$$ 
are t-exact. Now by (\cite{G}, ch. II.1, Lemma 1.2.4), $Shv(Y)^{G, ren}$ acquires a unique t-structure compatible with filtered colimits for which the natural map $Shv(Y)^{G, constr}\to Shv(Y)^{G, ren}$ is t-exact. 

 For each $i$ we have a full embedding $Shv(Y_i)^{G, constr}\subset Shv(Y)^{G, constr}$. In fact, 
$$
\mathop{\colim}\limits_{i\in I} Shv(Y_i)^{G, constr}\,\iso\, Shv(Y)^{G, constr},
$$ 
where the colimit is taken in $\DGCat^{non-cocmpl}$. Applying the functor $\Ind$ to the later equivalence, we obtain $Shv(Y)^{G, ren}\,\iso\,\mathop{\colim}\limits_{i\in I} Shv(Y_i)^{G, ren}$, where the colimit is taken in $\DGCat_{cont}$. 

 The self-dualities $(Shv(Y_i)^{G,ren})^{\vee}\,\iso\, Shv(Y_i)^{G, ren}$ yield in the colimit a canonical self-duality
$$
(Shv(Y)^{G,ren})^{\vee}\,\iso\, Shv(Y)^{G, ren} 
$$
It actually comes from the Verdier duality
$$
\DD: (Shv(Y)^{G, constr})^{op}\,\iso\, Shv(Y)^{G, constr}
$$
Passing to the colimit over $i\in I$ in the adjoint pair $\ren: Shv(Y_i)^G\leftrightarrows Shv(Y_i)^{G, ren}: \unren$, one gets the adjoint pair 
$$
\ren: Shv(Y)^G\leftrightarrows Shv(Y)^{G, ren}: \unren 
$$
in $\DGCat_{cont}$, where the left adjoint is fully faithful.

\ssec{About averaging functors} 
\label{Sect_A.5}

\sssec{} 
\label{Sect_A.5.1}
Let $Y$ be an ind-scheme of ind-finite type. Let $U, G$ be placid group schemes with $U$ prounipotent, $G$ prosmooth. Assume $G$ acts on $U$ by automorphisms, let $H=G\rtimes U$ be the semi-direct product. Assume $H$ acts on $Y$, and $Y\,\iso\,\colim_{i\in I} Y_i$, where $I$ is a small filtered category, if $i\in I$ then $Y_i\hook{} Y$ is a closed $H$-invariant subscheme of finite type, and for $i\to j$ in $I$ the map $Y_i\to Y_j$ is a closed immersion.

 Since we are in the constructible context, the functor $\oblv: Shv(Y)^H\to Shv(Y)^G$ admits a left adjoint $\Av^U_!: Shv(Y)^G\to Shv(Y)^H$. Since $U$ is prounipotent, $\oblv: Shv(Y)^H\to Shv(Y)^G$ is fully faithful.
 
 Assume in addition $Y'$ is another ind-scheme of ind-finite type with a $H$-action satisfying the same assumptions, and $f: Y\to Y'$ is a $H$-equivariant morphism, where $f$ is schematic of finite type.
 
 Then $f_*: Shv(Y)^G\to Shv(Y')^G$ admits a left adjoint 
\begin{equation} 
\label{functor_f^*_for_Sect_A.5.1}
f^*: Shv(Y')^G\to Shv(Y)^G
\end{equation}
Both these functors preserve the full subcategories of $H$-invariants, and give rise to an adjoint pair in $\DGCat_{cont}$
$$ 
f^*: Shv(Y')^H\leftrightarrows Shv(Y)^H: f_*
$$
Besides, the following diagram canonically commutes
\begin{equation}
\label{diag_for_Sect_A.5.1}
\begin{array}{ccc}
Shv(Y)^G & \getsup{f^*} & Shv(Y')^G\\
\downarrow\lefteqn{\scriptstyle \Av^U_!} &&\downarrow\lefteqn{\scriptstyle \Av^U_!}\\
Shv(Y)^H & \getsup{f^*} & Shv(Y')^H
\end{array}
\end{equation}

 The proof is left to a reader, let us only indicate a construction of (\ref{functor_f^*_for_Sect_A.5.1}). Pick a presentation $Y'\,\iso\, \colim_{j\in J} Y'_j$, where $J$ is small filtered $\infty$-category, $Y'_j\in\Sch_{ft}$, for $j\to j'$ in $J$ the map $Y'_j\to Y'_{j'}$ is a $G$-equivariant closed immersion. Let $Y_j=Y'_j\times_{Y'} Y$ for $j\in J$, so $\colim_{j\in J} Y_j\,\iso\, Y$ in $\PreStk_{lft}$. For each $j\in J$ the $G$-action on $Y'_j$ factors through a quotient group scheme of finite type. Let $f_j: Y_j\to Y'_j$ be the restriction of $f$. Then there is a left adjoint $f_j^*: Shv(Y'_j)^G\to Shv(Y_j)^G$ of $(f_j)_*: Shv(Y_j)^G\to Shv(Y'_j)^G$. Besides, $f_j^*$ are compatible with the transition maps given by $!$-extensions for $j\to j'$ in $J$, so that we may pass to the colimit $\colim_{j\in J} f_j^*$. The latter is the desired functor (\ref{functor_f^*_for_Sect_A.5.1}).
 
\sssec{} 
\label{Sect_A.5.2}
Assume now $\bar U=\colim_{j\in J} U_j$, where $J$ is a small filtered category, if $j\in J$ then $U_j$ is a placid prounipotent group scheme, and for $j\to j'$ in $J$ the map $U_j\to U_{j'}$ is a homomorphism of group schemes and a placid closed immersion.

 Let $G, Y$ be as in Section~\ref{Sect_A.5.1}. Assume $G$ acts by conjugation on each $U_j$ in a way compatible with closed immersions $U_j\to U_{j'}$ for $j\to j'$ in $J$. Write $H_j=G\rtimes U_j$ and $\bar H=G\rtimes \bar U$ for the corresponding semi-direct products. So, $\bar H$ is a placid ind-scheme. Assume $\bar H$ acts on $Y$. By Lemma~\ref{Lm_A.2.3},  $\oblv: Shv(Y)^{\bar H}\to Shv(Y)^G$ admits a left adjoint denoted $\Av^{\bar U}_!: Shv(Y)^G\to Shv(Y)^{\bar H}$ given as
$$
\Av^{\bar U}_!(K)\,\iso\, \mathop{\colim}\limits_{j\in J} \Av^{U_j}_!(K),
$$
the colimit being taken in $Shv(Y)^G$.

 Assume in addition $f: Y\to Y'$ is a schematic morphism of finite type, where $Y'$ is an ind-scheme of ind-finite type. Assume $\bar H$ acts on $Y'$, and $f$ is $\bar H$-equivariant. The functors $f^*: Shv(Y')^{H_j}\to Shv(Y)^{H_j}$ yield after passing to the limit over $J^{op}$ the functor $f^*: Shv(Y')^{\bar H}\to Shv(Y)^{\bar H}$. The commutativity of (\ref{diag_for_Sect_A.5.1}) implies that the diagram is canonically commutative
$$
\begin{array}{ccc}
Shv(Y)^G & \getsup{f^*} & Shv(Y')^G\\
\downarrow\lefteqn{\scriptstyle \Av^{\bar U}_!} &&\downarrow\lefteqn{\scriptstyle \Av^{\bar U}_!}\\
Shv(Y)^{\bar H} & \getsup{f^*} & Shv(Y')^{\bar H}
\end{array} 
$$

\begin{Lm} 
\label{Lm_A.5.3}
Keep the assumptions of Section~\ref{Sect_A.5.1}. \\
i) The functor $f^!: Shv(Y')^H\to Shv(Y)^H$ admits a left adjoint $f_!: Shv(Y)^H\to Shv(Y')^H$, and the diagram commutes naturally
\begin{equation}
\label{diag_for_Lm_A.5.3}
\begin{array}{ccc}
Shv(Y) & \toup{f_!} & Shv(Y')\\
\uparrow\lefteqn{\scriptstyle\oblv} && \uparrow\lefteqn{\scriptstyle\oblv} \\
Shv(Y)^H & \toup{f_!} & Shv(Y')^H
\end{array}
\end{equation}
ii) Assume in addition that $f: Y\to Y'$ is an open immersion. Then $f_!: Shv(Y)^H\to Shv(Y')^H$ is fully faithful. 
\end{Lm}
\begin{proof}
i) {\bf Step 1} Assume first $Y, Y'$ are schemes of finite type. Then the $H$-action on $Y$ automatically factors through an action of a quotient group scheme $H\to H_0$, where $H_0$ is smooth of finite type, and $\Ker(H\to H_0)$ is prounipotent. We get the cartesian square
$$
\begin{array}{ccc}
Y &\toup{f} & Y'\\
\downarrow && \downarrow\\
Y/H_0 & \toup{\bar f} & Y'/H_0
\end{array}
$$
and the desired functor is $\bar f_!: Shv(Y)^H\to Shv(Y')^H$. The commutativity of (\ref{diag_for_Lm_A.5.3}) follows from the $(^*, {}_!)$-base change.

\smallskip\noindent
{\bf Step 2} Pick a presentation $Y'\,\iso\, \colim_{i\in I} Y'_i$, where $I$ is a small filtered $\infty$-category, $Y'_i\in\Sch_{ft}$ is $H$-invariant closed subscheme of $Y'$, and for $i\to j$ in $I$, $Y'_i\to Y'_j$ is a $H$-equivariant closed immersion. Set $Y_i=Y'_i\times_{Y'} Y$. Note that $Y\,\iso\,\colim_{i\in I} Y_i$ in $\PreStk_{lft}$. For $i\to j$ in $I$ write $g_{ij}: Y_i\to Y_j$ and $g'_{ij}: Y'_i\to Y'_j$ for the transition maps. Write $f_i: Y_i\to Y'_i$ for the restriction of $f$. 

 Recall that $Shv(Y)^H\,\iso\,\lim_{i\in I^{op}} Shv(Y_i)^H$ with the transition maps being $g_{ij}^!$. Passing to the left adjoints, we get $Shv(Y)^H\,\iso\, \colim_{i\in I} Shv(Y_i)^H$ in $\DGCat_{cont}$ for the transition maps $(g_{ij})_!$. 

We get a functor $\cF: I\times [1]\to \DGCat_{cont}$ sending $i\in I$ to $(Shv(Y_i)^H\toup{(f_i)_!} Shv(Y'_i)^H)$, here for $i\to j$ in $I$ the transition functors are $(g_{ij})_!, (g'_{ij})_!$. Passing to the colimit in $(f_i)_!: Shv(Y_i)^H\to Shv(Y'_i)^H$, one gets the desired functor $f_!: Shv(Y)^H\to Shv(Y')^H$. This functor is the left adjoint to $f^!: Shv(Y')^H\to Shv(Y)^H$ by (\cite{Ly}, 9.2.39). The commutativity of (\ref{diag_for_Lm_A.5.3}) is obtained by passing to the colimit over $i\in I$ from the commutativity of
$$
\begin{array}{ccc}
Shv(Y_i) & \toup{(f_i)!} & Shv(Y'_i)\\
\uparrow\lefteqn{\scriptstyle\oblv} && \uparrow\lefteqn{\scriptstyle\oblv} \\
Shv(Y_i)^H & \toup{(f_i)!} & Shv(Y'_i)^H
\end{array}
$$

\smallskip\noindent
ii) Keep notations of Step 2. Then for each $i\in I$, 
$$
(f_i)_!: Shv(Y_i)^H\to Shv(Y'_i)^H
$$ 
is fully faithful by coinstruction. Besides, each $Shv(Y_i)^H, Shv(Y'_i)^H$ is compactly generated, and we may pass to right adjoints in the functor $\cF$. So, our claim follows from (\cite{Ly}, 9.2.47). 
\end{proof}

\begin{Rem} Actually, in the situation of Lemma~\ref{Lm_A.5.3} i) the diagram commutes
$$
\begin{array}{ccc}
Shv(Y)^G & \toup{f_!} & Shv(Y')^G\\
\uparrow\lefteqn{\scriptstyle\oblv} && \uparrow\lefteqn{\scriptstyle\oblv} \\
Shv(Y)^H & \toup{f_!} & Shv(Y')^H,
\end{array}
$$
and the vertical functors are fully faithful. 
\end{Rem}

\sssec{} 
\label{Sect_A.5.5}
Assume we are in the situation of Section~\ref{Sect_A.5.2}. For each $j\in J$ we have the fully functor $\oblv: Shv(Y)^{H_i}\to Shv(Y)^G$. Passing to the limit over $j\in J^{op}$, they yield a fully faithful functor $\oblv: Shv(Y)^{\bar H}\to Shv(Y)^G$. Since for each $j\in J$ the diagram 
$$
\begin{array}{ccc}
Shv(Y)^G & \toup{f_!} & Shv(Y')^G\\
\uparrow\lefteqn{\scriptstyle \oblv} && \uparrow\lefteqn{\scriptstyle \oblv}\\
Shv(Y)^{H_j} & \toup{f_!} & Shv(Y')^{H_j}
\end{array}
$$
commutes, passing to the limit over $j\in J$ this gives a commutativity of
$$
\begin{array}{ccc}
Shv(Y)^G & \toup{f_!} & Shv(Y')^G\\
\uparrow\lefteqn{\scriptstyle \oblv} && \uparrow\lefteqn{\scriptstyle \oblv}\\
Shv(Y)^{\bar H} & \toup{f_!} & Shv(Y')^{\bar H}
\end{array}
$$

\ssec{Some intertwining functors}
\label{Sect_A.7}

\sssec{} 
\label{Sect_A.7.1}
Let $G$ be a smooth algebraic group of finite type. Let $Y\in \PreStk_{lft}$ with a $G$-action. We take the convention that the natural identification $Shv(Y/G)\,\iso\, Shv(Y)^G$ is such that $\oblv:  Shv(Y)^G\to Shv(Y)$ corresponds to $q^!: Shv(Y/G)\to Shv(Y)$ for $q: Y\to Y/G$. 

\sssec{} 
\label{Sect_A.7.2}
Let $P,Q\subset G$ be closed subgroups. We define the convolution $Shv(Q\backslash G/P)\otimes Shv(Y/P)\to Shv(Y/Q)$ by 
$$
F\boxtimes K\mapsto F\ast K=\act_*q^!(F\boxtimes K)
$$
for the diagram 
$$
(Q\backslash G/P)\times (Y/P)\getsup{q} Q\backslash G\times^P Y\toup{\act} Y/Q.
$$ 
In particular, this is known to be the underlying binary product of a monoidal structure on $Shv(P\backslash G/P)$. 

\sssec{} Let now $C\in Shv(G)-mod(\DGCat_{cont})$. We use the identification
$$
\Fun_{Shv(G)}(Shv(G/Q), C)\,\iso\, C^Q
$$
coming from $Shv(G/Q)\,\iso\, Shv(G)^Q\,\iso\, Shv(G)_Q$, where the first isomorphism is as in Section~\ref{Sect_A.7.1}. 

 We write $\delta_1\in Shv(G/Q)$ for the constant sheaf at $1$ extended by zero to $G/Q$.
 
\begin{Lm} 
\label{Lm_2.0.5_inverse_equivalence}
Consider the equivalence $\Fun_{Shv(G)}(Shv(G/Q), Shv(Y))\,\iso\, Shv(Y/Q)$ given by $f\mapsto f(\delta_1)$. The inverse equivalence sends $K\in Shv(Y/Q)$ to the functor $Shv(G/Q)\to Shv(Y)$ given by $F\mapsto m_*q^!(F\boxtimes K)$ for the diagram 
$$
G/Q\times (Y/Q)\getsup{q} G\times^Q Y\toup{m} Y,
$$ 
where $m$ comes from the action map $\act: G\times Y\to Y$. So, $q\in Q$ acts on $(g, y)\in G\times Y$ as $(gq^{-1}, qy)$. 
\end{Lm}
\begin{proof}
We have the canonical equivalence $Shv(G)\otimes_{Shv(Q)}\Vect\to Shv(G/Q)$ sending $F\boxtimes e$ to $\alpha_*F$, where $\alpha: G\to G/Q$ is the natural map. In the following commutative diagram the square is cartesian
$$
\begin{array}{ccccc}
G\times (Y/Q) & \getsup{\id\times\beta} & G\times Y\\
\downarrow\lefteqn{\scriptstyle \alpha\times\id} && \downarrow\lefteqn{\scriptstyle\bar\alpha} & \searrow\lefteqn{\scriptstyle \act}\\
(G/Q)\times (Y/Q) & \getsup{q} & G\times^Q Y  & \toup{m} & Y
\end{array}
$$
So, for $F\in Shv(G)$, $K\in Shv(Y/Q)$ one has canonically $q^!((\alpha_* F)\boxtimes K)\,\iso\, \bar\alpha_*(F\boxtimes \beta^! K)$, hence also
$$
F\ast \beta^! K\,\iso\, m_*q^!((\alpha_* F)\boxtimes K)
$$
Since the objects $F\boxtimes V$ for $F\in Shv(G), V\in\Vect$ generate $Shv(G)\otimes_{Shv(Q)}\Vect$, our claim follows.
\end{proof} 

\sssec{} 
\label{Sect_A.7.5}
Let $\upsilon: Shv(Q\backslash G/P)\,\iso\,Shv(P\backslash G/Q)$ be the equivalence coming from the map $G\,\iso\, G$, $g\mapsto g^{-1}$. Let $\cK\in Shv(Q\backslash G/P)$.
 
 Denote by 
\begin{equation}
\label{functor_for_A.7.5}
 \cK\ast\_: C^P\to C^Q
\end{equation} 
and also by $\_\ast \upsilon(\cK)=\cK\ast\_$ the functor obtained from the $Shv(G)$-linear functor
\begin{equation}
\label{functor_for_Sect_A.7.5}
\_\ast \cK: Shv(G/Q)\to Shv(G/P)
\end{equation}
by applying $\Fun_{Shv(G)}(\_, C)$. In the case of $C=Shv(Y)$ this is unambiguous thanks to the following.

\begin{Lm} For $C=Shv(Y)$ the functor (\ref{functor_for_A.7.5})
identifies with the convolution functor $ \cK\ast\_$ from Section~\ref{Sect_A.7.2}.
\end{Lm} 
\begin{proof}
This follows from Lemma~\ref{Lm_2.0.5_inverse_equivalence}. Namely, let $\tau: G/P\to Q\backslash G/P$ be the natural map. By definition, (\ref{functor_for_A.7.5}) sends $K\in Shv(Y/P)$ to the object of $Shv(Y/Q)$ whose $!$-pullback to $Y$ is $m_*q^!(\tau^!\cK\boxtimes K)$ for the diagram
$$
\begin{array}{ccccc}
(G/P)\times (Y/P) & \getsup{q} & G\times^P Y & \toup{m} & Y\\
\downarrow\lefteqn{\scriptstyle \tau\times\id} && \downarrow &&\downarrow\lefteqn{\scriptstyle\beta}\\
(Q\backslash G/P)\times (Y/P) & \getsup{\bar q} & Q\backslash G\times^P Y & \toup{\bar m} & Q\backslash Y,
\end{array}
$$
where both squares are cartesian. Our claim follows by base change.
\end{proof} 

Sometimes, we denote (\ref{functor_for_A.7.5}) by $\cK\astP\_$ to underline that the convolution is calculated with respect to $P$.

\sssec{} Assume for this subsection that $Q\subset P$ and $P/Q$ is smooth. For the natural map $\alpha: G/Q\to G/P$ we have the adjoint pair 
$$
\alpha^*: Shv(G/P)\leftrightarrows Shv(G/Q): \alpha_*
$$ 
in $Shv(G)-mod(\DGCat_{cont})$. Applying $\Fun_{Shv(G)}(\_, C)$, it gives an adjoint pair $\oblv: C^P \leftrightarrows C^Q: \Av^{P/Q}_*$ in $\DGCat_{cont}$.  The functor $\alpha^*$ identifies with 
$$
\_\ast i_!e_{P\backslash P/Q}[-2\dim P]: Shv(G/P)\to Shv(G/Q)
$$ 
for the closed immersion $i: P\backslash P/Q\hook{} P\backslash G/Q$. So, $\Av^{P/Q}_*: C^Q\to C^P$ in our notations is the functor
$$
i_!e_{P\backslash P/Q}[-2\dim P]\ast\_.
$$

 The functor $\alpha_*$ identifies with
$$
\_\ast s_!e_{Q\backslash P/P}[-2\dim Q]: Shv(G/Q)\to Shv(G/P)
$$
for the closed immersion $s: Q\backslash P/P\hook{} Q\backslash G/P$. So, $\oblv: C^P\to C^Q$ in our notations is the functor
$$
s_!e_{Q\backslash P/P}[-2\dim Q]\ast\_.
$$ 

 The functor $\alpha^*$ has a left adjoint $\alpha_![2\dim(P/Q)]$. If $P/Q$ is proper then $\alpha$ is proper, and we get an adjoint pair 
$$
\alpha_*[2\dim(P/Q)]: Shv(G/Q)\leftrightarrows Shv(G/P): \alpha^*
$$
in $Shv(G)-mod(\DGCat_{cont})$. It yields an adjoint pair in $\DGCat_{cont}$
$$
\Av^{P/Q}_*: C^Q\leftrightarrows C^P: \oblv[2\dim(P/Q)]
$$  

\sssec{} Let now $P, Q\subset G$ be as in Section~\ref{Sect_A.7.2}. Define the functor $^Q\Av_*^P: C^P\to C^Q$ as the composition
$$
C^P\,\toup{\oblv}\, C^{P\cap Q}\,\toup{\Av_*^{Q/P\cap Q}}\, C^Q
$$
 
 Write $j: Q\backslash QP/P\to Q\backslash G/P$ for the natural inclusion. Then in our notations $^Q\Av^P_*$ is the functor $j_*e_{Q\backslash QP/P}[-2\dim Q]\ast\_$. 
 
 Assume in addition that $Q/(P\cap Q)$ is proper. Then $^Q\Av^P_*$ admits a right adjoint given as the composition 
$$
C^Q\,\toup{\oblv[2\dim(Q/P\cap Q)]}\, C^{P\cap Q}\;\toup{\Av^{P/(P\cap Q)}_*} \,C^P
$$ 
That is, $^P\Av^Q_*[2\dim(Q/P\cap Q)]$ is the right adjoint of $^Q\Av^P_*$. Let $j': P\backslash PQ/Q\to P\backslash G/Q$ be the embedding. Then 
$^P\Av^Q_*[2\dim(Q/(P\cap Q))]$ identifies with the functor 
$$
j'_*e_{P\backslash PQ/Q}[-2\dim P+2\dim(Q/(P\cap Q))]\ast\_.
$$

\sssec{} 
\label{Sect_A.7.9}
Let now $P, Q\subset G$ be as in Section~\ref{Sect_A.7.2}. Assume $G/P$ proper. Write $j': P\backslash PQ/Q\to P\backslash G/Q$ for the embedding. 
Recall that the functor $^Q\Av^P_*: C^P\to C^Q$ comes from $\alpha_*\beta^*: Shv(G/Q)\to Shv(G/P)$ for the diagram 
$$
G/P\getsup{\alpha} G/(P\cap Q)\toup{\beta} G/Q
$$ 
Assume $Q/(P\cap Q), P/(P\cap Q)$ smooth. The left adjoint to $\alpha_*\beta^*$ is 
$$
\beta_!\alpha^*[2\dim(Q/(P\cap Q)]: Shv(G/P)\to Shv(G/Q)
$$ 

 We claim that the latter functor is $Shv(G)$-linear and identifies with 
\begin{equation}
\label{functor_for_Sect_A.7.9}
\_\ast j'_!e[-2\dim P+2\dim(Q/(P\cap Q))]: Shv(G/P)\to Shv(G/Q) 
\end{equation}
Indeed, consider the diagram, where the square is cartesian
$$
\begin{array}{ccccc}
(G/P)\times (P\backslash G/Q) & \getsup{q} & G\times^P G/Q & \toup{m} & G/Q\\
\uparrow\lefteqn{\scriptstyle \id\times j'} && \uparrow\lefteqn{\scriptstyle \tilde j} & \nearrow\lefteqn{\scriptstyle \tilde m}\\
(G/P)\times (P\backslash PQ/Q) & \getsup{\tilde q} & G\times^P PQ/Q
\end{array}
$$
Since $G/P$ is proper, $m$ is proper, so for $F\in Shv(G/P)$, 
$$
F\ast j'_!e\,\iso\,
m_*q^!(F\boxtimes j'_!e)\,\iso\, m_! \tilde j_!\tilde q^!(F\boxtimes e)\,\iso\,\beta_!\alpha^*F[2\dim P],
$$
because $q$ is smooth, and $\tilde m$ identifies with $\beta$. 

 So, $^Q\Av^P_*$ admits a right adjoint denoted $^P\Av^Q_!$
 obtained from (\ref{functor_for_Sect_A.7.9}) by applying $\Fun_{Shv(G)}(\_, C)$. 
In general, $PQ$ is not closed in $G$, so this functor is, in this generality, distinct from $^P\Av^Q_*$.

%\sssec{} In the situation of Section~\ref{Sect_A.7.5} assume $\cK\in Shv(Q\backslash G/P)^{constr}$. Then the functor (\ref{functor_for_Sect_A.7.5}) admits a left adjoint 
%$$
%Shv(G/P)\to Shv(G/Q), \, L\mapsto (p_1)_!(m^*L\otimes p_2^*\DD\cK)
%$$ 
%for the diagram
%$$
%\begin{array}{ccccc}
%G/Q & \getsup{p_1} & G\times^Q G/P & \toup{m} & G/P\\
%&&\downarrow\lefteqn{\scriptstyle p_2}\\
%&& Q\backslash G/P
%\end{array}
%$$
%Here $m$ comes from the product map $G\times G\to G$, and $p_i$ comes from the projection $G\times G\to G$ sending $(g_1, g_2)$ to $g_i$. 

\section{On the invertibility of some standard objects in parabolic Hecke categories}
\label{Sect_appendixB}

\ssec{Associated parabolic subgroups} 

\sssec{} Let $T\subset B\subset G$ and $W$ be as in Section~\ref{Sect_1.4.1}. Let $P, Q\subset G$ be parabolics containing $T$. 
\begin{Rem} Any pair of parabolics in $G$ contain a common maximal torus. We fix this torus to be $T$ to simplify some notations.
\end{Rem}

\sssec{} Write $L_P\subset P, L_Q\subset Q$ for the unique Levi subgroups containing $T$. Write $W_P, W_Q\subset W$ for the Weyl groups of $L_P, L_Q$. Write also $L_{P\cap Q}$ for the unique Levi subgroup of $P\cap Q$ containing $T$. 

  Let $C\in Shv(G)-mod(\DGCat_{cont})$. In Section~\ref{Sect_A.7.9} we introduced the adjoint pair
\begin{equation}
\label{adj_functors_for_B1}
^Q\Av^P_*: C^P\leftrightarrows C^Q: {^P\Av^Q_!}
\end{equation}  
Our goal here is two determine for which pairs $(P, Q)$ as above these functors are equivalences.

\begin{Def} Say that $P$ and $Q$ are associated if we have $L_P=L_Q$. 
\end{Def}

 Note that $P$ and $Q$ are associated iff $L_P=L_{P\cap Q}=L_Q$.
 
\sssec{Example} The opposite parabolics are associated.
 
\begin{Thm}
\label{Thm_B.1.2}
The adjoint functors (\ref{adj_functors_for_B1}) are equivalences (for any $C\in Shv(G)-mod(\DGCat_{cont})$) if and only if $P$ and $Q$ are associated.
\end{Thm} 

\begin{Rem}
We note that the answer given by Theorem~\ref{Thm_B.1.2} differs from its function-theoretic counterpart. For parabolic Hecke algebras it is true that if $P$ and $Q$ are associated then  the indicator function $T_{PQ}$ of $TP\subset G$ is invertible. This follows by taking the trace of Frobenius. 

 However, typically there are more invertible elements. For example, consider $G=GL(V)$ for a finite-dimensional vector space $V$ with $\dim V\ge 3$. Let $P\subset G$ be the parabolic preserving a line $L\subset V$, so $G/P\,\iso\, \PP(V)$. There are two $P$-orbits on $\PP(V)$, namely $\{L\}$ and its complement. Let $w\in W$ such that $PwP/P\subset G/P$ is open. While $P$ and $wPw^{-1}$ are not associated, the indicator  function of the double coset $PwP$ is invertible. The parabolic Hecke algebra here is the usual Hecke algebra for $\GL_2$ with parameter $q+q^2+\ldots+q^{\dim \PP(V)}$ (if we work over a finite field of $q$ elements).
\end{Rem}
 
\begin{Lm} 
\label{Lm_B.1.5}
The parabolics $P$ and $Q$ are associated if and only if both $P/(P\cap Q)$ and $Q/(P\cap Q)$ are homologically contractible.
\end{Lm}
\begin{proof}
The only if direction is obvious.

Assume both $P/(P\cap Q)$ and $Q/(P\cap Q)$ are contractible. Note that $P/(P\cap Q)$ deformation retracts onto $L_P/L_P\cap Q$. Further, $L_P\cap Q$ is a parabolic of $L_P$. Indeed, if $B'\subset Q$ is a Borel subgroup containing $T$ then $L_P\cap B'$ is a Borel subgroup of $L_P$. So, $L_P/(L_P\cap Q)$ is a partial flag variety of $L_P$. In particular, it is homologically contractible iff $L_P\subset Q$, that is, $L_P\subset L_Q$.
Interchanging the roles of $P$ and $Q$ we get also $L_Q\subset L_P$.
\end{proof}

\begin{proof}[Proof of Theorem~\ref{Thm_B.1.2}]
{\bf Step 1} Assume $P$ and $Q$ are associated. Pick Borel subgroups $B\subset P$, $B'\subset Q$ containing $T$. Let us show that the diagram canonically commutes
$$
\begin{array}{ccc}
C^P & \toup{^Q\Av^P_*} & C^Q\\
\downarrow\lefteqn{\scriptstyle \oblv} && \downarrow\lefteqn{\scriptstyle\oblv}\\
C^B & \toup{^{B'}\Av^B_*} & C^{B'}
\end{array}
$$

 Let 
$$
j: Q\backslash QP/P\to Q\backslash G/P, \;\; j': B'\backslash B'B/B\to B'\backslash G/B
$$ 
and 
$$
\pi: G/B\to G/P, \;\;\pi': G/B'\to G/Q
$$ 
be the natural maps. By Section~\ref{Sect_A.7}, it suffices to show that the diagram commutes
$$
\begin{array}{ccc}
Shv(G/Q) & \toup{\_\ast j_*e_{Q\backslash QP/P}[d]} & Shv(G/P)\\
\uparrow\lefteqn{\scriptstyle  \pi'_*}&& \uparrow\lefteqn{\scriptstyle  \pi_*}\\
Shv(G/B') & \toup{\_\ast j'_*e_{B'\backslash B'B/B}} & Shv(G/B)
\end{array}
$$
for $d=2\dim B'-2\dim Q$. Consider the closed immersions 
$$
s: B\backslash P/P\to B\backslash G/P,\;\; s': B'\backslash Q/Q\to B'\backslash G/Q
$$ 
The functor $\pi_*$ is $\_\ast s_!e_{B\backslash P/P}[-2\dim B]$. The functor $\pi'_*$ is $\_\ast s'_!e_{B'\backslash Q/Q}[-2\dim B']$. 
So, we must establish an isomorphism
\begin{equation}
\label{iso_to_prove_for_ThmB.1.3}
s'_!e_{B'\backslash Q/Q}\ast j_*e_{Q\backslash QP/P}[-2\dim Q]\,\iso\,  j'_*e_{B'\backslash B'B/B}\ast s_!e_{B\backslash P/P}[-2\dim B]
\end{equation}
in $Shv(B'\backslash G/P)$. Let 
$$
\bar j: B'\backslash QP/P\to B'\backslash G/P
$$ 
be the natural map. By base change, the LHS of (\ref{iso_to_prove_for_ThmB.1.3}) identifies with $\bar j_*e$. Let 
$$
m: B'\backslash B'B\times^B P/P\to B'\backslash QP/P
$$ 
be the map induced by the product map $B'B\times^B P\to Q\times P$. The RHS of (\ref{iso_to_prove_for_ThmB.1.3}) identifies by base change with $\bar j_* m_*e$. In fact, $m_*e\,\iso\, e$. Indeed, consider the diagram
$$
B'B/B\to B'P/P\to QP/P
$$
In this diagram the second map is an isomorphism, because $B'\cap L_P\subset L_Q$, and the first map identifies with the affine fibration $U(B')/U(B')\cap U(B)\to U(B')/U(B')\cap U(P)$. Here $U(B), U(B'), U(P)$ denotes the unipotent radical of the corresponding group. Our claim follows.

 A dual argument shows that the diagram canonically commutes
$$
\begin{array}{ccc}
C^Q & \toup{^P\Av^Q_!} & C^P\\
\downarrow\lefteqn{\scriptstyle \oblv} && \downarrow\lefteqn{\scriptstyle\oblv}\\
C^{B'} & \toup{^B\Av^{B'}_!} & C^{B}
\end{array}
$$
Since the functors $\oblv: C^P\to C^B$ and $\oblv: C^{Q}\to C^{B'}$ are conservative, our claim follows from the fact that the adjoint functors 
$$
^{B'}\Av^B_*: C^B\leftrightarrows C^{B'}: {^B\Av^{B'}_!}
$$
are equivalences, which is standard.

\medskip\noindent
{\bf Step 2} Assume $^Q\Av^P_*$, $^P\Av^Q_!$ are equivalences. Let $\tilde j: P\backslash PQ/Q\hook{} P\backslash G/Q$ be the natural map. We have
$$
j_*e_{Q\backslash QP/P}[-2\dim Q]\astP \tilde j_! e_{P\backslash PQ/Q}
[-2\dim P+2\dim(Q/(P\cap Q))]\,\iso\,i_!\omega_{Q\backslash Q/Q},
$$
where $i: Q\backslash Q/Q\hook{} Q\backslash G/Q$ is the closed immersion. 
Let $\inv: Q\backslash QP/P\,\iso\, P\backslash PQ/Q$ be the inversion. We have the cartesian square
$$
\begin{array}{ccc}
Q\backslash (G\times^P G)/Q & \toup{m} & Q\backslash G/Q\\
\uparrow{\lefteqn{\scriptstyle\tilde i}} && \uparrow{\lefteqn{\scriptstyle i}}\\
Q\backslash G/P & \toup{\tilde m} & Q\backslash Q/Q
\end{array}
$$
So, 
\begin{multline*}
i^!(j_*e_{Q\backslash QP/P}\astP \tilde j_! e_{P\backslash PQ/Q})\,\iso\, \tilde m_*(j_*e_{Q\backslash QP/P}\otimes^! \inv^!(\tilde j_! e_{P\backslash PQ/Q})))\,\iso\\  \tilde m_*(j_*e_{Q\backslash QP/P}\otimes^! j_!e_{Q\backslash QP/P})\,\iso\,  \tilde m_*j_*e_{Q\backslash QP/P}[2\dim(Q\cap P)]\\ \iso\,  \omega_{Q\backslash Q/Q}[2\dim P+2\dim(P\cap Q)]
\end{multline*}
For the map $\eta: \Spec k\to Q\backslash Q/Q$ applying $\eta^!$
this gives
$$
e[2\dim P-2\dim Q]\,\iso\,\RG(QP/P, e)
$$
Since $\H^0(QP/P, e)\,\iso\, e$, this shows that $\dim P=\dim Q$, and $QP/P$ is homologically contractible. 

 Reversing the roles of $P$ and $Q$, one similarly shows that $PQ/P$ is homologically contractible. Our claim follows now from Lemma~\ref{Lm_B.1.5}.
\end{proof}

\begin{Rem} Our proof of Theorem~\ref{Thm_B.1.2} also shows that if $P$ and $Q$ are associated then $\dim P=\dim Q$.
\end{Rem}

\section{Category of lax central objects}
\label{Sect_lax-appendix}

\sssec{} Let $A\in Alg(1-\Cat)$. Let $C\in 1-\Cat$ be equipped with commuting left and right lax actions of $A$. Under this assumption, in this section we define the category $\Cent^{lax,\ra}_A(C)$ of lax central objects in $C$ for these actions. Informal definition of this category has first appeared in (\cite{Gai19SI}, 2.7).

\sssec{} Given $M, N\in A-mod(1-\Cat)$, one has the category $\Fun_{A-rlax}(M, N)$ of right lax functors $M\to N$ of $A$-module categories. For convenience, recall its definition given in (\cite{G}, ch. I.1, 3.5.1). Let $\bfitDelta^+\subset 1-\Cat$ be the full subcategory defined in (\cite{G}, ch. I.1, 3.4.1), its objects are categories $[n]$ and $[n]^+$ for $n\ge 0$. One has an equivalence $\bfitDelta^+\,\iso\, \bfitDelta\times[1]$ such that $\bfitDelta\subset\bfitDelta^+$ identifies with $\bfitDelta\times \{0\}$. Under this equivalence $[n]$ (resp., $[n]^+$) corresponds to $([n], 0)$ (resp., to $([n], 1)$). 

Let $h_M: X_M\to \bfitDelta^{+, op}$ be the cocartesian fibration classifying a pair $(A, M)$, and similarly for $h_N: X_N\to  \bfitDelta^{+, op}$. In particular, $X_M\mid_{\bfitDelta^{op}}\to \bfitDelta^{op}$ is the cocartesian fibration attached to $A$. Then $\Fun_{A-rlax}(M, N)$ is the category of functors $e: X_M\to X_N$ over $\bfitDelta^{+, op}$, whose restriction to $\bfitDelta^{op}$ is identified with the identity, and which satisfy the following property: if $\rho$ is a morphism in $\bfitDelta^+$ of the form $[n]\to [n]^+, i\mapsto i$ or $[0]^+\to [n]^+, 0\mapsto n, +\mapsto +$ then $e$ sends a $h_M$-cocartesian arrow to a $h_N$-cocartesian arrow.

 So, for $f\in \Fun_{A-rlax}(M, N)$ we have a map functorial in $a\in A, m\in M$
\begin{equation}
\label{map_lax_for_f_for_A-modules}
af(m)\to f(am)
\end{equation}
Note that $\Fun_A(M, N)\subset \Fun_{A-rlax}(M, N)$ is the full subcategory of those functors, which are strictly $A$-linear. 

\begin{Rem} If $A\in Alg(\DGCat_{cont})$ is rigid, and $M,N\in A-mod(\DGCat_{cont})$ then $\Fun_A(M, N)\,\iso\,\Fun_{A-rlax}(M, N)$ by (\cite{G}, ch. I.1, 9.3.6).
\end{Rem}

\ssec{Case of strict actions}

\sssec{} Write $A^{rm}$ for $A$ with reversed multiplication. First, assume that both actions of $A$ on $C$ are strict, so $C\in (A\times A^{rm})-mod(1-\Cat)$. In other words, $C$ is a $A$-bimodule category. For $a\in A, c\in C$ we write $a\ast c, c\ast a\in C$ for the corresponding actions.

 In this case the category $\Fun_{A\times A^{rm}}(A, C)$ of $A\times A^{rm}$-linear functors plays the role of the \select{category of central objects} for these actions. Namely, $f\in \Fun_{A\times A^{rm}}(A, C)$ can be seen as the object $c:=f(1)\in C$ together with a collection of functorial isomorphisms 
$$
\gamma_a: a\ast c\,\iso\, c\ast a
$$
in $C$ for $a\in A$ together with a coherent system of higher compatibilities. In particular, for $a_1, a_2\in A$ we are given a commutativity datum for the diagram
\begin{equation}
\label{diag_for_lax_central_object}
\begin{array}{ccc}
a_1\ast(a_2\ast c) & \toup{\gamma_{a_2}} & a_1\ast (c\ast a_2)\\
\downarrow\lefteqn{\scriptstyle \wr} && \downarrow\lefteqn{\scriptstyle \wr}\\
(a_1a_2)\ast c && (a_1\ast c)\ast a_2\\
\downarrow\lefteqn{\scriptstyle\gamma_{a_1a_2}} && \downarrow\lefteqn{\scriptstyle\gamma_{a_1}}\\
c\ast (a_1a_2) & \iso & (c\ast a_1)\ast a_2
\end{array}
\end{equation}
\sssec{} 
\label{Sect_C.1.2}
Consider the category $\Fun_{A\times A^{rm}-rlax}(A, C)$. Its object is a functor $f: A\to C$ together with functorial maps for $a_1, a_2\in A$ 
\begin{equation}
\label{map_for_lax_central_objects_a_1_a_2}
a_1\ast f(a_2)\to f(a_1a_2)\;\;\;\;\mbox{and}\;\;\;\; f(a_1)\ast a_2\to f(a_1a_2)
\end{equation}
together with a coherent system of higher compatibilities.

 Define $\Cent^{lax,\ra}_A(C)\subset \Fun_{A\times A^{rm}-rlax}(A, C)$ as the full subcategory of those $f$ such that for any $a_1, a_2\in A$ the map
$f(a_1)\ast a_2\to f(a_1a_2)$ is an isomorphism. That is, we require that for the right $A$-action on $C$ the corresponding lax structure on $f$ be strict.

 We then may see $f\in \Cent^{lax,\ra}_A(C)$ as the object $c=f(1)\in C$ together with 
morphism $\gamma_a: a\ast c\to c\ast a$ functorial in $a\in A$ and obtained from the above map $a\ast f(1)\to f(a)\,\iso\, f(1)\ast a$. The maps $\gamma_a$ then satisfy a coherent system of higher compatibilities. In particular, for $a_1, a_2\in A$ we are given a commutativity datum for the diagram (\ref{diag_for_lax_central_object}). 

\sssec{} 
\label{Section_def_of_lax_central_objects_left_direction}
Let $\Cent^{lax,\la}_A(C)\subset \Fun_{A\times A^{rm}-rlax}(A, C)$ be the full subcategory of those $f$ such that for any $a_1, a_2\in A$ the map $a_1\ast f(a_2)\to f(a_1a_2)$ is an isomorphism. 

 We may see $f\in \Cent^{lax,\la}_A(C)$ as the object $c=f(1)\in C$ together with morphisms $\gamma_a: c\ast a\to a\ast c$ functorial in $a$ and obtained from $f(1)\ast a\to f(a)\,\iso\, a\ast f(1)$. They satisfy a similar coherent system of higher compatibilities.
 
\begin{Rem} Let $A\in Alg(\DGCat_{cont})$ be rigid, $C\in A\otimes A^{rm}-mod(\DGCat_{cont})$. Then in the situation of Section~\ref{Sect_C.1.2} the natural map $\Fun_{A\otimes A^{rm}}(A, C)\to\Cent^{lax,\ra}_A(C)$ is an equivalence. Indeed, $A\otimes A^{rm}$ is also rigid. 
\end{Rem} 

\ssec{Case of lax actions}

\sssec{} For the notion of an $\infty$-operad in the model independent setting we refer to (\cite{Ly}, 3.0.1). 

\sssec{} 
\label{Sect_1.2.1_lax-appendix}
Recall the $\infty$-operads $\Assoc^{\otimes}, \cL\cM^{\otimes}$ defined in (\cite{HA}, 4.1.1.3 and 4.2.1.7). Following the notations of (\cite{HA}, 4.2.1.11), given a map\footnote{We do not require this map to be a `fibration of $\infty$-operads' as in (\cite{HA}, 4.2.1.11), because in the model-independent setting any map of $\infty$-operads should be considered a fibration of $\infty$-operads in the sense of quasi-categories.} of $\infty$-operads $\cE^{\otimes}\to \cL\cM^{\otimes}$, we have the fibres $\cE^{\otimes}_{\ga}=\cE^{\otimes}\times_{\cL\cM^{\otimes}} \Assoc^{\otimes}$ and $\cE^{\otimes}_{\gm}=\cE^{\otimes}_{\cL\cM^{\otimes}} \{\gm\}$. In the terminologie of (\cite{HA}, 4.2.1.12), $\cE^{\otimes}_{\gm}$ is weakly enriched over $\cE^{\otimes}_{\ga}$.
 
 Let $A\in Alg(1-\Cat)$, let $M\in 1-\Cat$ be equipped with a lax action of $A$ on the left. This pair gives rise to a map of $\infty$-operads $\cE^{\otimes}\to \cL\cM^{\otimes}$ with $\cE^{\otimes}_{\ga}\to \Assoc^{\otimes}$ the $\infty$-operad classifying $A$ and $M\,\iso\, \cE^{\otimes}_{\gm}$. So, $M$ is weakly enriched over $A$. 
 
  Let us write $\cE^{\otimes}_M=\cE^{\otimes}$ if we need to express the dependence of $\cE^{\otimes}$ on $M$ (this is ambigouous, as $\cE^{\otimes}$
also depends on $A$ and on the lax action of $A$ on $M$).

\sssec{} Let us explain the construction of $\cE^{\otimes}$ informally first. Recall that in the notations of (\cite{HA}, 4.2.1.6) objects of $\cL\cM^{\otimes}$ are pairs $(\<n\>, S)$, where $\<n\>\in \cFin_*$, and $S\subset \<n\>^0$ is a subset. A morphism in $\cL\cM^{\otimes}$ from $(\<n\>, S)$ to $(\<n'\>, S')$ is a map $\alpha: \<n\>\to \<n'\>$ in $\Assoc^{\otimes}$ such that 
\begin{itemize}
\item $\alpha$ sends $S\cup *$ to $S'\cup *$;
\item if $s'\in S'$ then $\alpha^{-1}(s')$ contains exactly one element of $S$ and that element is maximal with respect to the linear order on $\alpha^{-1}(s')$. 
\end{itemize}
The fibre $\cE^{\otimes}_{(\<n\>, S)}$ of $\cE^{\otimes}$ over $(\<n\>, S)\in \cL\cM^{\otimes}$ is 
$$
(\prod_{k\in \<n\>^0-S} A)\times (\prod_{k\in S} M).
$$ 

 Let $\alpha: (\<n\>, S)\to (\<n'\>, S')$ be a morphism in $\cL\cM^{\otimes}$, pick $X\in \cE^{\otimes}_{(\<n\>, S)}, X'\in \cE^{\otimes}_{(\<n'\>, S')}$. As in \cite{HA}, we write 
$$
X=X_1\oplus\ldots\oplus X_n\;\;\mbox{and}\;\; X'=X'_1\oplus\ldots\oplus X'_{n'}
$$ 
for the components of $X$ and $X'$ respectively. The union of the connected components of $\Map_{\cE^{\otimes}}(X, Y)$ lying over $\alpha$ is denoted by $\Map_{\cE^{\otimes}}^{\alpha}(X, X')$. 
 
 One has canonically
$$
\Map_{\cE^{\otimes}}^{\alpha}(X, X')\,\iso\, \underset{1\le k\le n}{\prod} \Mul_{\cE}(\{X_i\}_{i\in \alpha^{-1}(k)}, X'_k),
$$
in $Spc$, where $\Mul_{\cE}(\{X_i\}_{i\in \alpha^{-1}(k)}, X'_k)\in\Spc$ is defined as follows. If $\alpha^{-1}(k)=\{j_1,\ldots, j_r\}$ with its linear order then 
$$
\Mul_{\cE}(\{X_i\}_{i\in \alpha^{-1}(k)}, X'_k)=\left\{
\begin{array}{ll}
\Map_C(X_{j_1}\ast(X_{j_2}\ldots \ast(X_{j_{r-1}}\ast X_{j_r}))), X'_k), & k\in S\\ \\
\Map_A(X_{j_1}\ast(X_{j_2}\ldots \ast(X_{j_{r-1}}\ast X_{j_r}))), X'_k) & k\notin S
\end{array}
\right.
$$

 Here $\ast$ stands for both the product in $A$ and the action map in $C$. The composition of morphisms uses the right lax structure on the action map $\ast: A\times C\to C$. 

 The rigorous construction of $\cE^{\otimes}$ is given in Section~\ref{Sect_1.3.6_now_lax-appendix} below.

\sssec{} 
\label{Sect_1.2.3_lax-appendix}
Assume now $M, N\in 1-\Cat$ are equipped with a lax action of $A$ on the left. We get the corresponding maps of $\infty$-operads $\cE^{\otimes}_M\to \cL\cM^{\otimes}$ and $\cE^{\otimes}_N\to \cL\cM^{\otimes}$. We have canonically $\cE^{\otimes}_{M, \ga}\,\iso\, \cE^{\otimes}_{N, \ga}$ over $\Assoc^{\otimes}$. 

 Following \cite{HA}, we denote by $Alg_{\cE_M/\cL\cM}(\cE_N)$ the full subcategory of $\Fun_{\cL\cM^{\otimes}}(\cE_M^{\otimes}, \cE_N^{\otimes})$ spanned by maps of $\infty$-operads. We similarly have the full subcategory
$$
Alg_{\; \cE_{M,\ga}/\Assoc}(\cE_{N,\ga})\subset \Fun_{\Assoc^{\otimes}}(\cE^{\otimes}_{M, \ga}, \cE^{\otimes}_{N, \ga}).
$$ 
 
  Define the category $\Fun_{A-rlax}(M, N)$ of right lax functors $M\to N$ of lax $A$-module categories as 
$$
\{\id\}\times_{Alg_{\; \cE_{M,\ga}/\Assoc}(\cE_{N,\ga})} Alg_{\cE_M/\cL\cM}(\cE_N) 
$$

 As in the case of strict actions, for $f\in \Fun_{A-rlax}(M, N)$ we will define the map (\ref{map_lax_for_f_for_A-modules}) in $N$ functorial in $a\in A, m\in M$.
We define $\Fun_A(M, N)\subset \Fun_{A-rlax}(M, N)$ as the full subcategory of those functors for which all the maps (\ref{map_lax_for_f_for_A-modules}) are isomorphisms.

\sssec{} Let now $C\in 1-\Cat$ be equipped with commuting left and right lax actions of $A$. We spell this structure as a lax action of $A\times A^{rm}$ on $C$ on the left. 

 In this case the category $\Fun_{A\times A^{rm}}(A, C)$ plays the role of the \select{category of central objects} for these actions.
 
 Consider now the category $\Fun_{A\times A^{rm}-rlax}(A, C)$. For $f\in \Fun_{A\times A^{rm}-rlax}(A, C)$ we have natural maps (\ref{map_for_lax_central_objects_a_1_a_2}) in $M$ functorial in $a_1, a_2\in A$. 
 
 Define $\Cent^{lax, \ra}_A(C)\subset \Fun_{A\times A^{rm}-rlax}(A, C)$ as the full subcategory of those $f$ such that for any $a_1, a_2\in A$ the map $f(a_1)\ast a_2\to f(a_1a_2)$ is an isomorphism. 
 
 One may also define $\Cent^{lax, \la}_A(C)$ along the lines of Section~\ref{Section_def_of_lax_central_objects_left_direction}.  
 
 We may think of an object of $\Cent^{lax, \ra}_A(C)$ as $c\in C$ together with a collection of morphisms $\gamma_a: a\ast c\to c\ast a$ functorial in $a\in A$ together with a commutativity datum for the diagram
\begin{equation}
\label{diagram_Sect_C.2.5}
\begin{array}{ccc}
a_1\ast(a_2\ast c) & \toup{\gamma_{a_2}} & a_1\ast (c\ast a_2)\\
\downarrow && \downarrow\lefteqn{\scriptstyle \wr}\\
(a_1a_2)\ast c && (a_1\ast c)\ast a_2\\
\downarrow\lefteqn{\scriptstyle\gamma_{a_1a_2}} && \downarrow\lefteqn{\scriptstyle\gamma_{a_1}}\\
c\ast (a_1a_2) & \gets & (c\ast a_1)\ast a_2
\end{array}
\end{equation}
along with a coherent system of higher compatibilities. 
 
\sssec{} 
\label{Sect_C.2.6}
Note that if $M_1, M_2, M_3\in 1-\Cat$ are equipped with a lax action of $A$ on the left then we have the natural composition map $\Fun_{A-rlax}(M_2,M_3)\times\Fun_{A-rlax}(M_1, M_2)\to \Fun_{A-rlax}(M_1, M_3)$. 

\sssec{} 
\label{Sect_C.2.7}
If $E, E'\in A-mod(1-\Cat)$ let $L: E\leftrightarrows E': R$ be an adjoint pair, where both $L, R$ are $A$-linear functors. Let $C$ be equpped with a lax action of $A$ on the left. This gives an adjoint pair $\Fun_{A-rlax}(E, C)\leftrightarrows \Fun_{A-rlax}(E', C)$ obtained by composing with $L$ and $R$.

\ssec{Formal construction}
\sssec{} 
\label{Sect_1.3.2_special_case_lax-appendix}
Let $A, B\in Alg(1-\Cat)$ and $C\in B-mod(1-\Cat)$. Assume given a right lax non-unital monoidal functor $f: A\to B$. This yields a lax action of $A$ on $C$ on the left, namely $a\in A$ acts on $c\in C$ as $f(a)\ast c$. 

 In this case we construct below the $\infty$-operad $\cE^{\otimes}\to \cL\cM^{\otimes}$ classifying the lax $A$-action on $C$. 
 
 Let $\cV^{\otimes}\to \cL\cM^{\otimes}$ be the cocartesian fibraion of $\infty$-operads classifying the $B$-module $C$, so $\cV^{\otimes}_{\ga}\to \Assoc^{\otimes}$ is the cocartesian fibration of $\infty$-operads attached to $B$. Let $\cE^{\otimes}_{\ga}\to \Assoc^{\otimes}$ be the cocartesian fibration of $\infty$-operads classifying $A$. Let $\bar f: \cE^{\otimes}_{\ga}\to \cV^{\otimes}_{\ga}$ be the morphism of $\infty$-operads over $\Assoc^{\otimes}$ classifying $f$. 

 We have a diagram of maps of $\infty$-operads
$$
\begin{array}{ccc}
\cL\cM^{\otimes} &\toup{p}& \Assoc^{\otimes} \\
& \nwarrow &\uparrow\lefteqn{\scriptstyle\id}\\
&& \Assoc^{\otimes},
\end{array}
$$
where $p$ is defined as follows. We may think of an object of $\cL\cM^{\otimes}$ as a pointed finite set $(*\subset I)$ together with a subset $S\subset I-*$. Then $p$ sends $(*\in I, S)$ to $(*\in J)$, where $J$ is
the push-out of the diagram $*\gets (S\sqcup *)\to I$ in the category of sets, and $*\in J$ is the image of $*\to I\to J$. Given a map $\alpha: (*\in I, S)\to (*\in I', S')$ in $\cL\cM^{\otimes}$ let $\bar\alpha: (*\in J)\to (*\in J')$ be obtained by functoriality of the construction. For $j'\in J'-*$ there is a unique $i'\in I'-S'$ such that $\bar\alpha^{-1}(j')=\alpha^{-1}(i')$, so that $\bar\alpha^{-1}(j')$ is equipped with a linear order. Thus, $\bar\alpha$ is a map in $\Assoc^{\otimes}$.    

 For $(*\in I, S)\in \cL\cM^{\otimes}$ we have the evident projection $\xi: (*\in I, S)\to (*\in J, \emptyset)$, which is a morphism in $\cL\cM^{\otimes}$. These maps are functorial in $(*\in I, S)\in\cL\cM^{\otimes}$, so define a natural transformation $\xi: \id\to p$ in $\Fun(\cL\cM^{\otimes}, \cL\cM^{\otimes})$. 
 
 Write $\bar\cV^{\otimes}_{\ga}$ (resp., $\bar\cE^{\otimes}_{\ga}$) for the base change of $\cV^{\otimes}_{\ga}\to \Assoc^{\otimes}$ (resp., of $\cE^{\otimes}_{\ga}\to \Assoc^{\otimes}$) by $p$. 
 
  The natural transformation $\xi$ yields by functoriality a morphism $\cV^{\otimes}\to \bar\cV^{\otimes}_{\ga}$ of $\infty$-operads over $\cL\cM^{\otimes}$, this is the map induced by $\xi$ between the pullbacks of $\cV^{\otimes}$ along $\id$ and $\cL\cM^{\otimes}\toup{p} \Assoc^{\otimes}\hook{}\cL\cM^{\otimes}$ respectively. 
  
  The map $\bar f$ yields via base change along $p$ a morphism $\tilde f: \bar\cE^{\otimes}_{\ga}\to \bar\cV^{\otimes}_{\ga}$ of $\infty$-operads over $\cL\cM^{\otimes}$. Finally, we set 
$$
\cE^{\otimes}=\bar\cE^{\otimes}_{\ga}\times_{\bar\cV^{\otimes}_{\ga}} \cV^{\otimes},
$$
the product taken in the category of $\infty$-operads over $\cL\cM^{\otimes}$ (equivalently, by Lemma~\ref{Lm_fibred_products_of_operads}, in the category $1-\Cat_{/\cL\cM^{\otimes}}$).   
 
\sssec{} If $\cO^{\otimes}\to\cFin_*$ is an $\infty$-operad, we denote by $Op^{\infty}_{/\cO^{\otimes}}$ the 1-full subcategory of $1-\Cat_{/\cO^{\otimes}}$, where we keep only those objects $\cO'^{\otimes}\to \cO^{\otimes}$ which are maps of $\infty$-operads, and only those morphismes, which are maps of $\infty$-operads over $\cO^{\otimes}$. 
 
\begin{Lm} 
\label{Lm_fibred_products_of_operads}
The inclusion $Op^{\infty}_{/\cO^{\otimes}}\subset 1-\Cat_{/\cO^{\otimes}}$ is stable under the fibred product. So, $Op^{\infty}_{\cO^{\otimes}}$ admits fibred products, and the latter inclusion preserves fibred products.
\end{Lm}
\begin{proof} This is (\cite{Ly}, 3.0.9).
\end{proof}

\sssec{} Let $\alpha: (\<2\>, \{2\})\to (\<1\>, \{1\})$ be the map in $\cL\cM^{\otimes}$ given by $\alpha(2)=\alpha(1)=1$. Then $p(\alpha)$ is the unique map $\<1\>\to \<0\>$ in $\Assoc^{\otimes}$. 

\begin{Lm}
\label{Lm_cocartesian_arrow_lax-appendix}
 In the setting of Section~\ref{Sect_1.3.2_special_case_lax-appendix} for any $z\in \cE^{\otimes}$ over $(\<2\>, \{2\})$ there is a cocartesian arrow $z\to z'$ over $\alpha: (\<2\>, \{2\})\to (\<1\>, \{1\})$. If $z=a\oplus c\in A\times C\,\iso\, \cE^{\otimes}_{(\<2\>, \{2\})}$ then this cocartesian arrow is a map 
$$
a\oplus c\to f(a)\ast c
$$ 
giving the lax action map of $A$ on $C$. 
\end{Lm}
\begin{proof}
Our $z$ is given by a pair $f(a)\oplus c\in B\times C\,\iso\, \cV^{\otimes}_{(\<2\>, \{2\})}$ and $a\in A\,\iso\,(\bar\cE^{\otimes}_{\ga})_{(\<2\>, \{2\})}$. 

 Let $\beta: f(a)\oplus c\to f(a)\ast c$ be the arrow in $\cV^{\otimes}$ cocartesian over $\alpha$. The image of $\beta$ in $\bar\cV^{\otimes}_{\ga}$ is the arrow $f(a)\to *$. We have the arrow $\gamma: a\to *$ in $\bar\cE^{\otimes}_{\ga}$ over the arrow $f(a)\to *$ in $\bar\cV^{\otimes}_{\ga}$. So, $(\gamma, \beta): z\to z'$ is an arrow in $\cE^{\otimes}$ over $\alpha$. This $(\gamma, \beta)$ is an arrow in $\cE^{\otimes}$ cocartesian over $\alpha$ by (\cite{HTT}, 2.4.1.3).
 
  Indeed, it suffices to show that $\gamma$ is cocartesian with respect to $\tilde f$. Let $(\<n\>, S)\in \cL\cM^{\otimes}$ and $x$ an element of the fibre $(\bar\cE^{\otimes}_{\ga})_{(\<n\>, S)}$. Let $\tau: a\to x$ be a morphism in $\bar\cE^{\otimes}_{\ga}$ over the composition $(\<2\>, \{2\})\toup{\alpha}(\<1\>, \{1\})\to (\<n\>, S)$. Then $\tau$ comes from a unique morphism $*\to x$ in $\bar\cE^{\otimes}_{\ga}$ over $(\<1\>, \{1\})\to (\<n\>, S)$.
\end{proof}

\sssec{} 
\label{Sect_1.3.6_now_lax-appendix}
Assume now we are in the setting of Section~\ref{Sect_1.2.1_lax-appendix}. To define $\cE^{\otimes}$ formally, we take $B=\End(C)$ and apply the construction of Section~\ref{Sect_1.3.2_special_case_lax-appendix} to the right-lax non-unital monoidal functor $f: A\to B$ defining the right lax action of $A$ on $C$. 

\sssec{} Let us now be in the setting of Section~\ref{Sect_1.2.3_lax-appendix}. Then for $f\in \Fun_{A-rlax}(M, N)$ the map (\ref{map_lax_for_f_for_A-modules}) is formally obtained as follows.

 Given $a\in A, m\in M$ the cocartesian arrow $a\oplus m\to a\ast m$ in $\cE^{\otimes}_M$ over $\alpha$ is sent by $f: \cE^{\otimes}_M\to \cE^{\otimes}_N$ to some arrow $a\oplus f(m)\to f(a\ast m)$. So, the cocartesian arrow $a\oplus f(m)\to a\ast f(m)$ over $\alpha$ provided by Lemma~\ref{Lm_cocartesian_arrow_lax-appendix} yields the desired morphism (\ref{map_lax_for_f_for_A-modules}). 

\ssec{Additional results}

\sssec{} We have the projection $\tau: \cL\cM^{\otimes}\to \Assoc^{\otimes}$ defined in (\cite{HA}, 4.2.1.9). Namely, it sends $(\<n\>, S)$ to $\<n\>$. It sends a morphism $\alpha: (\<n\>, S)\to (\<n'\>, S')$ to the underlying map $\alpha: \<n\>\to \<n'\>$ in $\Assoc^{\otimes}$. 

 Let $A\in Alg(1-\Cat)$, let $\cE^{\otimes}_{\ga}\to \Assoc^{\otimes}$ be the $\infty$-operad classifying $A$. Then $\cE^{\otimes}:=\cE^{\otimes}_{\ga}\times_{\Assoc^{\otimes}} \cL\cM^{\otimes}$ classifies $A$ as a left $A$-module by (\cite{HA}, 4.2.1.16), here we used $\tau$ to form the fibred product. 
 
\begin{Lm} 
\label{Lm_C.4.2}
Let $A,B\in Alg(1-\Cat)$ and $f: A\to B$ be a right-lax non-unital monoidal functor. View $A$ as a left $A$-module category. Equip $B$ with a left lax $A$-action such that $a\in A$ acts on $b\in B$ as $f(a)\ast b$. Then $f\in \Fun_{A-rlax}(A, B)$. 
\end{Lm}
\begin{proof} Let $\cV^{\otimes}_{\ga}\to\Assoc^{\otimes}$ be the $\infty$-operad classifying $B$. Let $\bar f: \cE^{\otimes}_{\ga}\to \cV^{\otimes}_{\ga}$ be the map of $\infty$-operads over $\Assoc^{\otimes}$ classifying $f$. Base change by $\tau$ yields a morphism 
$$
f\times\id: \cE^{\otimes}\to \cV^{\otimes}_{\ga}\times_{\Assoc^{\otimes}} \cL\cM^{\otimes}=:\cV^{\otimes}
$$ 
of $\infty$-operads over $\cL\cM^{\otimes}$.   

  Let now $\cV^{\otimes}\to \bar\cV^{\otimes}_{\ga}$ and $\cE^{\otimes}\to \bar\cE^{\otimes}_{\ga}$ be the morphisms of $\infty$-operads over $\cL\cM^{\otimes}$ defined as in Section~\ref{Sect_1.3.2_special_case_lax-appendix}. Consider the map of $\infty$-operads $\bar\cE^{\otimes}_{\ga}\to\bar\cV^{\otimes}_{\ga}$ obtained by base change of $\bar f$ via $p: \cL\cM^{\otimes}\to\Assoc^{\otimes}$. Then the corresponding morphism of $\infty$-operads over $\cL\cM^{\otimes}$ is the natural functor $\cE^{\otimes}\to \bar\cE^{\otimes}_{\ga}\times_{\bar\cV^{\otimes}_{\ga}} \cV^{\otimes}$, where the latter fibred product is taken over $\cL\cM^{\otimes}$.
\end{proof}

\begin{Lm} 
\label{Lm_C.4.3}
Let $A\in Alg(1-\Cat)$, $C\in 1-\Cat$ equipped with a lax $A$-actiion on the left. For $c\in C$ let $h_c: A\to C$ be the functor $a\mapsto a\ast c$.
Then the functor $C\to \Fun_{A-rlax}(A, C)$, $c\mapsto h_c$ admits a right adjoint sending $h$ to $h(1)\in C$.
\end{Lm}
\begin{proof}(sketch) Given $f\in \Fun_{A-rlax}(A, C)$, $c\in C$, a map $h_c\to f$ in $\Fun_{A-rlax}(A, C)$ is a natural transformation of functors from $A\to C$, which is compatible with right-lax structure. It is essentially given by a map $\beta_a: a\ast c\to f(a)$ in $C$ functorial in $a\in A$ together with a commutativity datum (functorial in $b\in A$) for the diagram
$$
\begin{array}{ccc}
b\ast (a\ast c)& \toup{\id\ast\beta_a} & b\ast f(a)\\
\downarrow && \downarrow\\
(ba)\ast c & \toup{\beta_{ba}} & f(ba).
\end{array}
$$
Here the vertical arrows are given by our lax structures. Taking $a=1$ we get a commutative diagram
$$
\begin{array}{ccc}
b\ast c &\toup{\id\ast\beta_1}& b\ast f(1)\\
\downarrow\lefteqn{\scriptstyle \id} && \downarrow\\
b\ast c & \toup{\beta_b} & f(b)
\end{array}
$$
We see that $\beta_b$ is completely determined by $\beta_1: c\to f(1)$. So, te datum of such natural transformation is the same as the datum of a map $\beta_1: c\to f(1)$ in $C$.
\end{proof}

\sssec{} Let $A,B\in CAlg(\DGCat_{cont})$ with $A$ rigid. Let $f: A\to B$ be a morphism in $\DGCat_{cont}$, which is a right-lax non-unital symmetric monoidal functor. Let $m^R: A\to A\otimes A$ be the right adjoint to the multiplication $m: A\otimes A\to A$. One has $m^R(1)\in CAlg(A\otimes A)$, because $m^R$ is right-lax non-unital symmetric monoidal. Set $\cR=(f\otimes\id)m^R(1)\in CAlg(B\otimes A)$. We used that $f\otimes\id$ is right-lax non-unital symmetric monoidal. 

 Let $C\in B\otimes A-mod(\DGCat_{cont})$. We denote the action of $a\in A$ on $c\in C$ by $c\ast a$. View  $C\in A-mod$ as equpped with an additional left lax $A$-action, where $a\in A$ acts on $c\in C$ as $f(a)\ast c$.  
 
 The following result has appeared in (\cite{Gai19Ran}, 5.3.6) without a proof.
\begin{Pp} 
\label{Pp_Dennis_result}
There is a canonical equivalence $\Cent^{lax,\ra}_A(C)\,\iso\, \cR-mod(C)$.
\end{Pp} 
\begin{proof}
Equip $B\otimes A$ with a left lax $A\otimes A$-action, where $a_1\boxtimes a_2\in A\otimes A$ sends $b\boxtimes a\in B\otimes A$ to 
$$
(f(a_1)\ast b)\boxtimes (a\ast a_2).
$$ 
Similarly, $A\otimes A$ is a $A\otimes A$-module category naturally.  As in Lemma~\ref{Lm_C.4.2} one shows that $f\otimes\id: A\otimes A\to B\otimes A$ lies in $\Fun_{A\otimes A-rlax}(A\otimes A, B\otimes A)$. Since $m^R\in \Fun_{A\otimes A}(A, A\otimes A)$, we get by Section~\ref{Sect_C.2.6}
$$
(f\otimes\id)\comp m^R\in \Fun_{A\otimes A-rlax}(A, B\otimes A).
$$
Moreover, for $a\in A$ one has $m^R(a)\,\iso\, m^R(1)\otimes (1\boxtimes a)$ in $A\otimes A$, hence 
$$
(f\otimes\id)m^R(a)\,\iso\, ((f\otimes\id)m^R(1))\otimes (1\boxtimes a),
$$
where in the RHS of the latter formula $\otimes$ stands for the product in $B\otimes A$. This shows that $(f\otimes\id)\comp m^R\in \Cent^{lax, \ra}_A(B\otimes A)$. 

 Note that if $C_i\in B\otimes A-mod(\DGCat_{cont})$ then the composition yields a
functor 
$$
\Fun_{B\otimes A}(C_1, C_2)\times \Cent^{lax,\ra}_A(C_1)\to \Cent^{lax,\ra}_A(C_2).
$$ 
Since $\Fun_{B\otimes A}(B\otimes A, C)\,\iso\ C$ canonically, the composition with $(f\otimes\id)\comp m^R$ yields a functor
$$
L: C\to \Cent^{lax,\ra}_A(C).
$$
If $c\in C$ then $L(c)$ is the functor 
$$
A\to C, \; a\mapsto ((f\otimes\id)m^R(a))\ast c\,\iso\,(\cR\otimes (1\boxtimes a))\ast c. 
$$ 
Let $R: \Cent^{lax,\ra}_A(C)\to C$ be the functor sending $h: A\to C$ to $h(1)$.   
 
 Let us show that $(L, R)$ is an adjoint pair in $\DGCat_{cont}$. Applying  Section~\ref{Sect_C.2.7} to the adjoint pair $m: A\otimes A\leftrightarrows A: m^R$ in $A\otimes A-mod(\DGCat_{cont})$, one gets an adjoint pair 
$$
\Fun_{A\otimes A-rlax}(A\otimes A, C)\leftrightarrows \Fun_{A\otimes A-rlax}(A, C)
$$ 
Here $\Fun_{A\otimes A-rlax}(A, C)\,\iso\, \Cent^{lax,\ra}_A(C)$, because $A$ is rigid. Lemma~\ref{Lm_C.4.3} gives an adjoint pair 
$$
C\leftrightarrows \Fun_{A\otimes A-rlax}(A\otimes A, C)
$$
Composing the two, we get the adjoint pair $(L, R)$. 
 
Since $R$ is continuous and conservative, $R$ is monadic. The corresponding monad is given by the action of $\cR$.
\end{proof} 

\sssec{} The following observation goes back to (\cite{Gai19SI}, 2.7.2). 

\begin{Lm} 
\label{Lm_C.4.7_Dennis}
Let $A\in Alg(1-\Cat)$ be an ordinary category (i.e., its mapping spaces are discrete). Let $C\in 1-\Cat$ be equipped with a lax $A\times A^{rm}$-action on the left. Let $c\in C$. Assume given a morphism $\gamma_a: a\ast c\to c\ast a$ in $C$ functorial in $a\in A$. Assume that for any $n\ge 1$, $a_1,\ldots, a_n\in A$ the space
\begin{equation}
\label{space_for_Lm_C.4.7}
\Maps_C(a_1\ast (a_2\ast\ldots (a_n\ast c)\ldots), c\ast (a_1\ldots a_n))
\end{equation}
is discrete. Assume for any $a_1, a_2\in A$ the diagram (\ref{diagram_Sect_C.2.5}) commutes\footnote{This is a property, not an additional structure under other assumptions.} in the homotopy category of $C$. Then the above datum lifts uniquely to an object of $\Cent^{lax, \ra}_A(C)$.
\end{Lm}
\begin{proof} The higher compatibilities are given by $n$-morphisms for $n\ge 2$ lying in the spaces (\ref{space_for_Lm_C.4.7}), so they are all trivial. For example, for $a_1, a_2, a_3\in A$ the corresponding higher compatibility is expressed in the commutativity of the diagram
$$
\begin{array}{ccc}
a_1\ast (a_2\ast (c\ast a_3)) & \getsup{\gamma_{a_3}} & a_1\ast (a_2\ast (a_3\ast c))\\
\downarrow\lefteqn{\scriptstyle \tau_{12}} && \downarrow\lefteqn{\scriptstyle \tau_{23}}\\
(c\ast (a_1a_2))\ast a_3 && a_1\ast (c\ast (a_2a_3))\\
\downarrow && \downarrow\lefteqn{\scriptstyle \gamma_{a_1}}\\
c\ast (a_1a_2a_3) & \gets & (c\ast a_1)\ast (a_2a_3),
\end{array}
$$
where we denoted by $\tau_{ij}$ the corresponding morphism obtained as the composition in the diagram (\ref{diagram_Sect_C.2.5}) for the pair $a_i, a_j$. 
\end{proof}

\sssec{} 
\label{Sect_C.4.8}
Let $A$ be a commutative monoid in $Sets$ written multiplicatively. Recall that it is called cancellative if for $a,b,z\in A$, $az=bz$ implies $a=b$. Assume $A$ is cancellative. Let $B\in CAlg(1-\Cat)$, $C\in B-mod(1-\Cat)$. Assume $C$ is equipped with a lax $A$-action on the right commuting with the $B$-module structure. Let $f: A\to B$ be a symmetric monoidal functor. View $C$ as equipped with a $A$-action on the left via $f$. Assume for any $a\in A$ the object $f(a)$ is invertible\footnote{that is, dualizable, and its unit and counit are isomorphisms.} in $B$. 

 Make $A$ into a category $\cA$ whose objects are elements of $A$. For $a_1, a_2\in A$ a morphism $a_1\to a_2$ in $\cA$ is given by $b\in A$ with $a_2=a_1b$ (so, such a morphism is unique if it exists).

\begin{Lm} 
\label{Lm_C.4.9}
In the situation of Section~\ref{Sect_C.4.8} let $c\in\Cent^{lax,\ra}_A(C)$. Then there is a functor $h: \cA\to C$ given by $h(a)=f(a)^{\vee}\ast c\ast a$. Given a morphism $a_1\to a_2$ in $\cA$ with $a_2=ba_1$, $b\in A$, the functor $h$ sends it to the composition
\begin{multline*}
f(a_1)^{\vee}\ast c\ast a_1\toup{unit} (f(a_1)^{\vee}f(b)^{\vee}f(b))\ast c\ast a_1\to f(a_2)^{\vee}\ast (f(b)\ast c\ast a_1)\to\\ f(a_2)^{\vee}\ast ((c\ast b)\ast a_1)\to f(a_2)^{\vee}\ast c\ast a_2.
\end{multline*}
\end{Lm}
\begin{proof}
The higher compatibilities for compositions in the definition of $h$ are precisely given by the higher compatibilities for the lax central object structure on $c$. Let us illustrate this for compositions of two morphisms. Write $\tau: A\to C$ for the object of $\Cent^{lax,\ra}_A(C)$ given by $c$, so $c=\tau(1)$. 

Let $b, b', a_i\in A$ with $a_2=ba_1$ and $a_3=b'a_2$. We have canonically
\begin{multline*}
\Map_C(f(a_1)^{\vee}\ast c\ast a_1, f(a_2)^{\vee}\ast c\ast a_2)\,\iso\,\Map_C(f(a_2)\ast\tau(a_1), f(a_1)\ast \tau(a_2))\\
\iso\, \Map_C(f(b)\ast \tau(a_1), \tau(a_2)).
\end{multline*}
Similarly,
$$
\Map_C(f(a_2)^{\vee}\ast c\ast a_2, f(a_3)^{\vee}\ast c\ast a_3)
\iso\, \Map_C(f(b')\ast \tau(a_2), \tau(a_3))
$$
and
$$
\Map_C(f(a_1)^{\vee}\ast c\ast a_1, f(a_3)^{\vee}\ast c\ast a_3)
\iso\, \Map_C(f(bb')\ast \tau(a_1), \tau(a_3)).
$$
The composition
\begin{multline*}
\Map_C(f(a_2)^{\vee}\ast c\ast a_2, f(a_3)^{\vee}\ast c\ast a_3)\times
\Map_C(f(a_1)^{\vee}\ast c\ast a_1, f(a_2)^{\vee}\ast c\ast a_2)\to \\
\Map_C(f(a_1)^{\vee}\ast c\ast a_1, f(a_3)^{\vee}\ast c\ast a_3)
\end{multline*}
identifies with the composition map
\begin{multline*}
\Map_C(f(b')\ast \tau(a_2), \tau(a_3))\times \Map_C(f(b)\ast \tau(a_1), \tau(a_2))\to 
\Map_C(f(bb')\ast \tau(a_1), \tau(a_3)).
\end{multline*}
The corresponding 2-morphism is given by the right-lax structure on $\tau$, namely, the commutativity datum for the diagram
$$
\begin{array}{ccc}
f(b')\ast (f(b)\ast \tau(a_1)) & \to & f(b')\ast \tau(a_2)\\
\downarrow\lefteqn{\scriptstyle\wr}& &\downarrow\\
f(b'b)\ast \tau(a_1) &\to & \tau(a_3).
\end{array}
$$
\end{proof}

\section{Corrections for \cite{Gai19SI}}
\label{Section_corrections_for_Gai19SI}

\sssec{} The paper \select{D. Gaitsgory, The semi-infinite intersection cohomology sheaf, Adv. in Math., Volume 327 (2018), 789 - 868} has been corrected by the author after its publication, the latest corrected version is \cite{Gai19SI} the arxiv version 6 dating October 31 (2021). In this appendix, we collect for the convenience of the reader what may be  some further errata.\footnote{We thank Dennis Gaitsgory for related correspondence.}  

\iffalse 
In our opinion, it contains some mistakes, which we collect in this section for the convenience of the reader.\footnote{We have sent our corrections to D. Gaitsgory, but no more recent corrected version of \cite{Gai19SI} is publicly available to the best of our knowledge.}\fi 

\sssec{} In the 2nd displayed formula in Section 2.8.3 in the shifts both times one should remove the minus. The correct shift is $[\<\lambda, 2\check{\rho}\>]$.

\sssec{} In Sect. 3.7.2 line 5 one should replace $-\<\mu, 2\check{\rho}\>$ by
$-\<\mu, \check{\rho}\>$.

\sssec{} In the displayed square in Sect. 3.7.3 the right vertical arrow should go up and not down. 

\sssec{} In 3.4.7 it is claimed that  $\Maps(., .)$ identifies with $e$, here $\Maps$ in the internal hom in $\Vect$. This is wrong as stated. Namely, the corresponding $\Maps$ is a complex in $\Vect$ placed in degrees $\ge 0$, and its 0-th cohomology is indeed $e$. So, the corresponding space $\Map\in\Spc$ (image under Dold-Kan) is indeed discrete as desired.

\sssec{}  in Section 3.4.2: the datum of a map (3.3) is equivalent to a datum of a vector in the fibre as is claimed, but one should add: in 0-th cohomological degree of the fibre.

 It is claimed there that the fibre in question identifies with $e$. This is wrong. Only its 0-th cohomology identifies with $e$.

\sssec{} In Section 3.4.3 the first claim that
$$
\act^{-1}(\pi(t^{-\lambda}))\cap (\Bunb_N\ttimes \ov{\Gr}_G^{-\lambda})
$$
coincides with its open subset (3.4) is wrong and not needed. It is not true that this preimage is contained in a single $N(F)$-orbit. One actually needs the 0-th cohomology of the correspinding !-fibre, which is indeed identifies with $e$. 

\sssec{}  In Prop. 3.3.4 it is claimed that the isomorphism is canonical. In fact, it is not. The canonical answer is given in terms of the universal enveloping algebra $U(\check{n})$, and is given in Proposition 4.4 of Braverman, Gaitsgory, Deformations of local systems and Eis series. Similarly, in Corollary 3.3.5 the isomorphism is not canonical.

\sssec{} For Section 3.9.3: in the diagram (3.9) the right vertical map $\gq$ does not exist. (The proof can be corrected as in the present paper). 

\sssec{} In 4.3.1 line 6 replace comonad by monad. 

\sssec{} In Sect. 5.2.6 in the middle replace $t^{\lambda}\cF''$ by $t^{-\lambda}\cF''$.

\sssec{} For Sect. 5.3.3. It is claimed that for $\lambda$ dominant and regular
 $$
 \ell(t^{-w_0(\lambda)})=\ell(w_0)+\ell(t^{-\lambda}w_0),
 $$
where $\ell(.)$ is the length function on the affine extended Weyl group. This is wrong as stated. The formula for the given on p. 93 of the paper \select{Neil Chriss and Kamal Khuri-Makdisi, On the Iwahori-Hecke Algebra of a p-adic Group, IMRN 1998, No. 2}. According to that formula we have instead
$$
\ell(t^{-\lambda}w_0)=\ell(w_0)+\ell(t^{-w_0(\lambda)})
$$
So, the proof of (\cite{Gai19SI}, Theorem 5.3.1) similarly needs to be corrected.

\end{document}